%% file: main.tex
\documentclass{amsart}
\usepackage{fullpage}
\title
{Cohomological Hall algebras for $3$-Calabi--Yau categories}
\date{}
\author{Tasuki Kinjo}
\address{Research Institute for Mathematical Sciences,
Kyoto University, Kyoto 606-8502, Japan}
\email{tkinjo@kurims.kyoto-u.ac.jp}
\author{Hyeonjun Park}
\address{June E Huh Center for Mathematical Challenges, Korea Institute for Advanced Study, 85 Hoegiro, Dongdaemun-gu, Seoul 02455, Republic of Korea}
\email{hyeonjunpark@kias.re.kr}

\address{School of Mathematics, University of Edinburgh, Edinburgh, UK}
\email{p.safronov@ed.ac.uk}
\author{Pavel Safronov}

\makeatletter
 
  \@addtoreset{equation}{section}
 \makeatother

\usepackage{amsmath, amssymb, amsthm, amscd, comment, mathtools, amsfonts, thmtools, enumitem}
\usepackage{stmaryrd}
\usepackage{graphicx}
\usepackage[frame,cmtip,curve,arrow,matrix,line,graph]{xy}
\usepackage[mathscr]{euscript}
\usepackage{mathrsfs}
\usepackage{tikz}
\usepackage{tikz-cd}
\usetikzlibrary{intersections, calc}

\usepackage[utf8]{inputenc}
\usepackage[T1]{fontenc}
\usepackage[backend=biber,
            isbn=false,
            doi=false,
            maxbibnames=6,
            giveninits=true,
            style=alphabetic,
            citestyle=alphabetic]{biblatex}
\usepackage{url}
\renewbibmacro{in:}{}
\bibliography{biblio.bib}

\usepackage{xcolor}
\definecolor{e-mail}{rgb}{0,.40,.80}
\definecolor{reference}{rgb}{.20,.60,.22}
\definecolor{citation}{rgb}{0,.40,.80}

\usepackage[colorlinks=true,
            linkcolor=reference,
            citecolor=citation,
            urlcolor=e-mail, 
            breaklinks=true]{hyperref}

\input{macros}

\begin{document}

\begin{abstract}

The aim of this paper is to construct the cohomological Hall algebras for $3$-Calabi--Yau categories admitting a strong orientation data.
This can be regarded as a mathematical definition of the algebra of BPS states, whose existence was first mathematically conjectured by Kontsevich and Soibelman.
Along the way, we prove Joyce's conjecture on the functorial behaviour of the Donaldson--Thomas perverse sheaves for the attractor Lagrangian correspondence of $(-1)$-shifted symplectic stacks.
This result allows us to construct a parabolic induction map for cohomological Donaldson--Thomas invariants of $G$-local systems on $3$-manifolds for a reductive group $G$, which can be regarded as a $3$-manifold analogue of the Eisenstein series functor in the geometric Langlands program.

\end{abstract}

\maketitle

\setcounter{tocdepth}{1}
\tableofcontents

\section{Introduction}

\subsection*{Cohomological Hall algebras}

As a model of the algebra of BPS states in supersymmetric QFTs \cite{hm98} Kontsevich and Soibelman \cite{ks10} defined a cohomological Hall algebra (CoHA) associated to smooth algebras with potentials which we briefly recall now in the case the smooth algebra is a path algebra of a quiver. Let $Q$ be a quiver and $W$ a potential, which is a formal sum of cyclic paths in the quiver. Let $\fM_Q$ be the moduli stack of representations of $Q$ which admits a decomposition $\fM_Q = \coprod_{\alpha\in\bN^{Q_0}} \fM_{Q, \alpha}$ by the dimension vectors. Each $\fM_{Q, \alpha}$ is a quotient stack of a vector space by a product of general linear groups. The potential $W$ defines a function $f_W\colon \fM_Q\rightarrow \C$ given by the trace of $W$ in a given representation, and so we can consider the hypercohomology
\[\cH_{Q, W, \alpha} = H^\bullet(\fM_{Q, \alpha}, \varphi_{f_W}(\bQ_{\fM_{Q, \alpha}}[\dim \fM_{Q, \alpha}]))\]
of the sheaf of vanishing cycles with respect to $f_W$.

Let $\fM^{2-\filt}_Q$ be the moduli stack parametrizing short exact sequences $(0\rightarrow E_1\rightarrow E_2\rightarrow E_3\rightarrow 0)$ of representations of $Q$. We obtain a correspondence
\[
\xymatrix{
& \fM^{2-\filt}_Q \ar_{\ev_1\times \ev_3}[dl] \ar^{\ev_2}[dr] & \\
\fM_Q\times \fM_Q && \fM_Q
}
\]
of moduli stacks, where $\ev_i\colon \fM^{2-\filt}_Q\rightarrow \fM_Q$ is given by extracting $E_i$. If we denote by $f_W\boxplus f_W\colon \fM_Q\times \fM_Q\rightarrow \C$ the sum of the functions on the two factors, then we have $(\ev_1\times \ev_3)^\ast (f_W\boxplus f_W) = \ev_2^\ast f_W$. Moreover, $\ev_2\colon \fM^{2-\filt}_Q\rightarrow \fM_Q$ is proper on each connected component, so this correspondence defines a multiplication
\begin{equation}
\cH_{Q, W, \alpha}\otimes \cH_{Q, W, \beta}\longrightarrow \cH_{Q, W, \alpha+\beta}[\chi_Q(\alpha, \beta) - \chi_Q(\beta, \alpha)],
\label{eq:KSCoHA}
\end{equation}
where $\chi_Q(-, -)$ is the Euler form on the category of finite-dimensional representations of $Q$.

Given a quiver with potential $(Q, W)$, one can construct the corresponding Ginzburg dg algebra $\Gamma_3(Q, W)$ \cite{gin06} so that its stable dg category $\eC=\Mod_{\Gamma_3(Q, W)}^{\Perf}$ of perfect modules carries a $3$-Calabi--Yau structure. It was conjectured in \cite{ks10} that one can generalize the construction of the CoHA to other $3$-Calabi--Yau dg categories $\eC$. For this one has to solve the following two problems: 1) define the space $\cH_{\eC}$, the \emph{cohomological Donaldson--Thomas invariant} of $\eC$, and 2) define the multiplication $\cH_{\eC}\otimes\cH_{\eC}\rightarrow \cH_{\eC}$.

The first problem was solved in \cite{bbdjs15,bbbbj15,kl16}. Namely, if $\eC$ is a $3$-Calabi--Yau category of finite type, its derived moduli stack of representations $\fM_{\eC}$ in the sense of \cite{tv07} carries a $(-1)$-shifted symplectic structure \cite{bdii}. Consider an open substack $\fM\subset \fM^{}_{\eC}$ which is 1-Artin. Given an orientation of $\fM$, i.e. a square root of the canonical bundle $K_{\fM_{\eC}}|_{\fM}$, one can define a perverse sheaf $\varphi_{\fM}\in\Perv(\fM)$ locally modelled on the sheaf of vanishing cycles. Set $\Gamma = \pi_0(\fM)$ and define the cohomological DT invariant
\[\cH = \bigoplus_{\alpha\in\Gamma} H^\bullet(\fM_\alpha, \varphi_{\fM_\alpha}).\]
The monoid $\Gamma$ carries an antisymmetric Euler pairing $\chi\colon \Gamma_{}^2\rightarrow \bZ$.

The goal of this paper is to construct a multiplication on the above space. The following statement combines Corollary \ref{cor:CoHA}, Proposition \ref{prop:modulilikemoduli} and Proposition \ref{prop:CoHAcomparison}.

\begin{maintheorem}
Consider the following data:
\begin{itemize}
    \item $\eC$ is a stable dg category of finite type equipped with a $3$-Calabi--Yau structure.
    \item $\fM\subset \fM^{}_{\eC}$ is an open substack closed under direct sums and containing the zero module. Moreover, assume that the underlying classical stack $\fM^{\cl}$ is a quasi-separated Artin stack with affine stabilizers which is $\Theta$-reductive (see \cite[Definition 5.1.1]{hlp14}), such that the inclusion of the zero module $\pt\rightarrow \fM^{\cl}$ is a closed immersion.
    \item Strong orientation data for $\fM$, i.e., an orientation on the $(-1)$-shifted symplectic stack $\fM$ compatible with direct sums and satisfying associativity (see Definition \ref{defin:CYorientationdata}).
\end{itemize}
Then there is an associative multiplication
\[\cH_{\eC, \alpha}\otimes \cH_{\eC, \beta}\longrightarrow \cH_{\eC, \alpha+\beta}[\chi(\alpha, \beta)].\]
In the case $\eC=\Mod_{\Gamma_3(Q, W)}^{\Perf}$ is the $3$-Calabi--Yau category associated to a quiver with potential $(Q, W)$, this multiplication reduces to \eqref{eq:KSCoHA}.
\label{maintheorem:hallmultiplication}
\end{maintheorem}

Before we explain the construction of the multiplication, let us list some examples of $3$-Calabi--Yau categories equipped with strong orientation data. In particular, one can apply our construction of the multiplication to define a CoHA for such categories.
\begin{itemize}
    \item Given a smooth dg category $\eD$ with a $0$-th Hochschild homology class, one may consider its deformed $3$-Calabi--Yau completion $\eC$. For instance, $\eC$ can be given by the dg derived category of coherent sheaves on local curves or surfaces \cite{km21} or $\eC=\Mod_{\Gamma_3(Q, W)}^{\Perf}$ for a quiver with potential. Then $\fM = \fM^{}_{\eC}$ carries a
    strong orientation data by Proposition \ref{prop:strong_orientation_completion}.
    \item Given a closed spin $3$-manifold $M$, consider $\eC=\LocSys(M)^\omega$ the stable dg category of complexes of local systems on $M$. Let $\fM\subset \fM^{}_{\eC}$ be the substack of local systems of vector spaces. Then $\fM$ carries a strong orientation data by Proposition \ref{prop:LocSysorientation}.
\end{itemize}

For instance, let $S$ be a smooth quasi-projective surface. Let $\eC$ be the dg category of perfect complexes on the Calabi--Yau threefold $\Tot(K_S)$ and $\fM\subset \fM^{}_{\eC}$ the substack of compactly supported coherent sheaves. Combining the results of \cite{bcs20,kin21} we get $\cH_{\eC}\cong \HBM_{\vdim(\Coh(S))-\bullet}(\Coh(S))$, where $\Coh(S)$ is the stack of compactly supported coherent sheaves on $S$. We expect that in this case, the cohomological Hall algebra structure we construct in the paper reduces to the one constructed in \cite{kv19}.

Note that for any compact Calabi--Yau threefold $X$ the paper \cite{ju21} constructs an orientation on the moduli stack of coherent sheaves on $X$ compatible with direct sums. However, it is not known at present if the compatibility isomorphisms can be chosen to be associative. Thus, using our results we obtain a CoHA multiplication for any compact Calabi--Yau threefold, but we know the associativity only up to some sign.

\subsection*{Hyperbolic localization}

The construction of the CoHA multiplication goes through an intermediate step of independent interest. Given a smooth algebraic space $X$ with a $\bG_m$-action $\mu$ and a $\bG_m$-equivariant function $f\colon X\rightarrow \C$ consider the \emph{attractor correspondence}
\[
\xymatrix{
& X^{\mu, +} \ar_{\gr_\mu}[dl] \ar^{\ev_\mu}[dr] & \\
X^\mu && X,
}
\]
where $X^\mu\subset X$ is the subscheme of fixed points and $X^{\mu, +}$ is the subscheme of points $x\in X$ such that the limit $\lim_{t\rightarrow 0}\mu(t) x$ exists. The map $\ev_\mu$ is the obvious inclusion and $\gr_\mu\colon X^{\mu, +}\rightarrow X^\mu$ is given by taking the limit. Using a theorem of Braden \cite{bra03,dg14} one can show (see \cite[Proposition 5.4.1]{nak16} and \cite[Proposition 3.3]{des22} as well as Proposition \ref{prop:van_commute_hyp}) that the functor of vanishing cycles $\varphi_f$ commutes with the hyperbolic localization functor $(\gr_\mu)_*\ev_\mu^!$. In particular, if we denote by $I^\mu_X\colon X^\mu\rightarrow \Z$ the locally constant function given by assigning to a fixed point $x\in X$ its index, i.e. the difference $\rank(T_{X, x}^{+}) - \rank(T_{X, x}^{-})$ of the ranks of positive and negative weight subspaces of the tangent space, then
\[\varphi_{f|_{X^\mu}}(\bQ_{X^\mu}[\dim X^\mu])[I^\mu_X]\cong (\gr_\mu)_*\ev_\mu^!\varphi_f(\bQ_X[\dim X]).\]

The construction of the sheaf of vanishing cycles is globalized in \cite{bbdjs15,bbbbj15} to general $(-1)$-shifted symplectic stacks: given a $(-1)$-shifted symplectic stack $(X, \omega_X)$ equipped with an orientation $o$, there is a perverse sheaf $\varphi_{X, \omega_X, o}\in\Perv(X)$. Correspondingly, the hyperbolic localization statement was extended to $(-1)$-shifted symplectic algebraic spaces (and more generally to d-critical algebraic spaces) in \cite[Theorem 4.2]{des22} (see also Theorem \ref{thm:hyp_DT}) as follows. Given a $(-1)$-shifted symplectic algebraic space $(X, \omega_X)$ equipped with an orientation $o$ and a $\bG_m$-action preserving the $(-1)$-shifted symplectic structure, let $u\colon X^\mu\rightarrow X$ be the inclusion of the fixed points. Then one can pull back the $(-1)$-shifted symplectic structure and the orientation to $X^\mu$ (see Lemma \ref{lem:Lagrangiancorrespondenceretract}), so that there is an isomorphism
\[\varphi_{X^\mu, u^{\star} \omega_X, u^{\star} o}[I^\mu_{X^{\cl}}]\cong (\gr_\mu)_* \ev^!_\mu \varphi_{X, \omega_X, o}.\]

Our second main result is a construction of the hyperbolic localization isomorphism for $(-1)$-shifted symplectic Artin stacks equipped with an orientation. The attractor correspondence for a $\bG_m$-action on a stack $\fX$ can be extracted from the attractor correspondence for the \emph{trivial} $\bG_m$-action on the quotient stack $[\fX/\bG_m]$ (see Remark \ref{rmk:nontrivialactionGrad}). So, from now on we restrict to the case of trivial $\bG_m$-actions, so that the attractor correspondence becomes
\begin{equation}\label{eq:attractorcorrespondence}
\begin{aligned}
\xymatrix{
& \Filt(\fX) \ar_{\gr}[dl] \ar^{\ev}[dr] & \\
\Grad(\fX) && \fX,
}
\end{aligned}
\end{equation}
where $\Grad(\fX) = \Map(B\bG_m, \fX)$ and $\Filt(\fX) = \Map([\bA^1/\bG_m], \fX)$.

First of all, we show in Corollary \ref{cor:lagattractorcorrespondence} that if $\fX$ carries a $k$-shifted symplectic structure, then the attractor correspondence \eqref{eq:attractorcorrespondence} carries a natural $k$-shifted Lagrangian structure. For instance, when $\fX$ is the moduli stack of perfect complexes on an elliptic curve, the induced $0$-shifted Poisson structure on $\Filt(\fX)$ is related to the Feigin--Odesskii Poisson structure \cite{hp17}.

Next, there is a locally constant function $I_{\fX}\colon \Grad(\fX)\rightarrow \Z$ given by the index of the corresponding point. Let $u\colon \Grad(\fX)\rightarrow \fX$ be the map obtained by evaluation at the basepoint of $B\bG_m$. The following is Corollary \ref{cor:Joyce_conj_attractor} (see also Theorem \ref{thm:main_thm_nontrivial} for the case of nontrivial $\bG_m$-actions):

\begin{maintheorem}\label{maintheorem}
Let $\fX$ be a quasi-separated derived
Artin stack with affine stabilizers equipped with a $(-1)$-shifted symplectic structure $\omega_{\fX}$ and an orientation $o$. Then there is a natural isomorphism
\[\varphi_{\Grad(\fX), u^{\star} \omega_{\fX}, u^{\star} o }[I_{\fX^{\cl}}]\cong \gr_*\ev^! \varphi_{\fX,  \omega_{\fX}, o}.\]
\label{mainthm:hyperboliclocalization}
\end{maintheorem}

We call this isomorphism the \defterm{integral isomorphism}.
The integral isomorphism first appeared in the work of Kontsevich--Soibelman \cite[Theorem 13]{ks10} in the name of the integral identity where
they consider the hyperbolic localization of the vanishing cycle complex inside a $\bG_m$-equivariant vector bundle whose weight is $-1$, $0$ or $1$.
The motivic version  of Kontsevich--Soibelman's integral identity \cite[\S 4.4]{KS08} (see also \cite{Le15})  was very recently generalized by Bu \cite{Bu24} to the attractor correspondence of $(-1)$-shifted symplectic stacks with good moduli spaces.

\begin{rmk*}
While we were preparing this manuscript, we were informed that Pierre Descombes independently came up with the idea of using a hyperbolic localization statement for the perverse sheaf $\varphi_\fX$ to construct a cohomological Hall algebra structure. This work has recently appeared in \cite{descombes2025hyperbolic} as a part of a revised version of \cite{des22}.
\end{rmk*}

The attractor correspondence reduces to several other important correspondences in particular cases. For instance, let $\eC$ be a dg category of finite type equipped with a $3$-Calabi--Yau structure. Then by Proposition \ref{prop:gradedmoduli} the correspondence \eqref{eq:attractorcorrespondence} for $\fX=\fM_{\eC}$, the derived moduli stack of objects, is the union of the correspondences
\[
\xymatrix{
& \fM^{n-\filt}_{\eC} \ar[ld]_-{\gr} \ar[rd]^-{\ev} & \\
\fM_{\eC}^n && \fM_{\eC},
}
\]
where $\fM^{n-\filt}_{\eC}$ parametrizes sequences $x_{n-1}\rightarrow\dots\rightarrow x_0$ and $\gr\colon \fM^{n-\filt}_{\eC}\rightarrow \fM_{\eC}^n$ is given by taking the cofibres of the maps $x_k\rightarrow x_{k-1}$. Moreover, by Proposition \ref{prop:moduliattractorlagrangian} the corresponding Lagrangian structure on the attractor correspondence reduces to the one constructed in \cite{bdii} using relative Calabi--Yau structures. The hyperbolic localization isomorphism from Theorem \ref{mainthm:hyperboliclocalization} by adjunction gives rise to the multiplication on the cohomological Hall algebra in Theorem \ref{maintheorem:hallmultiplication}.

\subsection*{Cohomological DT invariants of $3$-manifolds}

Let $G$ be a connected reductive group and $M$ a closed oriented $3$-manifold. Then by Proposition \ref{prop:LocGattractor} the correspondence \eqref{eq:attractorcorrespondence} for $\fX=\Loc_G(M)$, the derived character stack of $M$, is the union of the correspondences
\[
\xymatrix{
& \Loc_P(M) \ar[dl] \ar[dr] & \\
\Loc_L(M) && \Loc_G(M),
}
\]
where $P\subset G$ are parabolic subgroups and $L$ their Levi factors. Moreover, the corresponding Lagrangian structure on the attractor correspondence reduces to the one constructed in \cite{saf17}. Given a spin structure on $M$ we construct a natural orientation of this Lagrangian correspondence in Proposition \ref{prop:characterstackorientation} following \cite{ns23}, so the hyperbolic localization isomorphism from Theorem \ref{mainthm:hyperboliclocalization} gives rise to a parabolic induction map
\[\ind_P\colon H^\bullet(\Loc_L(M), \varphi_{\Loc_L(M)})\longrightarrow H^\bullet(\Loc_G(M), \varphi_{\Loc_G(M)}).\]

The hypercohomology groups $H^\bullet(\Loc_G(M), \varphi_{\Loc_G(M)})$ appear as invariants associated to $3$-manifolds in the geometric Langlands program for generic quantum parameters. In particular, the parabolic induction map is a 3-dimensional analogue of the Eisenstein series functor for surfaces. There is a close relationship between these hypercohomology groups and skein modules of $3$-manifolds, see \cite[Conjecture D]{gs23} and related hypercohomology groups for $G=\SL_2$ were studied by Abouzaid and Manolescu \cite{am20}.

In Theorem \ref{thm:manifoldhall} we use the parabolic induction map to construct a cohomological Hall algebra structure on
\[\cH^{\GL} = \bigoplus_{n\geq 0} H^\bullet(\Loc_{\GL_n}(M), \varphi_{\Loc_{\GL_n}(M)})\]
and cohomological Hall $\cH^{\GL}$-module structures on
\[\cH^{\SO} = \bigoplus_{n\geq 0} H^\bullet(\Loc_{\SO_n}(M), \varphi_{\Loc_{\SO_n}(M)}),\qquad \cH^{\Sp} = \bigoplus_{n\geq 0} H^\bullet(\Loc_{\Sp_{2n}}(M), \varphi_{\Loc_{\Sp_{2n}}(M)}).\]

We expect that these will be useful in establishing structural results for the cohomological DT invariants, e.g. their Langlands duality.

\subsection*{A conjecture of Joyce}

Given an oriented Lagrangian correspondence between oriented $(-1)$-shifted symplectic stacks $\fX_{1} \xleftarrow{f_1} \fL \xrightarrow{f_2} \fX_2$, Joyce has conjectured (see \cite[Conjecture 1.1]{js15} and \cite[Conjecture 5.22]{ab17}) the existence of a natural map
\[
f_1^* \varphi_{\fX_1}[\vdim \fL] \longrightarrow f_2^! \varphi_{\fX_2}
\]
satisfying certain natural compatibilities. For instance, if $\fX_1=\fX_2=\pt$, it defines a virtual Lagrangian cycle $[\fL]^{\Lag}_{\vir}\in \HBM_{\vdim(\fL)}(\fL)$ which should coincide with the one constructed in \cite{bj17,ot23,pridham18}. See \cite[\S 5]{kinsurvey} for a survey on this conjecture. One of the motivations for the conjecture of Joyce is to construct a cohomological Hall algebra structure.

If $\fX$ is an oriented $(-1)$-shifted symplectic stack which is quasi-separated and has affine stabilizers, the hyperbolic localization isomorphism from Theorem \ref{mainthm:hyperboliclocalization} by adjunction provides a map
\[\gr^*\varphi_{\Grad(\fX)}[\vdim \Filt(\fX)] \longrightarrow \ev^! \varphi_{\fX}\]
thus resolving the conjecture of Joyce in this case. Note, however, that Theorem \ref{mainthm:hyperboliclocalization} is stronger than the version of the conjecture for a general Lagrangian correspondence as it asserts that the adjoint map is an isomorphism.

\subsection*{BPS states in physics}

In this section we briefly explain a connection between cohomological DT invariants considered in this paper and spaces of BPS states in supersymmetric theories. The starting point is that of 4-dimensional quantum field theories with $\cN=2$ supersymmetry. These theories have a Hilbert space $\cH$ which carries an action of the $\cN=2$ superPoincar\'e group. We will be interested in 1-particle states (corresponding to irreducible representations of the Poincar\'e group $\Spin(3, 1)\ltimes \bR^4$) satisfying the BPS condition $M=|Z|>0$, where $M\in\bR$ is the eigenvalue of a Casimir operator of the Poincar\'e group and $Z\in\mathbb{C}$ the eigenvalue of the central charge. Such representations are determined by spaces $\cH^{BPS}_Z$ equipped with an action of Wigner's little group $\Spin(3)\subset \Spin(3, 1)$. Often the Hilbert space $\cH$ carries a grading by the lattice $\Gamma$ of charges and $Z$ is determined by $\gamma$. So, one may equivalently study the spaces $\cH^{BPS}_\gamma$ of BPS states of charge $\gamma$. We refer to \cite{nei21} for an overview of the theory of BPS states.

Physically, the spaces $\cH^{BPS}_\gamma$ have the following alternative interpretation. The worldvolume theory of a particle of charge $\gamma$ is specified by $\cN=4$ supersymmetric quantum mechanics with R-symmetry group $\Spin(3)$. There is a square-zero supercharge $Q$ and the $Q$-twisted Hilbert space coincides with $\cH^{BPS}_\gamma$. The simplest example of $\cN=4$ supersymmetric quantum mechanics is given by a $\sigma$-model specified by a K\"ahler manifold $M$ and a holomorphic function $W\colon M\rightarrow \C$. In this case the $Q$-twisted Hilbert space is given by the cohomology of the sheaf $\varphi_W$ of vanishing cycles of $W$. It is explained in \cite{bfk22,sw23} that, more generally, one may define a certain $(-1)$-shifted symplectic stack $\fX$ (the moduli space of supersymmetric vacua) such that the $Q$-twisted Hilbert space of states can be modeled by $H^\bullet(\fX, \varphi_{\fX})$. Moreover, as explained in \cite{bfk22}, mass parameters in supersymmetric quantum mechanics correspond to $\bG_m$-actions $\mu$ on the $(-1)$-shifted symplectic stack $\fX$ and the $Q$-twisted Hilbert space of states with a nonzero mass corresponds to the cohomology $H^\bullet(\fX^\mu, \varphi_{\fX^\mu})$ of the fixed point locus. Certain open patches of $\fX$ may be described as moduli stacks of representations of a quiver with potential (and, consequently, the hypercohomology $H^\bullet(\fX, \varphi_{\fX})$ restricts to the cohomology of the sheaf of vanishing cycles); this is known as the BPS quiver and we refer to \cite{accerv} for many examples.

A particular class of examples of 4d $\cN=2$ supersymmetric theories is given by considering type II superstring compactifications on Calabi--Yau threefolds $M$. In this case D$p$-branes wrapping supersymmetric $p$-cycles in $M$ give rise to BPS particles in the compactified 4d theory. Mathematically, one encodes the data of branes into a dg category $\eC$ equipped with a $3$-Calabi--Yau structure ($\eC$ is the derived category of coherent sheaves on $M$ in type IIA and the Fukaya category of $M$ in type IIB) as well as a stability condition $\sigma$. The moduli stack $\cM^\sigma_\eC\subset \cM_\eC$ of $\sigma$-semistable objects is then the mathematical definition of the moduli stack of D-branes. The substack $\cM^\sigma_\gamma\subset \cM^\sigma_\eC$ of D-branes of charge $\gamma$ can be physically interpreted as the moduli space of supersymmetric vacua in the supersymmetric quantum mechanics on the worldvolume of the D-brane of charge $\gamma$. Thus, combining this observation with the description of the twisted Hilbert space in supersymmetric quantum mechanics we get that the cohomology $\cH^{BPS}_\gamma = H^\bullet(\cM^\sigma_\gamma, \varphi_{\cM^\sigma_\gamma})$ is the space of BPS states of charge $\gamma$.

\begin{ACK}
    We thank Chenjing Bu, Andr\'es Ib\'a\~nez N\'u\~nez, Ben Davison, Dominic Joyce and Yukinobu Toda for helpful discussions related to this paper.
    T.K. was supported by JSPS KAKENHI Grant Number 23K19007.
    H.P. was supported by Korea Institute for Advanced Study (SG089201).
\end{ACK}

\begin{NaC}$ $

\begin{itemize}

      \item We let $\eS$ denote the $\infty$-category of spaces (see \cite[Definition 1.2.16.1]{htt}).

    \item By $\cdga^{\leq 0}$ we denote the $\infty$-category of non-positively graded commutative differential graded $\C$-algebras (cdgas).

     \item All classical Artin stacks are assumed to be locally  of finite type over the complex number field.

     \item All derived Artin stacks are assumed to be $1$-Artin and locally of finite presentation over the complex number field.

     \item A derived stack that is $n$-geometric for some integer $n$ and is locally of finite presentation is called a derived higher Artin stack.

    \item For an Artin stack $\fX$, we let $D^b_c(\fX)$ denote the bounded derived category of sheaves with respect to the analytic topology in $\bQ$-vector spaces on $\fX$ with constructible cohomology. We let $\Perv(\fX)$ denote the abelian category of perverse sheaves on $\fX$.

    \item Given a perfect complex $E$ on a derived stack $\fX$, we denote by $\rdet(E) = \det(E|_{X^{\red}})$ the restriction of the determinant line bundle to the reduced stack.

    \item For a derived Artin stack $\fX$ with a trivial $\bG_m$-action and a $\bG_m$-equivariant quasi-coherent complex $\cF \in \QCoh(\fX)$, we let $\cF^{+}$ (resp. $\cF^{-}$) denote the direct summand with positive (resp. negative) weights.

    \item For a smooth morphism $f \colon X_1 \to X_2$, we set $K_{X_1 / X_2} \coloneqq \rdet(\Omega_{X_1 / X_2})$.

    \item
 If there is no risk of confusion, we use expressions such as $f_*$, $f_!$, and $\sHom$ for the derived functors
$Rf_*$, $Rf_!$, and $R\sHom$.
\end{itemize}

\end{NaC}

\section{Determinant line bundles}\label{sec:det}

The aim of this section is to discuss the construction and some properties of the determinant functor of perfect complexes following \cite{KM76,stv15}.

\subsection{Algebraic K-theory of stable $\infty$-categories}\label{ssec:alg_K}

Here, we briefly recall some basic facts on algebraic K-theory of stable $\infty$-categories: See \cite{bgt13} for a detailed discussion.

Let $\eC$ be a small stable $\infty$-category.
The Grothendieck group $K_0(\eC)$ is defined as a group generated by elements $[E]$ labelled by objects $E \in \eC$ and relations $[E_2] = [E_1] + [E_3]$ for fibre sequences $E_1 \to E_2 \to E_3$.
One can define a spectral refinement of $K_0(\eC)$, i.e., a spectrum $K(\eC) \in \Sp$ with a natural isomorphism $\pi_0(K(\eC)) \cong K_0(\eC)$, which is called the algebraic K-theory spectrum: See \cite[\S 9]{bgt13} for the detail.

Let $K^{\mathrm{cn}}(\eC)$ be the connective part of the algebraic K-theory spectrum.
Then by the construction of $K^{\mathrm{cn}}(\eC)$ (see \cite[\S 7.1]{bgt13}), we have the following:
    \begin{itemize}
        \item For an object $E \in \eC$, there exists a corresponding map
        \[
        [E] \in \Map(\Delta^1, \Omega^{\infty-1}K^{\mathrm{cn}}(\eC))
        \]
        with homotopy $[E] \circ \delta^0 \simeq *$, $[E] \circ \delta^1 \simeq *$ where $\delta^i$ is the $i$-th face map and $*$ is the base point.

        \item For a fibre sequence $\Delta \colon E_1 \to E_2 \to E_3$ in $\eC$, there exists a corresponding map 
        \[
        [\Delta] \in \Map(\Delta^2, \Omega^{\infty-1}K^{\mathrm{cn}}(\eC))
        \]
        with homotopy $[\Delta] \circ \delta^0 \simeq [E_3]$, $[\Delta] \circ \delta_1 \simeq [E_2]$, $[\Delta] \circ \delta_2 \simeq [E_1]$.

        \item For a fibre double sequence in $\eC$, i.e., a fibre sequence in the $\infty$-category of fibre sequences in $\eC$ depicted as
        \begin{equation}\label{eq:Gamma}\tag{$\Gamma$}
        \begin{aligned}
        \begin{tikzcd}
	\begin{array}{c}  
    \substack{
    \displaystyle{\phantom{\Delta_2'} } \\ \displaystyle{  \phantom{\rotatebox{100}{\mbox{:}}}} \\
    \displaystyle{\Delta_{12} \colon E_1}}
    \end{array} & \begin{array}{c} \substack{\displaystyle{\ \Delta_{23} } \\ \displaystyle{  \rotatebox{90}{\mbox{:}}} \\  \displaystyle{E_{2} }} \end{array} & \begin{array}{c} \substack{\displaystyle{\Delta' } \\ \displaystyle{\rotatebox{90}{\mbox{:}}} \\  \displaystyle{E_{12} }} \end{array} \\
	{\Delta_{13}:E_{1}} & {E_{3}} & {E_{13}} \\
	{ \quad \quad \, \, 0} & {E_{23}} & {E_{23},}
	\arrow[shift right=3, from=1-1, to=1-2, 
     start anchor={[xshift=-1ex]},
    end anchor={[xshift=2ex]}]
	\arrow[equal, shift left=5, from=1-1, to=2-1]
	\arrow[shift right=3, from=1-2, to=1-3,
 start anchor={[xshift=-1.65ex]},
    end anchor={[xshift=1.35ex]}]
	\arrow[from=1-2, to=2-2]
	\arrow[from=1-3, to=2-3]
	\arrow[from=2-1, to=2-2]
	\arrow[shift left=5, from=2-1, to=3-1]
	\arrow[from=2-2, to=2-3]
	\arrow[from=2-2, to=3-2]
	\arrow[from=2-3, to=3-3]
	\arrow[from=3-1, to=3-2]
	\arrow[equal, from=3-2, to=3-3]
\end{tikzcd}
\end{aligned}
\end{equation}
there exists a corresponding map
\[
[\Gamma] \in \Map(\Delta^3, \Omega^{\infty-1}K^{\mathrm{cn}}(\eC))
\]
with homotopy 
$[\Gamma] \circ \delta_0 \simeq [\Delta']$, 
$[\Gamma] \circ \delta_1 \simeq [\Delta_{23}]$,
$[\Gamma] \circ \delta_2 \simeq [\Delta_{13}]$,
$[\Gamma] \circ \delta_3 \simeq [\Delta_{12}]$.
    \end{itemize}
In particular, for each $E \in \eC$, there exists a corresponding point 
$[E] \in \Omega^{\infty}K^{\mathrm{cn}}(\eC)$
such that for a fibre sequence $\Delta \colon E_1 \to E_2 \to E_3$, there exists a natural homotopy
\begin{equation}\label{eq:fibre_seq_homotopy}
I(\Delta) \colon [E_1] \cdot [E_3] \sim [E_2]
\end{equation}
and for a  fibre double sequence \eqref{eq:Gamma},
the following homotopy commutative diagram exists
\begin{equation}\label{eq:Gamma_comm}
   \begin{aligned}
   \xymatrix@C=50pt{
       {[E_1] \cdot [E_{12}] \cdot [E_{23}]} \ar[r]^-{\sim}_-{I(\Delta_{12})} \ar[d]_-{\sim}^-{I(\Delta')}
       & {[E_{2}] \cdot [E_{23}]} \ar[d]_-{\sim}^-{I(\Delta_{23})} \\
       {[E_1] \cdot [E_{13}]} \ar[r]_-{\sim}^-{I(\Delta_{13})}
       & {[E_3].} 
     }
   \end{aligned} 
\end{equation}

\subsection{Determinant functor}\label{ssec:det}

Here we briefly recall the notion of determinant line bundle of perfect complexes following \cite{KM76,stv15}. Consider the following $\infty$-categories for $R$ a connective commutative dg $\C$-algebra:
\begin{itemize}
    \item $\Perf(R)$ is the stable $\infty$-category of perfect complexes. It has a symmetric monoidal structure given by the direct sum.
    \item $\Vect(R)\subset \Perf(R)$ is the subcategory of vector bundles. It inherits the symmetric monoidal structure.
    \item $\sPic(R)$ is the $\infty$-groupoid of $\bZ/2\Z$-graded line bundles. It carries a symmetric monoidal structure given by
    \[
    (L_1, \alpha_1) \otimes (L_2, \alpha_2) = (L_1 \otimes L_2, \alpha_1 + \alpha_2)
    \]
    with the braiding defined using the Koszul sign rule, which we call the \textit{standard symmetric monoidal structure}. It also carries another symmetric monoidal structure with the same monoidal product with the braiding of line bundle with no extra sign, which we call the \textit{naive symmetric monoidal structure}. Otherwise stated, we use the standard symmetric monoidal structure.
    \item We define the \textit{standard $\bZ / 2\bZ$-action} on $\sPic(R)$ by the monoidal dual with respect to the standard symmetric monoidal structure: $(L, \alpha) \mapsto (L, \alpha)^{\vee}$.
     We define the \textit{naive $\bZ / 2 \bZ$-action} on $\sPic(R)$ by the monoidal dual with respect to the naive monoidal structure: $(L, \alpha) \mapsto (L^{\vee}, \alpha)$.
    Unless stated otherwise, we use the standard $\bZ / 2 \bZ$-action.
    
    \item We will fix an isomorphism $(L, \alpha)^{\vee} \cong (L^{\vee}, \alpha)$ so that the 
    canonical trivializations $(L, \alpha) \otimes (L, \alpha)^{\vee} \cong (\cO_R, 0)$ and
    $(L, \alpha) \otimes (L^{\vee}, \alpha) \cong (\cO_R, 0)$ are identified.
    Under this identification,
    the natural isomorphism $(L_1^\vee, \alpha_1) \otimes (L_2^\vee, \alpha_2) \cong ((L_1\otimes L_2)^\vee, \alpha_1 + \alpha_2)$ is given by the braiding of line bundles with the extra sign $(-1)^{\alpha_1\alpha_2}$ and the trivialization $(L^{\vee \vee}, \alpha) \cong (L, \alpha)$ is given with the intervention of the sign $(-1)^{\alpha}$.
\end{itemize}

If there is no risk of confusion, a graded line bundle $(L, \alpha)$ is simply denoted as $L$.
When we work with the graded line bundles, we need particular care about the sign.
For example, for a graded line bundle $(L, \alpha)$,
the following diagram commutes only up to sign $(-1)^{\alpha}$:
\begin{equation}\label{eq-sign-L-dimension}
\begin{tikzcd}
	{L^{\vee \vee} \otimes L^{\vee}} & {L \otimes L^{\vee}} & {L^{\vee} \otimes L} \\
	{(L \otimes L^{\vee})^{\vee}} & {(L^{\vee} \otimes L)^{\vee}} & {\mathcal{O}_{R}.}
	\arrow["\cong", from=1-1, to=1-2]
	\arrow["\cong"', from=1-1, to=2-1]
	\arrow["\cong", from=1-2, to=1-3]
	\arrow["\cong", from=1-3, to=2-3]
	\arrow["\cong"', from=2-1, to=2-2]
	\arrow["\cong"', from=2-2, to=2-3]
\end{tikzcd}
\end{equation}
Equivalently, the trace of the identity map $\id_L$ is given by $(-1)^{\alpha}$.

The following is shown in \cite{KM76}.

\begin{prop}
Let $R$ be a (discrete) commutative $\C$-algebra. Then there is a symmetric monoidal functor
\[\det\colon \Vect(R)^{\simeq}\longrightarrow \sPic(R),\]
natural in $R$ and compatible with the $\Z/2\Z$-actions on the left side given by passing to the dual vector bundle and on the right side the naive $\bZ / 2 \bZ$-action.
\end{prop}

As shown in \cite[Section 2]{Heleodoro} and \cite[Appendix A]{EHKSY} the derived stacks $\Vect(-)$ and $\sPic(-)$ are left Kan extended from discrete commutative algebras, so we get a symmetric monoidal functor of $\infty$-categories $\det\colon \Vect(R)^{\simeq}\rightarrow \sPic(R)$ defined for any connective cdga $R$. Moreover, as $\sPic(R)$ is grouplike, it factors through algebraic $K$-theory:
\[\Vect(R)^{\simeq}\longrightarrow K(\Vect(R))\longrightarrow \sPic(R).\]

Consider the commutative diagram
\[
\xymatrix{
\Perf(R)^{\simeq} \ar[r] & K(\Perf(R)) \\
\Vect(R)^{\simeq} \ar[r] \ar[u] & K(\Vect(R)) \ar[u] \ar^{\det}[r] & \sPic(R).
}
\]
The vertical map on the right is an equivalence by the results of \cite[Corollary 8.1.3]{HebestreitSteimle}, \cite[Corollary 1.40]{Heleodoro} and \cite{Fontes}, so in total we obtain the \defterm{determinant map}
\[\det\colon \Perf(R)^{\simeq}\longrightarrow \sPic(R).\]
For a derived stack $\fX$ both $\Perf(\fX)$ and $\sPic(\fX)$ are defined by a right Kan extension from their values on derived affine schemes, so we obtain the determinant map
\[\det\colon \Perf(\fX)^{\simeq}\longrightarrow \sPic(\fX).\]

We will often be interested in restrictions of determinant line bundles to reduced stacks, so for a derived stack $\fX$ and a perfect complex $E\in\Perf(\fX)$ we introduce the notation
\[\rdet(E) \coloneqq \det(E|_{\fX^{\red}}).\]

We now prove some properties of the determinant functor for perfect complexes on derived Artin stacks $\fX$.
Firstly, for a fibre sequence
$\Delta \colon E_1 \to E_2 \to E_3$,
the homotopy $I(\Delta)$ in \eqref{eq:fibre_seq_homotopy} induces a natural isomorphism
\begin{equation}\label{eq:fibre_transform}
i(\Delta) \colon \det(E_1) \otimes \det(E_3) \cong \det(E_2).
\end{equation}
Assume that we are given a fibre double sequence as in \eqref{eq:Gamma}. Then the commutativity of the diagram \eqref{eq:Gamma_comm} implies the commutativity of the following diagram
\begin{equation}\label{eq:KMpre}
   \begin{aligned}
   \xymatrix@C=50pt{
       {\det(E_1) \otimes \det(E_{12}) \otimes \det(E_{23})} \ar[r]^-{\cong}_-{i(\Delta_{12})\otimes \id} \ar[d]_-{\cong}^-{\id\otimes i(\Delta')}
       & {\det(E_2) \otimes \det(E_{23})} \ar[d]_-{\cong}^-{i(\Delta_{23})} \\
       {\det(E_1) \otimes \det(E_{13})} \ar[r]_-{\cong}^-{i(\Delta_{13})}
       & {\det(E_3).} 
     }
   \end{aligned} 
\end{equation}

For an object $E \in \Perf(\fX)$, consider the fibre sequence
$\Delta_E \colon E \to 0\to E[1]$ given by the rotation of $E \xrightarrow{\id} E \to 0$.
In other words, $\Delta_E$ classifies the map $- \id_E$.
It induces an isomorphism
\[
i(\Delta_E) \colon \det(E) \otimes \det(E[1]) \cong \cO_{\fX}.
\]
We define a natural map
\begin{equation}\label{eq:det_shift}
    \theta_E \colon \det(E[1]) \cong \det(E)^\vee
\end{equation}
so that the following diagram commutes:
\[
\xymatrix@C=60pt@R=30pt{
\det(E) \otimes \det(E[1]) \ar[d]^-{(\id, \theta_E)}_-{\cong}  \ar[r]_-{\cong}^-{i(\Delta_E)} 
& \cO_{\fX}  \\
\det(E) \otimes \det(E)^{\vee}. \ar[ru]_-{\cong}
& {}
}
\]
For $n \in \bZ_{> 0}$, we inductively construct an isomorphism
\[
\theta_{E, n} \colon \det(E[n]) \cong \det(E)^{\vee^{(n)}}
\]
where $(-)^{\vee^{(n)}}$ denotes the $n$-th iterated dual, by the following composition
\[
\det(E[n]) = \det(E[1][n - 1]) \xrightarrow[\cong]{\theta_{E[1], n-1}} \det(E[1])^{\vee^{(n - 1)}} \xrightarrow[\cong]{\theta_{E}}  \det(E)^{\vee^{(n)}}.
\] 
 Since the $n$-th iterated dual only depends on the parity of $n$ up to the canonical isomorphism, we can define $(-)^{\vee^{(n)}}$ for all integer $n \in \bZ$.
We define $\theta_{E, -n} \colon \det(E[-n]) \cong \det(E)^{\vee^{(-n)}}$ by $(\theta_{E[-n], n}^{\vee^{(-n)}})^{-1}$.
By construction, for $n, m \in \bZ$, the following diagram commutes:
\[
\xymatrix@C=50pt{
{\det(E[n + m])} \ar@{=}[d] \ar[r]_-{\cong}^-{\theta_{E[n], m}}
& {\det(E[n])^{\vee^{(m)}}} \ar[d]_-{\cong}^-{\theta_{E, n}} \\
{\det(E[n + m])} \ar[r]_-{\cong}^-{\theta_{E, n + m}}
& {\det(E)^{\vee^{(n + m)}}.}
}
\]
Also, it is clear from the construction that the following diagram commutes:
\begin{equation}\label{eq:-1and1_compatible}
\begin{aligned}
\xymatrix@C=50pt{
{\det(E[-1]) \otimes \det(E)} \ar[r]_-{\cong}^-{\id \otimes \theta_{E[-1]}} \ar[d]_-{\cong}^-{\theta_{E, -1} \otimes \id}
& {\det(E[-1]) \otimes \det(E[-1])^{\vee}} \ar[d]_-{\cong} \\
{\det(E)^{\vee} \otimes \det(E)} \ar[r]^-{\cong}
& {\cO_{\fX}.}
}
\end{aligned}
\end{equation}

\begin{lem}\label{lem:rotate}
    Let $\Delta \colon E_1 \xrightarrow{f} E_2 \xrightarrow{g} E_3$ be a fibre sequence in $\Perf(\fX)$.
    Let $\rho(\Delta) \colon E_2 \xrightarrow{g} E_3 \xrightarrow{h} E_1[1]$ and $\rho^{-1}(\Delta) \colon E_3[-1] \xrightarrow{-h[-1]} E_1 \xrightarrow{f} E_2$ be rotated fibre sequences.
    Then the following diagrams commute:
    \begin{equation}\label{eq:rotate_det1}
    \begin{aligned}
    \xymatrix@C=50pt{
    {\det(E_1) \otimes \det(E_3) \otimes \det(E_1[1])} \ar[r]_-{\cong}^-{i(\Delta) \otimes \id} \ar[d]_-{\cong}^-{\id \otimes \theta_{E_1}}
    & {\det(E_2) \otimes \det(E_1[1]) } \ar[d]_-{\cong}^-{i(\rho(\Delta))} \\
    {\det(E_1) \otimes \det(E_3) \otimes \det(E_1)^{\vee}} \ar[r]^-{\cong}
    & \det(E_3),
    }
    \end{aligned}
 \end{equation}
    \begin{equation}\label{eq:rotate_det2}
    \begin{aligned}
    \xymatrix@C=50pt{
    {\det(E_1) \otimes \det(E_2) \otimes \det(E_1[1])} \ar[r]_-{\cong}^-{\id \otimes i(\rho(\Delta)) } \ar[d]_-{\cong}^-{\id \otimes \theta_{E_1}}
    & {\det(E_1) \otimes \det(E_3) } \ar[d]_-{\cong}^-{(-1)^{\rank E_1} \cdot i(\Delta)} \\
    {\det(E_1) \otimes \det(E_2) \otimes \det(E_1)^{\vee}} \ar[r]^-{\cong}
    & \det(E_2),
    }
    \end{aligned}
    \end{equation}
    \begin{equation}\label{eq:rotate_inverse}
        \begin{aligned}
    \xymatrix@C=50pt{
    {\det(E_3[-1]) \otimes \det(E_1) \otimes \det(E_3)} \ar[r]_-{\cong}^-{\id \otimes i(\Delta)} \ar[d]_-{\cong}^-{\theta_{E_3, -1} \otimes \id}
    & {\det(E_3[-1]) \otimes \det(E_2)} \ar[d]_-{\cong}^-{(-1)^{\rank E_3} \cdot i(\rho^{-1}(\Delta))} \\
    {\det(E_3)^{\vee} \otimes \det(E_1) \otimes \det(E_3)} \ar[r]^-{\cong} 
    & {\det(E_1).}
    }
        \end{aligned}
    \end{equation}
\end{lem}

\begin{proof}
Consider the following fibre double sequences:
\[
\xymatrix{
{E_1} \ar[r] \ar@{=}[d]
& {0} \ar[r] \ar[d]
& {E_1[1]} \ar[d]^-{\begin{psmallmatrix} 0\\ \id \end{psmallmatrix}} \\
{E_1} \ar[d] \ar[r]^-{0}
& {E_3 } \ar@{=}[d] \ar[r]^-{\begin{psmallmatrix} \id \\ 0 \end{psmallmatrix}}
& {E_3 \oplus E_1[1]} \ar[d]^-{(\id, 0)} \\
{0} \ar[r]
& {E_3} \ar@{=}[r]
& {E_3,}
}\quad \quad
\xymatrix{
{E_1} \ar[r]^-{f} \ar@{=}[d]
& {E_2} \ar[r]^-{g} \ar[d]^-{g}
& {E_3} \ar[d]^-{\begin{psmallmatrix} \id \\ -h \end{psmallmatrix}} \\
{E_1} \ar[d] \ar[r]^-{0}
& {E_3 } \ar[d]^-{h} \ar[r]^-{\begin{psmallmatrix} \id \\ 0 \end{psmallmatrix}}
& {E_3 \oplus E_1[1]} \ar[d]^-{(h, \id)} \\
{0} \ar[r]
& {E_1[1]} \ar@{=}[r]
& {E_1[1].}
}
\]
Here, the upper horizontal fibre sequence in the left diagram is $\Delta_{E_1}$ and the middle horizontal  fibre sequences in the left and right diagrams are equivalent.
By considering the commutative diagrams \eqref{eq:KMpre} associated with these diagrams and using the fact that the automorphism of the complex $E_3 \oplus E_1[1]$ given by $\begin{psmallmatrix} \id & 0 \\ h & \id \end{psmallmatrix}$ induces the identity map on the determinant line bundle, we conclude the commutativity of the diagram \eqref{eq:rotate_det1}.

To prove the commutativity of the diagram \eqref{eq:rotate_det2}, consider the following diagram:
    \[
    \xymatrix@C=50pt{
    {}
    & {\det(E_1) \otimes \det(E_1) \otimes \det(E_1[1]) \otimes \det(E_3)} \ar[d]^-{\id \otimes \theta_{E_1} \otimes \id}_-{\cong} \\
    {\det(E_1) \otimes \det(E_1) \otimes \det(E_3) \otimes \det(E_1[1])} \ar[ru]_-{\cong}^-{\id \otimes \mathrm{sw}} \ar[d]_-{\cong}^-{\id \otimes i(\Delta) \otimes \id}
    & {\det(E_1) \otimes \det(E_1) \otimes \det(E_1)^{\vee} \otimes \det(E_3)} \ar[d]_-{\cong}^-{} \\
    {\det(E_1) \otimes \det(E_2) \otimes \det(E_1[1])} \ar[r]_-{\cong}^-{\id \otimes i(\rho(\Delta)) } \ar[d]_-{\cong}^-{\id \otimes \theta_{E_1}}
    & {\det(E_1) \otimes \det(E_3) } \ar[d]_-{\cong}^-{(-1)^{\rank E_1} \cdot i(\Delta)} \\
    {\det(E_1) \otimes \det(E_2) \otimes \det(E_1)^{\vee}} \ar[r]^-{\cong}
    & \det(E_2),
    }
    \]
    where $\mathrm{sw}$ denotes the brading isomorphism.
    The commutativity of the upper diagram follows from the commutativity of the diagram \eqref{eq:rotate_det1}.
    The commutativity of the outer diagram follows from the fact that the brading isomorphism for the object $\det(E_1) \otimes \det(E_1)$ is given by multiplication of $(-1)^{\rank E_{1}}$.
    In particular, we conclude the commutativity of the diagram \eqref{eq:rotate_det2}.

    The commutativity of the diagram \eqref{eq:rotate_inverse} immediately follows from the commutativity of \eqref{eq:-1and1_compatible} and \eqref{eq:rotate_det2}.
\end{proof}

\begin{cor}
    Consider the following fibre double sequence:
    \begin{equation}
        \begin{aligned}
        \begin{tikzcd}
	\begin{array}{c}  
    \substack{
    \displaystyle{\phantom{\Delta_2'} } \\ \displaystyle{  \phantom{\rotatebox{90}{\mbox{:}}}} \\
    \displaystyle{\quad \quad \, \,  0}}
    \end{array} & \begin{array}{c} \substack{\displaystyle{\ \Delta_{23} } \\ \displaystyle{ \rotatebox{90}{\text{\textup{:}}}} \\  \displaystyle{E_{23}' }} \end{array} & \begin{array}{c} \substack{\displaystyle{\Delta' } \\ \displaystyle{\rotatebox{90}{\text{\textup{:}}}} \\  \displaystyle{E_{23}' }} \end{array} \\
	{\Delta_{12}:E_{1}} & {E_{2}} & {E_{12}} \\
	{ \Delta_{13}:E_{1}} & {E_{3}} & {E_{13}.}
	\arrow[shift right=3, from=1-1, to=1-2, 
     start anchor={[xshift=-1ex]},
    end anchor={[xshift=2ex]}]
	\arrow[ shift left=5, from=1-1, to=2-1]
	\arrow[equal, shift right=3, from=1-2, to=1-3,
 start anchor={[xshift=-1.65ex]},
    end anchor={[xshift=1.35ex]}]
	\arrow[from=1-2, to=2-2]
	\arrow[from=1-3, to=2-3]
	\arrow[from=2-1, to=2-2]
	\arrow[equal, shift left=5, from=2-1, to=3-1]
	\arrow[from=2-2, to=2-3]
	\arrow[from=2-2, to=3-2]
	\arrow[from=2-3, to=3-3]
	\arrow[from=3-1, to=3-2]
	\arrow[from=3-2, to=3-3]
\end{tikzcd}
\end{aligned}
\end{equation}
Then the following diagram commutes:
\begin{equation}\label{eq:KMpre2}
   \begin{aligned}
   \xymatrix@C=50pt{
       {\det(E_{23}') \otimes \det(E_1) \otimes \det(E_{13})} \ar[d]^-{\cong}_-{i(\Delta_{13})} \ar[r]_-{\cong}^-{i(\Delta')}
       & {\det(E_1) \otimes \det(E_{12})}        
       \ar[d]_-{\cong}^-{i(\Delta_{12})} \\
       { \det(E_{23}') \otimes \det(E_3)} 
        \ar[r]_-{\cong}^-{i(\Delta_{23})}
       & {\det(E_2).} 
     }
   \end{aligned} 
\end{equation}

\end{cor}

\begin{proof}
    Consider the following diagram:
    \[
    \xymatrix@C=50pt@R=40pt{
    {\det(E_{23}') \otimes \det(E_1) \otimes \det(E_{13})} \ar[r]_-{\cong}^-{i(\Delta')} \ar[ddd]^-{\cong}_-{i(\Delta_{13})} \ar[r]_-{\cong}^-{i(\Delta')}
    & {\det(E_1) \otimes \det(E_{12})} \\
        {}
    & {\det(E_{23}') \otimes \det(E_1) \otimes  \det(E_{12}) \otimes \det(E_{23}'[1]) } \ar[lu]^-{\cong}_-{i(\rho(\Delta'))} \ar[u]^-{\cong}_-{(-1)^{\rank E_{23}'} \cdot \theta_{E_{23}'}} \ar[d]_-{\cong}^-{i(\Delta_{12})}  \\
        {}
    & {\det(E_{23}') \otimes \det(E_2) \otimes \det(E_{23}'[1])} \ar[d]^-{(-1)^{\rank E_{23}'} \cdot \theta_{E_{23}'}}_-{\cong} \ar[ld]^-{\cong}_-{i(\rho(\Delta_{23}))} \\
        {\det(E_{23}') \otimes \det(E_3)} \ar[r]_-{\cong}^-{i(\Delta_{23})}
    & {\det(E_2).} 
    }
    \]
    The upper and lower triangles commute by the commutativity of the diagram \eqref{eq:rotate_det2}.
    The middle square commutes by \eqref{eq:KMpre}.
    Hence, we conclude.
    
\end{proof}

The following statement can be thought of as an $\infty$-categorical refinement of \cite[Theorem 1]{KM76}:

\begin{lem}\label{lem:KM}
    Let $\fX$ be a derived stack.
    Consider the fibre double sequence in $\Perf(\fX)$:
\[\begin{tikzcd}
	\begin{array}{c}  \substack{\quad \quad \displaystyle{\Delta_1' } \\ \displaystyle{\quad \ \, \, \rotatebox{90}{\text{\textup{:}}}} \\  \displaystyle{\Delta_1:E_{11} }} \end{array} & \begin{array}{c} \substack{\displaystyle{\Delta_2' } \\ \displaystyle{\rotatebox{90}{\text{\textup{:}}}} \\  \displaystyle{E_{12} }} \end{array} & \begin{array}{c} \substack{\displaystyle{\Delta_3' } \\ \displaystyle{\rotatebox{90}{\text{\textup{:}}}} \\  \displaystyle{E_{13} }} \end{array} \\
	{\Delta_2:E_{21}} & {E_{22}} & {E_{23}} \\
	{\Delta_3:E_{31}} & {E_{32}} & {E_{33}.}
	\arrow[shift right=3, from=1-1, to=1-2]
	\arrow[shift left=3, from=1-1, to=2-1]
	\arrow[shift right=3, from=1-2, to=1-3]
	\arrow[from=1-2, to=2-2]
	\arrow[from=1-3, to=2-3]
	\arrow[from=2-1, to=2-2]
	\arrow[shift left=3, from=2-1, to=3-1]
	\arrow[from=2-2, to=2-3]
	\arrow[from=2-2, to=3-2]
	\arrow[from=2-3, to=3-3]
	\arrow[from=3-1, to=3-2]
	\arrow[from=3-2, to=3-3]
\end{tikzcd}\]
Then the following diagram commutes:
\begin{equation}
\begin{aligned}\label{eq:KM}
\xymatrix@C=70pt{
{\left( \det(E_{11}) \otimes \det(E_{13}) \right) \otimes \left( \det(E_{31}) \otimes \det(E_{33}) \right)} \ar[r]^-{i(\Delta_1) \otimes i(\Delta_3)}_-{\simeq} \ar[dd]^-{\simeq}
& {\det(E_{12}) \otimes \det(E_{32})} \ar[d]^-{i(\Delta_2')}_-{\simeq} \\
{}
& {\det(E_{22})} \\
{ \left( \det(E_{11}) \otimes \det(E_{31}) \right) \otimes \left( \det(E_{13}) \otimes \det(E_{33}) \right)} \ar[r]_-{\simeq}^-{i(\Delta_1') \otimes i(\Delta_3')}
& {\det(E_{21}) \otimes \det(E_{23}).} \ar[u]_-{i(\Delta_2)}^-{\simeq}
}
\end{aligned}
\end{equation}
\end{lem}

\begin{proof}

 Let $F$ be the cofibre of the map $E_{11} \to E_{22}$. Consider the following fibre double sequences:
\begin{equation}\label{eq:three_fibre}
\begin{aligned}
\xymatrix{
{E_{11}} \ar[r] \ar@{=}[d]
& {E_{12}} \ar[r] \ar[d]
& {E_{13}} \ar[d] \\
{E_{11}} \ar[d] \ar[r]
& {E_{22}} \ar[d] \ar[r]
& {F} \ar[d] \\
{0} \ar[r]
& {E_{32}} \ar@{=}[r]
& {E_{32}}
}
\quad \quad
\xymatrix{
{E_{11}} \ar[r] \ar@{=}[d]
& {E_{21}} \ar[r] \ar[d]
& {E_{31}} \ar[d] \\
{E_{11}} \ar[d] \ar[r]
& {E_{22}} \ar[d] \ar[r]
& {F} \ar[d] \\
{0} \ar[r]
& {E_{23}} \ar@{=}[r]
& {E_{23}}
}
\quad \quad
\xymatrix{
{0} \ar[r] \ar[d]
& {E_{13}} \ar@{=}[r] \ar[d]
& {E_{13}} \ar[d] \\
{E_{31}} \ar@{=}[d] \ar[r]
& {F} \ar[r] \ar[d]
& {E_{23}} \ar[d] \\
{E_{31}} \ar[r]
& {E_{32}} \ar[r]
& {E_{33}.} 
}
\end{aligned}
\end{equation}
Consider the following diagram:
\[
\xymatrix{
{}
& {\substack{ \displaystyle{ \left( \det(E_{11}) \otimes \det(E_{13}) \right)} \\ \displaystyle{
\otimes \left( \det(E_{31}) \otimes \det(E_{33}) \right)}}}  \ar[ld]_-{\cong} \ar[rd]^-{\cong}
& {} \\
{\det(E_{11}) \otimes \left( \det(E_{13})  \otimes \det(E_{32}) \right)} \ar[r]^-{\cong} \ar[d]^-{\cong}
& {\det(E_{11}) \otimes \det(F)} \ar[d]_-{\cong}
& {\det(E_{11}) \otimes \left( \det(E_{31}) \otimes \det(E_{23}) \right)} \ar[l]_-{\cong} \ar[d]_-{\cong} \\
{\det(E_{12}) \otimes \det(E_{32})} \ar[r]^-{\cong}
& {\det(E_{22}) }
& {\det(E_{21}) \otimes \det(E_{23}).} \ar[l]_-{\cong}
}
\]
The left bottom (resp. right bottom) square commutes by the commutativity of the diagram \eqref{eq:KMpre} applied to the first (resp. second) fibre double sequence in \eqref{eq:three_fibre}.
Also, the upper triangle commutes by applying the commutativity of \eqref{eq:KMpre2} to the third fibre double sequence in \eqref{eq:three_fibre}.
Hence, we conclude.

\end{proof}

\begin{cor}\label{cor:det_shift_fibre}
    Let $\Delta \colon E_1 \to E_2 \to E_3$ be a fibre sequence. Then the following diagram commutes:
    \[
    \xymatrix@C=50pt@R=40pt{
    {\det(E_{1}[1]) \otimes \det(E_{3}[1])} \ar[r]^-{i(\Delta)}_-{\cong} \ar[d] \ar[d]_-{\cong}^-{\theta_{E_1} \otimes \theta_{E_3}}
    & { \det (E_{2}[1])}  \ar[d]_-{\cong}^-{\theta_{E_2}} \\
    {\det(E_{1})^{\vee} \otimes \det(E_{3})^{\vee}} \ar[r]^-{i(\Delta[1])}_-{\cong}
    & {\det(E_2)^{\vee}.}
    }
    \]
\end{cor}

\begin{proof}
    This follows by applying Lemma \ref{lem:KM} to the following fibre double sequence:
    \[
    \xymatrix{
    {E_1} \ar[r] \ar[d]
    & {E_2}\ar[r] \ar[d] 
    & {E_3} \ar[d] \\
        {0} \ar[r] \ar[d]
    & {0} \ar[r] \ar[d]
    & {0} \ar[d] \\
        {E_1[1]}  \ar[r]
    & {E_2[1]} \ar[r]
    & {E_3[1].} 
    }
    \]
\end{proof}

Recall by construction that the natural map
\[
\det \colon K(\Perf(\fX)) \to \Pic^{\mathrm{gr}}(\fX)
\]
is $\bZ / 2 \bZ$-equivariant, where $\bZ / 2 \bZ$-equivariant structure on $K(\Perf(\fX))$ is induced by the duality map
\[
\Perf(\fX) \to \Perf(\fX)^{\mathrm{op}}, \quad E \mapsto E^{\vee}
\]
with the canonical identification $E \simeq E^{\vee \vee}$
and similarly for $\Pic^{\mathrm{gr}}(\fX)$ (with respect to the naive monoidal structure). 
The $\bZ / 2 \bZ$-equivariant structure gives a natural isomorphism
\begin{equation}\label{eq:det_dual}
\iota_E \colon \det(E^\vee)\cong \det(E)^{\vee}.
\end{equation}
By definition, the following diagram commutes:
\begin{equation}\label{eq:double_dual}
\begin{aligned}
\xymatrix@C=80pt{
\det(E) \ar[rr]_-{\cong}^-{(-1)^{\rank E} \cdot \id} \ar[d]_-{\cong} 
& {}
& {\det(E)} \\
\det(E^{\vee \vee}) \ar[r]_-{\cong} ^-{\iota_{E^{\vee}}}
& \det(E^{\vee})^{\vee} \ar[r]_-{\cong} ^-{(\iota_{E}^{\vee})^{-1}}
& \det(E)^{\vee \vee}. \ar[u]_-{\cong}
}
\end{aligned}
\end{equation}
The appearance of the sign intervention is due to the fact that the trivializations of the double dual with respect to the standard and naive monoidal structures differ by a sign.

Let $ \Delta \colon E_1 \to E_2 \to E_3$ be a fibre sequence in $\Perf(\fX)$ and $\Delta^{\vee} \colon E_3^{\vee} \to E_{2}^{\vee} \to E_{1}^{\vee}$ be the dual fibre sequence.
By the construction of algebraic K-theory, we clearly have the following equivalence of $2$-cells
\[
\sigma([\Delta]) \simeq [\Delta^{\vee}]
\]
where $\sigma$ denotes the non-trivial element in $\bZ / 2 \bZ$.
In particular, the following diagram commutes:
\begin{equation}\label{eq:dual_fibreseq}
\begin{aligned}
    \xymatrix@C=80pt{
    \det(E_3^{\vee}) \otimes \det(E_1^{\vee}) \ar[r]_-{\cong}^-{i(\Delta^{\vee})} \ar[d]_-{\cong}^-{\iota_{E_3} \otimes \iota_{E_1}}
    & \det(E_2^{\vee}) \ar[d]_-{\cong}^-{\iota_{E_2}} \\
    \det(E_3)^{\vee} \otimes \det(E_1)^{\vee} \ar[r]_-{\cong}^-{(i(\Delta)^{\vee})^{-1}}
    & \det(E_2)^{\vee}.
    }
    \end{aligned}
\end{equation}
Applying this to the fibre sequence $\Delta_E \colon E \to 0 \to E[1]$ and using \eqref{eq-tauto-fibreseq-equiv}, we see that the following diagram commutes:
\begin{equation}
\begin{aligned}
    \xymatrix@C=80pt{
    {\det(E^{\vee}[-1])} \otimes \det(E^{\vee}) \ar[d]_-{\cong}^-{\beta_1} \ar[r]_-{\cong}^-{(-1)^{\rank E} \cdot i(\Delta_{E^{\vee}[-1]})} 
    & \cO_{\fX} \ar@{=}[dd] \\
    \det(E[1]^{\vee}) \otimes \det(E^{\vee}) \ar[d]_-{\cong}^-{\iota_{E[1]} \otimes \iota_{E}}
    & {} \\
    \det(E[1])^{\vee} \otimes \det(E)^{\vee} \ar[r]_-{\cong}^-{i(\Delta_E)^{\vee}}
    & \cO_{\fX}.
    }
    \end{aligned}
\end{equation}
The appearance of the sign is caused by the sign insertion in the definition of the map $\beta_1$ in \eqref{eq-beta-def}.
Combining this with the commutativity of the diagram \eqref{eq-sign-L-dimension} up to sign and the commutativity of \eqref{eq:-1and1_compatible}, we see that the following diagram commutes:
\begin{equation}\label{eq:dual_shift}
\begin{aligned}
    \xymatrix@C=80pt{
\det(E^{\vee}[-1]) \ar[r]_-{\cong}^-{\beta_1} \ar[d]_-{\cong}^-{\theta_{E^{\vee}, -1}}
& \det(E[1]^{\vee}) \ar[r]_-{\cong}^-{\iota_{E[1]}}
& \det(E[1])^{\vee} \ar[d]_-{\cong}^-{(\theta_{E}^{\vee})^{-1}} \\
\det(E^{\vee})^{\vee} \ar[rr]_-{\cong}^-{(\iota_{E^{\vee}}^{\vee})^{-1}}
&{}
& \det(E)^{\vee \vee}.
 }
    \end{aligned}
\end{equation}

By repeatedly using the commutativity of the above diagram, for a positive integer $n \in \bZ_{> 0}$,
we see that the following diagram commutes:
\begin{equation}\label{eq:dual_shift_n}
\begin{aligned}
    \xymatrix@C=80pt{
\det(E^{\vee}[-n]) \ar[r]_-{\cong}^-{\beta_n} \ar[d]_-{\cong}^-{\theta_{E^{\vee}, -n}}
& \det(E[n]^{\vee}) \ar[r]_-{\cong}^-{\iota_{E[n]}}
& \det(E[n])^{\vee} \ar[d]_-{\cong}^-{(\theta_{E, n}^{\vee})^{-1}} \\
\det(E^{\vee})^{\vee^{(n)}} \ar[rr]_-{\cong}^-{(\iota_{E^{\vee})^{-1}}^{\vee}}
&{}
& \det(E)^{\vee^{(n + 1)}}.
 }
    \end{aligned}
\end{equation}

By substituting $E = F[-n]$, we see that the same diagram commutes for negative integers $n$  as well.

\section{Shifted symplectic and d-critical structures}

We will recall the theory of shifted symplectic structures following \cite{ptvv13} and 
d-critical structures following \cite{joy15} which can be regarded a classical shadow of $(-1)$-shifted symplectic structures. 
These structures are a key ingredient in the construction of the Donaldson--Thomas perverse sheaves, which we will recall in the subsequent section.

\subsection{Shifted symplectic structures}

Here we briefly recall the theory of shifted symplectic structures. 
See \cite{ptvv13} for the original source, \cite{PY} for a quick survey, and \cite{cal15,ab17,chs21} for the material on Lagrangian correspondences.

Let $A$ be a connective commutative dg algebra. We define the graded mixed commutative dg algebra $\bDR(A)$ of differential forms as in \cite{ptvv13}. Its underlying graded commutative algebra is $\Sym(\bL_A[-1])$. We consider the following related objects:
\begin{itemize}
    \item The complex of closed $p$-forms
    \[\Omega^{\geq p}(A) = \left(\prod_{n\geq p} \bDR(A)(n), d + \ddr\right)[p].\]
    There are natural maps $\Omega^{\geq (p+1)}(A)\rightarrow \Omega^{\geq p}(A)[1]$.
    \item The space of $n$-shifted closed $p$-forms $\cA^{p, \cl}(A, n) = \Map(k[-n], \Omega^{\geq p}(A))$. There are natural forgetful maps $\cA^{p+1, \cl}(A, n-1)\rightarrow \cA^{p, \cl}(A, n)$.
    \item The space of $n$-shifted exact $p$-forms $\cA^{p, \exact}(A, n)$ which fits into a fibre sequence of pointed spaces
    \[\cA^{p, \exact}(A, n)\xrightarrow{\ddr} \cA^{p, \cl}(A, n)\longrightarrow \cA^{0, \cl}(A, p+n).\]
    There are natural forgetful maps $\cA^{p+1, \exact}(A, n-1)\rightarrow \cA^{p, \exact}(A, n)$.
    \item The space of $n$-shifted $p$-forms $\cA^p(A, n)\cong \Map(k[-n], \wedge^p\bL_A)$. There are natural forgetful maps
    \begin{equation}\label{eq:closedformforget}
    \cA^{p, \cl}(A, n)\rightarrow \cA^p(A, n)
    \end{equation}
\end{itemize}

\begin{prop}\label{prop:DRconvergent}
The functor $\bDR$ is convergent, i.e. for any connective commutative dg algebra $A$ the natural morphism
\[\bDR(A)\longrightarrow \lim_n \bDR(\tau^{\geq -n}(A))\]
is an isomorphism.
\end{prop}
\begin{proof}
Let $A_n = \tau^{\geq -n}(A)$ and $i_n\colon A\rightarrow A_n$ the projection. We have to show that the natural morphism
\[\Sym^p_A(\bL_A[-1])\longrightarrow \lim_n \Sym^p_{A_n}(\bL_{A_n}[-1])\]
is an isomorphism for every $p\geq 0$. This morphism factors as the composite
\[\Sym^p_A(\bL_A[-1])\longrightarrow \lim_n \Sym^p_{A_n}(i_n^*\bL_A[-1])\longrightarrow \lim_n \Sym^p_{A_n}(\bL_{A_n}[-1]).\]
Since $\bL_A$ is connective, $\Sym^p_A(\bL_A[-1])$ is bounded on the right. Therefore, the first map is an isomorphism by \cite[Proposition 19.2.1.5]{sag}. It remains to show that
\[\lim_n(\fib(\Sym^p_{A_n}(i_n^*\bL_A[-1])\longrightarrow \Sym^p_{A_n}(\bL_{A_n}[-1]))) = 0.\]
Consider an integer $N$. By construction $\cofib(i_N)$ is $(N+2)$-connective. Therefore, by \cite[Corollary 7.4.3.2]{ha} the relative cotangent complex $\bL_{A_N/A}$ is $(N+2)$-connective. Using the fibre sequence
\[\bL_{A_N/A}[-1]\longrightarrow i_N^*\bL_A\longrightarrow \bL_{A_N}\]
we obtain a filtration on $\Sym^p_{A_N}(i_N^*\bL_A[-1])$ with associated graded pieces
\[\Sym^k_{A_N}(\bL_{A_N/A}[-2])\otimes \Sym^{p-k}_{A_N}(\bL_{A_N}[-1]).\]
Therefore, the fibre
\[\fib(\Sym^p_{A_N}(i_N^*\bL_A[-1])\longrightarrow \Sym^p_{A_N}(\bL_{A_N}[-1]))\]
is $(N+1-p)$-connective. For an inverse system $\{x_n\}$, the inverse limit may be computed as the fibre
\[\lim_n x_n = \fib\left(\prod_{n\geq N} x_n\rightarrow \prod_{n\geq N} x_n\right)\]
for any $N$. We obtain that the left-hand side is $(N-p)$-connective for any $N$ and is, therefore, $0$.
\end{proof}

For any derived prestack $\fX$ we define $\bDR(\fX)$ by a right Kan extension, i.e.
\[\bDR(\fX) = \lim_{\Spec A\rightarrow \fX}\bDR(A)\]
and similarly for the other spaces. By \cite[Proposition 1.14]{ptvv13} for any derived Artin stack $\fX$ we have
\[
\cA^p(\fX, n) \cong \Map_{\QCoh(\fX)} (\cO_{\fX}, \wedge^p \bL_{\fX}[n]).
\]

For a morphism between derived stacks $f \colon \fX_1 \to \fX_2$,
we can define pullback morphisms of (closed, exact) $n$-shifted $p$-forms
\[
f^{\star} \colon \cA^{p}(\fX_2, n) \to \cA^{p}(\fX_1, n), \quad f^{\star} \colon \cA^{p, \cl}(\fX_2, n) \to \cA^{p, \cl}(\fX_1, n),\qquad f^\star\colon \cA^{p, \exact}(\fX_2, n) \to \cA^{p, \exact}(\fX_1, n).
\]

We will particularly be interested in the case of $(-1)$-shifted $2$-forms, in which case the following statement is useful. Recall from \cite[Chapter 2, Section 1.7.2]{gr1} that a derived prestack $\fX$ is \defterm{locally almost of finite type} if it is convergent and its $n$-truncation $\tau^{\leq n}(\fX)$ is locally of finite type.

\begin{prop}\label{prop:closedexactdecomposition}
Let $\fX$ be a derived prestack locally almost of finite type. Then there is a canonical isomorphism $\cA^{2, \exact}(\fX, -1)\cong \cA^{2, \cl}(\fX, -1)\times \Gamma(\fX, \C_X)$, so that the projection on the first factor is given by the de Rham differential, where $\Gamma(\fX, \C_X)$ is the set of locally constant functions.
\end{prop}
\begin{proof}
By Proposition \ref{prop:DRconvergent} the functor $\bDR$ is convergent, so we have
\[\bDR(\fX) = \lim_n \bDR(\tau^{\leq n}(\fX)).\]
Since $\fX$ is locally almost of finite type, $\tau^{\leq n}(\fX)$ is locally of finite type and can, therefore, be written as a colimit (in the $\infty$-category of $n$-truncated derived prestacks) of a diagram of $n$-truncated derived affine schemes of finite type. Thus, it is enough to prove the claim for $\fX=\Spec A$ a derived affine scheme of finite type. We have a fibre sequence
\[
\cA^{0, \cl}(\fX, 0)\longrightarrow \cA^{2, \exact}(\fX, -1)\xrightarrow{\ddr} \cA^{2, \cl}(\fX, -1)
\]
of pointed spaces. Moreover, there is a natural forgetful map $\cA^{2, \exact}(\fX, -1)\rightarrow \cA^{1, \exact}(\fX, 0)\cong\cA^0(\fX, 0)$ so that the composite
\[\cA^{0, \cl}(\fX, 0)\longrightarrow \cA^{2, \exact}(\fX, -1)\longrightarrow \cA^0(\fX, 0)\]
identifies with the forgetful map \eqref{eq:closedformforget}. By \cite[Proposition 5.6]{bbj19} we have the following:
\begin{itemize}
    \item The restriction $\cA^{0, \cl}(\fX, 0)\rightarrow \cA^0(\fX, 0)\rightarrow \cA^0(\fX^{\red}, 0)$ identifies $\cA^{0, \cl}(\fX, 0)$ with the \emph{set} of locally constant $\C$-valued functions on $\fX^{\red}$, i.e. there is a natural isomorphism $\cA^{0, \cl}(\fX, 0)\cong \Gamma(\fX, \C)$.
    \item The de Rham differential $\ddr\colon \cA^{2, \exact}(\fX, -1)\rightarrow \cA^{2, \cl}(\fX, -1)$ induces an isomorphism on $\pi_k$ for $k\geq 1$.
    \item On the level of $\pi_0$ we have a short exact sequence of abelian groups
    \[0\longrightarrow \pi_0(\cA^{0, \cl}(\fX, 0))\longrightarrow \pi_0(\cA^{2, \exact}(\fX, -1))\longrightarrow \pi_0(\cA^{2, \cl}(\fX, -1))\longrightarrow 0.\]
\end{itemize}
Let $\cA^{2, \exact}_0(\fX, -1)$ be the homotopy fibre of the composite
\[\cA^{2, \exact}(\fX, -1)\longrightarrow \cA^0(\fX, 0)\longrightarrow \cA^0(\fX^{\red}, 0).\]
By the above facts we get that the composite
\[\cA^{2, \exact}_0(\fX, -1)\longrightarrow \cA^{2, \exact}(\fX, -1)\longrightarrow \cA^{2, \cl}(\fX, -1)\]
as well as the morphism
\[\cA^{2, \exact}_0(\fX, -1)\times \cA^{0, \cl}(\fX, 0)\longrightarrow \cA^{2, \exact}(\fX, -1)\]
are isomorphisms.
\end{proof}

\begin{rmk}\label{rmk:classicalaft}
Let $\fX$ be a classical prestack locally of finite type and let $\fX^{\der}$ be the corresponding derived prestack. Let $\fX^{\der, \mathrm{conv}}$ be the convergent completion of $\fX^{\der}$ (see \cite[Chapter 2, Section 1.4.8]{gr1}). Then $\fX^{\der, \mathrm{conv}}$ is locally almost of finite type (see \cite[Chapter 2, Corollary 1.7.8]{gr1}). Since $\bDR$ is convergent by Proposition \ref{prop:DRconvergent}, $\bDR(\fX^{\der, \mathrm{conv}})\cong \bDR(\fX^{\der})$ and so Proposition \ref{prop:closedexactdecomposition} applies to $\fX^{\der}$ as well.
\end{rmk}

\begin{defin}
Let $\fX$ be a derived higher Artin stack. An \defterm{$n$-shifted symplectic structure on $\fX$} is a closed $n$-shifted $2$-form $\omega_{\fX} \in \cA^{2, \cl}(\fX, n)$ whose underlying $n$-shifted $2$-form induces an equivalence
\[
\bT_{\fX}=\bL_{\fX}^{\vee} \xrightarrow[\simeq]{\cdot \omega_{\fX}} \bL_{\fX}[n].
\]
\end{defin}

The notion of shifted symplectic structures extends to locally geometric derived stacks, that is, the unions of open substacks that are derived higher Artin stacks.

A derived higher Artin stack equipped with an $n$-shifted symplectic structure is called an \defterm{$n$-shifted symplectic stack}.
In this paper, we are mainly interested in $(-1)$-shifted symplectic stacks. Given a pair of $n$-shifted symplectic stacks $(\fX, \omega_\fX)$ and $(\fY, \omega_{\fY})$, their product $\fX\times \fY$ carries an $n$-shifted symplectic structure
\[\omega_\fX\boxplus \omega_\fY \coloneqq \pr_1^{\star} \omega_\fX + \pr_2^{\star}\omega_\fY.\]

\begin{defin}
Let $(\fX_1, \omega_{\fX_1})$ and $(\fX_2, \omega_{\fX_2})$ be a pair of $n$-shifted symplectic stacks. A \defterm{Lagrangian correspondence} $\fX_1\xleftarrow{\tau_1} \fL\xrightarrow{\tau_2} \fX_2$ is the choice of a homotopy $\eta\colon \tau_1^\star \omega_{\fX_1}\sim \tau_2^\star \omega_{\fX_2}$ in $\cA^{2, \cl}(\fL, n)$ such that the naturally induced commutative square
\[
\xymatrix{
\bT_{\fL} \ar[r] \ar[d] & \tau_1^*\bT_{\fX_1} \ar[d] \\
\tau_2^*\bT_{\fX_2} \ar[r] & \bL_{\fL}[n]
}
\]
is Cartesian. A \defterm{Lagrangian morphism} $\fL\rightarrow \fX$ is a Lagrangian correspondence $\pt\leftarrow \fL\rightarrow \fX$.    
\end{defin}

Suppose $\fX_1,\fX_2,\fX_3$ are three $n$-shifted symplectic stacks and $\fX_1\xleftarrow{\tau_1} \fL_1\xrightarrow{\tau_2} \fX_2$ and $\fX_2\xleftarrow{\sigma_1} \fL_2\xrightarrow{\sigma_2} \fX_3$ are two Lagrangian correspondences. Consider their composition
\[
\xymatrix{
&& \fL_1\times_{\fX_2} \fL_2 \ar[dl] \ar[dr] && \\
& \fL_1 \ar^{\tau_1}[dl] \ar^{\tau_2}[dr] && \fL_2 \ar^{\sigma_1}[dl] \ar^{\sigma_2}[dr] & \\
\fX_1 && \fX_2 && \fX_3.
}
\]
Then $\fX_1\leftarrow \fL_1\times_{\fX_2} \fL_2\rightarrow \fX_3$ is a Lagrangian correspondence \cite[Theorem 4.4]{cal15}.

\begin{defin}
A \defterm{$2$-fold Lagrangian correspondence} is a diagram
\[
\xymatrix{
& \fL_1 \ar_{\tau_1}[dl] \ar^{\tau_2}[dr] & \\
\fX_1 & \fM \ar_{\pi_1}[u] \ar^{\pi_2}[d] & \fX_2 \\
& \fL_2 \ar^{\sigma_1}[ul] \ar_{\sigma_2}[ur] &
}
\]
where $(\fX_1, \omega_{\fX_1})$ and $(\fX_2, \omega_{\fX_2})$ are $n$-shifted symplectic stacks, $\fX_1\xleftarrow{\tau_1} \fL_1\xrightarrow{\tau_2} \fX_2$ and $\fX_1\xleftarrow{\sigma_1} \fL_2\xrightarrow{\sigma_2} \fX_2$ are Lagrangian correspondences specified by the homotopies $\eta_1$ in $\cA^{2, \cl}(\fL_1, n)$ and $\eta_2$ in $\cA^{2, \cl}(\fL_2, n)$ as well as a homotopy $\xi$ from $\pi_1^\star \eta_1$ to $\pi_2^\star \eta_2$ in $\cA^{2, \cl}(\fM, n)$ such that the naturally induced diagram
\[
\xymatrix{
& \bT_{\fM} \ar[dl] \ar[dr] & \\
\pi_1^* \bT_{\fL_1} \ar[d] \ar[drr]|\hole && \pi_2^* \bT_{\fL_2} \ar[d] \ar[dll] \\
\pi_1^*\tau_1^* \bT_{\fX_1}\simeq \pi_2^*\sigma_1^* \bT_{\fX_1} \ar[dr] && \pi_1^*\tau_2^* \bT_{\fX_2}\simeq \pi_2^*\sigma_2^* \bT_{\fX_2} \ar[dl] \\
& \bL_{\fM}[n] &
}
\]
is a limit diagram.
\end{defin}

We now explain some examples of $n$-shifted symplectic structures and Lagrangians.

\begin{ex}
    Let $\Perf$ denote the derived stack of perfect complexes which is locally geometric \cite{tv07}.
    Let $\cE$ denote the universal perfect complex on $\Perf$.
    Then the cotangent complex of $\Perf$ is given by
    \[
    \bL_{\Perf} \simeq  \cE^{\vee} \otimes \cE [-1].
    \]
    It is shown in \cite[Theorem 2.12]{ptvv13} that $\Perf$ is equipped with a $2$-shifted symplectic structure. The underlying $2$-form corresponds to the equivalence
    \[
    \bL_{\Perf}^{\vee} \simeq (\cE^{\vee} \otimes \cE [-1])^{\vee} \simeq \cE \otimes \cE^{\vee}[1] \simeq (\cE^{\vee} \otimes \cE [-1])[2] \simeq \bL_{\Perf}[2].
    \]
\end{ex}

\begin{ex}
    Let $G$ be a reductive algebraic group.
    The cotangent complex of $BG$ is given by $\mathfrak{g}[-1]$, where $\mathfrak{g}$ denote the adjoint representation of $G$.
    It is shown in \cite[\S 1.2]{ptvv13} that a $2$-shifted symplectic structure for $BG$ corresponds to a $G$-equivariant non-degenerate bilinear form for $\mathfrak{g}$.
    In particular, there exists a $2$-shifted symplectic structure for $BG$.
\end{ex}

\begin{ex}
    Let $Y$ be a smooth projective Calabi--Yau $d$-fold or the Betti geometric realization of an oriented $d$-dimensional closed manifold.
    If we are given an $n$-shifted symplectic stack $\fX$,
    it is shown in \cite[Theorem 2.5]{ptvv13} that the mapping stack $\Map(Y, \fX)$ is naturally equipped with an $(n - d)$-shifted symplectic structure.
    Further, assume that we are given a Lagrangian morphism $\fL \to \fX$.
    Then it is shown in \cite[Theorem 2.10]{cal15} that the map $\Map(Y, \fL) \to \Map(Y, \fX)$ is naturally equipped with a Lagrangian structure.
\end{ex}

\begin{ex}\label{ex:critical}
    Let $\fX$ be a derived higher Artin stack.
    For an integer $n \in \bZ$, the \defterm{$n$-shifted cotangent stack} is defined to be the total space
    \[
    T^*[n]\fX \coloneqq \Tot_{\fX}(\bL_{\fX}[n]).
    \]
    Let $\lambda \in \cA^1(T^*[n]\fX, n)$ denote the tautological $n$-form.
    Then it is shown in \cite[Theorem 2.2]{cal19} that $\omega_{\fX} \coloneqq \ddr \lambda$ is $n$-shifted symplectic.

    Let $f \colon \fX \to \bA^1[n] \coloneqq \Tot_{\pt}(\bC[n])$ be an $n$-shifted function.
    The de Rham differential of $f$ defines a section $df \colon \fX \to T^*[n]\fX$.
    It is shown in \cite[Theorem 2.15]{cal19} that the map $df$ is canonically equipped with a Lagrangian structure.
    We define the derived critical locus $\Crit(f)$ by the fibre product
    \[
    \xymatrix{
    {\Crit(f)} \ar[d] \ar[r] \pbcorner
    & {\fX} \ar[d]^-{df} \\
    {\fX} \ar[r]^-{0}
    & {T^*[n]\fX.}
    }
    \]
    Since ${\Crit(f)}$ is defined as the fibre product of two Lagrangian morphisms, it is equipped with an $(n - 1)$-shifted symplectic structure $\omega_{\Crit(f)}$.
\end{ex}

The following Darboux theorem plays a very important role in the study of shifted symplectic stacks:

\begin{thm}[{\cite[Corollary 2.11]{bbbbj15}}]\label{thm:Darboux}$ $
    \begin{thmlist}
        \item[(i)]  Let $(X, \omega_{X})$ be a $(-1)$-shifted symplectic derived algebraic space.
        Then for each point $x \in X$ there is an \defterm{\'etale derived critical chart $\fR = (R, \eta, U, f, i)$}, i.e. 
        a derived scheme $R$,
        an \'etale morphism $\eta \colon R \to X$,
        a smooth scheme $U$ together with a regular function $f \colon U \to \bA^1$ satisfying $f |_{\mathrm{Crit}(f)^{\red}} = 0$ and a 
        closed embedding $i \colon R \hookrightarrow U$
        with $i(R) = \mathrm{Crit}(f)$ such that a homotopy
        $\eta^{\star} \omega_{X} \sim i^{\star} \omega_{\Crit(f)}$ exists.

        \item[(ii)] Let $(\fX, \omega_{\fX})$ be a $(-1)$-shifted symplectic derived Artin stack. 
        Then for each point $x \in \fX$, we can construct a diagram
        \begin{equation}\label{eq:BBBBJ_diagram}
        \begin{aligned}
        \xymatrix{
        {\widehat{V}} \ar[r]^-{\tau} \ar[d]^-{q}
        & {V = \Crit(f)} \\
        {\fX}
        & {}
        }
        \end{aligned}
        \end{equation}
        together with a homotopy $q^{\star} \omega_{\fX} \sim \tau^{\star} \omega_{\Crit(f)}$,
        where $q$ is a smooth morphism whose image contains $x$,
        $f$ is a regular function on a smooth scheme and $\tau$ is a morphism which induces an equivalence on the classical truncations.
    \end{thmlist}
\end{thm}

\begin{rmk}\label{rmk:auto_nondeg}
    The homotopy $q^{\star} \omega_{\fX} \sim \tau^{\star} \omega_{\Crit(f)}$ induces a morphism of fibre sequences
    \[
    \xymatrix{
    {\bT_{\widehat{V}/V}[1]} \ar[r] \ar[d]^-{\eta}
    & { \bT_{\widehat{V}}[1]} \ar[r] \ar[d]
    & {\tau^* \bT_{V}[1]} \ar[d] \\
    {\bL_{\widehat{V} / \fX}[-1]} \ar[r]
    & {q^* \bL_{\fX} } \ar[r]
    & {\bL_{\widehat{V}}.}
    }
    \]
    Since $q$ is smooth and $\tau$ is an isomorphism on classical truncations, 
    we see that $\eta$ is invertible by the degree reason.
    In particular, the correspondence $\fX \leftarrow \widehat{V} \to V$ is a Lagrangian correspondence.
    
\end{rmk}

We also have the following equivariant version of the Darboux theorem:

\begin{prop}\label{prop:equivariant_Darboux}
Let $\fX$ be a quasi-separated derived Artin stack with affine stabilizers and $\omega_{\fX}$ be a $(-1)$-shifted symplectic structure.
If $x\in \fX$ is a point with reductive stabilizer $G$,
then there exists a smooth affine scheme $U$ with a $G$-action, a fixed point $u \in U$, a $G$-invariant function $f:U \to \bA^1$, and a pointed {\'e}tale symplectomorphism
\[(\Crit([f/G]),u) \to (\fX, x),\] 
that preserves the stabilizer at $u$, where $[f/G]\colon [U/G] \to \bA^1$ is the induced function.
\end{prop}
\begin{proof}
Proposition \ref{prop:closedexactdecomposition} shows that all $(-1)$-shifted symplectic forms admit canonical exact structures. The claim then follows from \cite[Theorem B]{Par24}.
\end{proof}

\begin{rmk}
If $\fX$ is a $(-1)$-symplectic derived Artin stack with affine diagonal such that the classical truncation $\fX^{\cl}$ has a good moduli space $Y$, then we can cover $\fX$ with the equivariant Darboux charts in Proposition \ref{prop:equivariant_Darboux}.
Indeed, for any point  $y\in Y$, we can find a point $x\in \fX$ that maps to $y$ and has a reductive stabilizer \cite[Proposition 12.14]{alper}.
Further, using Luna's fundamental lemma \cite[Proposition 4.13]{ahr20}, we may take the equivariant charts in a way that the following diagram is Cartesian:
\[
\xymatrix{
{\Crit([f/G])^{\mathrm{cl}}}
\ar[r]
\ar[d]
& {\mathfrak{X}^{\mathrm{cl}}}
\ar[d] \\
{\mathrm{Crit}(f)^{\mathrm{cl}} \GIT G}
\ar[r]
& {Y.}
}
\]
\end{rmk}

\subsection{Orientation for $(-1)$-shifted symplectic stacks}

We recall the notion of orientation for $(-1)$-shifted symplectic stacks and their Lagrangians.
\begin{defin}
Let $(\fX, \omega_{\fX})$ be a $(-1)$-shifted symplectic derived Artin stack. A (graded) \defterm{orientation of $\fX$} is the pair of a graded line bundle $\cL \in \sPic(\fX^{\red})$ and an isomorphism $o \colon \cL^{\otimes 2} \cong \rdet(\bL_{\fX})$. By abuse of notation, we simply say that $o$ is an orientation of $\fX$.
\end{defin}

Orientations naturally form a groupoid.

\begin{defin}\label{defin:isom_orientation}
An \defterm{isomorphism of orientations} $(\cL_1, o_1)\cong (\cL_2, o_2)$ is an isomorphism of graded line bundles $f\colon \cL_1\cong \cL_2$ on $\fX^{\red}$ such that the triangle
\[
\xymatrix{
\cL_1^{\otimes 2} \ar^{f^{\otimes 2}}[rr] \ar_{o_1}[dr] && \cL_2^{\otimes 2} \ar^{o_2}[dl] \\
& \rdet(\bL_{\fX}) &
}
\]
commutes.
\end{defin}

Orientations of $\fX$ localize to define the graded \defterm{orientation $\mu_2$-gerbe} $\ori_{\fX}$ on $\fX^{\red}$, so that an orientation of $\fX$ is a trivialization of $\ori_{\fX}$.

For a pair of $(-1)$-shifted symplectic stacks $\fX_1, \fX_2$, the natural isomorphism 
\[\Psi_{\fX_1, \fX_2}\colon \rdet(\bL_{\fX_1})\boxtimes \rdet(\bL_{\fX_2})\cong \rdet(\bL_{\fX_1\times \fX_2})\]
of graded line bundles on $\fX^{\red}_1\times \fX^{\red}_2$ induces an isomorphism \[\ori_{\fX_1}\boxtimes \ori_{\fX_2}\cong \ori_{\fX_1\times \fX_2}\]
of graded orientation gerbes. In particular, if $o_1$ and $o_2$ are orientations of $\fX_1$ and $\fX_2$, using the above isomorphism we obtain an orientation $o_1\boxtimes o_2$ of $\fX_1 \times \fX_2$.

Assume that $\fX_1\leftarrow \fL\rightarrow \fX_2$ is a Lagrangian correspondence of $(-1)$-shifted symplectic stacks. Then we have an equivalence
\begin{equation}\label{eq:Lag_corresp_TL}
\bT_{\fL / \fX_1}[1] \simeq \bL_{\fL / \fX_2}[-1].
\end{equation}
which induces an isomorphism of determinant line bundles
\begin{equation}\label{eq:Lag_corresp_TL_det}
\rdet(\bL_{\fL / \fX_1}) \cong \rdet(\bL_{\fL / \fX_2})^{\vee}.
\end{equation}
by \eqref{eq:det_shift} and \eqref{eq:det_dual}.
Using isomorphisms \eqref{eq:fibre_transform} for fibre sequences 
$\bL_{\fX_1} |_{\fL} \to \bL_{\fL} \to \bL_{\fL / \fX_1}$ and 
$\bL_{\fX_2} |_{\fL} \to \bL_{\fL} \to \bL_{\fL / \fX_2}$, 
we obtain isomorphisms
\[
\rdet(\bL_{\fX_1} |_{\fL})  \otimes \rdet(\bL_{\fL / \fX_1}) \cong \rdet(\bL_{\fL}), 
\quad \rdet(\bL_{\fL}) \otimes \rdet(\bL_{\fL / \fX_2})^{\vee} \cong \rdet(\bL_{\fX_2} |_{\fL}) .
\]
By combining them and using \eqref{eq:Lag_corresp_TL_det}, we obtain an isomorphism
\begin{equation}\label{eq:compositeLagrangiancorrespondencedeterminant}
\xi \colon \rdet(\bL_{\fX_1} |_{\fL})  \otimes (\rdet(\bL_{\fL/\fX_2})^{\vee})^{\otimes 2} \cong \rdet(\bL_{\fX_2} |_{\fL}) 
\end{equation}
of graded line bundles on $\fL^{\red}$.
Therefore, there is an isomorphism
\begin{equation}\label{eq:Lagrangiancorrespondenceorientation}
\beta_{\fL}\colon \ori_{\fX_1}|_{\fL}\cong \ori_{\fX_2}|_{\fL}
\end{equation}
of orientation gerbes given by sending a pair $(\cL, o)$ of a graded line bundle $\cL$ on $\fL^{\red}$ together with an isomorphism $o\colon \cL^{\otimes 2}\simeq \rdet(\bL_{\fX_1} |_{\fL})$ to the graded line bundle $\cL\otimes \rdet(\bL_{\fL/\fX_2})^{\vee}$ equipped with the isomorphism
\[(\cL\otimes \rdet(\bL_{\fL/\fX_2})^{\vee})^{\otimes 2}\xrightarrow{o} 
\rdet(\bL_{\fX_1} |_{\fL}) \otimes (\rdet(\bL_{\fL/\fX_2})^{\vee})^{\otimes 2} \xrightarrow{\xi} \rdet(\bL_{\fX_2} |_{\fL}).\]

\begin{defin}
An \defterm{orientation of the Lagrangian correspondence $\fX_1\xleftarrow{\tau_1} \fL\xrightarrow{\tau_2} \fX_2$} is the choice of orientations $o_1$ of $\fX_1$ and $o_2$ of $\fX_2$ together with an isomorphism $\rho$ of the corresponding orientations related by $\beta_{\fL}$.
\end{defin}

Consider a composition of Lagrangian correspondences
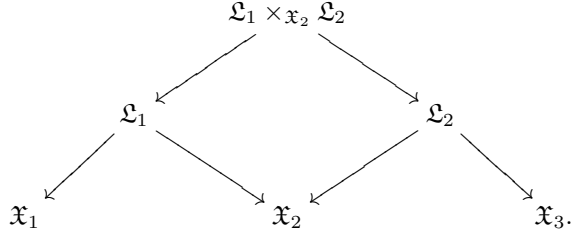
\[
\xymatrix{
&& \fL_1\times_{\fX_2} \fL_2 \ar[dl] \ar[dr] && \\
& \fL_1 \ar[dl] \ar[dr] && \fL_2 \ar[dl] \ar[dr] & \\
\fX_1 && \fX_2 && \fX_3.
}
\]
There is a natural fibre sequence
\[\bL_{\fL_2/\fX_3}|_{\fL_1\times_{\fX_2} \fL_2}\longrightarrow \bL_{\fL_1\times_{\fX_2} \fL_2/\fX_3}\longrightarrow \bL_{\fL_1/\fX_2}|_{\fL_1\times_{\fX_2} \fL_2}.\]
Using the isomorphism \eqref{eq:fibre_transform} we obtain an isomorphism of determinant line bundles
\[
\pi 
\colon \rdet(\bL_{\fL_2/\fX_3}|_{\fL_1\times_{\fX_2} \fL_2})\otimes \rdet(\bL_{\fL_1/\fX_2}|_{\fL_1\times_{\fX_2} \fL_2})\cong \rdet(\bL_{\fL_1\times_{\fX_2} \fL_2/\fX_3}).\]
We now show that the map $\pi$ is compatible with $\xi$:
\begin{lem}\label{lem:xi_theta_Lagcorresp}
Set $\widehat{\fL} \coloneqq \fL_{1} \times_{\fX_2} \fL_2$.
    Then the following diagram commutes:
    \[
    \xymatrix@C=40pt{
    {\rdet(\bL_{\fX_1} |_{\widehat{\fL}}) \otimes (\rdet(\bL_{\fL_1 \times_{\fX_2} \fL_2 / \fX_3})^{\vee})^{\otimes 2}  } \ar[r]^-{\xi}_-{\cong}
    & {\rdet(\bL_{\fX_3} |_{\widehat{\fL}})} \\
    {\rdet(\bL_{\fX_1} |_{\widehat{\fL}}) \otimes 
    (\rdet(\bL_{\fL_1  / \fX_2} |_{\widehat{\fL}})^{\vee})^{\otimes 2} 
    \otimes (\rdet(\bL_{\fL_2  / \fX_3} |_{\widehat{\fL}})^{\vee})^{\otimes 2}} \ar[u]_-{\cong}^-{\id \otimes (\pi^{\vee})^{\otimes 2}} \ar[r]^-{\xi \otimes \id}_-{\cong}
    & {\rdet(\bL_{\fX_2} |_{\widehat{\fL}}) \otimes 
    (\rdet(\bL_{\fL_2  / \fX_3} |_{\widehat{\fL}})^{\vee})^{\otimes 2}.} 
    \ar[u]_-{\cong}^-{\xi}
    }
    \]
\end{lem}

\begin{proof}
    By the definition of the composite Lagrangian structure,
    we have the following equivalence of fibre sequences:
    \[
    \xymatrix{
    {\bT_{\fL_1 / \fX_{1}}[1] |_{\widehat{\fL}}} \ar[r] \ar[d]^-{\simeq}_-{\eqref{eq:Lag_corresp_TL}}
    & {\bT_{\widehat{\fL} / \fX_{1}}[1]} \ar[r] \ar[d]^-{\simeq}_-{\eqref{eq:Lag_corresp_TL}}
    & {\bT_{\fL_2 / \fX_{2}}[1]|_{\widehat{\fL}}} \ar[d]^-{\simeq}_-{\eqref{eq:Lag_corresp_TL}} \\
        {\bL_{\fL_1 / \fX_{2}}[-1]|_{\widehat{\fL}}} \ar[r]
    & {\bL_{\widehat{\fL} / \fX_{3}}[-1]} \ar[r]
    & {\bL_{\fL_2 / \fX_{3}}[-1] |_{\widehat{\fL}}.} 
    }
    \]
    This equivalence of fibre sequences combined with the Lemma \ref{lem:KM} implies the commutativity of the following diagram:
    \begin{equation}\label{eq:xi_theta_Lagcorresp_Step1}
    \begin{aligned}
        \xymatrix{
    {\rdet(\bL_{\fX_1} |_{\widehat{\fL}}) \otimes \rdet(\bL_{\widehat{\fL} / \fX_3})^{\vee}  } 
    \ar[r]^-{\eqref{eq:Lag_corresp_TL_det}}_-{\cong}
    \ar[d]^-{\cong}
    & {\rdet(\bL_{\fX_1} |_{\widehat{\fL}}) \otimes \rdet(\bL_{\widehat{\fL} / \fX_1})}
    \ar[dd]^-{\cong} \\
                { \rdet(\bL_{\fX_1} |_{\widehat{\fL}}) \otimes \rdet(\bL_{\fL_{1} / \fX_2} |_{\widehat{\fL}})^{\vee} \otimes \rdet(\bL_{\fL_{2} / \fX_3} |_{\widehat{\fL}})^{\vee}}
                \ar[d]^-{\eqref{eq:Lag_corresp_TL_det}}_-{\cong}
        & {} \\
                { \rdet(\bL_{\fX_1} |_{\widehat{\fL}}) \otimes \rdet(\bL_{\fL_{1} / \fX_1} |_{\widehat{\fL}}) \otimes \rdet(\bL_{\fL_{2} / \fX_2} |_{\widehat{\fL}})}
                \ar[r]^-{}_-{\cong}
        & {\rdet(\bL_{\widehat{\fL}}). } 
        }
        \end{aligned}
    \end{equation}
    Using Lemma \ref{lem:KM} again, we obtain the commutativity of the following diagram:
    \begin{equation}\label{eq:xi_theta_Lagcorresp_Step2}
    \begin{aligned}
    \xymatrix{
    {\rdet(\bL_{\widehat{\fL}}) \otimes \rdet(\bL_{\widehat{\fL} / \fX_3} |_{\widehat{\fL}})^{\vee} }
    \ar[r]^-{\cong}
    \ar[d]^-{\cong}
    & {\rdet(\bL_{\fX_3} |_{\widehat{\fL}})} \\
    {\rdet(\bL_{\widehat{\fL}}) \otimes \rdet(\bL_{\fL_{1} / \fX_2} |_{\widehat{\fL}})^{\vee} \otimes \rdet(\bL_{\fL_{2} / \fX_3} |_{\widehat{\fL}})^{\vee}}
    \ar[r]^-{\cong}
    & { \rdet(\bL_{\fL_{2}} |_{\widehat{\fL}}) \otimes \rdet(\bL_{\fL_{2} / \fX_3} |_{\widehat{\fL}})^{\vee}.}
    \ar[u]^-{\cong}
    }
    \end{aligned}
    \end{equation}
    The claim of the lemma follows from the commutativity of diagrams \eqref{eq:xi_theta_Lagcorresp_Step1} and \eqref{eq:xi_theta_Lagcorresp_Step2}.
\end{proof}

By Lemma \ref{lem:xi_theta_Lagcorresp},  we see that the isomorphism $\pi$ induces a natural 2-isomorphism
\begin{equation}\label{eq:ori_assoc_general}
\beta_{\fL_2}\circ \beta_{\fL_1}\cong \beta_{\fL_1\times_{\fX_2} \fL_2}
\end{equation}
of isomorphisms $\ori_{\fX_1}|_{\fL_1\times_{\fX_2} \fL_2}\simeq \ori_{\fX_3}|_{\fL_1\times_{\fX_2} \fL_2}$ of orientation gerbes. In particular, if $\fX_1\leftarrow \fL_1\rightarrow \fX_2$ and $\fX_2\leftarrow \fL_2\rightarrow \fX_3$ are oriented, their composition $\fX_1\leftarrow \fL_1\times_{\fX_2} \fL_2\rightarrow \fX_3$ is oriented as well.

We will often consider the following setting, which allows us to pullback orientations along Lagrangian correspondences with sections.

\begin{lem}\label{lem:Lagrangiancorrespondenceretract}
Let $\fX_1\xleftarrow{\tau_1} \fL\xrightarrow{\tau_2} \fX_2$ be a $(-1)$-shifted Lagrangian correspondence together with a morphism $i\colon \fX_1\rightarrow \fL$ such that $\tau_1\circ i\cong \id_{\fX_1}$. Then there is a natural isomorphism
\[i^*\tau_2^*\ori_{\fX_2}\cong \ori_{\fX_1}.\]
Moreover, suppose $\fL\rightarrow \fX_1$ is an $\bA^1$-deformation retract, i.e. there is a map $F\colon \bA^1\times \fL\rightarrow \fL$ satisfying the following conditions:
\begin{enumerate}
    \item $F|_{\{1\}\times \fL}\simeq \id_{\fL}$.
    \item $F|_{\{0\}\times \fL}\simeq i\circ \tau_1$.
\end{enumerate}
Then the forgetful functor from the groupoid of orientations
\[\left\{ (o_1, o_2, \rho) \middle| \substack{
\text{$o_1$: orientation for $\fX_1$} \\
\text{$o_2$: orientation for $\fX_2$} \\
\text{$\rho$: orientation for $\fX_1 \xleftarrow{\tau_1} \fL \xrightarrow{\tau_2} \fX_2$}
} \right\}\]
of the Lagrangian correspondence $\fX_1\leftarrow \fL\rightarrow \fX_2$ to the groupoid of orientations $\{o_2\}$ of $\fX_2$ is an equivalence.
\end{lem}
\begin{proof}
The isomorphism \eqref{eq:Lagrangiancorrespondenceorientation} provides an isomorphism
\[\beta_{\fl}\colon \ori_{\fX_1}|_{\fL}\cong \ori_{\fX_2}|_{\fL}.\]
Restricting it along $i$, we obtain the isomorphism
\[\ori_{\fX_2}|_{\fX_1}\cong \ori_{\fX_1}.\]

Let $\eC_{\fX_1}$ be the 2-groupoid of graded $\mu_2$-gerbes on $\fX_1$ and similarly for $\fL$. Since $\tau_1\circ i\simeq \id_{\fX_1}$, $i^*\circ \tau_1^*\simeq \id_{\eC_{\fX_1}}$. So, it is enough to prove that $\tau_1^*\colon \eC_{\fX_1}\rightarrow \eC_{\fL}$ is an equivalence. The 2-groupoid $\eC_{\fX_1}$ has homotopy groups given by
\[\pi_{2-n}(\eC_{\fX_1})\simeq H^n_{et}(\fX_1; \mu_2)\times H^{n-1}_{et}(\fX_1; \Z/2\Z)\]
for $n=0, 1, 2$, so it is enough to show that $\tau_1^*\colon H^n_{et}(\fL; \mu_2)\rightarrow H^n_{et}(\fX_1; \mu_2)$ is an isomorphism for $n=0, 1, 2$ and similarly for cohomology with $\Z/2\Z$ coefficients. The projection $\pi_2\colon \bA^1\times \fL\rightarrow \fL$ is acyclic \cite[Chapter VI, Corollary 4.20]{milne}, so $\pi_2^* = F^*$ as maps $H^n_{et}(\fL; \mu_2)\rightarrow H^n_{et}(\bA^1\times \fL; \mu_2)$. Restricting to $0,1\in\bA^1$, we conclude that $\tau_1^*\circ i^* = \id$ and hence $\tau_1^*\colon H^n_{et}(\fX_1; \mu_2)\rightarrow H^n_{et}(\fL_1; \mu_2)$ is an isomorphism.
\end{proof}

Finally, we remark that derived critical loci admit canonical orientations:

\begin{ex}\label{ex:can_ori}
    Let $\fY$ be a derived higher Artin stack and $f \colon \fY \to \bA^1$ be a regular function.
    As we have seen in Example \ref{ex:critical}, the derived critical locus 
    $\fX \coloneqq \Crit(f)$ admits a $(-1)$-shifted symplectic structure $\omega_{\fX}$.
    By construction, we have a natural fibre sequence
     \[
    \bL_{\fY} |_{\fX} \to \bL_{\fX} \to \bT_{\fY}|_{\fX}[1].
    \]
    where the latter map is given by the composition
    $\bL_{\fX} \xrightarrow[\simeq]{ (\cdot \omega_{\fX})^{-1}} \bT_{\fX}[1] \to \bT_{ \fY} |_{\fX}[1]$.
    Applying \eqref{eq:fibre_transform}, \eqref{eq:det_shift} and \eqref{eq:det_dual},
    we obtain an isomorphism
    \[
  \bar{o}_{\fX}^{\can} \colon  \rdet(\bL_{\fY} |_{\fX})^{\otimes 2} \cong \rdet(\bL_{\fX}).
    \]
   We set $o_{\fX}^{\can} \coloneqq (-1)^{\vdim \fY} \cdot \bar{o}_{\fX}^{\can}$ and call it the \defterm{canonical orientation} for the derived critical locus $\fX$\footnote{The insertion of the sign is to make it consistent with the canonical orientation for d-critical stacks: see \eqref{eq:can_ori} and Remark \ref{rmk:sign_canori}}.
\end{ex}

\subsection{D-critical structures}

Here we recall Joyce's theory of d-critical structures, which can be regarded as classical truncations of $(-1)$-shifted symplectic structures. The original reference is \cite{joy15}.

For an algebraic space $X$, Joyce \cite[Theorem 2.1]{joy15} defined \'etale sheaves
\[
\cS_{X},\cS_{X}^0 \in \Mod_{\bC_{X}}
\]
of $\bC$-vector spaces with the following properties:
\begin{enumerate}
    \item There is a decomposition $\cS_{X}\cong \cS_{X}^0\oplus \bC_X$.
    \item Consider an \'etale morphism $\eta\colon R\to X$ from a scheme $R$. Since $X$ is locally of finite type, $R$ is locally of finite type as well. If we denote by $R^{\der}$ the corresponding derived scheme, then there exist isomorphisms
    \begin{equation}\label{eq:Sforms}
    \Gamma(R, \cS_{X}|_R)\cong \cA^{2, \exact}(R^{\der}, -1),\qquad \Gamma(R, \cS_{X}^0|_R)\cong \cA^{2, \cl}(R^{\der}, -1).
    \end{equation}
    Under these isomorphisms the decomposition $\cS_X\cong \cS_X^0\oplus \bC_X$ corresponds to the isomorphism $\cA^{2, \exact}(R, -1)\cong \cA^{2, \cl}(R, -1)\times \Gamma(R, \C_R)$ from Proposition \ref{prop:closedexactdecomposition} (see Remark \ref{rmk:classicalaft} why it applies).
    \item For any \'etale morphism $\eta \colon R \to X$ and any closed embedding $i \colon R \hookrightarrow U$ where $U$ is a smooth scheme, there exists the following short exact sequence:
    \[
    0 \to \cS_{X}|_R \to i^{-1} \cO_{U} / I_{R, U}^2 \xrightarrow{\ddr} i^{-1} \Omega_U / (I_{R, U} \cdot i^{-1} \Omega_U)
    \]
    where $I_{R, U}$ is the ideal sheaf of $i^{-1} \cO_U$ defining $R$. The composite
    \[
    \cS_{X}|_R \to i^{-1} \cO_{U} / I_{R, U}^2 \to i^{-1} \cO_{U} / I_{R, U} \simeq \cO_R \to \cO_{R^{\red}}
    \]
    factors through locally constant functions $\C_{R^{\red}}\rightarrow \cO_{R^{\red}}$ which defines the projection $\cS_{X}\rightarrow \C_X$.
\end{enumerate}

Assume that we are given a regular function $f$ on a smooth scheme $U$ such that $f|_{\Crit(f)^{\red}} = 0$.
Set $X = \Crit(f)^{\cl}$.
Then $f$ defines a section of $\cS_{X}^0$.

\begin{defin}[{\cite[Definition 2.5]{joy15}}]
    Let $X$ be an algebraic space and $s \in \Gamma(X, \cS_{X}^0)$ is a section.
    \begin{itemize}
        \item An \defterm{\'etale d-critical chart} for $(X, s)$ is an \'etale morphism $\eta\colon R \to X$, a smooth scheme $U$, a closed embedding $i \colon R\hookrightarrow U$ and a regular function $f$ on $U$ with $f |_{i(R)^{\red}} = 0$, $i(R) = \Crit(f)$ and $f + I_{R, U}^2 = s|_{R}$.
        \item The section $s$ is a \defterm{d-critical structure} if for each point $x \in X$ there exists an \'etale d-critical chart $(R, \eta,  U, f, i)$ such that the image of $\eta$ contains $x$.
    \end{itemize}
\end{defin}

Let $f \colon X \to Y$ be a morphism of algebraic spaces.
Then there exists a natural morphism of sheaves
\[
f^\ast\colon f^{-1} \cS_{Y}^0 \to \cS_{X}^0.
\]
which corresponds under the isomorphism \eqref{eq:Sforms} to the pullback of differential forms. It is shown in \cite[Proposition 2.8]{joy15} that if $f$ is smooth and $s$ is a d-critical structure of $Y$,
then $f^{\star} s$ is a d-critical structure on $X$.

For an Artin stack $\fX$, the assignment
\[
\mathrm{Sch}^{\mathrm{sm}}_{/ \fX} \ni T \to \Gamma(T, \cS_{T}^0)
\]
together with the natural transition map
defines a sheaf in the lisse-\'etale topos of $\fX$,
which will be denoted by $\cS_{\fX}^0$.
\begin{defin}
Let $\fX$ be an Artin stack. A \defterm{d-critical structure on $\fX$} is a choice of a section $s \in \Gamma(\fX, \cS^0_{\fX})$ such that the restriction $s |_T \in \Gamma(T, \cS^0_T)$ is a d-critical structure for each $T \in \mathrm{Sch}^{\mathrm{sm}}_{/ \fX}$.
\end{defin}

By \cite[Proposition 2.8]{joy15}, $s \in \Gamma(\fX, \cS^0_{\fX})$ is a d-critical structure if and only if there exists a smooth surjection from an algebraic space $T \to \fX$ such that $s |_{T}$ is a d-critical structure.

The following statement summarizes the relation between $(-1)$-shifted symplectic structures and d-critical structures. See \cite[Theorem 3.18 (a)]{bbbbj15} and also \cite[Theorem 4.6]{kin21} for the details.

\begin{thm}\label{thm:BBBBJ}
    Let $(\fX, \omega_{\fX})$ be $(-1)$-shifted symplectic stack and $j\colon \fX^{\cl}\rightarrow \fX$ the inclusion of the classical truncation. Let $s\in\Gamma(X, \cS^0_X)$ be the section corresponding to $j^\star\omega_{\fX}$ under the isomorphism \eqref{eq:Sforms}. Then $s$ is a d-critical structure.
    \begin{thmlist}
        \item If $\fX$ is an algebraic space and $\fR = (R, \eta, U, f, i)$ is an \'etale derived critical chart for $\fX$, then its classical truncation $\fR^{\cl} = (R^{\cl}, \eta^{\cl}, U, f, i^{\cl})$ defines an \'etale d-critical chart for $\fX^{\cl}$.

        \item  If we are given a diagram \eqref{eq:BBBBJ_diagram} such that $f |_{\Crit(f)^{\red}} = 0$, the following equality holds:
        \[
        q^{\cl, \star} s = \tau^{\cl, \star} (f + I_{\Crit(f)^{\cl}, U}^2).
        \]
    \end{thmlist}
\end{thm}

For later use, we explain how to compare two \'etale d-critical charts following \cite[\S 2.2]{joy15}.
Let $(X, s)$ be a d-critical algebraic space.
An \'etale morphism between \'etale d-critical charts $(R_1, \eta_1, U_1, f_1, i_1)$ and $(R_2, \eta_2, U_2, f_2, i_2)$ is given by \'etale morphisms $h \colon U_1 \to U_2$ and $h_0 \colon R_1 \to R_2$ such that $f_1 = f_2 \circ h$, $\eta_1 = \eta_2 \circ h_0$ and $i_2 \circ h_0 = h \circ i_1$ hold.
The following technical statement will be useful to find an \'etale cover of a given d-critical chart:

\begin{lem}\label{lem:etale_ambient}
    Let $h_0 \colon R_1 \to R_2$ be an \'etale  morphism of schemes and $i_2 \colon R_2 \hookrightarrow U_2$ be a closed embedding of $R_2$ to a smooth scheme.
    Then for each $x \in R_1$, by possibly replacing $R_1$ with an open neighbourhood of $x$, 
    there exist an \'etale morphism $h \colon U_1 \to U_2$ and a closed embedding $R_1 \hookrightarrow U_1$ such that the natural map $R_1 \to R_2 \times_{U_2} U_1$ is an isomorphism.
    
\end{lem}

\begin{proof}
    By possibly shrinking $R_1$ around $x$, we may assume that $R_1$, $U_1$ and $U_2$ are affine.
    We take a closed embedding $\iota \colon R_1 \hookrightarrow \bA^n$.
     Since the map $(h_0, \iota) \colon R_1 \to R_2 \times \bA^n$ is a regular closed embedding of relative dimension $n$,
     there exists a regular functions $g_1, \ldots, g_n \colon R_2 \times \bA^n \to \bA^1$ such that the natural map $R_1 \to Z(g_1, \ldots, g_n)$ is an isomorphism around $x$.
     By shrinking $R_1$ around $x$ again, we may assume that the natural map  $R_1 \to Z(g_1, \ldots, g_n)$ is an open immersion.
     We choose a lift $\tilde{g}_i \colon U_2 \times \bA^n \to \bA^1$ of $g_i$
     such that the morphism $\tilde{h} \colon Z(\tilde{g_1}, \ldots, \tilde{g_n}) \to U_2$ is \'etale at $(h_0(x), \iota(x))$.
     We may take an open subset $U_1 \subset Z(\tilde{g_1}, \ldots, \tilde{g_n})$
     containing $(h_0(x), \iota(x))$ 
     such that $\tilde{h}|_{U_1}$
     is \'etale and
     the natural map
     $R_1 \to R_2 \times_{U_2} U_1$ is an isomorphism, hence we conclude.

\end{proof}

We now introduce the notion of embedding of \'etale d-critical charts.
Let $(X, s)$ be a d-critical algebraic space.
An embedding of \'etale d-critical charts $(R_1, \eta_1, U_1, f_1, i_1) \hookrightarrow (R_2, \eta_2, U_2, f_2, i_2)$ is given by an open embedding $\Phi_0 \colon R_1 \hookrightarrow R_2$ and a locally closed embedding $\Phi \colon U_1 \hookrightarrow U_2$ such that 
$\eta_1 = \eta_2 \circ \Phi_0$, $f_1 = f_2 \circ \Phi$ and $\Phi \circ i_1 = i_2 \circ \Phi_0$ hold.
The following statement is useful when comparing two \'etale d-critical charts:

\begin{thm}\label{thm:compare_d-critical}
    Let $(X, s)$ be a d-critical algebraic space.
    \begin{thmlist}
        \item Let $\fR_1 = (R_1, \eta_1, U_1, f_1, i_1)$ and $\fR_2 = (R_2, \eta_2, U_2, f_2, i_2)$ be two \'etale d-critical charts and $x_1 \in R_1$ and $x_2 \in R_2$ be points such that $\eta_1(x_1) = \eta_2(x_2)$ holds.
        Then by possibly replacing $\fR_1$ and $\fR_2$ by \'etale neighbourhoods of $x_1$ and $x_2$,
        we may find a third d-critical chart $\fR_3 = (R_3, \eta_3, U_3, f_3, i_3)$ and $x_3 \in R_3$ such that there exists embeddings $\fR_1 \hookrightarrow \fR_3$ and $\fR_2 \hookrightarrow \fR_3$ which map $x_1$ to $x_3$ and $x_2$ to $x_3$ respectively.

        \item Let $(\Phi_0, \Phi) \colon \fR_1 =(R_1, \eta_1, U_1, f_1, i_1) \hookrightarrow \fR_2 = (R_2, \eta_2, U_2, f_2, i_2)$ be an embedding of \'etale d-critical charts.
        Then for each $x_1 \in R_1$, there exist morphisms $h_1 \colon U_1' \to U_1$ which contains $x_1$ in its image, $h_2 \colon U_2' \to U_2$, $\Phi' \colon U_1' \to U_2'$, $\alpha \colon U_2' \to U_1$ and $\beta \colon U_2' \to \bA^n$ 
        where $n = \dim U_1 - \dim U_2$, $h_1$, $h_2$ and $(\alpha, \beta) \colon U_2' \to U_1 \times \bA^n$ are \'etale, $h_2 \circ \Phi' = \Phi \circ h_1$, $\alpha \circ \Phi' = h_1$, $\beta \circ \Phi' = 0$ and $f_2 \circ h_2  = f_1 \circ \alpha + (z_1^2 + \cdots + z_n^2) \circ \beta$ hold.
        
        In particular, there exists an \'etale cover of d-critical charts $\fR_2' \to \fR_2$ containing $\Phi_0(x_1)$ in its image and an \'etale morphism of d-critical charts
        $\fR_2' \to (R_1, \eta_1, U_1 \times \bA^n, f_1 \boxplus (z_1^2 + \cdots + z_n^2), (i_1, 0) )$.
    \end{thmlist}
\end{thm}

\begin{proof}
    The second statement is nothing but \cite[Theorem 2.23]{joy15}.
    To prove the first statement, by replacing $R_1$ and $R_2$ by $R_1 \times_X R_2$ and using Lemma \ref{lem:etale_ambient},
    we may assume that $R_1 = R_2$ holds.
    In this case, the statement follows from \cite[Theorem 2.20]{joy15}.
\end{proof}

The following statement is useful to understand the behaviour of d-critical structures under smooth morphisms, which can be found in the proof of \cite[Proposition 2.8]{joy15}:

\begin{prop}\label{prop:sm_chart}
    Let $h \colon X_1 \to X_2$ be a smooth  morphism of algebraic spaces. Let $s_2$ be a  d-critical structure and set $s_1 \coloneqq h^{\star} s_2$.
    Then for each point $x_1$ with $h(x_1) = x_2$,
    we can find an \'etale d-critical chart $(R_1, \eta_1, U_1, f_1, i_1)$ of $x_1$, $(R_2, \eta_2, U_2, f_2, i_2)$ of $x_2$, smooth  morphisms $H_0 \colon R_1 \to R_2$ and $H \colon U_1 \to U_2$ such that $\eta_2 \circ H_0 = h \circ \eta_1$, $f_2 \circ H = f_1$ and $i_2 \circ H_0 = H \circ i_1$.
\end{prop}

We define a \defterm{smooth morphism of \'etale d-critical charts} to be the pair of morphisms $(H_0, H)$ as in Proposition \ref{prop:sm_chart}.

Finally, we introduce the external sum of d-critical structures.
Let $\fX$ and $\fY$ be Artin stacks and $\pr_i$ denotes the $i$-th projection from $\fX \times \fY$.
For a section $s \in \Gamma(\fX, \cS^0_{\fX})$ and $t \in \Gamma(\fY, \cS^0_{\fY})$,
we set
\[
s \boxplus t \coloneqq \pr_1^\star s + \pr_2^{\star} t.
\]

\begin{lem}\label{lem:external_sum_d-critical}
    If $s$ and $t$ are d-critical structures,
    $s \boxplus t$ is a d-critical structure for $\fX \times \fY$.
\end{lem}

\begin{proof}
   Since being a d-critical structure is a local condition,
   we may assume that $\fX$ (resp. $\fY$) is a d-critical algebraic space covered by a single \'etale d-critical chart
   $\fR  = (\fX, \id, U, f, i)$ (resp. $(\fS = (\fY, \id, V, g, j))$).
   Then the claim follows from the identity
    \[
    s \boxplus t = (f \boxplus g) + I_{R \times S, U \times V}^2 \in \Gamma(\fX \times \fY, \cS^0_{\fX \times \fY}).
    \]
\end{proof}

\subsection{Torus-equivariant d-critical spaces}

Here we will discuss torus equivariant d-critical algebraic spaces.

\begin{defin}
Let $T$ be a torus and $X$ be an algebraic space with a $T$-action $\mu$.
\begin{itemize}
    \item A \defterm{$T$-invariant d-critical structure on $X$} is a d-critical structure on $X$ pulled back from $[X / T]$.
    \item A \defterm{$T$-equivariant d-critical algebraic space} is a triple $(X, \mu, s)$ of algebraic spaces $X$ with $T$-action $\mu$ together with a $T$-invariant d-critical structure $s$.  
    \item A \defterm{$T$-equivariant \'etale d-critical chart} is a tuple $(R, \eta, U, f, i)$ where $R$ is a quasi-separated scheme with a $T$-action, $\eta \colon R \to X$ is a $T$-equivariant \'etale morphism, $U$ is a smooth scheme with a $T$-action, $f$ is a $T$-invariant function on $U$ such that $f|_{\Crit(f)^{\red}} = 0$ and $i \colon R \hookrightarrow U$ is a $T$-equivariant closed embedding such that $i(R) = \Crit(f)$.
\end{itemize}
\end{defin}

\begin{prop}\label{prop:minimal_chart_space}
    Let $(X, \mu, s)$ be a quasi-separated $T$-equivariant d-critical algebraic space. Then for each $x \in X$, there exists a $T$-equivariant \'etale d-critical chart $(R, \eta, U, f, i)$ of $X$ containing $x$ such that $\dim T_{X, x} = \dim U$ holds.
\end{prop}

\begin{proof}
    By the \'etale version of Sumihiro's theorem proved in \cite[Theorem 4.1]{ahr20},
    there exists a $T$-equivariant affine scheme $R$ and $T$-equivariant \'etale morphism $R \to X$ whose image contains $x$.
    Then the claim follows from \cite[Proposition 2.43]{joy15}.
\end{proof}

For an algebraic space $X$ with a $T$-action $\mu$, we denote by $X^{\mu}\subset X$ the fixed point subspace.

\begin{cor}\label{cor:dcriticallocalization}
    Let $X$ be a quasi-separated algebraic space with a $T$-action $\mu$ and $s$ be a $T$-invariant d-critical structure on $X$.
    Then $s^{\mu} \coloneqq (X^{\mu} \hookrightarrow X)^{\star} s$ is a d-critical structure.
\end{cor}

\begin{proof}
    We may assume that $X$ is covered by a single $T$-equivariant \'etale d-critical chart $(R, \eta, U, f, i)$.
    Then the tuple
    $(R^{T}, \eta^{T}, U^T, f|_{U^{T}}, i^{T})$ defines an \'etale d-critical chart for $(X^{\mu}, s^{\mu})$.
\end{proof}

The notion of $T$-equivariant embedding of $T$-equivariant d-critical structures is defined in the same manner as the non-equivariant case.
We have the following $T$-equivariant version of Theorem \ref{thm:compare_d-critical}:

\begin{thm}\label{thm:T_compare_d-critical}
    Let $X$ be a quasi-separated algebraic space with a $T$-action $\mu$ and $s$ be a $T$-invariant d-critical structure.
    \begin{thmlist}
        \item Let $\fR_1 = (R_1, \eta_1, U_1, f_1, i_1)$ and $\fR_2 = (R_2, \eta_2, U_2, f_2, i_2)$ be two \'etale $T$-equivariant d-critical charts and $x_1 \in R_1^{T}$ and $x_2 \in R_2^{T}$ be points such that $\eta_1(x_1) = \eta_2(x_2)$ holds.
        Then by possibly replacing $\fR_1$ and $\fR_2$ by $T$-equivariant \'etale neighbourhoods of $x_1$ and $x_2$,
        we may find a third $T$-equivariant d-critical chart $\fR_3 = (R_3, \eta_3, U_3, f_3, i_3)$ and $x_3 \in R_3$ such that there exists $T$-equivariant embeddings $\fR_1 \hookrightarrow \fR_3$ and $\fR_2 \hookrightarrow \fR_3$ which map $x_1$ to $x_3$ and $x_2$ to $x_3$ respectively.

        \item Let $(\Phi_0, \Phi) \colon (R_1, \eta_1, U_1, f_1, i_1) \hookrightarrow (R_2, \eta_2, U_2, f_2, i_2)$ be a $T$-equivariant embedding of $T$-equivariant \'etale d-critical charts.
        Then for each $x_1 \in R_1$, there exist $T$-equivariant morphisms $\eta_1 \colon U_1' \to U_1$ containing $x_1$ in its image, $\eta_2 \colon U_2' \to U_2$, $\Phi' \colon U_1' \to U_2'$, $\alpha \colon U_2' \to U_1$, $\beta \colon U_2' \to \bA^n$ and a $T$-invariant non-degenerate bilinear form $q$ on $\bA^n$
        where $n = \dim U_1 - \dim U_2$,  
        $\bA^n$ is equipped with some linear $T$-action,
        $\eta_1$, $\eta_2$ and $(\alpha, \beta) \colon U_2' \to U_1 \times \bA^n$ are \'etale, $\eta_2 \circ \Phi' = \Phi \circ \eta_1$, $\alpha \circ \Phi' = \eta_1$, $\beta \circ \Phi' = 0$ and $f_2 \circ \eta_2  = f_1 \circ \alpha + q \circ \beta$ hold.

         In particular, there exists a $T$-equivariant \'etale cover of d-critical charts $\fR_2' \to \fR_2$ containing $\Phi_0(x_1)$ in its image and an \'etale morphism of d-critical charts
        $\fR_2' \to (R_1, \eta_1, U_1 \times \bA^n, f_1 \boxplus q, (i_1, 0) )$ where $q$ is a $T$-invariant non-degenerate quadratic function.
    \end{thmlist}
\end{thm}

\begin{proof}
    Note that we can easily make each choice in Lemma \ref{lem:etale_ambient} $T$-equivariant if we may replace $R_1$ with an \'etale cover.
    Therefore, the first claim follows by arguing as the proof of Theorem \ref{thm:compare_d-critical} using \cite[Proposition 2.44]{joy15}.
    The latter claim follows by arguing as the proof of \cite[Proposition 2.23]{joy15}. See \cite[Lemma 5.1]{des22} for the proof of the corresponding statement in the analytic situation.
\end{proof}

Assume that we are given algebraic spaces $X_1$ and $X_2$ with $T$-actions $\mu_1$ and $\mu_2$, $T$-equivariant smooth morphism $h \colon X_1 \to X_2$ which fit into a commutative diagram of fixed point spaces
\[
\xymatrix{
X_1^{\mu_1}\ar[r] \ar^{h^T}[d] & X_1 \ar^{h}[d] \\
X_2^{\mu_2} \ar[r] & X_2.
}
\]
Consider a d-critical structure $s_2$ of $X_2$ and set $s_1 \coloneqq h^{\star} s_2$. Then by construction, we have the identity
\begin{equation}\label{eq:dcrit_local_commutes_sm}
h^{T, \star} (s_2^{\mu_2}) = s_1^{\mu_1}.
\end{equation}
The following $T$-equivariant version of Proposition \ref{prop:sm_chart} will be useful to compare $T$-invariant d-critical structures related by a $T$-equivariant smooth morphism. See \cite[Lemma 5.2]{des22} for the proof:

\begin{prop}\label{prop:T_sm_chart}
    Let $h \colon X_1 \to X_2$ be a smooth $T$-equivariant morphism of quasi-separated algebraic spaces with $T$-actions. Let $s_2$ be a $T$-invariant d-critical structure and set $s_1 \coloneqq h^{\star} s_2$.
    Then for each point $x_1$ with $h(x_1) = x_2$,
    there exists a $T$-equivariant \'etale d-critical chart $(R_1, \eta_1, U_1, f_1, i_1)$ of $x_1$, $(R_2, \eta_2, U_2, f_2, i_2)$ of $x_2$, smooth $T$-equivariant morphisms $H_0 \colon R_1 \to R_2$ and $H \colon U_1 \to U_2$ such that $\eta_2 \circ H_0 = h \circ \eta_1$, $f_2 \circ H = f_1$ and $i_2 \circ H_0 = H \circ i_1$.
\end{prop}

\subsection{Virtual canonical bundles and orientations}

We now explain the definition of the virtual canonical bundle for d-critical algebraic spaces and stacks, following \cite[\S 2.4]{joy15}.
Firstly, we will introduce the notion of the virtual cotangent complexes for \'etale d-critical charts:
\begin{defin}\label{defin:virtual_cot_complex}
    For an \'etale d-critical chart $\fR = (R, \eta, U, f, i)$ of a d-critical algebraic space $(X, s)$,
    we define the \defterm{virtual cotangent complex} $\bL_{\fR} \in \Perf (R)$ by
    \[
    \bL_{\fR} \coloneqq [T_U |_{R} \xrightarrow{\Hess_{f}} \Omega_U |_R]
    \]
    concentrated in degree $[-1, 0]$.
\end{defin}

Let $(\fX, s)$ be a d-critical Artin stack.
Joyce \cite[Theorem 2.56]{joy15} constructed a line bundle 
\[
K_{\fX, s} \in \Pic(\fX^{\red})
\]
called the \defterm{virtual canonical bundle} by gluing $\rdet(\bL_{\fR})$ for each \'etale d-critical charts $\fR = (R, \eta, U, f, i)$ of smooth covers of $\fX$.
Throughout the paper, we equip $K_{\fX, s}$ with the trivial grading.
The following proposition summarizes the properties of the virtual canonical bundle that we will use later:

\begin{prop}\label{prop:virtual_canonical}
    The virtual canonical bundle $K_{\fX, s}$ for a d-critical stack $(\fX, s)$ has the following properties:
    \begin{thmlist}
        \item If $\fX$ is an algebraic space and $\fR$ is an \'etale d-critical chart of $(\fX, s)$, there exists a natural isomorphism
        \begin{equation}\label{eq:iota}
    \iota_{\fR} \colon K_{\fX, s} |_{R^{\red}} \simeq \rdet(\bL_{\fR}).
    \end{equation}
        \item For each point $x \in \fX$, there exists a natural isomorphism
        \begin{equation}\label{eq:kappa}
        \kappa_x \colon K_{\fX, s} |_{x} \cong  \det(\tau^{\geq 0}(\bL_{\fX, x}))^{\otimes 2}.
        \end{equation}

        \item For a smooth morphism $q \colon X \to \fX$ from an algebraic space $X$,
        there exists a natural isomorphism
        \begin{equation}\label{eq:Upsilon}
        \Upsilon_q \colon q^{\red, *} K_{\fX,  s} \otimes K_{X / \fX}^{\otimes 2} \cong K_{X, q^{\star}s}.
        \end{equation}
        
        \item Assume that we are given another d-critical stack $(\fY, t)$.
        Then there exists a natural isomorphism
        \begin{equation}\label{eq:Psi}
        \Psi_{\fX, \fY} \colon K_{\fX, s} \boxtimes K_{\fY, t} \cong K_{\fX \times \fY, s \boxplus t}.
        \end{equation}
    \end{thmlist}
\end{prop}
The properties (i) -- (iii) are proved in \cite[Theorem 2.56]{joy15}.
Concerning the construction of $\Psi_{\fX, \fY}$, see the discussion after \eqref{eq:kappa_Psi_stack}.

For later use, we discuss compatibility between isomorphisms in Proposition \ref{prop:virtual_canonical}.
Assume first that $X$ is an algebraic space and take an \'etale d-critical chart
$\fR = (R, \eta, U, f, i)$.
Then for each $x \in X$ with a lift $\hat{x} \in R$, the map $\kappa_x$ in \eqref{eq:kappa} is compatible with $\iota_{\fR}$, i.e., the following diagram commutes:
\begin{equation}\label{eq:iota_kappa_stack}
    \begin{aligned}
\xymatrix@C=80pt{
{K_{X, s} |_{x} } \ar[rd]_-{\iota_{\fR}|_{x}}^-{\cong} \ar[r]^-{\kappa_x}_{\cong}
& {\det(\Omega_{X, x})^{\otimes 2} } \ar[d]^-{\cong} \\
{}
& {\det(\bL_{\fR}|_{\hat{x}}) }.
}
\end{aligned}
\end{equation}
Here the right vertical isomorphism is constructed using the fibre sequence
\begin{equation}\label{eq:kappa_fibre_seq}
T_{X, x}[1] \to \bL_{\fR}|_{\hat{x}} \to \Omega_{X, x}
\end{equation}
together with \eqref{eq:det_shift} and \eqref{eq:det_dual}.

Let $(\fX, s)$ be a d-critical stack and $q \colon X \to \fX$ be a smooth morphism from an algebraic space.
Then the map $\Upsilon_q$ in \eqref{eq:Upsilon} is the unique morphism such that the following diagram commutes for each $\hat{x} \in X$ with $x = q(\hat{x})$:
\begin{equation}\label{eq:Upsilon_kappa_stack}
\begin{aligned}
\xymatrix@C=100pt{
{q^{\red, *} K_{\fX, s} |_{x} \otimes K_{X / \fX}|_{\hat{x}}^{\otimes 2}}  \ar[r]^-{\Upsilon_q}_-{\cong} \ar[d]^-{\kappa_{x} \otimes \id}
& {K_{X, q^{\star}s} |_{\hat{x}}} \ar[d]^-{\kappa_{\hat{x}}}_-{\cong} \\
{ \det(\tau^{\geq 0}(\bL_{\fX, \hat{x}}))^{\otimes 2} \otimes K_{X / \fX}|_{\hat{x}}^{\otimes 2} } \ar[r]_-{\cong}
& {\det(\Omega_{X, \hat{x}})^{\otimes 2}}
}
\end{aligned}
\end{equation}
where the bottom horizontal arrow is induced by the fibre sequence
$\tau^{\geq 0}(\bL_{\fX, x}) \to \Omega_{X, \hat{x}} \to \Omega_{X / \fX}|_{\hat{x}}$.
The map $\Upsilon_h$ is associative in the following sense: if we are given a smooth morphism $h \colon X_1 \to X_2$ between algebraic spaces and a smooth morphism $q_2 \colon X_2 \to \fX$,
the following diagram commutes:
\begin{equation}\label{eq:Upsilon_assoc_stack}
\begin{aligned}
\xymatrix@C=100pt{
{\substack{\displaystyle h^{\red, *} q_2^{\red, * } K_{\fX, s}  
\displaystyle
\otimes h^{\red, *}K_{X_2/\fX}^{\otimes 2} \otimes K_{X_1 / X_2}^{\otimes 2} }}  \ar[d]_-{\cong} \ar[r]^-{h^{\red, *} \Upsilon_{q_2}}_-{\cong}
& {\substack{
\displaystyle
h^{\red, *} K_{X_2, s_2} \otimes K_{X_1 / X_2}^{\otimes 2}}} \ar[d]^-{\Upsilon_{h}}_-{\cong} \\
{\substack{\displaystyle
q_1^{\red, *} K_{\fX, s} \otimes K_{X_1 / \fX}^{\otimes 2} }} \ar[r]^-{\Upsilon_{q_1}}_-{\cong}
& {K_{X_1, s_1}}
}
\end{aligned}
\end{equation}
where we write $q_1 = q_2 \circ h$.

Let $(\fX, s)$ and $(\fY, t)$ be d-critical algebraic stacks.
The map $\Psi_{\fX, \fY}$ in \eqref{eq:Psi} is uniquely characterized by the property that the following diagram commutes for each $x\in \fX$ and $y \in \fY$:
\begin{equation}\label{eq:kappa_Psi_stack}
    \begin{aligned}
        \xymatrix@C=60pt{
    {K_{\fX, s} |_{{x}} \otimes K_{\fY, t}|_{y}} \ar[r]_-{\cong}^-{(\Psi_{\fX \times \fY})|_{(x, y)}} \ar[d]_-{\cong}^-{\kappa_{x} \otimes \kappa_{y}}
    & {(K_{\fX, s} \otimes K_{\fY, y})|_{(x, y)}} \ar[d]_-{\cong}^-{\kappa_{(x, y)}} \\
    {\rdet(\tau^{\geq 0} (\bL_{\fX, x}))^{\otimes 2} \otimes \rdet(\tau^{\geq 0} (\bL_{\fY, y}))^{\otimes 2}} \ar[r]^-{\cong}
    & {\rdet(\tau^{\geq 0} (\bL_{\fX, x} \oplus \bL_{\fY, y}))^{\otimes 2}.}
}
    \end{aligned}
\end{equation}
To prove the existence of $\Psi_{\fX, \fY}$ with the above commutativity property,
using the commutativity of \eqref{eq:Upsilon_kappa_stack}, we may assume that $\fX$ and $\fY$ are algebraic spaces covered by single critical charts $\fR = (R,  \id, U, f, i)$ and $\fS = (S, \id, V, g, j)$.
In this case, the isomorphism $\bL_{\fR} \boxplus \bL_{\fS} \simeq \bL_{\fR \times \fS}$ induces an isomorphism $\rdet(\bL_{\fR}) \boxtimes \rdet(\bL_{\fS}) \cong \rdet(\bL_{\fR \times \fS})$.
Using \eqref{eq:iota_kappa_stack}, we see that it induces the isomorphism $K_{\fX,s} \boxtimes K_{\fY, t} \cong K_{\fX \times \fY, s \boxplus t}$ with the desired property.

Assume that we are given smooth morphisms $q \colon X \to \fX$ and $r \colon Y \to \fY$ from algebraic spaces and given d-critical structures $s$ of $\fX$ and $t$ of $\fY$.
Then the commutativity of the diagrams \eqref{eq:Upsilon_kappa_stack} and \eqref{eq:kappa_Psi_stack} implies the commutativity of the following diagram:
\begin{equation}\label{eq:Upsilon_Psi_stack}
\begin{aligned}
    \xymatrix@C=100pt{
    \substack{\displaystyle(q^{\red, *} K_{\fX, s} \otimes K_{X / \fX})
    \displaystyle \boxtimes (r^{\red, *} K_{\fY, t} \otimes K_{Y / \fY}  ) }
    \ar[r]_-{\cong}^-{\Upsilon_q \boxtimes \Upsilon_r}
   \ar[d]_-{\cong}^-{\Psi_{\fX, \fY}} 
    & {K_{X, q^{\star} s} \boxtimes K_{Y, r^{\star} t} }
    \ar[d]_-{\cong}^-{\Psi_{X, Y}} \\
    {(q \times r)^* K_{\fX \times \fY, s \boxplus t} \otimes K_{X \times Y / \fX \times \fY}}
    \ar[r]_-{\cong}^-{\Upsilon_{ q \times r}}
    & {K_{X \times Y, q^{\star} s \boxplus r^{\star} t}.}
    }
\end{aligned}
\end{equation}

We now describe the virtual canonical bundle for underlying d-critical stacks of $(-1)$-shifted symplectic stacks.
Let $(X, \omega_{X})$ be a $(-1)$-shifted symplectic algebraic space and $(X^{\cl}, s)$ be its underlying d-critical algebraic space.
For an \'etale derived critical chart $\fR = (R, \eta, U, f, i)$,
we have a natural equivalence
\begin{equation}\label{eq:cot_Hess}
\bL_{\fX}|_{R^{\red}} \simeq \bL_{\fR^{\cl}}
\end{equation}
where $\fR^{\cl}$ denote the underlying \'etale d-critical chart.
Under this identification, the natural equivalence $\bT_{X}|_{R^{\red}} \simeq \bL_{X}|_{R^{\red}}[-1]$ induced from the $(-1)$-shifted symplectic structure corresponds to the map of complexes\footnote{See \cite[\S 1.3]{Con} for the sign convention on derived categories. In particular, the appearance of the sign in the top horizontal map follows from the sign convention of the dual complex, and the fact that the right vertical map is the identity follows from (1.3.6) in loc.cit. and the explicit description of the $(-1)$-shifted symplectic structure on the critical locus in \cite[Example 5.15]{bbj19}. }
\begin{equation}\label{eq:map_of_complexes}
\begin{aligned}
\xymatrix@C=50pt{
{\bT_{X} |_{R^{\red}}} \ar[d]^-{ \cdot \omega_{X}|_{R^{\red}} } \ar@{}[r]|-{\simeq}
& {[ T_U |_{R^{\red}}} \ar[r]^-{- \mathrm{Hess}(f)} \ar[d]^-{\id}
& {\Omega_U |_{R^{\red}}]}  \ar[d]^-{\id} \\
{\bL_{X}[-1] |_{R^{\red}} } \ar@{}[r]|-{\simeq}
& {[T_U |_{R^{\red}}} \ar[r]^-{- \mathrm{Hess}(f)}
& {\Omega_U |_{R^{\red}}].} 
}
\end{aligned}
\end{equation}
The equivalence \eqref{eq:cot_Hess} implies an isomorphism
\[
\Lambda_{X, \fR} \colon   \rdet(\bL_{X}) |_{R^{\red}}  \cong K_{X^{\cl}, s}|_{R^{\red}}. 
\]
We will show that this isomorphism glues to define a global isomorphism.
To see this, for a point $x\in X$ we construct an isomorphism
\[
\kappa^{\mathrm{der}}_{x} \colon \det(\bL_{X, x}) \cong \det(\tau^{\geq 0}(\bL_{X, x}))^{\otimes 2}
\]
by the following composition
\begin{align*}
\det(\bL_{X, x}) \cong \det(\tau^{\leq -1}(\bL_{X, x})) \otimes \det(\tau^{\geq 0}(\bL_{X, x})) 
&\cong \det(\tau^{\geq 0}(\bL_{X, x})^{\vee}[1]) \otimes \det(\tau^{\geq 0}(\bL_{X, x})) \\
&\cong \det(\tau^{\geq 0}(\bL_{X, x}))^{\otimes 2}.
\end{align*}
Here the second isomorphism is constructed using the isomorphism $(- \cdot \omega_{X})^{-1}|_{x} \colon \tau^{\leq -1}(\bL_{X, x}) \simeq \tau^{\geq 0}(\bL_{X, x})^{\vee}[1]$ and the third isomorphisms  from \eqref{eq:det_shift} and \eqref{eq:det_dual}.
By the construction and the description of the symplectic structure in $\eqref{eq:map_of_complexes}$,
we see that the following diagram commutes:
\begin{equation}\label{eq:kappa_Lambda_scheme}
\begin{aligned}
\xymatrix@C=100pt{
{\det(\bL_{X, x})} \ar[r]^-{\Lambda_{X, \fR}|_x}_-{\cong} \ar[d]^-{\cong}_-{\kappa_{x}^{\mathrm{der}}}
& {K_{X^{\cl}, s}|_{x}} \ar[d]_-{\cong}^-{\kappa_x} \\
{\det(\tau^{\geq 0}(\bL_{X, x}))^{\otimes 2} } \ar[r]^-{\cong}
& {\det(\tau^{\geq 0}(\bL_{X^{\mathrm{cl}}, x}))^{\otimes 2}. }
}
\end{aligned}
\end{equation}
In particular, the map $\Lambda_{X, \fR}$ does not depend on the choice of $\fR$ at each point, hence it glues to define an isomorphism
\[
\Lambda_{X} \colon   \rdet(\bL_{X}) \cong K_{X^{\cl}, s}.
\]

 Let $(\fX, \omega_{\fX})$ be a $(-1)$-shifted symplectic stack and $(\fX^{\cl}, s)$ be the underlying d-critical stack.
For a point $x \in \fX$, we define an isomorphism
\[
\kappa^{\mathrm{der}}_{x} \colon \det(\bL_{\fX, x}) \cong \det(\bL_{\fX^{\cl}, x})^{\otimes 2}
\]
exactly in the same manner as for algebraic spaces.
We have an isomorphism of one-dimensional vector spaces
\[
\Lambda_{\fX, x}^{\mathrm{pre}} \colon \rdet(\bL_{\fX, x}) \cong K_{\fX^{\cl}, s}|_{x} 
\]
given by $\kappa_{x}^{-1} \circ \kappa_{x}^{\mathrm{der}}$.
The following statement is proved in \cite[Theorem 3.18]{bbbbj15} (see also \cite[Theorem 4.9]{kin21} for a proof with careful choices of sign and Proposition \ref{prop:compare_index_line} for a closely related statement):

\begin{lem}{\cite[Theorem 3.18]{bbbbj15}}\label{lem:Lambda_stack}
    The map $x \mapsto \Lambda_{\fX, x}^{\mathrm{pre}}$ is algebraic and hence defines an isomorphism of line bundles 
    \[
    \Lambda_{\fX} \colon \rdet(\bL_{\fX}) \cong K_{\fX^{\cl}, s}.
    \]
\end{lem}

By the construction of the map $\Lambda_{\fX} \colon \rdet(\bL_{\fX}) \cong K_{\fX^{\cl}, s}$, the following diagram commutes:
 \begin{equation}\label{eq:Lambda_kappa_stack}
 \begin{aligned}
\xymatrix@C=100pt{
{\det(\bL_{\fX, x})} \ar[r]^-{\Lambda_{\fX}|_x}_-{\cong} \ar[d]_-{\cong}^-{\kappa_x^{\mathrm{der}}}
& {K_{\fX^{\cl}, s} |_{x} } \ar[d]_-{\cong}^-{\kappa_x} \\
{\det(\tau^{\geq 0}(\bL_{\fX, x}))^{\otimes 2} } \ar[r]^-{\cong}
& {\det(\tau^{\geq 0}(\bL_{\fX^{\mathrm{cl}}, x}))^{\otimes 2}.}
}
\end{aligned}
\end{equation}

 The map $\Lambda_{\fX}$ is compatible with the isomorphism $\Psi_{\fX, \fY}$:
 Namely, assume that we are given $(-1)$-shifted symplectic stacks
 $(\fX, \omega_{\fX})$ and $(\fY, \omega_{\fY})$ with underlying d-critical stacks
 $(\fX^{\cl}, s)$ and $(\fY^{\cl}, t)$.
 Then it follows from the commutativity of diagrams \eqref{eq:kappa_Psi_stack} and \eqref{eq:Lambda_kappa_stack} that the following diagram commutes:
 \begin{equation}\label{eq:Lambda_Psi}
     \begin{aligned}
         \xymatrix@C=80pt{
         {\rdet(\bL_{\fX}) \boxtimes \rdet(\bL_{\fY})  }
         \ar[r]^-{\cong}
         \ar[d]_-{\cong}^-{\Lambda_{\fX} \boxtimes \Lambda_{\fY}}
         & {\rdet(\bL_{\fX \times \fY})}
         \ar[d]_-{\cong}^-{\Lambda_{\fX \times \fY}} \\
         {K_{\fX^{\cl}, s} \boxtimes K_{\fY^{\cl}, t}}
         \ar[r]_-{\cong}^-{\Upsilon_{\fX^{\cl}, \fY^{\cl}}}
         & {K_{\fX^{\cl} \times \fY^{\cl}, s \boxplus t}.}
         }
     \end{aligned}
 \end{equation}

\subsection{Orientation for d-critical stacks}

Let $(\fX, s)$ be a d-critical stack. 
An \defterm{orientation of $(\fX, s)$} is a choice of a graded line bundle $\cL \in \sPic(\fX^{\red})$ and an isomorphism $\cL^{\otimes 2} \cong K_{\fX, s}$.
An isomorphism of orientations is defined in the same manner as Definition \ref{defin:isom_orientation}.
By abuse of notation, we simply say that $o$ is an orientation for $(\fX, s)$.
Let $q \colon X \to \fX$ be a smooth morphism from an algebraic space $X$.
For an orientation $o \colon \cL^{\otimes 2} \cong K_{\fX, s}$,
we define an orientation $q^{\star}o$ by the composition
\[
q^{\star} o \colon (\cL \otimes \rdet(\Omega_{X / \fX}))^{\otimes 2} \xrightarrow[\cong]{o} K_{\fX, s} \otimes \rdet(\Omega_{X / \fX})^{\otimes 2} \xrightarrow[\cong]{\Upsilon_q} K_{X, q^{\star}s}.
\]
Assume further that we are given a smooth morphism $h \colon Y \to X$.
Then the diagram \eqref{eq:Upsilon_assoc_stack} implies a natural equivalence
\begin{equation}\label{eq:ori_pull_assoc}
h^{\star} q^{\star} o \cong (q \circ h)^{\star} o.
\end{equation}

Let $(\fX, s)$ and $(\fY, t)$ be d-critical stacks equipped with orientations 
$o_{\fX} \colon \cL_{\fX}^{\otimes 2} \cong K_{\fX, s}$ and
$o_{\fY} \colon \cL_{\fY}^{\otimes 2} \cong K_{\fY, t}$.
Then we define an orientation $o_{\fX} \boxtimes o_{\fY}$ by
\begin{equation}\label{eq:product_ori}
o_{\fX} \boxtimes o_{\fY} \colon (\cL_{\fX} \boxtimes \cL_{\fY})^{\otimes 2}
\xrightarrow[\cong]{o_{\fX} \boxtimes o_{\fY}} K_{\fX, s} \boxtimes K_{\fY, t}
\xrightarrow[\cong]{\Psi_{\fX, \fY}} K_{\fX \times \fY, s \boxplus t}.
\end{equation}
Assume further that we are given another oriented d-critical stack $(\fZ, u, o_{\fZ})$.
Then it easily follows from \eqref{eq:kappa_Psi_stack} that the isomorphism $\Psi_{\fX, \fY}$ is associative hence there exists a natural isomorphism
\begin{equation}\label{eq:ori_prod_assoc}
    (o_{\fX} \boxtimes o_{\fY}) \boxtimes o_{\fZ} \cong o_{\fX} \boxtimes (o_{\fY} \boxtimes o_{\fZ}).
\end{equation}
Assume that we are given smooth morphisms $q \colon X \to \fX$ and $r \colon Y \to \fY$.
Then it follows from \eqref{eq:Upsilon_Psi_stack} that there exists a natural isomorphism
\begin{equation}\label{eq:ori_prod_pull}
    q^{\star}o_{\fX} \boxtimes r^{\star} o_{\fY} \cong (q \times r)^{\star} o_{\fX} \boxtimes o_{\fY}.
\end{equation}

Let $(\fX, \omega_{\fX}, o)$ be an oriented $(-1)$-shifted symplectic stack 
with underlying d-critical structure $s$.
By using Lemma \ref{lem:Lambda_stack}, $o$ induces an orientation for $(\fX^{\cl}, s)$ which is denoted by $o^{\cl}$.
The formation of classical truncation of orientations commutes with the product:
Namely, given oriented $(-1)$-shifted symplectic stacks $(\fX, \omega_{\fX}, o_{\fX})$ and $(\fY, \omega_{\fY}, o_{\fY})$, there exists a natural isomorphism
\begin{equation}\label{eq:ori_cl_prod}
    (o_{\fX} \boxtimes o_{\fY})^{\cl} \cong o_{\fX}^{\cl} \boxtimes o_{\fY}^{\cl}.
\end{equation}
This is a consequence of the commutativity of the diagram \eqref{eq:Lambda_Psi}.
Further, it is clear from the construction that this isomorphism is associative under the identification \eqref{eq:ori_prod_assoc}.

We now introduce a d-critical version of the canonical orientation discussed in Example \ref{ex:can_ori}.
Let $U$ be a smooth scheme, $f \colon U \to \bA^1$ be a regular function with $f |_{\Crit(f)^{\red}} = 0$ and $X = \Crit(f)^{\cl}$. Set $s \coloneqq f + I_{X / U}^{2}$ and
$\fR^{\mathrm{taut}} = (X, \id, U, f, i)$ be the tautological \'etale d-critical chart.
Then the fibre sequence
\begin{equation}\label{eq:ambient_fibre_seq}
\Omega_U |_{X^{\red}} \to \bL_{\fR^{\mathrm{taut}}} \to T_U |_{X^{\red}}[1]
\end{equation}
together with \eqref{eq:det_shift} and \eqref{eq:det_dual} induce an isomorphism
$\bar{o}^{\mathrm{can}}_{X} \colon K_U |_{X^{\red}}^{\otimes 2} \cong K_{X, s}$.
We set 
\begin{equation}\label{eq:can_ori}
o^{\mathrm{can}}_{X} \coloneqq (-1)^{\dim U} \cdot \bar{o}^{\mathrm{can}}_{X} 
\end{equation}
and call it the \defterm{canonical orientation} for $(X, s)$ with respect to the critical locus description $X = \Crit(f)^{\cl}$.
This is compatible with Example \ref{ex:can_ori}: Namely, there exists a natural isomorphism
\begin{equation}\label{eq:can_ori_cl}
(o_{\Crit(f)}^{\can})^{\cl} \cong o_{\Crit(f)^{\cl}}^{\can}.
\end{equation}
\begin{rmk}\label{rmk:sign_canori}
    The sign insertion seems to be unnatural, but this choice is more convenient for us in that it reduces the number of signs appearing in the paper.
This is because, we did and will control the behaviour of the virtual canonical bundles using the isomorphism $\kappa_{x}$, and it is defined using the fibre sequence \eqref{eq:kappa_fibre_seq}.
On the other hand, the isomorphism $\bar{o}^{\mathrm{can}}_{X}$ is constructed using the fibre sequence \eqref{eq:ambient_fibre_seq}. Therefore, it is natural to swap $K_U$ when defining the canonical orientation, which corresponds to multiplying $(-1)^{\dim U}$ since we twist the sign of the brading isomorphism: see \S \ref{ssec:det}.
As an example, if we do not insert this sign, the isomorphism of underlying line bundles in \eqref{eq:can_ori_pull} needs to be modified by $\sqrt{-1}^{\pm \dim \Omega_{U_1/U_2}}$ to define an isomorphism of orientations. 
\end{rmk}

Assume that we are given a smooth morphism $H \colon U_1 \to U_2$ between smooth schemes and  a regular function $f_2$ on $U_2$ with $f_2 |_{\Crit(f_2)^{\red}} = 0$.
Write $f_1 \coloneqq f_2 \circ H$.
Set $X_1 \coloneqq \Crit(f_1)$ and $X_2 \coloneqq \Crit(f_2)$ equipped with the natural d-critical structures and $h \colon X_1 \to X_2$ be the restriction of $H$.
Then we see that the natural morphism $H^*K_{U_2} \otimes K_{U_1 / U_2} \cong K_{U_1}$ induces an 
isomorphism
\begin{equation}\label{eq:can_ori_pull}
b_H \colon o^{\can}_{X_1} \cong h^{\star} o^{\can}_{X_2}.
\end{equation}
See the proof of the commutativity of the diagram \eqref{eq:Upsilon_kappa_Ind_pre} for a related discussion.
This diagram is associative in the following sense: Assume further that we are given a smooth morphism $G \colon U_0 \to U_1$, set $X_0 = \Crit(f_1 \circ G)$ and $g \colon X_0 \to X_1$ be the restriction of $G$.
Then the following diagram commutes:
\begin{equation}\label{eq:can_ori_pull_assoc}
    \begin{aligned}
        \xymatrix@C=60pt{
        {o^{\can}_{X_0}}
        \ar[r]_-{\cong}^-{\eqref{eq:can_ori_pull}}
        \ar[d]_-{\cong}^-{\eqref{eq:can_ori_pull}} 
        & {g^{\star} o^{\can}_{X_1}} 
        \ar[d]_-{\cong}^-{\eqref{eq:can_ori_pull}} \\
        {(h \circ g)^{\star} o^{\can}_{X_2}}
        \ar[r]_-{\cong}^-{\eqref{eq:ori_pull_assoc}}
        & {g^{\star} h^{\star} o^{\can}_{X_2}}
        }
    \end{aligned}
\end{equation}

Assume now that we are given smooth schemes $U$ and $V$ and regular functions $f$ on $U$ and $g$ on $V$ with $f |_{U^{\red}} = 0$ and $g |_{V^{\red}} = 0$.
Set $X = \Crit(f)$ and $Y = \Crit(g)$ equipped with the natural d-critical structures.
Then there exists a natural isomorphism
\begin{equation}\label{eq:can_ori_prod}
    \begin{aligned}
        o^{\can}_{X} \boxtimes o^{\can}_{Y} \cong o^{\can}_{X \times Y}. 
    \end{aligned}
\end{equation}
Assume further that we are given a d-critical scheme $Z = \Crit(h)$.
Then by construction, the following diagram commutes:
\begin{equation}\label{eq:can_ori_prod_assoc}
    \begin{aligned}
        \xymatrix{
        {(o^{\can}_{X} \boxtimes o^{\can}_{Y}) \boxtimes o^{\can}_{Z}}
        \ar[rr]_-{\cong}^-{\eqref{eq:ori_prod_assoc}}
        \ar[d]_-{\cong}^-{\eqref{eq:can_ori_prod}}
        & {} 
        & {o^{\can}_{X} \boxtimes (o^{\can}_{Y} \boxtimes o^{\can}_{Z})}
        \ar[d]_-{\cong}^-{\eqref{eq:can_ori_prod}} \\
        {o^{\can}_{X \times Y} \boxtimes o^{\can}_{Z}}
        \ar[r]_-{\cong}^-{\eqref{eq:can_ori_prod}}
        & {o^{\can}_{X \times Y \times Z}}
        & {o^{\can}_{X} \boxtimes o^{\can}_{Y \times Z}}
        \ar[l]^-{\cong}_-{\eqref{eq:can_ori_prod}}.
        }
    \end{aligned}
\end{equation}
Assume that we are given smooth morphisms $H_{U} \colon U' \to U$ and $H_V \colon V' \to V$.
Set $X' = \Crit(f \circ H_U)$ and $Y' = \Crit(g \circ H_V)$ and
$h_X \colon X' \to X$ and $h_Y \colon Y' \to Y$ be the restrictions of $H_U$ and $H_V$.
Then the following diagram commutes:
\begin{equation}\label{eq:can_ori_prod_pull}
    \begin{aligned}
    \xymatrix@C=80pt{
    {h_{X}^{\star} o^{\can}_{X} \boxtimes h_{Y}^{\star} o^{\can}_Y }
    \ar[r]_-{\cong}^-{\eqref{eq:can_ori_pull}}
    \ar[d]_-{\cong}^-{\eqref{eq:can_ori_prod}}
    & {o^{\can}_{X'} \boxtimes o^{\can}_{Y'}}
    \ar[d]_-{\cong}^-{\eqref{eq:can_ori_prod}} \\
    {(h_X \times h_Y)^{\star} o^{\can}_{X \times Y} }
    \ar[r]_-{\cong}^-{\eqref{eq:can_ori_pull}}
    & {o^{\can}_{X' \times Y'}.}
    }
    \end{aligned}
\end{equation}

\section{Donaldson--Thomas perverse sheaves}

Here we briefly recall the definition and some properties of the vanishing cycle functor (\S \ref{ssec:vanfunct}) and the construction of the Donaldson--Thomas perverse sheaves (\S \ref{ssec:DT_perv}) associated with oriented d-critical stacks following \cite{bbdjs15} which is a global version of the vanishing cycle complexes.
We also prove in \S \ref{ssec:TS} the Thom--Sebastiani theorem for the Donaldson--Thomas perverse sheaves, whose detailed proof seems to be not available in the literature.

\subsection{Vanishing cycle functors}\label{ssec:vanfunct}

In this section, we recall the theory of the vanishing cycle complexes.
Let $U$ be a smooth separated algebraic space and $f \colon U \to \bA^1$ be a regular function. We set 
\[
U_{\leq 0} \coloneqq f^{-1} (\bA^1_{\leq 0}), \quad U_0 \coloneqq f^{-1}(0)
\]
where $\bA^1_{\leq 0} \coloneqq \{ z \in \bA^1 \mid \Re(z) \leq 0 \}$. Define the vanishing cycle functor $\varphi_f \colon D(U) \to D(U_0)$ by 
\[
\varphi_{f} \coloneqq (U_0 \hookrightarrow U_{\leq 0})^* (U_{\leq 0} \hookrightarrow U)^!.
\]
It is proved in \cite[Corollary 4.2.6, Theorem 4.2.8]{Achar} that the functor $\varphi_f$ preserves constructible complexes and perverse sheaves.
For simplicity, we write
\begin{equation}\label{eq:we_are_shifting_vanishing}
\varphi_{f} \coloneqq \varphi_f(\bQ_U[\dim U]).
\end{equation}
One can easily show that the complex $\varphi_f$ is supported on the critical locus.

The definition of the vanishing cycle functor can be extended to functions on an Artin stack (hence to non-separated spaces).
Let $\fU$ be a smooth Artin stack.
We set $\fU_{> 0} \coloneqq f^{-1}(\bA^1_{> 0})$, which is regarded as a complex analytic substack of $\fU$.
Then we define the vanishing cycle functor by
\[
\varphi_{f} \coloneqq \fib((\fU_0 \hookrightarrow \fU)^* \to (\fU_0 \hookrightarrow \fU)^* (\fU_{> 0 } \hookrightarrow \fU)_* (\fU_{> 0 } \hookrightarrow \fU)^*).
\]
It is clear from the definition that this definition matches with the previous definition of the vanishing cycle functor for functions on a smooth separated scheme.

Assume that we are given a morphism $h \colon U \to V$ between smooth algebraic spaces and a regular function $f$ on $V$. Set $f_0 \colon U_0 \to V_0$ be the induced morphism on the zero locus. Then by construction, we have natural transforms
\begin{align}\label{eq:van_natural}
\begin{aligned}
   h_{0, !} \circ \varphi_{f \circ h} \to \varphi_f \circ h_!&, \quad  
   h_0^* \circ \varphi_f \to \varphi_{f \circ h} \circ h^* \\
   \varphi_f \circ h_* \to  h_{0, *} \circ \varphi_{f \circ h}&, \quad
   \varphi_{f \circ h} \circ h^! \to h_0^! \circ \varphi_f. 
   \end{aligned}
\end{align}
The proper base change theorem implies that left maps are invertible if $h$ is proper.
The smooth base change theorem implies that right maps are invertible if $h$ is smooth.
The same base change properties hold more generally for morphisms between stacks.

Let $f \colon U \to \bA^1$ and $g \colon V \to \bA^1$ be regular functions on smooth separated algebraic spaces and $\cF \in D^b_c(U)$ and $\cG \in D^b_c(U)$ be constructible complexes.
We define a function $f \boxplus g$ on $U \times V$ by $(f \boxplus g)(x, y) = f(x) +g(y)$.
Then there exist natural morphisms
\begin{align*}
\varphi_{f}(\cF) \boxtimes \varphi_{g}(\cG) &= (U_0 \hookrightarrow U_{\leq 0})^* (U_{\leq 0} \hookrightarrow U)^! \cF \boxtimes (V_0 \hookrightarrow V_{\leq 0})^* (V_{\leq 0} \hookrightarrow V)^! \cG \\
&\to (U_0 \times V_0 \hookrightarrow U_{\leq 0} \times V_{\leq 0})^* (U_{\leq 0} \times V_{\leq 0} \hookrightarrow U \times V)^! (\cF \boxtimes \cG)
\end{align*}
\begin{align*}
    \varphi_{f \boxplus g}(\cF \boxtimes \cG)|_{U_ 0 \times V_0} &\cong (U_0 \times V_0 \hookrightarrow (U \times V)_{\leq 0})^* ((U \times V)_{\leq 0} \hookrightarrow U \times V)^! \\
    & \leftarrow (U_0 \times V_0 \hookrightarrow U_{\leq 0} \times V_{\leq 0})^* (U_{\leq 0} \times V_{\leq 0} \hookrightarrow U \times V)^! (\cF \boxtimes \cG)
\end{align*}
which are shown to be isomorphisms in \cite[Proposition 1.4]{masts} and \cite[Lemma 1.2]{masts} respectively.
In particular, we obtain an isomorphism
\begin{equation}\label{eq:Thom_Sebastiani_functor}
\TS \colon \varphi_{f}(\cF) \boxtimes \varphi_{g}(\cG) \cong \varphi_{f \boxplus g}(\cF \boxtimes \cG)|_{U_0 \times V_0}. 
\end{equation}
Assume further that we are given a regular function $h \colon W \to \bA^1$ on a smooth separated algebraic space and a constructible complex $\cH \in D^b_c(W)$.
Then it is clear from the construction that the following diagram commutes:
\begin{equation}\label{eq:Thom_Sebastiani_assoc_functor}
\begin{aligned}
    \xymatrix@C=80pt{
    {\varphi_{f}(\cF) \boxtimes \varphi_{g}(\cG) \boxtimes \varphi_{h}(\cH)} \ar[r]^-{(\TS, \id)}_-{\cong} \ar[d]_-{\cong}^-{(\id, \TS)}
    & {\varphi_{f \boxplus g}(\cF \boxtimes \cG) |_{U_0 \times V_0} \boxtimes \varphi_{h}(\cH)} \ar[d]_-{\cong}^-{\TS} \\
    {\varphi_{f}(\cF) \boxtimes \varphi_{g \boxplus h}(\cG \boxtimes \cH)|_{V_0 \times W_0}} \ar[r]_-{\cong}^-{\TS}
    & {\varphi_{f \boxplus g \boxplus h}(\cF \boxtimes \cG \boxtimes \cH)|_{U_0 \times V_0 \times W_0}.}
    }
\end{aligned}
\end{equation}

Assume that we are given a smooth morphism $h \colon U \to V$ and $h' \colon U' \to V'$ and a regular function $f$ on $V$ and $f'$ on $V'$. Take a constructible complex $\cF \in D^b_c(V)$ and $\cF' \in D^b_c(V')$. Then by the construction of the Thom--Sebastiani isomorphism, we have the following commutative diagram:
\begin{equation}\label{eq:sm_pull_TS_functor}
\begin{aligned}
\xymatrix@C=80pt{
{h_0^* \varphi_{f}(\cF) \boxtimes h_0'^* \varphi_{f'}(\cF') } \ar[r]_-{\cong}^-{\TS} 
\ar[d]_-{\cong}^-{\eqref{eq:van_natural}}
& {(h_0 \times h'_0)^* \varphi_{f \boxplus f'}(\cF \boxtimes \cF') } \ar[d]_-{\cong}^-{\eqref{eq:van_natural}} \\
{\varphi_{f \circ h}(h^* \cF) \boxtimes \varphi_{f' \circ h'}({h'} ^*\cF')} \ar[r]_-{\cong}^-{\TS}
& {\varphi_{f \circ h \boxplus f' \circ h'}((h \times h')^*(\cF \boxtimes \cF')).}
}
\end{aligned}
\end{equation}

It can be easily shown that the Thom--Sebastiani theorem \eqref{eq:Thom_Sebastiani_functor} holds more generally for functions on smooth Artin stacks by pulling back to a smooth chart and using \eqref{eq:van_natural}.

For later use, we compute the vanishing cycle complex with respect to quadratic functions.
Let $ q = x_1 y_1 + \cdots + x_n y_n \colon \bA^{2n} \to \bA^1$ be a quadratic function.
Set 
\begin{align*}
V_x &\coloneqq \{ (x_1, \ldots, y_n) \in \bA^{2n} \mid y_1 = \cdots = y_n = 0 \}, \\
V_{y} &\coloneqq \{ (x_1, \ldots, y_n) \in \bA^{2n} \mid x_1 = \cdots = x_n = 0 \}.
\end{align*}
Then the inclusions $i_x \colon V_{x} \hookrightarrow \bA^{2n}$ and $i_y \colon V_{y} \hookrightarrow \bA^{2n}$
and the exchange map \eqref{eq:van_natural} implies natural morphisms of complexes
\[
\tilde{c}_x \colon i_{x} ^* \varphi_{q} \to \varphi_0(\bQ_{V_{x}}[2n]) = \bQ_{V_{x}}[2n], \quad \tilde{c}_y \colon i_{y} ^* \varphi_{q} \to \varphi_0(\bQ_{V_{y}}[2n]) = \bQ_{V_{y}}[2n].
\]
Note that $\varphi_q$ has its support contained in the origin. 
Therefore, by taking the compactly supported cohomology, we obtain natural morphisms
\begin{equation}\label{eq:quad_trivialize}
 c_x \colon  \varphi_{q} |_{0} \to \bQ_{0}, \quad c_y \colon  \varphi_{q} |_{0} \to \bQ_{0}.  
\end{equation}
It follows from \cite[Theorem A.1]{dav17} that $c_x$ and $c_y$ are isomorphisms.
We note that $c_x$ and $c_y$ differ by the sign $(-1)^n$ \footnote{This can be shown e.g. by constructing a proper homotopy between $V_x$ and $V_y$ inside $q^{-1}(0)$ which reverses the orientation induced from the complex structure. See the proof of \cite[Proposition 2.24]{kk24pre} for a related discussion.}.
This discrepancy corresponds to the ambiguity of the trivialization of the vanishing cycle complexes with respect to the quadratic function explained in \cite[Example 2.14]{bbdjs15}.
Assume further that we are given a quadratic function $q' = z_1w_1 + \cdots + z_m w_m \colon \bA^{2m} \to \bA^1$. We let $V_{w} \subset \bA^{2m}$ and $V_{y, w} \subset \bA^{2n + 2m}$ in the same manner as $V_y$.
Then it is clear from the construction that the following diagram commutes:
\begin{equation}\label{eq:c_associative}
\begin{aligned}
\xymatrix@C=80pt{
{\varphi_{q}|_{0}} \boxtimes {\varphi_{q'}|_{0}} \ar[d]_-{\cong}^-{(c_y, c_w)} \ar[r]_-{\cong}^-{\TS}
&\varphi_{q \boxplus q'}|_{0} \ar[d]^-{c_{y, w}}_-{\cong} \\
{\bQ_0 \boxtimes \bQ_0} \ar[r]^-{\cong}
& {\bQ_0}.
}
\end{aligned}
\end{equation}

\subsection{Donaldson--Thomas perverse sheaf}\label{ssec:DT_perv}

We now briefly explain the construction of the Donaldson--Thomas perverse sheaf associated with oriented d-critical algebraic stacks following \cite[\S 6]{bbdjs15} and \cite[\S 4]{bbbbj15}.
Recall that for a d-critical scheme $(X, s)$ and an \'etale d-critical chart $\fR = (R, \eta, U, f, i)$,
we have constructed in \eqref{eq:can_ori} the canonical orientation
\[
o^{\can}_{\fR} \colon K_{U}^{\otimes 2} |_{R^{\red}} \cong K_{R, \eta^{\star } s}
\]
for $(R, \eta^{\star} s)$. 
For a given orientation $(\cL, o)$ for $(X, s)$ such that the parity of $\dim U$ and the grading of $\cL$ agrees, we define a $\bZ / 2 \bZ$-local system $Q_{\fR}^o$ on $R$ parametrizing local isomorphisms of orientations $ o^{\can}_{\fR}  \cong o |_{R} $.
Assume that we are given an embedding
\begin{equation}\label{eq:standard_embedding}
    \fR = (R, \eta, U, f, i) \hookrightarrow \fR' = (R, \eta, U \times \bA^{2n}, f \boxplus (x_1y_1 + \cdots + x_n y_n), (i, 0)).
\end{equation}
Then we construct an isomorphism $o_{\fR}^{\can} \cong o_{\fR'}^{\can}$ by
\begin{equation}\label{eq:isom_canori_quadratic}
   K_U  \cong K_{U \times \bA^{2n}}, \quad v \mapsto  v \wedge (dx_1 \wedge \cdots \wedge d{x_n} \wedge d{y_n} \wedge \cdots \wedge dy_1 ).
\end{equation}
This induces an isomorphism of local systems
\[
b_{y} \colon Q_{\fR'}^o \cong Q_{\fR}^{o}.
\]
Note that we can define $b_{x}$ by swapping $x$-coordinates and $y$-coordinates.
We have an identity $b_x = (-1)^n b_y$.
Consider the following d-critical chart
\[
\fR'' = (R, \eta, U \times \bA^{2n + 2m}, f \boxplus (x_1y_1 + \cdots + x_n y_n) \boxplus (z_1w_1 + \cdots + z_m w_m), (i, 0)).
\]
Then we have an identity
\begin{equation}\label{eq:b_assoc}
b_{y, w} = b_{y} \circ b_{w}.
\end{equation}

We now explain the construction of the DT perverse sheaf.
It is shown \cite[Theorem 6.9]{bbdjs15} that there exists a perverse sheaf called the \defterm{Donaldson--Thomas (DT) perverse sheaf}
\[
\varphi_{X, s, o} \in \Perv(X)
\]
with a choice of an isomorphism 
\begin{equation}\label{eq-omega-def}
\omega_{\fR} \colon \varphi_{X, s, o} \cong \varphi_{f} |_{R} \otimes_{\bZ / 2 \bZ} Q_{\fR}^o
\end{equation}
which satisfies the following properties:
\begin{enumerate}
    \item If we are given an \'etale morphism $(h_0, h) \colon \fR = (R, U, f, \eta, i) \to \fR' = (R', U', f', \eta', i') $,
    the following diagram commutes:
    \[
    \xymatrix@C=60pt{
    {\varphi_{X, s, o} |_{R}} \ar[r]^-{\omega_{\fR'} |_{R}}_-{\cong} \ar@{=}[d] 
    & {\varphi_{f'} |_{R} \otimes_{\bZ / 2 \bZ} Q_{\fR'}^o} \ar[d]^-{\cong} \\
    {\varphi_{X, s, o} |_{R}} \ar[r]_-{\cong}^-{\omega_{\fR}}
    & {\varphi_f|_{R} \otimes_{\bZ / 2 \bZ} Q_{\fR}^o}
    }
    \]
    where the right vertical map is defined in the canonical manner.

    \item If we are given an embedding \eqref{eq:standard_embedding},
    then the following diagram commutes\footnote{Note that our choice for the morphism $c_y$ and $b_y$ are different from the one in \cite{bbbbj15} up to some sign. We do not compute the explicit difference here. However, since our choice of $b_y$ and $c_y$ are associative with respect to the composition of embeddings, the associativity of the Thom--Sebastiani isomorphism \eqref{eq:Thom_Sebastiani_assoc_functor} implies that the difference of these signs are associative. In particular, one can glue the vanishing cycle complexes using our choice of sign. See \S \ref{ssec-sign-dep} for more explanations and an advantage of our sign conventions.}:
    \begin{equation}\label{eq-omega-commutes}
    \xymatrix@C=70pt{
    {\varphi_{X, s, o} |_{R}} \ar[d]^-{\omega_{\fR'}}_-{\cong} \ar[r]^-{\omega_{\fR}}_-{\cong}
    & {\varphi_{f} |_{R} \otimes_{\bZ / 2 \bZ} Q^o_{\fR} } \ar[d]_-{\cong}^-{(c_{y}^{-1}, b_{y}^{-1})} \\
    {\varphi_{f \boxplus (x_1y_1 + \cdots + x_ny_n) }  |_R \otimes_{\bZ / 2 \bZ} Q^{o}_{\fR'} } \ar[r]_-{\cong}^-{\TS^{-1}}
    & {\varphi_f |_R \boxtimes \varphi_{x_1y_1 + \cdots + x_ny_n} \otimes_{\bZ / 2 \bZ} Q^{o}_{\fR'}. } 
    }
    \end{equation}
\end{enumerate}
Note that the DT perverse sheaf $\varphi_{X, s, o}$ is uniquely characterized up to unique isomorphism by the above property.
If there is no confusion, we simply write $\varphi_{X} = \varphi_{X, s, o}$.

Assume that we are given a smooth morphism $h \colon X_1 \to X_2$, $s_2$ be a d-critical structure on $X_2$ and $o_2$ be an orientation for $(X_2, s_2)$. We set $s_1 \coloneqq h^{\star} s_2$ and $o_1 \coloneqq h^{\star} o_2$.
Then we can construct a natural isomorphism
\[
\Theta_{h} = \Theta_{h, s_2, o_2} \colon   h^* \varphi_{X_2, s_2, o_2} \cong \varphi_{X_1, s_1, o_1} [- \dim h]
\]
characterized by the following property: If we are given \'etale d-critical charts $\fR_j = (R_j, \eta_j, U_j, f_j, i_j)$ for $(X_j, s_j)$ and a morphism $(H_0, H) \colon \fR_1 \to \fR_2$ as in Proposition \ref{prop:sm_chart},
the following diagram commutes:
\[
\xymatrix@C=100pt{
{\varphi_{X_1, s_1, o_1}} \ar[r]^-{\omega_{\fR_1}}_-{\cong} \ar[d]_-{\Theta_{h}^{-1} |_{R_1}}^-{\cong}
& {\varphi_{f_1} |_{R_1} \otimes_{\bZ / 2 \bZ}} Q_{\fR_1}^{o_1} |_{R_1} \ar[d]^-{\cong}_-{(\Theta_{H}^{-1}, b_H)} \\
{ h^* \varphi_{X_2, s_2, o_2}[\dim h] } \ar[r]^-{\omega_{\fR_2}}_-{\cong}
& { (H_0^{*} \varphi_{f_2})|_{R_1} \otimes_{\bZ / 2 \bZ} H_0^* Q_{\fR_2}^{o_2}[\dim h]}
}
\]
where $\Theta_{H}$ is the morphism constructed in \eqref{eq:van_natural} and the map $b_H$ is given using \eqref{eq:can_ori_pull}. 
By using \eqref{eq:can_ori_pull_assoc}, one sees that the isomorphism $\Theta_{h}$ is associative with respect to the composition of smooth morphisms:
Namely, if we are given a smooth morphism $g \colon X_0 \to X_1$ and if we set $s_0 \coloneqq g^{\star} s_1$ and $o_0 \coloneqq g^{\star} o_{1}$, the following diagram commutes:
\begin{equation}\label{eq:assoc_pull_van}
\begin{aligned}
\xymatrix@C=80pt{
{g^* h^* \varphi_{X_2, s_2, o_2} } \ar[r]_-{\cong}
\ar[d]_-{\cong}^-{\Theta_{h }}
& { (h \circ g)^{*} \varphi_{X_2, s_2, o_2} 
\ar[d]_-{\cong}^-{\Theta_{h\circ g}}} \\
{g^* \varphi_{X_1, s_1, o_1}[- \dim h]}  \ar[r]_-{\cong}^-{\Theta_g}
& {\varphi_{X_0, s_0, o_0}[-\dim h \circ g ].}  
}
\end{aligned}
\end{equation}

Now take an oriented d-critical stack $(\fX, s, o)$.
Then the commutativity of the diagram \eqref{eq:assoc_pull_van} implies that there exists a perverse sheaf
\[
\varphi_{\fX} = \varphi_{\fX, s, o} \in \Perv(\fX)
\]
with a choice of an isomorphism $\Theta_q \colon q^* \varphi_{\fX, s, o} \cong \varphi_{X, q^{\star} s, q^{\star} o}[- \dim q]$ for a smooth morphism $q \colon X \to \fX$ from an algebraic space, uniquely characterized up to unique isomorphism by the following property:
If we are given a smooth morphism $h \colon X_1 \to X_2$ between algebraic spaces and a smooth morphism $q_2 \colon X_2 \to \fX$, the following diagram commutes:
\begin{equation}\label{eq:assoc_pull_van_stack}
\begin{aligned}
\xymatrix@C=80pt{
{h^* q_2^* \varphi_{\fX, s, o}}  \ar[d]_-{\cong}^-{\Theta_{q_2}} \ar[r]^{\cong}
& {q_1^{*} \varphi_{\fX, s, o}} \ar[d]_-{\cong}^-{\Theta_{q_1}} \\
{h^* \varphi_{X_2, s_2, o_2}[- \dim q_2]} \ar[r]_-{\cong}^-{\Theta_{h}}
& {\varphi_{X_1, s_1, o_1}[- \dim q_1]}
}
\end{aligned}
\end{equation}
where we set $q_1 \coloneqq q_2 \circ h$, $s_1 \coloneqq q_1^{\star} s$, $s_2 \coloneqq q_2^{\star} s$, $o_1 \coloneqq q_1^{\star} o$ and $o_2 \coloneqq q_2^{\star} o$.
It is shown in \cite[Theorem 4.8]{bbbbj15} that the DT perverse sheaf is Verdier self-dual; namely,
there exists a natural isomorphism
\begin{equation}\label{eq:dual_DT}
\bD_{\fX}(\varphi_{\fX, s, o}) \cong \varphi_{\fX, s, o}
\end{equation}
where $\bD_{\fX}$ denotes the Verdier dualizing functor.
For an oriented $(-1)$-shifted symplectic stack $(\fX, \omega_{\fX}, o)$ whose underlying d-critical structure is denoted by $s$,
we set
\[
\varphi_{\fX, \omega_{\fX}, o} \coloneqq \varphi_{\fX^{\cl}, s, o^{\cl}}.
\]

\subsection{Thom--Sebastiani theorem for DT perverse sheaves}\label{ssec:TS}

The aim of this section is to prove the Thom--Sebastiani theorem for Donaldson--Thomas perverse sheaves on d-critical stacks and $(-1)$-shifted symplectic stacks.
We first prove it for algebraic spaces:

\begin{prop}\label{prop:TS_DT_space}
    Let $(X_1, s_1, o_1)$ and $(X_2, s_2, o_2)$ be oriented d-critical algebraic spaces.
    Then there exists a natural isomorphism
    \[
    \TS \colon \varphi_{X_1, s_1, o_1} \boxtimes \varphi_{X_2, s_2, o_2} \cong \varphi_{X_1 \times X_2, s_1 \boxplus s_2, o_1 \boxtimes o_2}.
    \]
    Further, if we are given another oriented d-critical algebraic space $(X_3, s_3, o_3)$, the following diagram commutes:
    \begin{equation}\label{eq:TS_assoc_space}
    \begin{aligned}
    \xymatrix@C=80pt{
    {\varphi_{X_1, s_1, o_1} \boxtimes \varphi_{X_2, s_2, o_2} \boxtimes \varphi_{X_3, s_3, o_3}} \ar[r]_-{\cong}^-{(\TS, \id)} \ar[d]_-{\cong}^-{(\id, \TS)}
    & {\varphi_{X_1 \times X_2, s_1 \boxplus s_2, o_1 \boxtimes o_2} \boxtimes \varphi_{X_3, s_3, o_3}} \ar[d]_-{\cong}^-{\TS} \\
    {\varphi_{X_1, s_1, o_1} \boxtimes \varphi_{X_2 \times X_3, s_2 \boxplus s_3, o_2 \boxtimes o_3}} \ar[r]_-{\cong}^-{\TS}
    & {\varphi_{X_1 \times X_2 \times X_3, s_1 \boxplus s_2 \boxplus s_3, o_1 \boxtimes o_2 \boxtimes o_3}}.
    }
    \end{aligned}
    \end{equation}
\end{prop}

\begin{proof}
    Take \'etale d-critical charts $\fR_1 = (R_1, \eta_1, U_1, f_1, i_1)$ for $X_1$ and $\fR_2 = (R_2, \eta_2, U_2, f_2, i_2)$.
    As we have seen in \eqref{eq:can_ori_prod}, there exists a natural isomorphism
    \begin{equation}\label{eq:ori_TS}
    o_{\fR_1}^{\can} \boxtimes o_{\fR_2}^{\can} \cong o_{\fR_1 \times \fR_2}^{\can}.
    \end{equation}
    In particular, we have a natural isomorphism
    \begin{equation}\label{eq:alpha_TS}
    \alpha_{\fR_1, \fR_2} \colon Q_{\fR_1}^{o_1} \boxtimes Q_{\fR_2}^{o_2} \cong Q_{\fR_1 \times \fR_2}^{o_1 \boxtimes o_2}.
    \end{equation}
    Define an isomorphism $\TS_{\fR_1, \fR_2} \colon \varphi_{X_1, s_1, o_1}|_{R_1} \boxtimes \varphi_{X_2, s_2, o_2}|_{R_2} \cong \varphi_{X_1 \times X_2, s_1 \boxplus s_2, o_1 \boxtimes o_2} |_{R_1 \times R_2}$ by the following composition:
    \begin{align}\label{eq:TS_chart}
    \begin{aligned}
        \varphi_{X_1, s_1, o_1}|_{R_1} \boxtimes \varphi_{X_2, s_2, o_2}|_{R_2} 
        &\xrightarrow[\cong]{\omega_{\fR_1 \boxtimes \fR_2}}  (\varphi_{f_1}|_{R_1} \otimes_{\bZZ} Q_{\fR_1}^{o_1} ) \boxtimes (\varphi_{f_2}|_{R_2} \otimes_{\bZZ} Q_{\fR_2}^{o_2} ) \\
        & \xrightarrow[\cong]{(\TS, \alpha_{\fR_1, \fR_2})} \varphi_{f_1 \boxplus f_2}|_{R_1 \times R_2} \otimes_{\bZZ} Q_{\fR_1 \times \fR_2}^{o_1 \boxtimes o_2} \\
        &\xrightarrow[\cong]{\omega_{\fR_1 \times \fR_2}^{-1}} \varphi_{X_1 \times X_2, s_1 \boxplus s_2, o_1 \boxtimes o_2}|_{R_1 \times R_2}.
        \end{aligned}
    \end{align}
    We will show that these isomorphisms glue to define the desired Thom--Sebastiani isomorphism.
    Firstly, it is clear that the map $\TS_{\fR_1, \fR_2}$ is compatible with \'etale cover of \'etale d-critical charts.
    Therefore, by using Theorem \ref{thm:compare_d-critical}, it is enough to prove an identity 
    \begin{equation}\label{eq:TS_independent}
        \TS_{\fR_1, \fR_2} = \TS_{\fR_1', \fR_2'}
    \end{equation}
    where we set $\fR_1' = (R_1, \eta_1, U \times \bA^{2n}, f_1 \boxplus  (x_1y_1 + \cdots + x_n y_n), (i_1, 0))$ and
    $\fR_2' = (R_2, \eta_2, U \times \bA^{2m}, f_2 \boxplus (z_1 w_1 + \cdots + z_m w_m), (i_2, 0))$.
    We set $q_1 = x_1y_1 + \cdots + x_n y_n$ and $q_2 = z_1 w_1 + \cdots + z_m w_m$.
    By construction of the Thom--Sebastiani isomorphism of the vanishing cycle complexes and the commutativity of the diagram \eqref{eq:c_associative}, we see that the following diagram commutes:
    \[
    \xymatrix@C=80pt{
    {\varphi_{f_1} \boxtimes \varphi_{q_1} \boxtimes \varphi_{f_2} \boxtimes \varphi_{q_2}} \ar[d]_-{\cong}^-{(\id, c_y, \id, c_w)}
     \ar[r]_-{\cong}^-{\TS} 
    &  \varphi_{f_1 \boxplus q_1 \boxplus f_2 \boxplus q_2 } \ar[d]_-{\cong }^{\TS} \\
    {\varphi_{f_1} \boxtimes \varphi_{f_2}} \ar[r]_-{\cong}^-{(\TS, c_{y, w}^{-1})}
    & {\varphi_{f_1 \boxplus f_2} \boxtimes \varphi_{q_1 \boxplus q_2}} .
    }
    \]
    On the other hand, it follows from the construction of the isomorphism $b_y$ and the equality \eqref{eq:b_assoc} that the following diagram commutes:
    \[
    \xymatrix@C=80pt{
    Q_{\fR_1'}^{o_1} \boxtimes Q_{\fR_2'}^{o_2} \ar[d]_-{\cong}^-{(b_y, b_w)} \ar[r]_-{\cong}^-{\alpha_{\fR_1', \fR_2'}}
    & {Q_{\fR_1' \times \fR_2'}^{o_1 \boxtimes o_2}}  \ar[d]_-{\cong}^-{b_{y, w}} \\
    {Q_{\fR_1}^{o_1} \boxtimes Q_{\fR_2}^{o_2}} \ar[r]_-{\cong}^-{\alpha_{\fR_1, \fR_2}}
    & {Q_{\fR_1 \times \fR_2}^{o_1 \boxtimes o_2}}
    }
    \]
    This commutativity implies the identity \eqref{eq:TS_independent} as desired.

    To prove the commutativity of the diagram \eqref{eq:TS_assoc_space},
    by using \eqref{eq:Thom_Sebastiani_assoc_functor},
    it is enough to prove the commutativity of the following diagram for \'etale d-critical charts $\fR_j = (R_j, \eta_j, U_j, f_j, i_j)$ for $X_j$,
    \[
   \xymatrix@C=70pt{
       {Q_{\fR_1}^{o_1} \boxtimes Q_{\fR_2}^{o_2} \boxtimes Q_{\fR_3}^{o_3}} \ar[r]_-{\cong}^-{(\id, \alpha_{\fR_2, \fR_3})} \ar[d]_-{\cong}^-{(\alpha_{\fR_1, \fR_2}, \id)}
    & {Q_{\fR_1}^{o_1} \boxtimes Q_{\fR_2 \times \fR_3}^{o_2 \boxtimes o_3}} \ar[d]_-{\cong}^-{\alpha_{\fR_1, \fR_2 \times \fR_3}} \\
    {Q_{\fR_1 \times \fR_2}^{o_1 \boxtimes o_2} \boxtimes Q_{\fR_3}^{o_3}} \ar[r]_-{\cong}^-{\alpha_{\fR_1 \times \fR_2, \fR_3}}
    & {Q_{\fR_1 \times \fR_2 \times \fR_3}^{o_1 \boxtimes o_2 \boxtimes o_3}}.
    }
    \]
    which is obvious from the commutativity \eqref{eq:can_ori_prod_assoc}.

\end{proof}

Next, we will prove Thom--Sebastiani theorem for d-critical Artin stacks.
To do this, we first show that the Thom--Sebastiani isomorphism for algebraic spaces is compatible with smooth pullback.

\begin{lem}\label{lem:TS_DT_smooth_compatible}
Let $h_1 \colon X_1 \to Y_1$ and $h_2 \colon X_2 \to Y_2$ be smooth morphisms of algebraic spaces,
 $t_1$ and $t_2$ be the d-critical structures of $Y_1$ and $Y_2$, respectively.
Set $s_1 \coloneqq h_1^{\star} t_1$ and $s_2 \coloneqq h_2^{\star} t_2$.
    Then the following diagram commutes:
    \[
    \xymatrix@C=80pt{
    {h_1^{*} \varphi_{Y_1, t_1, o_1} \boxtimes h_2^{*} \varphi_{Y_2, t_2, o_2}}
     \ar[r]^-{\TS}_-{\cong} \ar[d]^-{(\Theta_{h_1}, \Theta_{h_2})}_-{\cong}
    & {(h_1 \times h_2)^{*} (\varphi_{Y_1 \times Y_2, t_1 \boxplus t_2, o_1 \boxtimes o_2})} \ar[d]_-{\cong}^-{\Theta_{h_1 \times h_2}} \\
    {\varphi_{X_1, s_1, h_1^{\star} o_1}[-d_1] \boxtimes \varphi_{X_2, s_2, h_2^{\star} o_2}[-d_2]} \ar[r]_-{\cong}^-{\TS}
    & {\varphi_{X_1 \times X_2, s_1 \boxplus s_2, o_1 \boxtimes o_2}[-d_1 - d_2].}
    }
    \]
    Here we set $d_1 = \dim h_1$ and $d_2 = \dim h_2$.
\end{lem}

\begin{proof}

Take a point $x_1 \in X_1$ and $x_2 \in X_2$. 
By using Proposition \ref{prop:sm_chart}, we take an \'etale d-critical chart $\fR_k = (R_k, \eta_k, U_k, f_k, i_k)$ of $x_k$ and \'etale d-critical chart $\fS_k = (S_k, \gamma_k, V_k, g_k, j_k)$  and a smooth morphism $(H_{k, 0}, H_k) \colon \fR_k \to \fS_k$ for $k = 1, 2$.
Then by the commutativity of the diagram \eqref{eq:sm_pull_TS_functor}, it is enough to prove the commutativity of the following diagram:
\[
\xymatrix@C=70pt{
{H_{1, 0}^* Q_{\fS_1}^{o_1} \boxtimes H_{2, 0}^*Q_{\fS_2}^{o_2}} \ar[r]_-{\cong}^-{\alpha_{\fS_1, \fS_2}} \ar[d]^-{\cong}_-{b_{H_1} \boxtimes b_{H_2}}
& {(H_{1, 0} \times H_{2, 0})^* (Q_{\fS_1 \times \fS_2}^{o_1 \boxtimes o_2} ) } \ar[d]^-{\cong}_-{b_{H_1 \times H_2}} \\
{Q_{\fR_1}^{h_1^{\star} o_1} \boxtimes Q_{\fR_2}^{h_2^{\star}o_2}} \ar[r]_-{\cong}^-{\alpha_{\fR_1, \fR_2}}
& {Q_{\fR_1 \times \fR_2}^{ h_1^{\star}o_1 \boxtimes h_2^{\star} o_2}}.
}
\]
This is obvious from the commutativity of \eqref{eq:can_ori_prod_pull}.

\end{proof}

By combining Lemma \ref{lem:TS_DT_smooth_compatible}, the commutativity of the diagram \eqref{eq:assoc_pull_van_stack} and \eqref{eq:TS_assoc_space},
we obtain the following statement by smooth descent:

\begin{prop}\label{prop:TS_DT_stack}
    Let $(\fX_1, s_1, o_1)$ and $(\fX_2, s_2, o_2)$ be oriented d-critical stacks.
    Then there exists a natural isomorphism
    \[
    \TS \colon \varphi_{\fX_1, s_1, o_1} \boxtimes \varphi_{\fX_2, s_2, o_2} \cong \varphi_{\fX_1 \times \fX_2, s_1 \boxplus s_2, o_1 \boxtimes o_2}.
    \]
    Further, if we are given another oriented d-critical algebraic space $(\fX_3, s_3, o_3)$, the following diagram commutes:
    \begin{equation}\label{eq:TS_assoc_stack}
    \begin{aligned}
    \xymatrix@C=100pt{
    {\varphi_{\fX_1, s_1, o_1} \boxtimes \varphi_{\fX_2, s_2, o_2} \boxtimes \varphi_{\fX_3, s_3, o_3}} \ar[r]_-{\cong}^-{(\TS, \id)} \ar[d]_-{\cong}^-{(\id, \TS)}
    & {\varphi_{\fX_1 \times \fX_2, s_1 \boxplus s_2, o_1 \boxtimes o_2} \boxtimes \varphi_{\fX_3, s_3, o_3}} \ar[d]_-{\cong}^-{\TS} \\
    {\varphi_{\fX_1, s_1, o_1} \boxtimes \varphi_{\fX_2 \times \fX_3, s_2 \boxplus s_3, o_2 \boxtimes o_3}} \ar[r]_-{\cong}^-{\TS}
    & {\varphi_{\fX_1 \times \fX_2 \times \fX_3, s_1 \boxplus s_2 \boxplus s_3, o_1 \boxtimes o_2 \boxtimes o_3}}.
    }
    \end{aligned}
    \end{equation}
\end{prop}

Let $(\fX_1, \omega_{\fX_1}, o_1)$ and $(\fX_2, \omega_{\fX_2}, o_2)$ be oriented $(-1)$-shifted symplectic stacks with underlying d-critical structures $s_1$ and $s_2$, respectively.
Then, using the isomorphism of orientations \eqref{eq:ori_cl_prod},
we obtain the following corollary:

\begin{cor}\label{cor:TS_derived}
    There exists a natural isomorphism
    \[
    \TS \colon \varphi_{\fX_1, \omega_{\fX_1}, o_1} \boxtimes \varphi_{\fX_2, \omega_{\fX_2}, o_2}
    \cong \varphi_{\fX_1 \times \fX_2, \omega_{\fX_1 } \boxplus \omega_{\fX_2}, o_1 \boxtimes o_2}
    \]
    which satisfies the associativity property as in \eqref{eq:TS_assoc_stack}.
\end{cor}

\section{Stack of graded/filtered objects}

In this section, we recall the notion of the stack of graded objects and filtered objects following \cite{hlp14}.

\subsection{Graded and filtered objects: definition}\label{ssec:grad_filt}

We first recall the definition of the graded and filtered objects.

\begin{defin}
    Let $\fX$ be a derived higher Artin stack equipped with a $\bG_m^n$-action $\mu$.
    \begin{itemize}
        \item The \defterm{fixed point stack} is
        \[\fX^{\mu} \coloneqq \Map^{\bG_m^n}(\pt, (\fX, \mu)).\]
        \item The \defterm{attracting stack} is
        \[\fX^{\mu, +, \ldots, +}\coloneqq \Map^{\bG_m^n}(\bA^n, (\fX, \mu)).\]
    \end{itemize}
\end{defin}

For the trivial $\bG_m^n$-action on $\fX$ we get the stacks
\[\Grad^n(\fX) \coloneqq \Map(B\bG_m^n, \fX)\]
of \defterm{$n$-graded objects} and
\[\Filt^n(\fX) \coloneqq \Map([\bA^1 / \bG_m]^n, \fX)\]
of \defterm{$n$-filtered objects}. We denote $\Grad(\fX)=\Grad^1(\fX)$ and $\Filt(\fX)=\Filt^1(\fX)$. We have the following algebraicity statement of $\Grad(\fX)$ and $\Filt(\fX)$:

\begin{prop}[{\cite[Proposition 1.1.2]{hlp14}}]\label{prop:grad_filt_algebraic}
    Let $\fX$ be a quasi-separated (classical) Artin stack with affine stabilizers.
    Then $\Grad^n(\fX)$ and $\Filt^n(\fX)$ are quasi-separated Artin stacks.
\end{prop}

\begin{rmk}
Let $X$ be a quasi-separated algebraic space with a $\bG_m^n$-action $\mu$. Then by \cite[Proposition 1.4.1]{hlp14} $X^{\mu}$ and $X^{\mu, +, \ldots, +}$ are represented by quasi-separated algebraic spaces.
\end{rmk}

We will be interested in the following maps:
\begin{itemize}
    \item The inclusion $i_1\colon \pt \rightarrow [\bA^1 / \bG_m]^n$ of $1$ which induces a map $\ev\colon \Filt^n(\fX)\rightarrow \fX$.
    \item The inclusion $i_0\colon B \bG_m^n\rightarrow [\bA^1 / \bG_m]^n$ of the origin, which induces a map $\gr\colon \Filt^n(\fX)\rightarrow \Grad^n(\fX)$.
    \item The projection $[\bA^1/\bG_m]^n\rightarrow B\bG_m^n$ which induces a map $\sigma\colon \Grad^n(\fX)\rightarrow \Filt^n(\fX)$ which is a section of $\gr$.
    \item The inclusion of the basepoint $\pt\rightarrow B\bG_m$ which induces a map $u\colon \Grad(\fX)\rightarrow \fX$ which coincides with $\ev\circ\sigma$.
\end{itemize}

For a derived higher Artin stack $\fX$ we consider the \defterm{attractor correspondence}
\begin{equation}\label{eq:grad_filt_diagram}
\begin{aligned}
\xymatrix{
{}
& \Filt(\fX)  \ar[ld]_-{\gr} \ar[rd]^-{\ev}
& {} \\
\Grad(\fX)
& {}
& \fX.
}
\end{aligned}
\end{equation}

More generally, for a derived higher Artin stack $\fX$ with a $\bG_m$-action $\mu$ we have the attractor correspondence
\begin{equation}\label{eq:grad_filt_diagram_twisted}
\begin{aligned}
\xymatrix{
{}
& \fX^{\mu,+}  \ar[ld]_-{\gr} \ar[rd]^-{\ev}
& {} \\
\fX^{\mu}
& {}
& \fX.
}
\end{aligned}
\end{equation}

If $\mu$ is trivial, \eqref{eq:grad_filt_diagram_twisted} reduces to \eqref{eq:grad_filt_diagram}.

\begin{rmk}\label{rmk:nontrivialactionGrad}
Let $\fX$ be a derived higher Artin stack equipped with a $\bG_m$-action $\mu$. By \cite[Theorem 1.4.7]{hlp14} we have $\Filt(B\bG_m)\cong \bZ\times B\bG_m$. Moreover, the fibre product $\Filt([\fX/\bG_m])\times_{\Filt(B\bG_m)} \pt$, where $\pt\rightarrow \Filt(B\bG_m)$ is given by the inclusion of $1\in\bZ$, may be identified with $[\fX^{\mu, +}/\bG_m]$. Thus, there is a morphism from the correspondence
\[
\xymatrix{
& [\fX^{\mu, +}/\bG_m] \ar[ld]_-{\gr} \ar[rd]^-{\ev} & \\
\fX^{\mu}\times B\bG_m && [\fX/\bG_m]
}
\]
to
\[
\xymatrix{
& \Filt([\fX/\bG_m]) \ar[ld]_-{\gr} \ar[rd]^-{\ev} & \\
\Grad([\fX/\bG_m]) && [\fX/\bG_m]
}
\]
which is both an open and closed immersion in each term. This will allow us to reduce the case of a nontrivial $\bG_m$-action on a stack to the case of a trivial $\bG_m$-action.
\end{rmk}

The commutativity of the attractor correspondence is expressed by the diagram
\begin{equation}\label{eq:grad2_diagram}
\begin{aligned}
\xymatrix{
& \Grad(\Filt(\fX)) \ar_{\Grad(\gr_{\fX})}[ddl] \ar^{\Grad(\ev_\fX)}[dr] && \Filt(\fX) \ar_{\gr_{\fX}}[dl] \ar^{\ev_\fX}[ddr] & \\
&& \Grad(\fX) && \\
\Grad^2(\fX) && \Filt^2(\fX) \ar^{\gr_{\Filt(\fX)}}[uul] \ar_{\Filt(\gr_{\fX})}[ddl] \ar_{\Filt(\ev_{\fX})}[uur] \ar^{\ev_{\Filt(\fX)}}[ddr] && \fX \\
&& \Grad(\fX) && \\
& \Filt(\Grad(\fX)) \ar_{\gr_{\Grad(\fX)}}[uul] \ar_{\ev_{\Grad(\fX)}}[ur] && \Filt(\fX). \ar^{\gr_{\fX}}[ul] \ar_{\ev_{\fX}}[uur] &
}
\end{aligned}
\end{equation}

\begin{rmk}
If we consider the attractor correspondence as a morphism $\attr_\fX\colon \fX\rightarrow \Grad(\fX)$ in the $(\infty, 2)$-category of correspondences of derived stacks, the iterated correspondence \eqref{eq:grad2_diagram} provides a two-morphism witnessing the commutativity of the diagram
\[
\xymatrix@C=2cm{
\fX \ar^{\attr_\fX}[r] \ar^{\attr_\fX}[d] & \Grad(\fX) \ar^{\attr_{\Grad(\fX)}}[d] \\
\Grad(\fX) \ar^{\Grad(\attr_\fX)}[r] & \Grad^2(\fX).
}
\]
\end{rmk}

Now let $G$ be an affine algebraic group and
\[X_\bullet(G) = \Hom(\bG_m, G),\qquad X^2_\bullet(G) = \Hom(\bG_m^2, G)\]
the sets of cocharacters of $G$. For a cocharacter $\lambda\colon \bG_m\rightarrow G$ we denote by
\[G_\lambda=\{g\in G\mid \lambda(t) g\lambda(t)^{-1} = g\}\]
the fixed point subgroup and $G^+_\lambda\subset G$ the subgroup of elements $g\in G$ such that the morphism $\bG_m\rightarrow G$ given by $t\mapsto \lambda(t) g\lambda(t)^{-1}$ admits an extension to a morphism $\bA^1\rightarrow G$. Informally,
\[G^+_\lambda = \{g\in G\mid \exists \lim_{t\rightarrow 0} \lambda(t) g\lambda(t)^{-1}\}.\]

\begin{rmk}
Since $G$ is affine and $R[t]\rightarrow R[t, t^{-1}]$ is injective for any commutative ring $R$, $\Map(\bA^1, G)\rightarrow \Map(\bG_m, G)$ is a monomorphism, i.e. the $\bA^1$-extension is unique if it exists.
\end{rmk}

We have a projection $G^+_\lambda\rightarrow G_\lambda$ given by $g\mapsto \lim_{t\rightarrow 0} \lambda(t) g\lambda(t)^{-1}$ and its section $G_\lambda\rightarrow G^+_\lambda$ given by the identity map on $G$.

Similarly, for a pair of commuting cocharacters $\hat{\lambda}=(\lambda_1,\lambda_2)\colon \bG_m^2\rightarrow G$ we denote by
\[G_{\hat{\lambda}}=\{g\in G\mid \hat{\lambda}(t_1, t_2) g\hat{\lambda}(t_1, t_2)^{-1} = g\}\]
and
\[G^{++}_{\hat{\lambda}} = \{g\in G\mid \exists \lim_{t_i\rightarrow 0} \hat{\lambda}(t_1,t_2) g\hat{\lambda}(t_1,t_2)^{-1}\}\]
the corresponding fixed point and attractor subgroups.

If $X$ is an algebraic space with a $G$-action, for any $\lambda\colon \bG_m\rightarrow G$ and $g\in G$ we have an isomorphism of stacks
\[[X^\lambda / G_\lambda]\longrightarrow [X^{g\lambda g^{-1}}/ G_{g\lambda g^{-1}}],\]
where $X^\lambda\rightarrow X^{g\lambda g^{-1}}$ is given by the action of $g\in G$ and $G_\lambda\rightarrow G_{g\lambda g^{-1}}$ is given by the conjugation by $g\in G$. Thus, it makes sense to consider the stack $[X^\lambda / G_\lambda]$ for a conjugacy class $\lambda\in X_\bullet(G)/G$ of cocharacters.

\begin{thm}\label{thm:Grad_quotient}
Let $G$ be an affine algebraic group acting on an algebraic space $X$ and set $\fX \coloneqq [X / G]$. Then the attractor correspondence \eqref{eq:grad_filt_diagram} is equivalent to
\[
\xymatrix{
& \coprod_{\lambda \in X_\bullet(G)/G} [X^{\lambda, +} / G^+_\lambda] \ar[dl] \ar[dr] & \\
\coprod_{\lambda \in X_\bullet(G)/G} [X^{\lambda} / G_\lambda] && [X/G]
}
\]
and the diagram \eqref{eq:grad2_diagram} is equivalent to
\[
\xymatrix@C=-0.2cm{
& \coprod_{\hat{\lambda} \in X^2_\bullet(G)/G} [(X^{\lambda_{2},+})^{\lambda_{1}} / (G^+_{\lambda_{2}})_{\lambda_{1}}] \ar[ddl] \ar[dr] && \coprod_{\lambda_{1} \in X_\bullet(G)/G} [X^{\lambda_{1},+}/G^+_{\lambda_{1}}] \ar[dl] \ar[ddr] & \\
&& \coprod_{\lambda_{1} \in X_\bullet(G)/G} [X^{\lambda_{1}}/G_{\lambda_{1}}] && \\
\coprod_{\hat{\lambda} \in X^2_\bullet(G)/G} [X^{\hat{\lambda}}/G_{\hat{\lambda}}] && \coprod_{\hat{\lambda} \in X^2_\bullet(G)/G} [X^{\hat{\lambda},+,+}/G^{+,+}_{\hat{\lambda}}] \ar[uul] \ar[ddl] \ar[uur] \ar[ddr] && [X/G] \\
&& \coprod_{\lambda_{2} \in X_\bullet(G)/G} [X^{\lambda_{2}}/G_{\lambda_{2}}] && \\
& \coprod_{\hat{\lambda} \in X^2_\bullet(G)/G} [(X^{\lambda_{1},+})^{\lambda_{2}} / (G^+_{\lambda_{1}})_{\lambda_{2}}] \ar[uul] \ar[ur] && \coprod_{\lambda_{2} \in X_\bullet(G)/G} [X^{\lambda_{2},+}/G^+_{\lambda_{2}}] \ar[ul] \ar[uur] &
}
\]
where we denote by $\lambda_1,\lambda_2\in X_\bullet(G)/G$ the images of $\hat{\lambda}$ under the two projections $X^2_\bullet(G)/G\rightarrow X_\bullet(G)/G$.
\end{thm}
\begin{proof}
This follows from the description of $\Grad(\fX)$ and $\Filt(\fX)$ given in \cite[Theorem 1.4.8]{hlp14}. The claim there is stated only for schemes (not algebraic spaces), but by reducing to the case of $G = \GL_n$ we can use \cite[Theorem 1.4.7]{hlp14}.
\end{proof}

\subsection{Perfect complexes}

We will use a description of perfect complexes on $B\bG_m$ and $[\bA^1/\bG_m]$. Let $\bZ$ be the set of integers considered as a discrete category and $\bZ_{\geq}$ the category where $\bZ$ is considered as a poset with respect to $\geq$. We will consider finite connected subposets $P\subset \bZ_{\geq}$. Consider the diagram
\[
\xymatrix{
& \Fun(P, \Perf_R) \ar_{\gr}[dl] \ar^{\ev}[dr] & \\
\Fun(P^\delta, \Perf_R) && \Perf_R
}
\]
where $\ev\colon \Fun(P, \Perf_R)\rightarrow \Perf_R$ is given by taking the colimit or, equivalently, by the evaluation at the minimum of $P$. Given an inclusion $P_1\subset P_2$ of posets there are morphisms of correspondences
\[
\left(
\begin{gathered}
\xymatrix@C=-0.5cm{
& \Fun(P_1, \Perf_R) \ar_{\gr}[dl] \ar^{\ev}[dr] & \\
\Fun(P_1^\delta, \Perf_R) && \Perf_R
}
\end{gathered}
\right)\longrightarrow\left(
\begin{gathered}
\xymatrix@C=-0.5cm{
& \Fun(P_2, \Perf_R) \ar_{\gr}[dl] \ar^{\ev}[dr] & \\
\Fun(P_2^\delta, \Perf_R) && \Perf_R
}    
\end{gathered}\right)
\]
given by left Kan extensions.

The following statement is well-known; see e.g. \cite[Chapter 9, Proposition 1.3.3]{gr2}, \cite[Sections 3.1 and 3.2]{lurrotation} and \cite[Section 3]{raksit}.

\begin{thm}\label{thm:QCohTheta}
The correspondence
\[
\xymatrix{
& \QCoh([\bA^1/\bG_m]\times \Spec R) \ar_{i_0^*}[dl] \ar^{i_1^*}[dr] & \\
\QCoh(B\bG_m\times \Spec R) && \Mod_R
}
\]
is equivalent to
\[
\xymatrix{
& \Fun(\Z_{\geq}, \Mod_R) \ar_{\gr}[dl] \ar^{\colim}[dr] & \\
\Fun(\Z, \Mod_R) && \Mod_R
}
\]
such that the natural symmetric monoidal structure on the left-hand side is sent to the Day convolution symmetric monoidal structure with respect to the abelian group structure on $\bZ$.
\end{thm}

\begin{prop}\label{prop:filteredperfect}
Let $R$ be a commutative dg algebra. Then there is an equivalence
\[
\begin{gathered}
\xymatrix@C=-1cm{
& \Perf([\bA^1/\bG_m]\times \Spec R) \ar_{i_0^*}[dl] \ar^{i_1^*}[dr] & \\
\Perf(B\bG_m\times \Spec R) && \Perf_R
}
\end{gathered}
\cong
\colim_{P\subset \bZ_{\geq}\text{, finite, connected}}
\left(
\begin{gathered}
\xymatrix@C=-0.5cm{
& \Fun(P, \Perf_R) \ar_{\gr}[dl] \ar^{\ev}[dr] & \\
\Fun(P^\delta, \Perf_R) && \Perf_R
}
\end{gathered}
\right)
\]

Under this equivalence, the natural symmetric monoidal structure on the left-hand side corresponds to the Day convolution symmetric monoidal structure.
\end{prop}
\begin{proof}
The stacks $[\bA^1/\bG_m]\times\Spec R$ and $B\bG_m\times\Spec R$ are perfect in the sense of \cite{bzfn}, so perfect complexes on them coincide with compact objects. The computation of compact objects is analogous for both cases, so we will stick with the case of $[\bA^1/\bG_m]$.

Writing $\bZ_{\geq}$ as a colimit over its finite connected subsets $P\subset \bZ_{\geq}$, by Theorem \ref{thm:QCohTheta} we have
\[\QCoh([\bA^1/\bG_m]\times \Spec R)\simeq \colim_{P\subset \bZ_{\geq}\text{, finite, connected}} \Fun(P, \Mod_R)\]
with respect to left Kan extensions, where the colimit is taken in the $\infty$-category of presentable $\C$-linear $\infty$-categories. For an inclusion $P_1\subset P_2\subset \bZ_{\geq}$, the left Kan extension along $P_1\subset P_2$ preserves compact objects as it has a colimit-preserving right adjoint given by restriction. Passing to compact objects, we get
\[\Perf([\bA^1/\bG_m]\times \Spec R)\simeq \colim_{P\subset \bZ_{\geq}\text{, finite, connected}} \Fun(P, \Mod_R)^\omega,\]
where $\Fun(P, \Mod_R)^\omega\subset \Fun(P, \Mod_R)$ is the full subcategory of compact objects. By \cite[Proposition 2.8]{ao23} we may further identify
\[\Fun(P, \Mod_R)^\omega\simeq \Fun(P, \Perf_R).\]
\end{proof}

\begin{rmk}
Under the equivalence given by Proposition \ref{prop:filteredperfect} an object $V\in\Perf([\bA^1/\bG_m]\times \Spec R)$ corresponds to a $\Z$-indexed diagram
\[\dots\rightarrow V_1\rightarrow V_0\rightarrow V_{-1}\rightarrow \dots\]
of perfect complexes of $R$-modules such that $V_m = 0$ for $m\gg 0$ and $V_m\rightarrow V_{m-1}$ is an isomorphism for $m\ll 0$. The Day convolution of two such objects is given by
\[(V\otimes W)_m = \colim_{i+j\geq m} V_i\otimes V_j.\]
\end{rmk}

If $\cF$ is a quasi-coherent complex on $\fX$, its pullback along $\Grad(\fX)\rightarrow \fX$ acquires a natural grading as follows. This map factors as $\Grad(\fX)\rightarrow \Grad(\fX)\times B\bG_m\xrightarrow{\ev}\fX$, where $\ev$ is the natural evaluation morphism. By Theorem \ref{thm:QCohTheta} the quasi-coherent complex $\ev^* \cF$ corresponds to a graded quasi-coherent complex on $\Grad(\fX)$ and its pullback under the inclusion of the basepoint $\pt\rightarrow B\bG_m$ is given by the total complex. For a quasi-coherent complex $\cF$ on $\fX$, we denote by $\cF|^0_{\Grad(\fX)}$ its weight $0$ part and by $\cF|^+_{\Grad(\fX)}$ (resp. $\cF|^-_{\Grad(\fX)}$) its positive (resp. negative) weight part. The following is \cite[Lemma 1.2.2]{hlp14}.

\begin{prop}\label{prop:GradFiltcotangent}
Let $\fX$ be a derived stack which admits an almost perfect cotangent complex. Then $\Grad(\fX)$ and $\Filt(\fX)$ admit almost perfect cotangent complexes, so that there are natural isomorphisms
\[\bL_{\Grad(\fX)}\cong \bL_{\fX}|^0_{\Grad(\fX)},\qquad \bL_{\Filt(\fX)/\Grad(\fX)}|_{\Grad(\fX)}\cong \bL_{\fX}|^-_{\Grad(\fX)}.\]
\end{prop}

For later use, we record the following technical statement here:

\begin{lem}\label{lem:Gradprojectionisomorphism}
    Let $\fX$ be a quasi-separated Artin stack with affine stabilizers.
    Let $\iota \colon \fX \to \Grad(\fX)$ be a section of $u \colon \Grad(\fX) \to \fX$.
    Then $\iota$ is an open immersion and the map
    $\fX \times_{\Grad(\fX)} \Filt(\fX) \to \fX$ defined by the base change of $\gr$ along $\iota$ is an isomorphism.
\end{lem}

\begin{proof}
    Let $\fX^{\der}$ be $\fX$ regarded as a derived stack. The composition
        \begin{equation}\label{eq:idem_cot}
        \bL_{\fX^{\der}} \cong \iota^{*} u^* \bL_{\fX^{\der}} \to \iota^* \bL_{\Grad(\fX^{\der})} \to \bL_{\fX^{\der}}
        \end{equation}
        is the identity map. The quasi-coherent complex $u^*\bL_{\fX^{\der}}$ carries a natural grading, which induces a grading on $\bL_{\fX^{\der}} \cong \iota^{*} u^* \bL_{\fX^{\der}}$. By Proposition \ref{prop:GradFiltcotangent} the first morphism is given by the projection to its weight $0$ part, so all maps appearing in \eqref{eq:idem_cot} are isomorphisms.
    In particular, $\iota$ is \'etale. 
    By \cite[Lemma 1.1.5]{hlp14}, the map $u$ is representable.
    Therefore, its section $\iota$ is a monomorphism, hence an open immersion by \cite[\href{https://stacks.math.columbia.edu/tag/025G}{Tag 025G}]{stacks-project}.
    
    By a similar argument, we see that $\sigma \circ \iota$ is \'etale.
    By using \cite[Lemma 1.3.5]{hlp14} for the open substack $\sigma \circ \iota(\fX) \subset \Filt(\fX)$,
    we obtain the identity $\sigma \circ \iota(\fX) = \gr^{-1}(\iota(\fX))$ as desired.
\end{proof}

\subsection{Symplectic and Lagrangian structures}

Our goal in this section is to endow the attractor correspondence with a Lagrangian structure. For this we use the AKSZ formalism developed in \cite{ptvv13,cal15,chs21}.

\begin{lem}\label{lem:positiveweightfunctions}
Let $Y$ be a finite product of stacks $B\bG_m$ and $[\bA^1/\bG_m]$. Then:
\begin{enumerate}
    \item Pullback to the origin $\pt\rightarrow Y$ defines an isomorphism \[\bR\Gamma(Y, \cO)\simeq \C.\]
    \item The stack $Y$ is $\cO$-compact.
\end{enumerate}
\end{lem}
\begin{proof}
Let $p\colon Y\rightarrow \pt$ be the projection. By \cite{bzfn} the stack $Y$ is perfect and hence $p_*$ is colimit-preserving.

To prove the first statement we have to show that $\id\rightarrow p_* p^*$ is an isomorphism. Since $\QCoh(Y_1\times Y_2)\simeq \QCoh(Y_1)\otimes \QCoh(Y_2)$ by \cite[Theorem 4.7]{bzfn}, it is enough to show this claim for $Y=B\bG_m$ and $Y=[\bA^1/\bG_m]$, where it is obvious by an explicit computation.

By \cite[Theorem 4.2.1, Theorem 2.4.3 and Proposition 5.1.6]{hlp23} the pullback $p^*\colon \QCoh(\pt)\rightarrow \QCoh(Y)$ admits a colimit-preserving left adjoint. Therefore, $Y$ is $\cO$-compact by \cite[Proposition 1.4]{ns23}.
\end{proof}

In particular, there is a natural isomorphism
\[[B\bG_m]\colon \bR\Gamma(B\bG_m, \cO)\longrightarrow \C.\]

\begin{prop}\label{prop:BGmorientation}
The morphism $[B\bG_m]\colon \bR\Gamma(B\bG_m, \cO)\rightarrow \C$ defines a $0$-orientation on $B\bG_m$.
\end{prop}
\begin{proof}
We have to show that, for every connective commutative dg algebra $R$ and a perfect complex $E\in\Perf(B\bG_m\times \Spec R)$, the induced morphism
\[\bR\Gamma(B\bG_m, E)\longrightarrow \bR\Gamma(B\bG_m, E^\vee)^\vee\]
is an isomorphism.

Using Proposition \ref{prop:filteredperfect} we identify $\Perf(B\bG_m\times \Spec R)$ with collections of perfect complexes $\{E_n\}_{n\in\Z}$ of $R$-modules, where only finitely many terms are nonzero. The dual of $E\in \Perf(B\bG_m\times \Spec R)$ corresponding to $n\mapsto E_n$ is given by $E^\vee\in\Perf(B\bG_m\times \Spec R)$ which corresponds to $n\mapsto (E_{-n})^\vee$ with the obvious evaluation map. We have $\bR\Gamma(B\bG_m, E)\simeq E_0$ and the induced morphism above is the canonical isomorphism
\[E_0\longrightarrow ((E_0)^\vee)^\vee.\]
\end{proof}

\begin{cor}\label{cor:grad_symplectic_lagrangian}
Let $(\fX, \omega_{\fX})$ be a $k$-shifted symplectic stack. Then $u^\star\omega_{\fX}$ defines a $k$-shifted symplectic structure on the stack $\Grad(\fX)$ of graded objects. Let $\tau\colon \fL\rightarrow \fX$ be a $k$-shifted Lagrangian morphism specified by a homotopy $\tau^\star \omega_{\fX}\sim 0$. Then its restriction under $u^\star$ defines a $k$-shifted Lagrangian structure on the morphism $\Grad(\tau)\colon \Grad(\fL)\rightarrow \Grad(\fX)$.
\end{cor}
\begin{proof}
Since $B\bG_m$ is $0$-oriented by Proposition \ref{prop:BGmorientation}, applying the AKSZ construction \cite[Theorem 2.5]{ptvv13} to the stack $\Grad(\fX)=\Map(B\bG_m, \fX)$ we obtain a $k$-shifted symplectic structure obtained by the composite
\begin{align*}
\C[2-k](2)&\xrightarrow{[\omega_\fX]} \bDR(\fX)\\
&\xrightarrow{\ev^\star}\bDR(B\bG_m\times \Grad(\fX))\\
&\xrightarrow{\kappa_{\fX, B\bG_m}} \bR\Gamma(B\bG_m, \cO)\otimes \bDR(\Grad(\fX))\\
&\xrightarrow{[B\bG_m]\otimes \id} \bDR(\Grad(\fX))
\end{align*}
where $\ev\colon B\bG_m\times \Grad(\fX)\rightarrow \fX$ is the evaluation map. By construction the orientation $[B\bG_m]\colon \bR\Gamma(B\bG_m, \cO)\rightarrow \C$ is given by the pullback along $i\colon \pt\rightarrow B\bG_m$. Unpacking the construction of the map $\kappa_{\fX, B\bG_m}$ given in \cite{ptvv13}, we see that the composite
\begin{align*}
\bDR(B\bG_m\times \Grad(\fX))&\xrightarrow{\kappa_{\fX, B\bG_m}} \bR\Gamma(B\bG_m, \cO)\otimes \bDR(\Grad(\fX))\\
&\xrightarrow{i^\star\otimes \id} \bDR(\Grad(\fX))
\end{align*}
is equivalent to
\[\bDR(B\bG_m\times \Grad(\fX))\xrightarrow{(i\times \id)^\star} \bDR(\Grad(\fX)).\]
But $\ev\circ (i\times \id)\colon \fX\rightarrow \Grad(\fX)$ is precisely $u$. The claim for a Lagrangian $\fL\rightarrow \fX$ is proven analogously.
\end{proof}

We will now show that the shifted symplectic structure on $\Grad(\fX)$ defined in Corollary \ref{cor:grad_symplectic_lagrangian} is compatible with taking derived critical loci. Let $\fY$ be a derived higher Artin stack locally of finite presentation and $f \colon \fY \to \bA^1[n]$ be an $n$-shifted function.
Set $\fX \coloneqq \Crit(f)$ and let $\omega_{\fX}$ be the $(n - 1)$-shifted symplectic structure in Example \ref{ex:critical}.

\begin{lem}\label{lem:crit_localize}
    There exists a natural equivalence of $(n-1)$-shifted symplectic stacks
    \[
        (\Grad(\fX), u^{\star} \omega_{\fX}) \simeq (\Crit(f |_{\Grad(\fY)} ), \omega_{\Crit(f |_{\Grad(\fY)} )}) 
    \]
    where $\omega_{\Crit(f |_{\Grad(\fY)}) }$ is the $(n-1)$-shifted symplectic structure constructed in Example \ref{ex:critical}.
\end{lem}

\begin{proof}

    Recall that the $(-1)$-shifted symplectic structure $\omega_{\fX}$ is constructed using the following Lagrangian intersection diagram:
    \[
    \xymatrix{
    \fX  \ar[r] \ar[d]
    & {\fY} \ar[d]^-{df} \\
    \fY \ar[r]^-{0}
    & T^* \fY.
    }
    \]
    Passing to the mapping stack from $B \bG_m$ and using Corollary \ref{cor:grad_symplectic_lagrangian}, we obtain the following Lagrangian intersection diagram:
     \begin{equation}\label{eq:Lag_intersection_1}
     \begin{aligned}
    \xymatrix{
    \Grad(\fX)   \ar[r] \ar[d]
    & {\Grad(\fY)} \ar[d]^-{\Grad(df)} \\
    \Grad(\fY) \ar[r]^-{0}
    & \Grad(T^* [n] \fY ).
    }
    \end{aligned}
    \end{equation}
    On the other hand, we have the following Lagrangian intersection diagram:
     \begin{equation}\label{eq:Lag_intersection_2}
          \begin{aligned}
    \xymatrix{
    \Crit(f |_{\Grad(\fY)}) \ar[r] \ar[d]
    & {\Grad(\fY)} \ar[d]^-{d \Grad(f)} \\
    \Grad(\fY) \ar[r]^-{0}
    & T^*[n] \Grad(\fY).
    }
         \end{aligned}
    \end{equation}
    
    It is enough to show that the diagrams \eqref{eq:Lag_intersection_1} and \eqref{eq:Lag_intersection_2} are identified as Lagrangian intersection diagrams.
    Let $\gamma_{\fY}$ (resp. $\gamma_{\Grad(\fY)}$) be the tautological one-form on $T^*[n] \fY$ (resp. $T^*[n] \Grad(\fY)$).
    Then we have a natural identification
    $u^{\star} \gamma_{\fY} \simeq \gamma_{\Grad(\fY)}$
    under the equivalence 
    $T^*[n] \Grad(\fY) \simeq \Grad(T^*[n] \fY)$,
    which implies that it preserves the canonical $n$-shifted symplectic structures.
    It is obvious to show that the Lagrangian structures are identified as well.

\end{proof}

It is obvious that the point $\pt$ has a canonical $0$-orientation
\[[\pt]\colon \bR\Gamma(\pt, \cO)\longrightarrow \C.\]
Now consider the cospan
\[
\xymatrix{
& [\bA^1 / \bG_m]  & \\
B\bG_m \ar^{i_0}[ur] && \pt \ar_{i_1}[ul]
}
\]
of derived stacks.

\begin{prop}\label{prop:orientedcospan}
There is a unique structure of a $0$-preoriented cospan on
\[
\xymatrix{
& [\bA^1 / \bG_m]  & \\
B\bG_m \ar^{i_0}[ur] && \pt \ar_{i_1}[ul]
}
\]
which restricts to the orientations $[B\bG_m]$ and $[\pt]$ on $B\bG_m$ and $\pt$. Moreover, this cospan is oriented.
\end{prop}
\begin{proof}
By Lemma \ref{lem:positiveweightfunctions} the functor $\bR\Gamma(-, \cO)$ applied to this cospan gives the correspondence
\[
\xymatrix{
& \C \ar[dl] \ar[dr] & \\
\C && \C
}
\]
In particular, the existence and uniqueness of a preorientation is obvious.

Let $R$ be a connective commutative dg algebra and $E\in\Perf([\bA^1/\bG_m]\times\Spec R)$, i.e. a filtered object $\{E_n\}$ by Proposition \ref{prop:filteredperfect}. By \cite[Definition 2.5.4]{chs21} we have to show that the diagram
\[
\xymatrix{
\bR\Gamma([\bA^1/\bG_m], E) \ar[r] \ar[d] & \bR\Gamma(\pt, i_1^* E) \ar[d] \\
\bR\Gamma(B\bG_m, i_0^* E) \ar[r] & \bR\Gamma([\bA^1/\bG_m], E^\vee)^\vee
}
\]
is Cartesian. Any filtered object $E$ fits into a fibre sequence
\[\tau_{\geq 0} E\rightarrow E\rightarrow \tau_{\leq -1} E,\]
where the associated graded of $\tau_{\geq 0} E$ is concentrated in weights $\geq 0$ and the associated graded of $\tau_{\leq -1} E$ is concentrated in weights $\leq -1$. By stability, it is enough to prove the claim for filtered objects $E$ with associated graded concentrated in weights $\geq 0$ and in weights $\leq -1$. Replacing $E$ by $E^\vee$ and using the symmetry of the diagram, it is enough to prove the claim for filtered objects whose associated graded is concentrated in weights $\geq 0$. Further, by stability, it is enough to prove the claim for filtered objects whose associated graded is $V\in\Perf_R$ concentrated in a single weight $n\geq 0$.

If $n > 0$, the diagram is
\[
\xymatrix{
V \ar^{\id}[r] \ar[d] & V \ar[d] \\
0 \ar[r] & 0
}
\]
while for $n=0$ the diagram is
\[
\xymatrix{
V \ar^{\id}[r] \ar^{\id}[d] & V \ar[d] \\
V \ar[r] & (V^\vee)^\vee
}
\]
Both of these are obviously Cartesian.
\end{proof}

\begin{rmk}
The proof of the above proposition is sketched in \cite[Lemma 7.4]{hp23}.
While we were preparing this paper, the same statement was proved in \cite[Theorem 3.1.5]{Bu24}.
\end{rmk}

\begin{cor}\label{cor:lagattractorcorrespondence}
Let $(\fX, \omega_{\fX})$ be a $k$-shifted symplectic stack. Then the attractor correspondence
\[
\xymatrix{
{}
& \Filt(\fX)  \ar[ld]_-{\gr} \ar[rd]^-{\ev}
& {} \\
\Grad(\fX)
& {}
& \fX
}
\]
carries a natural structure of a $k$-shifted Lagrangian correspondence.
\end{cor}
\begin{proof}
By definition, the attractor correspondence \eqref{eq:grad_filt_diagram} is obtained by applying $\Map(-, \fX)$ to the cospan from Proposition \ref{prop:orientedcospan}. The AKSZ construction \cite[Proposition 3.4.2]{chs21} implies that the attractor correspondence is Lagrangian.
\end{proof}

\begin{cor}
Let $(\fX, \omega_{\fX})$ be a $k$-shifted symplectic stack equipped with a $\bG_m$-action preserving $\omega_{\fX}$. Then the attractor correspondence
\[
\xymatrix{
{}
& \fX^{\mu, +}  \ar[ld]_-{\gr} \ar[rd]^-{\ev}
& {} \\
\fX^{\mu}
& {}
& \fX
}
\]
carries a natural structure of a $k$-shifted Lagrangian correspondence.
\end{cor}
\begin{proof}
By assumption the quotient map $[\fX/\bG_m]\rightarrow B\bG_m$ carries a relative $k$-shifted symplectic structure. Pulling it back along $[\bA^1/\bG_m]\rightarrow B\bG_m$ we obtain the map $[\fX\times \bA^1/\bG_m]\rightarrow [\bA^1/\bG_m]$ with a relative $k$-shifted symplectic structure. Using the oriented cospan $B\bG_m\rightarrow [\bA^1/\bG_m]\leftarrow \pt$ from Proposition \ref{prop:orientedcospan} and applying the relative AKSZ construction from \cite[Theorem 2.34]{cs24} we obtain the result.
\end{proof}

Next, consider the diagram
\[
\xymatrix{
& B\bG_m\times [\bA^1/\bG_m] \ar_{i_0\times \id}[ddr] && [\bA^1/\bG_m] \ar^{\id\times i_1}[ddl] & \\
&& B\bG_m \ar_{\id\times i_1}[ul] \ar^{i_0}[ur] && \\
B\bG_m^2 \ar^{\id\times i_0}[uur] \ar^{i_0\times \id}[ddr] && [\bA^1/\bG_m]^2 && \pt \ar^{i_1}[uul] \ar^{i_1}[ddl] \\
&& B\bG_m \ar^{i_1\times \id}[dl] \ar_{i_0}[dr] && \\
& [\bA^1/\bG_m]\times B\bG_m \ar^{\id\times i_0}[uur] && [\bA^1/\bG_m] \ar_{i_1\times \id}[uul] &
}
\]

\begin{prop}\label{prop:orientediteratedcospan}
There is a unique structure of a $0$-preoriented 2-fold cospan on
\[
\xymatrix{
& (B\bG_m\times [\bA^1/\bG_m])\coprod_{B\bG_m} [\bA^1/\bG_m] \ar[d] & \\
B\bG_m^2 \ar[dr] \ar[ur] & [\bA^1/\bG_m]^2 & \pt \ar[dl] \ar[ul] \\
& ([\bA^1/\bG_m]\times B\bG_m) \coprod_{B\bG_m} [\bA^1/\bG_m] \ar[u] &
}
\]
which restricts to the orientations $[B\bG_m]$ and $[\pt]$ on $B\bG_m$ and $\pt$. Moreover, this 2-fold cospan is oriented.
\end{prop}
\begin{proof}
By Lemma \ref{lem:positiveweightfunctions} the functor $\bR\Gamma(-, \cO)$ applied to these cospans gives the 2-fold correspondence
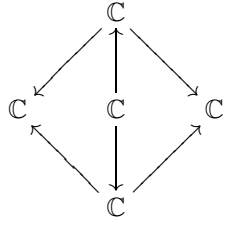
\[
\xymatrix{
& \C \ar[dl] \ar[dr] & \\
\C & \C \ar[u] \ar[d] & \C \\
& \C \ar[ul] \ar[ur] &
}
\]
with all maps given by the identities. In particular, the existence and uniqueness of a preorientation is obvious.

Let $R$ be a connective commutative dg algebra and $E\in\Perf([\bA^1/\bG_m]^2\times\Spec R)$, i.e. a doubly filtered object $\{E_{n,m}\}$ by Proposition \ref{prop:filteredperfect}. By \cite[Definition 2.7.5]{chs21} we have to show that the diagram
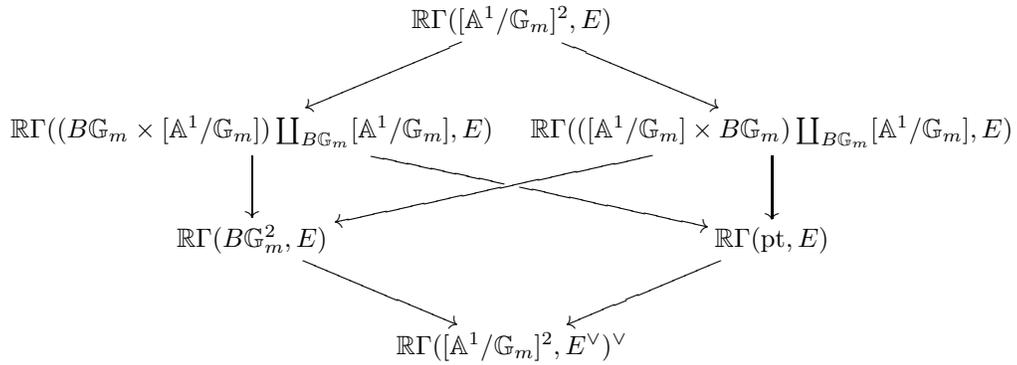
\[
\xymatrix@C=-1.5cm{
& \bR\Gamma([\bA^1/\bG_m]^2, E) \ar[dl] \ar[dr] & \\
\bR\Gamma((B\bG_m\times [\bA^1/\bG_m])\coprod_{B\bG_m} [\bA^1/\bG_m], E) \ar[d] \ar[drr]|\hole && \bR\Gamma(([\bA^1/\bG_m]\times B\bG_m) \coprod_{B\bG_m} [\bA^1/\bG_m], E) \ar[d] \ar[dll] \\
\bR\Gamma(B\bG_m^2, E) \ar[dr] && \bR\Gamma(\pt, E) \ar[dl] \\
& \bR\Gamma([\bA^1/\bG_m]^2, E^\vee)^\vee &
}
\]
is a limit diagram, where we omit obvious pullbacks of $E$ to the corresponding stacks. As in the proof of Proposition \ref{prop:orientedcospan}, it is enough to assume that the doubly filtered object $E$ has associated graded given by the object $V\in\Perf_R$ concentrated in a single weight $(n, m)$ with $n,m\geq 0$.

If $n, m > 0$, we have to show that the diagram
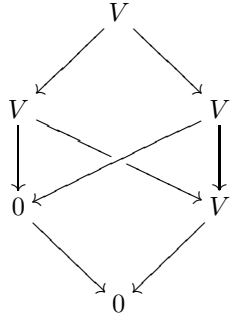
\[
\xymatrix{
& V \ar[dl] \ar[dr] & \\
V \ar[d] \ar[drr]|\hole && V \ar[d] \ar[dll] \\
0 \ar[dr] && V \ar[dl] \\
& 0 &
}
\]
is a limit diagram, where the nontrivial maps are all the identities. This is reduced to proving that the square
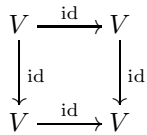
\[
\xymatrix{
V \ar^{\id}[r] \ar^{\id}[d] & V \ar^{\id}[d] \\
V \ar^{\id}[r] & V
}
\]
is Cartesian, which is obvious. The analysis of the other cases (when either $n$ or $m$ is zero) is done analogously.
\end{proof}

\begin{cor}\label{cor:Filt_Gr_twofolded}
Let $(\fX, \omega_{\fX})$ be a $k$-shifted symplectic stack. Then
\begin{equation}\label{eq:GradFiltLagrangiancorrespondence}
\xymatrix{
& \Filt(\Grad(\fX))\times_{\Grad(\fX)} \Filt(\fX) \ar[dl] \ar[dr] & \\
\Grad^2(\fX) & \Filt^2(\fX) \ar[u] \ar[d] & \fX \\
& \Grad(\Filt(\fX))\times_{\Grad(\fX)}\Filt(\fX) \ar[ul] \ar[ur] &
}
\end{equation}
has a natural structure of a 2-fold Lagrangian correspondence.
\end{cor}
\begin{proof}
The 2-fold correspondence in the statement is obtained by applying $\Map(-, \fX)$ to the 2-fold iterated cospan from Proposition \ref{prop:orientediteratedcospan}. The claim then follows from the AKSZ construction \cite[Proposition 3.4.2]{chs21}.
\end{proof}

\subsection{Matching locus}\label{ssec:Matching}

Let $u_1, u_2 \colon \Grad^2(\fX) \to \Grad(\fX)$ be morphisms induced by $\bG_m \times \{ 0 \} \to \bG_m^2$ and $ \{ 0 \} \times \bG_m \to \bG_m^2$, respectively.
We will introduce the notion of matching points:
\begin{defin}
    For a classical Artin stack (resp. derived Artin stack locally of finite presentation), we say that a $\bC$-valued point $x \in \Grad^2(\fX)$ is \defterm{matching} if for each $(m, l) \in \bZ^2$   with $m \cdot l < 0$ and $i = -1, 0, 1$ (resp. for any $i \in \bZ$), the weight $(m, l)$-part with respect to the $\bG_m^2$-action
    \begin{equation}\label{eq:vanishing_cot}
    H^i(\bL_{\fX, u(x)})_{(m, l)}
    \end{equation}
    vanishes.
\end{defin}
Roughly speaking, the matching point is a point where $u_1(x)$ is sufficiently close to $u_2(x)$.
Clearly, the collection of matching points $\Grad^2(\fX)_{\mathrm{mat}} \subset \Grad^2(\fX)$ form an open substack, which we call the matching locus.
The following proposition motivates the definition of the matching points\footnote{We thank Andr\'es Ib\'a\~nez N\'u\~nez for sharing with us the proof of a related statement.}:

\begin{prop}\label{prop:matching_invertible_stack}
   Let $\fX$ be a quasi-separated classical Artin stack with affine stabilizers. Consider the following commutative diagram.
    \[
    \xymatrix{
    {\Filt(\Grad(\fX)) \times_{\Grad(\fX)} \Filt(\fX)  } \ar[rd]
    & {\Filt^2(\fX)} \ar[l]_-{\alpha} \ar[r]^-{\beta}
    & {\Grad(\Filt(\fX)) \times_{\Grad(\fX)} \Filt(\fX)} \ar[ld] \\
    {}
    & {\Grad^2(\fX).}
    & {}
    }
    \]
    Then the maps $\alpha$ and $\beta$ are invertible over $\Grad^2(\fX)_{\mathrm{mat}}$.
    The same statement holds true for derived Artin stacks.
\end{prop}

\begin{proof}
    We first prove the statement for classical Artin stacks.
    Let $\alpha_{\mathrm{\mat}}$ denote the base change of $\alpha$ along the open inclusion $\Grad^2(\fX)_{\mat} \to \Grad^2(\fX)$.
    It is enough to prove the invertibility of $\alpha_{\mat}$.
    We first show that $\alpha_{\mat}$ is \'etale.
    To see this,  it is enough to show the vanishing of the cohomology of cotangent complex
    $H^i(\bL_{\alpha_{\mat}} |_{x})$ for each $\bC$-valued points $x \in \Filt^2(\fX)$ over the matching locus and $i = -1, 0, 1$.
    Since the set of $\bC$-valued points $x \in \Filt^2(\fX)$ over the matching locus with vanishing $H^i(\bL_{\alpha_{\mat}} |_{x})$ is open, by using \cite[Lemma 1.3.5]{hlp14}, it is enough to prove the vanishing for points of the form $y = \sigma(x)$ for some $y \in \Grad^2(\fX)$.
    In this case, the claim follows from the identification
    \[
H^i (\bL_{\alpha_{\mat}}|_{\sigma(y)}) \simeq \bigoplus_{m < 0, l > 0} H^i(\bL_{\fX, u(y)})_{(m, l)}
    \]
    and the vanishing \eqref{eq:vanishing_cot}.

    Next, we show that $\alpha_{\mat}$ is bijective on $\bC$-valued points.
    Consider the following diagram of stacks
    \[
    \xymatrix{
    { [0 / \bG_m] \times \{1 \}} \ar[r] \ar[d]
    & {[\bA^1 / \bG_m] \times \{1 \}} \ar[d] \\
    {[0 / \bG_m] \times [\bA^1 / \bG_m]} \ar[r]
    & {[\bA^2 / \bG_m^2]}.
    }
    \]
    where the horizontal arrows are canonical closed inclusions and the vertical arrows are canonical open immersions.
    By considering the mapping spaces from these stacks to $\fX$,
    we obtain the following commutative diagram:
    \[
    \xymatrix{
    {\Filt^2(\fX)}(\bC) \ar[r] \ar[d]
    & {\Filt(\Grad(\fX))(\bC)} \ar[d] \\
    {\Filt(\fX)(\bC)} \ar[r]
    & {\Grad(\fX)(\bC).}
    }
    \]
    It follows from \cite[Lemma 3.4.4]{hlp14} that this diagram is Cartesian over the matching locus as desired.

    Since the diagonal map $\Delta_{\alpha_{\mat}}$ is representable, \'etale and bijective on $\bC$-valued points,
    it is an isomorphism by  \cite[\href{https://stacks.math.columbia.edu/tag/05W5}{Tag 05W5}]{stacks-project}\footnote{Note that a morphism between algebraic spaces  locally of finite type over algebraically closed field $k$ is universally injective if and only if injective on $k$-valued points: see e.g. \cite{114409}. }.
    In particular, $\alpha_{\mat}$ is an \'etale representable morphism.
    Using \cite[\href{https://stacks.math.columbia.edu/tag/05W5}{Tag 05W5}]{stacks-project} again, we conclude that $\alpha_{\mat}$ is an isomorphism.

    The statement for derived Artin stacks follows since $\alpha_{\mat}$ is a classical isomorphism with a vanishing cotangent complex.
\end{proof}

By Corollary \ref{cor:Filt_Gr_twofolded}, we obtain the following corollary:
\begin{cor}\label{cor:matching_composite}
    Assume that $\fX$ is a $(-1)$-shifted symplectic derived Artin stack.
    Then we have a natural equivalence between the following composite Lagrangian correspondences:
    \[
    \xymatrix{
    {}
    & {}
    & {(\Filt(\Grad(\fX)) \times_{\Grad(\fX)} \Filt(\fX) )_{\mathrm{mat}}} \ar[ld] \ar[rd]
    & {}
    & {} \\
    {}
    & {\Filt(\Grad(\fX))_{\mathrm{mat}} } \ar[ld] \ar[rd]
    & {}
    & {\Filt(\fX)} \ar[ld] \ar[rd]
    & {} \\
    {\Grad^2(\fX)_{\mathrm{mat}}}
    & {}
    & {\Grad(\fX)}
    & {}
    & {\fX,}
    }
    \]
       \[
    \xymatrix{
    {}
    & {}
    & {(\Grad(\Filt(\fX)) \times_{\Grad(\fX)} \Filt(\fX) )_{\mathrm{mat}}} \ar[ld] \ar[rd]
    & {}
    & {} \\
    {}
    & {\Grad(\Filt(\fX))_{\mathrm{mat}} } \ar[ld] \ar[rd]
    & {}
    & {\Filt(\fX)} \ar[ld] \ar[rd]
    & {} \\
    {\Grad^2(\fX)_{\mathrm{mat}}}
    & {}
    & {\Grad(\fX)}
    & {}
    & {\fX.}
    }
    \]
    Here $(-)_{\mathrm{mat}}$ denotes the base change along the open inclusion $\Grad^2(\fX)_{\mathrm{mat}} \subset \Grad^2(\fX)$.
    Further, as a correspondence of derived Artin stacks,
    they are naturally equivalent to the following correspondence:
    \[
    \xymatrix{
    {}
    & {\Filt^2(\fX)_{\mathrm{mat}}} \ar[ld] \ar[rd]
    & {} \\
    {\Grad^2(\fX)_{\mathrm{mat}}}
    & {}
    & {\fX.}
    }
    \]
\end{cor}

\subsection{Equivariant chart of stacks}\label{ssec:QChart}

In this section, we will introduce the notion of a $\bG_m^n$-equivariant chart of $\fX$, which will be used to define charts for $\Grad(\fX)$.

\begin{defin}$ $

\begin{thmlist}
    \item[(1)]  Let $\fX$ be an Artin stack.
    A \defterm{$\bG_m^n$-equivariant chart} is a tuple $\fQ =
 (X, \mu, \bar{q})$ where $X$ is a quasi-separated algebraic space, 
 $\mu$ is a $\bG_m^n$-action on $X$ and  $\bar{q} \colon [X / \bG_m^n] \to \fX$ is a smooth morphism.

    \item[(2)] A \defterm{morphism of $\bG_m^n$-equivariant charts} from 
    $\fQ_1 = (X_1, \mu_1, \bar{q}_1)$ to $\fQ_2 = (X_2, \mu_2, \bar{q}_2)$ is given by a smooth
    $\bG_m^n$-equivariant morphism $f : X_1 \to X_2$ together with a homotopy $\bar{q}_1 \sim \bar{q}_2 \circ [f / \bG_m^n]$.
    We let $\mathrm{Cha}_{\fX }^{\bG_m^n}$ denote the category of $\bG_m^n$-equivariant charts on $\fX$.

\end{thmlist}

\end{defin}

For a $\bG_m^n$-equivariant chart of $\fX$ we introduce the following additional notation:
\begin{itemize}
    \item We let $q \colon X \to \fX$ be the composite $X\rightarrow [X/\bG_m^n]\xrightarrow{\bar{q}} \fX$.
    \item We let $q_\mu\colon X^\mu\to \Grad(\fX)$ be the composite $X^\mu\rightarrow \Grad([X/\bG_m^n])\xrightarrow{\Grad(\bar{q})} \Grad(\fX)$, where the first morphism is an inclusion of a component as in Theorem \ref{thm:Grad_quotient}. It follows from \cite[Proposition 1.3.1 (4)]{hlp14} that $q_{\mu}$ is a smooth morphism.
\end{itemize}

The following statement will be useful to glue sections of sheaves on $\Grad(\fX)$.

\begin{prop}[{\cite[Lemma 4.4.6]{hlp14}}]\label{prop:QCha_atlas}
Let $\fX$ be a quasi-separated Artin stack with affine stabilizers. Then the morphism
\[
\left (\coprod_{\substack{ \fQ = (X, \mu, \bar{q}) \in \mathrm{Cha}_{\fX}^{\bG_m^n} }} 
X^{\mu}_{}  \right)
\xrightarrow{\coprod q_{\mu}}
\Grad^n(\fX).
\]
is a smooth atlas of $\Grad^n(\fX)$.
\end{prop}
\begin{proof}
Since $q_{\mu}$ is smooth, the morphism in the statement is smooth, so we only need to show that it is surjective.
To see this,
take a point $x \in \fX$ and a map $\iota \colon \bG_m^n \to G_x$ where $G_x$ denote the stabilizer group.
Set $T \coloneqq \iota(\bG_m^n)$. 
Then by using \cite[Theorem 1.1]{ahr20},
we can find an affine scheme $\Spec A$ with a $T$-action $\mu$ together with a point $\hat{x} \in \Spec A^T$ and a smooth morphism $[\Spec A / T] \to \fX$ which restricts to the natural morphism
$BT \to B G_x$ at the point $\hat{x}$.
Let $\tilde{\mu}$ be a $\bG_m^n$-action on $\Spec A$ induced by $\mu$ and $\bar{q} \colon [\Spec A / \bG_m^n] \to \fX$ be the naturally induced morphism.
Then the tuple $(\Spec A, \tilde{\mu}, \bar{q})$ is a $\bG_m^n$-equivariant chart of $\fX$
and the morphism $q_{\tilde{\mu}} \colon \Spec A^{\bG_m^n} \to \Grad(\fX)$ maps $\hat{x}$ to $x$.
\end{proof}

\begin{cor}\label{cor:QCha_glue}
    Let $\fX$ be a quasi-separated Artin stack with affine stabilizers.
    Assume that we are given a lisse-\'etale sheaf $\cF$ on $\Grad^n(\fX)$.
    Assume further that we are given the following data:
    \begin{enumerate}
        \item For each $\fQ = (X, \mu, \bar{q}) \in \mathrm{Cha}_{\fX}^{\bG_m^n}$,
        a choice of a section $s_{\fQ} \in \Gamma(X^{\mu}, q_{\mu}^*\cF)$ which satisfies the following condition:
        For a morphism 
        $f \colon \fQ_1 = (X_1, \mu_1, \bar{q}_1) \to \fQ_2 = (X_2, \mu_2, \bar{q}_2)$ in $\mathrm{Cha}_{\fX }^{\bG_m^n}$,
        the following identity in $\Gamma(X_{1}^{\mu_1}, q_{1, \mu_1}^*\cF)$ holds:
        \begin{equation}\label{eq:overlap_same}
        f^{\bG_m^n, \star} s_{\fQ_2} = s_{\fQ_1}.
        \end{equation}
         Here $f^{\bG_m^n} \colon X_1^{\mu_1} \to X_2^{\mu_2}$ is the induced morphism.
    \end{enumerate}
    Then there exists a section $s \in \Gamma(\Grad(\fX), \cF)$ such that for each $\fQ = (X, \mu, \bar{q}) \in \mathrm{Cha}_{\fX }^{\bG_m^n} $,
    the identity $q_{\mu}^{\star} s = s_{\fQ}$ holds.
\end{cor}

\begin{proof}

Consider the following diagram:
\[
\left(\coprod_{\substack{\fQ_1, \fQ_2 \in \mathrm{Cha}_{\fX}^{\bG_m^n}, \\ \fQ_1 = (X_1, \mu_1, \bar{q}_1), \\
\fQ_2 = (X_2, \mu_2, \bar{q}_2).}} 
X_1^{\mu_1} \times_{\Grad(\fX)} X_2^{\mu_2} \right) \mathrel{\mathop{\rightrightarrows}^{\pr_1}_{\pr_2}}
\left (\coprod_{\substack{\fQ \in \mathrm{Cha}_{\fX}^{\bG_m^n}, \\ \fQ = (X,\mu,  q) }} 
X^{\mu}_{}  \right)
\xrightarrow{\coprod q_{\mu}}
\Grad^n(\fX).
\]
By using the descent of a section of a lisse-\'etale sheaf,
it is enough to prove the identity $\pr_1^{\star} 
 \{ s_{\fQ} \}_{\fQ} = \pr_2^{\star} 
 \{ s_{\fQ} \}_{\fQ}$.
We take objects $\fQ_i = (X_i, \mu_i,  q_i) \in \mathrm{Cha}_{\fX}^{\bG_m^n}$ for $i = 1, 2$.
We set $Y \coloneqq X_1 \times_{\fX} X_2$.
Then $Y$ admits a diagonal action $\nu$ of $\bG_m^n$.
Let $\bar{r} \colon [Y / \bG_m^n] \to \fX$ be the natural morphism and 
define a $\bG_m^n$-equivariant chart $\fR = (Y, \nu, \bar{r})$.
Consider the following Cartesian diagram:
\[
\xymatrix{
{[Y / \bG_m^n \times \bG_m^n]} \ar[r] \ar[d]
& {[X_1 / \bG_m^n]} \ar[d] \\
{[X_2 / \bG_m^n]} \ar[r]
& {\fX.}
}
\]
By taking the mapping stack from $B \bG_m^n$,
we obtain the following Cartesian diagram:
\[
\xymatrix{
{[Y^{\nu} / \bG_m^n \times \bG_m^n]} \ar[r] \ar[d]
& {[X_1^{\mu_1} / \bG_m^n]} \ar[d] \\
{[X_2^{\mu_2} / \bG_m^n]} \ar[r]
& {\Grad(\fX).}
}
\]
In particular, there exists a natural isomorphism $Y^{\nu} \cong X_1^{\mu_1} \times_{\Grad(\fX)} X_2^{\mu_2} $ over $\Grad(\fX)$.
Under this identification,
we have an identity 
\[
\pr_1^{\star} s_{\fQ_1} = s_{\fR} = \pr_2^{\star} s_{\fQ_2}
\]
as desired.

\end{proof}

By a similar argument, we can prove the following statement:

\begin{cor}\label{cor:QCha_glue_stack}
    Let $\fX$ be a quasi-separated Artin stack with affine stabilizers.
    Assume that we are given a lisse-\'etale stack in groupoids $\fF$ on $\Grad^n(\fX)$.
    Assume further that we are given the following data:
    \begin{enumerate}
        \item For each $\fQ = (X, \mu, \bar{q}) \in \mathrm{Cha}_{\fX}^{\bG_m^n}$,
        there exists a section $o_{\fQ} \in \Gamma(X^{\mu}, q_{\mu}^* \fF)$.

        \item For a morphism 
        $f \colon (\fQ_1 = (X_1, \mu_1, \bar{q}_1) ) \to (\fQ_2 = (X_2, \mu_2, \bar{q}_2) )$,
        there exists an isomorphism
        \[
         \Theta_{f} \colon o_{\fQ_1} \simeq f_{}^{\bG_m^n, \star} o_{\fQ_2}
        \]
        such that if we are further given a morphism $g \colon \fQ_2 \to \fQ_3$, the following diagram commutes:
        \begin{equation}\label{eq:assoc_general}
        \begin{aligned}
        \xymatrix{
        o_{\fQ_1} \ar[r]_-{\simeq}^-{\Theta_f} \ar[d]_-{\simeq}^-{\Theta_{g \circ f}}
        & {f_{}^{\bG_m^n, \star} o_{\fQ_2}} \ar[d]_-{\simeq }^-{\Theta_g} \\
        (g_{}^{\bG_m^n} \circ f_{}^{\bG_m^n})^{\star} o_{\fQ_3} \ar[r]^-{\simeq}
        & {f^{\bG_m^n, \star} g_{}^{\bG_m^n, \star} o_{\fQ_3}.}
        }
        \end{aligned}
        \end{equation}
        
    \end{enumerate}
    Then there exists a section $o \in \Gamma(\fX, \fF)$ such that for each $\fQ \in \mathrm{Cha}_{\fX}^{\bG_m^n}$,
    there exists an isomorphism $\Theta_{q} \colon o_{\fQ} \simeq q_{\mu}^{\star} o$ such that the associativity property as in \eqref{eq:assoc_general} holds.
\end{cor}

The following technical statements will be used for the associativity of the cohomological Hall algebras later:

\begin{lem}\label{lem:matching_chart}
    Let $\fX$ be a quasi-separated Artin stack with affine stabilizer groups.
    Then for any closed point $x \in \Grad^2(\fX)_{\mathrm{mat}}$,
    one can find a $\bG_m^2$-equivariant chart $\fQ = (X, \hat{\mu} = (\mu_1, \mu_2), \bar{q})$ and a point $\hat{x} \in X^{\hat{\mu}}_{\mathrm{mat}}$ such that $q_{\hat{\mu}}(\hat{x}) = x$ holds.
    Here $X^{\hat{\mu}}_{\mathrm{mat}}$ denotes the preimage of $[X^{\hat{\mu}} / \mathbb{G}_m^2]_{\mathrm{mat}}$.
\end{lem}

\begin{proof}
    By using \cite[Theorem 1.1]{ahr20},
    we may take a $\bG_m^2$-equivariant chart $(\Spec A, \hat{\mu} = (\mu_1, \mu_2), \bar{q})$
    and a point $\hat{x} \in \Spec A ^{\hat{\mu}}$ with $q_{\hat{\mu}}(\hat{x}) = x$.
    By possibly replacing $\Spec A$ with its $\bG_m^2$-equivariant subspace containing $x$,
    we may assume that $T_{\Spec A, \hat{x}}  \cong  T_{\fX , u(x)}$ holds.
    In particular, $\hat{x}$ is contained in the matching locus.
\end{proof}

As an application of Proposition \ref{prop:QCha_atlas}, we will explain that a d-critical structure is inherited to the stack of graded points:

\begin{prop}\label{prop:grad_dcrit}
    Let $(\fX, s)$ be a d-critical stack with $\fX$ quasi-separated with affine stabilizers.
    Then $(\Grad^n(\fX), u^\star s)$ is a d-critical stack.
\end{prop}

\begin{proof}
    By Proposition \ref{prop:QCha_atlas} the collection of $\bG_m^n$-equivariant charts $\fQ = (X, \mu, \bar{q}) \in \mathrm{Cha}^{\bG_m^n}_{\fX}$ defines a smooth atlas $\{q_{\mu} \colon X^{\mu} \rightarrow \Grad^n(\fX) \}$. Therefore, by \cite[Proposition 2.54]{joy15} it is enough to show that $q_{\mu}^\star u^\star s$ is a d-critical structure on $X^{\mu}$ for each equivariant chart. Using the commutative diagram
\[
\xymatrix{
X_{}^{{\mu}}\ar[r] \ar^{q_{\mu}}[d] & X \ar^{q}[d] \\
\Grad^n(\fX) \ar[r]^-{u} & \fX
}
\]
we have $q_{\mu}^\star u^\star s = (q^\star s)^{\mu}$. Since $(\fX, s)$ is a d-critical stack, $q^\star s$ is a d-critical structure and hence $(q^\star s)^{\mu}$ is a d-critical structure by Corollary \ref{cor:dcriticallocalization}.
\end{proof}

We now compare the d-critical structure constructed in Proposition \ref{prop:grad_dcrit} and the $(-1)$-shifted symplectic structure constructed in Corollary \ref{cor:grad_symplectic_lagrangian}.

\begin{prop}\label{prop:comparison_grad_dcrit}
    Let $(\fX, \omega_{\fX})$ be a $(-1)$-shifted symplectic Artin stack with $\fX^{\cl}$ quasi-separated with affine stabilizers.
    Let $(\fX^{\cl}, s)$ be the underlying d-critical locus. Then the d-critical structure underlying $(\Grad^n(\fX), u^\star \omega_{\fX})$ is $(\Grad^n(\fX^{\cl}), u^\star s)$.
\end{prop}

\begin{proof}
Consider the commutative diagram
\[
\xymatrix@C=70pt{
\Grad^n(\fX^{\cl}) \ar^{\Grad^n(j)}[r] \ar^{u^{\cl}}[d] & \Grad(\fX) \ar^{u}[d] \\
\fX^{\cl} \ar^{j}[r] & \fX.
}
\]
Identifying sections of $\cS^0$ with closed $2$-forms of degree $-1$ using \eqref{eq:Sforms}, the d-critical structure $s$ becomes $j^\star\omega_{\fX}$ by Theorem \ref{thm:BBBBJ}. Similarly, the d-critical structure underlying $u^\star \omega_{\fX}$ becomes $\Grad^n(j)^\star u^\star \omega_{\fX}= u^{\cl, \star} j^\star \omega_{\fX}$, hence we conclude.
\end{proof}

\section{Index line bundles and localization of orientations}

Let $(\fX, s, o)$ be an oriented d-critical stack.
The aim of this section is to define an orientation $u^{\star} o$ on $\Grad(\fX)$ in a way that is compatible with a similar construction for orientations of $(-1)$-shifted symplectic stacks we have already seen in Lemma \ref{lem:Lagrangiancorrespondenceretract}.
To do this, we will construct the index line bundle $\cL_{\fX, s}$ on $\Grad(\fX)^{\red}$
together with an isomorphism
$K_{\Grad(\fX), u^{\star} s} \otimes \cL_{\fX, s}^{\otimes 2} \cong K_{\fX, s} |_{\Grad(\fX)^{\red}}$.

In \S \ref{ssec:Indices}, we will introduce the notion of indices for d-critical stacks, which is a locally constant function on graded points.
In \S \ref{ssec:index_line_space} and \S \ref{ssec:index_line_stack}, we will introduce  index line bundles for $\bG_m$-equivariant d-critical spaces and d-critical stacks, respectively.
We recommend readers skip these two sections for the first read since it is technical.
In \S \ref{ssec:Localizing_ori}, we will explain the construction of localized orientation on the stack of graded points and prove its properties.

\subsection{Notation on rank and determinant for graded complexes}

Throughout the section, we will use the following notations:
If we are given a graded perfect complex $E^{\bullet} = \{ E^n \}_{n \in \bZ}$ on a derived Artin stack $\fX$,
we set $E^{+}$ (resp. $E^{-}$) denote the positive (resp. negative) part of $E$ and $E^{\pm} \coloneqq E^{+} \oplus E^{-}$.
We define

\begin{align*}
    &\rank_+ (E^{\bullet}) \coloneqq \rank (E^{+}), \quad \rank_- (E^{\bullet}) \coloneqq \rank (E^{-}), \quad \rank_0 (E^{\bullet}) \coloneqq \rank (E^{0}), \\
    &\rdet_+(E^{\bullet}) \coloneqq \rdet (E^{+}), \quad \rdet_-(E^{\bullet}) \coloneqq \rdet (E^{-}), \quad \rdet_0(E^{\bullet}) \coloneqq \rdet (E^{0}), \quad \rdet_{\pm}(E) \coloneqq \rdet(E^{\pm}). \\
\end{align*}

For a $\bG_m$-equivariant algebraic space $X$ and a $\bG_m$-equivariant perfect complex $E$ on $X$,
we regard $E |_{X^{\bG_m}}$ as a graded vector space using the $\bG_m$-action and 
define $E^{+}$ to be the positive part of $E |_{X^{\bG_m}}$.
We define $\rank_+(E) \coloneqq \rank(E^{+})$ and similarly for other operations.

For an Artin stack $\fX$ and a perfect complex $E$ on $\fX$, we regard $E|_{\Grad(\fX)}$ as a graded complex by using the $\bG_m$-action and define $E^{+}$ to be the positive part of $E |_{\Grad(\fX)}$.
We define $\rank_+(E) \coloneqq \rank (E^{+})$ and similarly for other operations.

For a doubly graded perfect complex $E^{\bullet, \bullet}$ and a subset $\Sigma \subset \bZ^2$,
we let $E^{\Sigma} \subset E$ denote the subspace consisting of complexes whose grading is concentrated in $\Sigma$. We set $\rank_{\Sigma} (E) \coloneqq \rank E^{\Sigma}$
and $\rdet_{\Sigma}(E) \coloneqq \rdet(E^{\Sigma})$.
We define $\rank_{\Sigma}$ and $\rdet_{\Sigma}$ for $\bG_m^2$-equivariant complexes on  algebraic spaces and Artin stacks similarly to $\rank_{+}$.

\subsection{Indices for d-critical stacks}\label{ssec:Indices}

Let $(\fX, s)$ be a d-critical Artin stack with $\fX$ quasi-separated with affine stabilizers.
The aim of this section is to introduce a locally constant function on $\Grad(\fX)$ called the index for $\fX$.
We first deal with the case of algebraic space with a $\bG_m$-action.

\begin{defin}\label{def:indexalgspace}
Let $X$ be an algebraic space with $\bG_m$-action $\mu$. The \defterm{index function} $I^{\mu}_X\colon X^{\mu}\rightarrow \bZ$ is given by
\[
I^{\mu}_X(x) \coloneqq \rank_{+}(T_{X, x}) - \rank_{-}(T_{X, x})
\]
for a point $x \in X^{\mu}$.
\end{defin}

Assume that $X$ admits a $\bG_m$-invariant \'etale d-critical structure $s \in \Gamma(X, \cS_X^{0})^{\bG_m}$.
Take a $\bG_m$-equivariant \'etale d-critical chart $\fR = (R, \eta, U, f, i)$ of $x$ such that $\dim T_{X, x} = \dim U$ holds.
Then we have an equality $I^{\mu}_{X}(x) = \rank(T_{U, x}^{+}) - \rank(T_{U, x}^{-})$.
Using Theorem \ref{thm:T_compare_d-critical}, we see that the same equality holds for arbitrary $\bG_m$-equivariant \'etale d-critical chart of $x$.
In particular, we have the following:

\begin{prop}\label{prop:index_locally_constant}
   Let $X$ be an algebraic space with $\bG_m$-action $\mu$, which admits a $\bG_m$-invariant d-critical structure. Then the function
   \[
    I^{\mu}_X\colon X^\mu\longrightarrow\bZ
   \]
   is locally constant.
\end{prop}

\begin{rmk}
    This proposition is false if we do not assume the existence of a $\bG_m$-invariant d-critical structure.
    For example, consider the scheme $X = \{ (x, y) \in \bA^2 \mid xy = 0   \}$ with a $\bG_m$-action $\mu$ of weight $(1, 0)$. Then the function $I^\mu_X$ is not locally constant near the origin.
\end{rmk}

Assume that we are given algebraic spaces $X_1$ and $X_2$ with $\bG_m$-actions $\mu_1$ and $\mu_2$,
$\bG_m$-equivariant smooth morphism $h \colon X_1 \to X_2$.
We define the locally constant function $I^{\mu_1}_h$ on $X_1^{\mu_1}$ by $I^{\mu_1}_h \coloneqq \rank_{+} (T_{X_1 / X_2}) - \rank_{-} (T_{X_1 / X_2})$.
Then we have the following identity of functions on $X_1^{\mu_1}$
\begin{equation}\label{eq:smooth_index}
    I_{X_2}^{\mu_2} + I_{h}^{\mu_1} = I_{X_1}^{\mu_1}.
\end{equation}

Let $X$ be a quasi-separated algebraic space with a $\bG_m^2$-action $\hat{\mu} = (\mu_1, \mu_2)$.
We say that a point $x \in X^{\hat{\mu}}$ is \defterm{matching} if its image in 
$[X^{\hat{\mu}} / \bG_m^2]$ is a matching point (see \S \ref{ssec:Matching} for the definition).
We let $X^{\hat{\mu}}_{\mathrm{mat}} \subset X^{\hat{\mu}}$ denote the open subspace of matching points. 
For an algebraic space $Y$ over $X^{\hat{\mu}}$, we let $Y_{\mat}$ denote the base change to $X_{\mat}^{\hat{\mu}}$.
The following proposition is a consequence of Proposition \ref{prop:matching_invertible_stack}:

\begin{prop}\label{prop:matching_invertible}
    Consider the following commutative diagram:
    \[
    \xymatrix{
    {(X^{\mu_1, +})^{\mu_2} \times_{X^{\mu_2}} X^{\mu_2, +}} \ar[rd]
    & {X^{\hat{\mu}, +, +}} \ar[l]_-{\alpha} \ar[r]^-{\beta}
    & {(X^{\mu_2, +})^{\mu_1} \times_{X^{\mu_1}} X^{\mu_1, +}} \ar[ld] \\
    {}
    & {X^{\hat{\mu}}.}
    & {}
    }
\]
Then the maps $\alpha$ and $\beta$ are isomorphisms over the matching locus $X^{\hat{\mu}}_{\mathrm{mat}}$.
\end{prop}

The above proposition implies the following:

\begin{cor}\label{cor:matching_correspondence}
The composites of the following correspondences coincide:
\[
\xymatrix{
{}
& {}
& {((X^{\mu_1, +})^{\mu_2} \times_{X^{\mu_2} } X^{\mu_2, +}})_{\mat}  \ar[ld] \ar[rd] 
& {}
& {} \\
{}
& {(X^{\mu_1, +})^{\mu_2}_{\mat}  } \ar[ld] \ar[rd]
& {}
& {X^{\mu_2, +}} \ar[ld] \ar[rd]
& {} \\
{X^{\hat{\mu}}_{\mathrm{mat}}}
& {}
& {X^{\mu_2}}
& {}
& {X,} 
}
\]
\[
\xymatrix{
{}
& {}
& {((X^{\mu_2, +})^{\mu_1} \times_{X^{\mu_1} } X^{\mu_1, +}})_{\mat} \ar[ld] \ar[rd] 
& {}
& {} \\
{}
& {(X^{\mu_2, +})^{\mu_1}_{\mat}} \ar[ld] \ar[rd]
& {}
& {X^{\mu_1, +}} \ar[ld] \ar[rd]
& {} \\
{X^{\hat{\mu}}_{\mathrm{mat}}}
& {}
& {X^{\mu_1}}
& {}
& {X.} 
}
\]
Further, these correspondences are naturally isomorphic to the following correspondence:
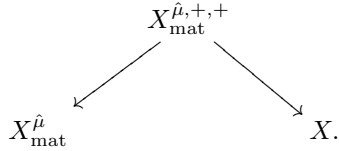
\[
\xymatrix{
{}
& {X^{\hat{\mu}, +, +}_{\mat}} \ar[ld] \ar[rd]
& {} \\
{X^{\hat{\mu}}_{\mathrm{mat}}}
& {}
& {X.}
}
\]
\end{cor}

From the definition of matching points, we get the following identity of locally constant functions on $X^{\hat{\mu}}_{\mathrm{mat}}$:
    \begin{equation}\label{eq:index_value_matching}
    I^{\mu_1}_X |_{X^{\hat{\mu}}_{\mathrm{mat}}} + (I^{\mu_2}_{X^{\mu_1}}) |_{X^{\hat{\mu}}_{\mathrm{mat}}} = I^{\mu_2}_X |_{X^{\hat{\mu}}_{\mathrm{mat}}} + (I^{\mu_1}_{X^{\mu_2}}) |_{X^{\hat{\mu}}_{\mathrm{mat}}}.
    \end{equation}

We now introduce the index function for Artin stacks.
    
\begin{defin}\label{def:indexstack}
Let $\fX$ be an Artin stack. The \defterm{index function} is the function $I_{\fX}\colon \Grad(\fX)\rightarrow \bZ$ given by
\[
I_{\fX}(x) \coloneqq (\rank_+ H^0(\bT_{\fX, x}) - \rank_- H^0(\bT_{\fX, x})) - (\rank_+ H^{-1}(\bT_{\fX, x}) - \rank_- H^{-1}(\bT_{\fX, x}))
\]
for a point $x \in \Grad(\fX)$.
\end{defin}
 
Assume that we are given a $\bG_m$-equivariant chart $(X, \mu, \bar{q})$ of $\fX$.
Set
\[
I_{q}^{\mu} \coloneqq \rank_+ T_{X / \fX} - \rank_- T_{X / \fX} .
\]
Then for a point $\hat{x} \in X^{\mu}$ with $q_{\mu}(\hat{x}) = x$, we have the following identity
        \begin{equation}\label{eq:smooth_index_stack}
        I_{\fX}(x) = I^{\mu}_X(\hat{x}) - I^{\mu}_{q}(\hat{x}).
        \end{equation}
Assume that $\fX$ admits a d-critical structure.
Then the equality \eqref{eq:smooth_index_stack} together with Proposition \ref{prop:index_locally_constant} implies that the function $I_{\fX}$ is locally constant.
Assume further that $\fX$ is the underlying stack of a $(-1)$-shifted symplectic stack $\widehat{\fX}$.
Since for each point $x \in \Grad(\fX)$ we have an equality
\[
\rank_+ H^{2}(\bT_{\widehat{\fX}, u(x)}) = \rank_{-} H^{-1}(\bT_{\widehat{\fX}}) , \quad
\rank_+ H^{1}(\bT_{\widehat{\fX}, u(x)}) = \rank_- H^{0}(\bT_{\widehat{\fX}, u(x)}),
\]
we obtain an identity $I_{\fX} = \rank_{+} \bT_{\widehat{\fX}}  $.
Since we have $(\bT_{\Filt(\widehat{\fX}) / \Grad(\widehat{\fX})}) |_{\Grad(\fX)} \simeq \bT_{\widehat{\fX}} ^{+}$ and $\vdim (\Grad(\widehat{\fX})) = 0$, we obtain
        \begin{equation}\label{eq:Ind=vdim}
        I_{\fX} = \vdim  \Filt(\widehat{\fX})|_{\Grad(\fX)}.
        \end{equation}

We set $I^1_{\fX} \coloneqq I_{\Grad(\fX)}\colon \Grad^2(\fX)\rightarrow\bZ$.
We let $\sigma \colon \Grad^2(\fX) \cong \Grad^2(\fX)$ denote the isomorphism induced by the swapping isomorphism $\bG_m^2 \cong \bG_m^2$ and set $I^2_{\fX} \coloneqq I^1_{\Grad^2(\fX)} \circ \sigma\colon \Grad^2(\fX)\rightarrow \bZ$.
By the definition of the matching locus, we have the following identity:
\begin{equation}\label{eq:index_matching}
    I^1_{\fX} |_{\Grad^2(\fX)_{\mathrm{mat}}} = I^2_{\fX} |_{\Grad^2(\fX)_{\mathrm{mat}}}.
\end{equation}

\subsection{Index line bundles for $\bG_m$-equivariant d-critical spaces}\label{ssec:index_line_space}

Now we will introduce the index line bundle for a $\bG_m$-equivariant d-critical algebraic spaces.
Let $X$ be a quasi-separated $\bG_m$-equivariant d-critical algebraic space.
For a $\bG_m$-equivariant \'etale d-critical chart $\fR = (R, \eta, U, f, i)$ of $(X, s)$,
define a graded line bundle on $R^{\bG_m, \red}$ by
\[
\cL^{\Ind}_{\mu, s, \fR} \coloneqq \rdet_+ (\bL_{\fR})
\]
where $\bL_{\fR}$ is the virtual cotangent complex defined in Definition \ref{defin:virtual_cot_complex}.
For each point $\hat{x} \in R^{\bG_m}$ with $\eta(\hat{x}) = x$, we construct an isomorphism 
\[
\kappa^{\Ind}_{\hat{x}, \fR} \colon \cL^{\Ind}_{\mu, s, \fR}|_{\hat{x}} \cong \det_{\pm}(\Omega_{X, x}^{})
\]
by using the following fibre sequence
\[
 T_{X, x}^{+}[1] \to \bL_{\fR}|_{\hat{x}}^+ \to \Omega_{X, x}^{+}
\]
and \eqref{eq:det_shift}, \eqref{eq:det_dual}.
If we are given two $\bG_m$-equivariant \'etale d-critical charts $\fR_1 = (R_1, \eta_1, U_1, f_1, i_1)$ and $\fR_2 = (R_2, \eta_2, U_2, f_2, i_2)$ and points $\hat{x}_1 \in R_1^{\bG_m}$ and $\hat{x}_2 \in R_2^{\bG_m}$ with $x = \eta_1(\hat{x}_1) = \eta_2(\hat{x}_2)$,
we can construct the following isomorphism
\[
 \cL^{\Ind}_{\mu, s, \fR_1}|_{\hat{x}_1} \xrightarrow[\cong]{\kappa^{\Ind}_{\hat{x}_1, \fR_1}}  \rdet_{\pm}(\Omega_{X, x}) \xrightarrow[\cong]{\kappa^{\Ind, -1}_{\hat{x}_2, \fR_2}} \cL^{\Ind}_{\mu, s, \fR_2}|_{\hat{x}_2}.
\]
We claim the composite algebraizes: namely, it induces an isomorphism of line bundles over $(R_1^{\bG_m} \times_{X^{\mu}} R_2^{\bG_m})^{\red}$.
To prove this, using Theorem \ref{thm:T_compare_d-critical}, we may assume $R_1 = R_2$, $\eta_1 = \eta_2$,  $U_2 = U_1 \times \bA^n$ and $f_2 = f_1 \boxplus q$ where $q$ is a $\bG_m$-invariant quadratic function.
Set $\bL_{\fR_2 / \fR_1}^+ \coloneqq [T_{\bA^n, 0}^+ \xrightarrow{\Hess_q^+} \Omega_{\bA^n, 0}^+]$.
The null-homotopy $\bL_{\fR_2 / \fR_1}^+ \simeq 0$ induces
 a  quasi-isomorphism
\[
\bL_{\fR_1}^+ \simeq \bL_{\fR_2}^+
\] 
hence an isomorphism 
\begin{equation}\label{eq:embed_L_Ind_isom}
\cL^{\Ind}_{\mu, s, \fR_1} \cong \cL^{\Ind}_{\mu, s, \fR_2}
\end{equation}
which induces the desired isomorphism on each stalk.

By construction, the isomorphism $\cL^{\Ind}_{\mu, s, \fR_1} |_{(R_1^{\bG_m} \times_{X^{\mu}} R_2^{\bG_m})^{\red}} \cong \cL^{\Ind}_{\mu, s, \fR_2} |_{(R_1^{\bG_m} \times_{X^{\mu}} R_2^{\bG_m})^{\red}}$ satisfies the cocycle condition, since it does on each stalk.
In particular, $\cL^{\Ind}_{\mu, s, \fR} $ glues to define a line bundle $\cL^{\Ind}_{\mu, s}$ on $X^{\mu, \red}$
with a choice of an isomorphism
\[
\iota_{\fR}^{\Ind}  \colon \cL^{\Ind}_{\mu, s} |_{R^{\bG_m, \red}} \cong \det_{+}(\bL_{\fR})
\]
for each $\bG_m$-equivariant d-critical chart $\fR = (R, \eta, U, f, i)$.
We call $\cL^{\Ind}_{\mu, s}$ the \defterm{index line bundle}.

We will study the behaviour of the index line bundle.
The following proposition summarizes a list of properties of index line bundles:

\begin{prop}\label{prop:index_line_bundle}
    \begin{thmlist}
        \item Let $(X, \mu, s)$ be a quasi-separated $\bG_m$-equivariant d-critical algebraic space.
        Then for each point $x \in X^{\mu}$, there exists a natural isomorphism 
       \begin{equation}
           \kappa^{\Ind}_{x} \colon \cL_{\mu, s}^{\Ind} |_{x} \cong \rdet_{\pm}(\Omega_{X, x}).
       \end{equation}

       \item Let $(X, \mu, s)$ be as above. Then there exists a natural isomorphism
       \begin{equation}
           \xi_{\mu, s} \colon K_{X^{\mu}, s^{\mu}} \otimes \cL_{\mu, s}^{\Ind, \otimes 2} \cong K_{X, s} |_{X^{\mu, \red}}.
       \end{equation}

        \item Let $(X_1, \mu_1, s_1)$ and $(X_2, \mu_2, s_2)$ be quasi-separated $\bG_m$-equivariant d-critical algebraic spaces and $h \colon X_1 \to X_2$ be a smooth morphism compatible with $\bG_m$-actions such that $h^{\star} s_2 = s_1$ holds.
        Then there exists a natural isomorphism
\[
\Upsilon_{h}^{\Ind} \colon h^{\bG_m, \red, *} \fL_{\mu_2, s_2}^{\Ind} \otimes K_{X_1 / X_2}^{\pm} \cong \fL_{\mu_1, s_1}^{\Ind}
\]
        where we set $K_{X_1 / X_2}^{\pm} \coloneqq \rdet_{\pm}(\Omega_{X_1 /X_2})$.

        \item Assume that we are given a quasi-separated $\bG_m^2$-equivariant algebraic space $(X, \hat{\mu} =   (\mu_1, \mu_2), s)$.
        Then there exists an isomorphism
        \[
        \theta_{\hat{\mu}, s} \colon \cL^{\Ind}_{\mu_1, s} |_{X^{\hat{\mu}, \red}} \otimes (\cL^{\Ind}_{\mu_2|_{X^{\mu_1}}, s^{\mu_1}})  
        \cong 
        \cL^{\Ind}_{\mu_2, s} |_{X^{\hat{\mu}, \red}} \otimes (\cL^{\Ind}_{\mu_1|_{X^{\mu_2}}, s^{\mu_2}}) .
        \]

        \item Let $(X, \mu, s)$ and $(Y, \nu, t)$ be quasi-separated $\bG_m$-equivariant d-critical algebraic spaces.
        Then there exists a natural isomorphism
        \[
        \Psi^{\Ind}_{\mu, \nu} \colon \cL_{\mu, s}^{\Ind} \boxtimes \cL_{\nu, t}^{\Ind}
        \cong \cL_{\mu \times \nu, s \boxplus t}^{\Ind}.
        \]

        \item Let $X$ be a quasi-separated derived algebraic space with a $\bG_m$-action $\mu$ and a $\bG_m$-invariant $(-1)$-shifted symplectic structure $\omega_X$.
        Let $(X^{\cl}, s)$ be the underlying $d$-critical structure. 
        Then there exists a natural isomorphism
        \[
        \Lambda_{\mu}^{\Ind} \colon \rdet_+(\bL_{X} ) \cong \cL_{\mu^{\cl}, s}^{\Ind}.
        \]
       
    \end{thmlist}
\end{prop}

\begin{proof}
    For (i), the existence of natural isomorphisms $\kappa^{\Ind}_{x}$ is obvious from the construction.
    Note that for each $\bG_m$-equivariant \'etale d-critical chart $\fR = (R, \eta, U, f, i)$ and a point $\hat{x} \in R^{\bG_m}$ lifting $x$, the following diagram commutes:
    \begin{equation}\label{eq:iota_kappa_Ind}
        \begin{aligned}
            \xymatrix@C=100pt{
            {\cL_{\mu, s}^{\Ind} |_{x}}
            \ar[r]_-{\cong}^-{\kappa_{x}^{\Ind}}
            \ar[d]_-{\cong}^-{\iota_{\fR}|_{\hat{x}}}
            & {\det_{\pm}(\Omega_{X, x})}
            \ar[d]_-{\cong}^-{\eqref{eq:det_shift}, \eqref{eq:det_dual}} \\
            {\det_+(\bL_{\fR} |_{\hat{x}})}
            \ar[r]^-{\cong}
            & {\det_+(T_{R, \hat{x}} [1] ) \otimes \det_+(\Omega_{R, \hat{x}}). }
            }
        \end{aligned}
    \end{equation}

    We now prove (ii). 
    For each $\bG_m$-equivariant \'etale d-critical chart $\fR = (R, \eta, U, f, i)$, there exists a natural equivalence of two-term complexes on $R^{\bG_m}$
    \begin{equation*}
     \bT_{\fR}^{+}[1] \simeq \bL_{\fR}^{+}
     \end{equation*}
     which is the identity map on each term.
     This equivalence together with the isomorphisms \eqref{eq:det_shift} and $\eqref{eq:det_dual}$ induce an isomorphism
     \begin{equation}\label{eq:action_TL_det}
     \rdet_-(\bL_{\fR} ) \cong \rdet_+(\bL_{\fR}). 
     \end{equation}
    For each $\hat{x} \in R^{\bG_m}$ with $\eta(\hat{x}) = x$, the following diagram commutes by \eqref{eq:double_dual}:
     \begin{equation}
     \begin{aligned}\label{eq:action_kappa_pm_comm}
         \xymatrix@C=100pt{
         {\det_-(\bL_{\fR} |_{\hat{x}})} \ar[r]_-{\cong}^-{\eqref{eq:action_TL_det}} \ar[d]^-{\cong}
         & {\det_+(\bL_{\fR} |_{\hat{x}})} \ar[d]^-{\cong} \\
         {\det_{\pm}(\Omega_{X, x})} \ar[r]^-{(-1)^{\rank_+ \Omega_{X, x}} \cdot \id}_-{\cong}
         & {\det_{\pm}(\Omega_{X, x})}.
         }
         \end{aligned}
     \end{equation}
     We define a map
    \[
    \xi_{\mu, s, \fR} \colon K_{X^{\mu}, s^{\mu}} |_{R^{\bG_m, \red}} \otimes \cL_{\mu, s}^{\Ind, \otimes 2} |_{R^{\bG_m, \red}} \cong K_{X, s} |_{R^{\bG_m, \red}}
    \]
    by the following composition
    \begin{align*}
        K_{X^{\mu}, s^{\mu}} |_{R^{\bG_m, \red}} \otimes \cL_{\mu, s}^{\Ind, \otimes 2} |_{R^{\bG_m, \red}} & \xrightarrow[\cong]{(\iota_{\fR^{\bG_m}},  \iota^{\Ind, \otimes 2}_{\fR})} \rdet_0(\bL_{\fR}) \otimes \rdet_+(\bL_{\fR} )^{\otimes 2} \\
        & \xrightarrow[\cong]{\id \otimes \id \otimes \eqref{eq:action_TL_det}} \rdet_0(\bL_{\fR}) \otimes \rdet_+(\bL_{\fR} ) \otimes \rdet_-(\bL_{\fR} ) \\
        & \cong \rdet(\bL_{\fR} |_{R^{\bG_m, \red}}) \xrightarrow[\cong]{\iota_{\fR}^{-1}} K_{X, s} |_{R^{\bG_m, \red}}.
    \end{align*}
    Note by construction and commutativity of the diagram  \eqref{eq:action_kappa_pm_comm}
    that the following diagram commutes:
    \[
    \xymatrix@C=100pt{
        {K_{X^{\mu}, s^{\mu}}|_{x} \otimes  \cL_{\mu, s}^{\Ind, \otimes 2}|_{x}} 
        \ar[r]_-{\cong}^-{(-1)^{\rank_+ \Omega_{X, x}} \cdot \xi_{\mu, s, \fR}|_{x}} 
        \ar[d]_-{\cong}^-{(\kappa_{x}, \kappa_{x}^{\Ind})}
    & {K_{X, s} |_{x}} 
    \ar[dd]_-{\kappa_{x}}^-{\cong}  \\
    {\det(\Omega_{X^{\mu}, x})^{\otimes 2} \otimes \det_{\pm}(\Omega_{X, x})^{\otimes 2}}
    \ar[d]_-{\cong}
    & {} \\
    {\det_{\leq 0}(\Omega_{X, x}) \otimes \det_{\geq 0}(\Omega_{X, x}) \otimes \det_{\pm}(\Omega_{X, x})} \ar[r]^-{\cong}
    & {\det(\Omega_{X, x})^{\otimes 2}.}
    }
    \]
    Since the brading isomorphism of the graded line bundle $\det_+(\Omega_{X, x})^{\otimes 2}$ is given by the multiplication of $(-1)^{\rank_+ \Omega_{X, x}}$, we conclude that the following diagram commutes:
    \begin{equation}\label{eq:xi_kappa_Ind_pre}
    \begin{aligned}
    \xymatrix@C=100pt{
    {K_{X^{\mu}, s^{\mu}}|_{x} \otimes  \cL_{\mu, s}^{\Ind, \otimes 2}|_{x}} \ar[r]_-{\cong}^-{\xi_{\mu, s, \fR}|_{x}} \ar[d]_-{\cong}^-{(\kappa_{x}, \kappa_{x}^{\Ind})}
    & {K_{X, s} |_{x}} \ar[d]_-{\kappa_{x}}^-{\cong} \\
    {\det_0(\Omega_{X, x})^{\otimes 2} \otimes \det_{\pm}(\Omega_{X, x})^{\otimes 2}} \ar[r]_-{\cong}
    & {\det(\Omega_{X, x})^{\otimes 2}}
    }
    \end{aligned}
    \end{equation}
    where the bottom map is the square of the canonical isomorphism
    $\det_0(\Omega_{X, x}) \otimes \det_{\pm}(\Omega_{X, x}) \cong \det(\Omega_{X, x})$.
    In particular, $\xi_{\mu, s, \fR}|_{x}$ does not depend on the choice of $\fR$ and hence $\xi_{\mu, s, \fR}$ glues to define $\xi_{\mu, s}$ as desired.

    We now prove (iii).
    Assume that we are given $\bG_m$-equivariant \'etale d-critical charts $\fR_j = (R_j, \eta_j, U_j, f_j, i_j)$ for $(X_j, s_j)$ for $j=1, 2$ and a morphism $(H_0, H) \colon \fR_1 \to \fR_2$ as in Proposition \ref{prop:T_sm_chart}.
    We define a complex on $R_1^{\bG_m}$ concentrated in degree $[-1, 0]$ by
    \[
    \bL_{\fR_1, H}^+ \coloneqq [ H^{*} T_{U_2} |_{R_1}^{+} \to  \Omega_{U_1} |_{R_1}^+]
    \]
    where the map is given by the composition 
    \[
    H^{*} T_{U_2} |_{R_1}^{+} \xrightarrow{\Hess_{f_2}}  H^{*} \Omega_{U_2} |_{R_1}^{+} \to  \Omega_{U_1} |_{R_1}^+.
    \]
Then we have the following fibre sequences
\begin{align}
    & H_0^{\bG_m, *} \bL_{\fR_2}^+ \to \bL_{\fR_1, H}^+ \to \bL_{H} |_{R_1}^+ \label{eq:fat_Hessian_fibre_1} \\
    & \bT_{H}[1] |_{R_1}^+ \to \bL_{\fR_1}^+ \to \bL_{\fR_1, H}^+. \label{eq:fat_Hessian_fibre_2}
\end{align}
    We construct an isomorphism
\[
\Upsilon_{H}^{\Ind} \colon h^{\bG_m, \red, *}  \fL_{\mu_2, s_2}^{\Ind} |_{R_1^{\bG_m, \red}} 
\otimes K_{X_1 / X_2}^{\pm}|_{R_1^{\bG_m, \red}} \cong \fL_{\mu_1, s_1}^{\Ind} |_{R_1^{\bG_m, \red}}
\]
so that the following diagram commutes:
\begin{equation}
    \begin{aligned}
        \xymatrix{
        {h^{\bG_m, \red, *}  \fL_{\mu_2, s_2}^{\Ind} |_{R_1^{\bG_m, \red}} \otimes K_{X_1 / X_2}^{\pm}|_{R_1^{\bG_m, \red}}}
        \ar[r]_-{\cong}^-{\Upsilon_{H}^{\Ind}}
        \ar[dd]_-{\cong}^-{\iota_{\fR_2}^{\Ind}}
        & {\cL_{\mu_1, s_1}^{\Ind} |_{R_1^{\bG_m, \red}}}
        \ar[d]_-{\cong}^-{\iota_{\fR_1}^{\Ind}} \\
        {}
        & {\rdet_+(\bL_{\fR_1}) }
        \ar[d]_-{\cong}^-{} \\
        {H_0^{\bG_m, \red, *} \rdet_+(\bL_{\fR_2}) \otimes K_{U_1 / U_2}^{\pm}|_{R_1^{\bG_m}}}
        \ar[r]_-{\cong}^-{\eqref{eq:det_shift}, \eqref{eq:det_dual}}
        & {H_0^{\bG_m, \red, *} \rdet_+(\bL_{\fR_2}) \otimes \rdet_+(\bL_{H}|_{R_1}) \otimes \rdet_{+}(\bT_{H} [1] |_{R_1}) }
        }
    \end{aligned}
\end{equation}
where the lower right vertical arrow is constructed using \eqref{eq:fibre_transform} for fibre sequences \eqref{eq:fat_Hessian_fibre_1} and \eqref{eq:fat_Hessian_fibre_2}.
Take a point $\hat{x}_1 \in R_1^{\bG_m}$ and set $H_0(\hat{x}_1) = \hat{x}_2$, $\eta_1(\hat{x}_1) = x_1$ and $\eta_2(\hat{x}_2) = x_2$.
We claim that the isomorphism $\Upsilon_{H}^{\Ind}$ is compatible with $\kappa^{\Ind}_{x_1}$ and $\kappa^{\Ind}_{x_2}$: i.e., the following diagram commutes:    
    \begin{equation}\label{eq:Upsilon_kappa_Ind_pre}
    \begin{aligned}
    \xymatrix@C=100pt{
        { \fL_{\mu_2, s_2}^{\Ind} |_{x_2} \otimes K_{X_1 / X_2}^{\pm}|_{x_1}} \ar[r]_-{\cong}^-{\Upsilon_{H}^{\Ind}|_{x_1}} \ar[d]_-{\cong}^-{\kappa_{x_2}^{\Ind} \otimes \id}
        & { \cL_{\mu_1, s_1}^{\Ind}|_{x_1} } \ar[d]_-{\cong}^-{\kappa_{x_1}^{\Ind}} \\
        {\det_{\pm}(\Omega_{X_2, x_2})  \otimes K_{X_1 / X_2}^{\pm}|_{x_1} } \ar[r]_-{\cong}
        & {\det_{\pm}(\Omega_{X_1, x_1}).}
     }
     \end{aligned}
    \end{equation}
    To see this, consider the following fibre double sequences:
    \[
    \xymatrix{
     {T_{X_2, x_2}^{+}[1]} \ar@{=}[r] \ar[d]
    & {T_{X_2, x_2}^{+}[1]} \ar[r] \ar[d]
    & {0} \ar[d] \\
        {\bL_{\fR_2} |_{\hat{x}_2}^+}\ar[r] \ar[d]
    & {\bL_{\fR_1, H} |_{\hat{x}_1}^+}\ar[r] \ar[d]
    & {\bL_{H_0} |_{\hat{x}_1}^+}  \ar[d]\\ 
        {\Omega_{X_2, x_2}^+} \ar[r] 
    & {\Omega_{X_1, x_1}^+} \ar[r]
    & {\bL_{h}|_{x_1}^+} 
    }
    \quad
    \xymatrix{
        {\bT_{h}[1]|_{x_1}^{+}} \ar[r] \ar[d]
    & {T_{X_1, x_1}^{+}[1]} \ar[r] \ar[d]
    & {T_{X_2, x_2}^{+}[1]}  \ar[d]
 \\
        {\bT_{H}[1]|_{\hat{x}_1}^+}\ar[r] \ar[d]
    & {\bL_{\fR_1}^{+}|_{\hat{x}_1}} \ar[r] \ar[d]
    & {\bL_{\fR_1, H}^{+}|_{\hat{x}_1}}  \ar[d]\\
        {0} \ar[r] 
    & {\Omega_{X_1, x_1}^{+}} \ar@{=}[r] 
    & {\Omega_{X_1, x_1}^{+}.} 
    }
    \]
    Using Lemma \ref{lem:KM}, we see that the following diagrams commute:
    \[
    \xymatrix@C=80pt{
    {\det_+(\bL_{\fR_1, H_1}|_{\hat{x}_1})} \ar[r]^-{\cong} \ar[d]^-{\cong}
    & {\det_+(\bL_{\fR_2}|_{\hat{x}_2}) \otimes \det_+(\bL_{H_0} |_{\hat{x}_1}) } \ar[d]^-{\cong} \\
    {\det_+(T_{X_2, x_2}[1]) \otimes \det_+(\Omega_{X_1, x_1}) } \ar[r]^-{\cong}
    & {\det_+(T_{X_2, x_2}[1]) \otimes \det_+(\Omega_{X_2, x_2}) \otimes \det_+(\bL_{h}|_{x_1}) ,}
    }
    \]
    \[
    \xymatrix@C=80pt{
    {\det_+( \bL_{\fR_1}|_{\hat{x}_1}) } \ar[r]^-{\cong} \ar[d]^-{\cong}
    & {\det_+(\bT_{H_0}[1]|_{\hat{x}_1}) \otimes \det_+(\bL_{\fR_1, H}|_{\hat{x}_1}) } \ar[d]^-{\cong} \\
    {  \det_+(T_{X_1, x_1}[1]) \otimes \det_+(\Omega_{X_1, x_1}) } \ar[r]^-{\cong}
    & {\rdet_+(\bT_{h}[1]|_{x_1}) \otimes \det_+(T_{X_2, x_2}[1]) \otimes \rdet_+(\Omega_{X_1, x_1}).}
    }
    \]
    By combining the commutativity of these diagrams and the commutativity of \eqref{eq:iota_kappa_Ind},
    we conclude the commutativity of the diagram \eqref{eq:Upsilon_kappa_Ind_pre}.
    In particular, the isomorphism $\Upsilon^{\Ind}_{H} |_{\hat{x}_1}$ only depends on the point $x_1 \in X_1$. Therefore, $\Upsilon^{\Ind}_{H}$ glues to define the isomorphism $\Upsilon^{\Ind}_{h}$.

    We now prove (iv).
    Take a $\bG_m^2$-equivariant \'etale d-critical chart $\fR = (R, \eta, U, f, i)$ .
    We will construct an isomorphism $\theta_{\hat{\mu}, s, \fR}$ over $R^{\hat{\mu}, \red}$ and show that it does not depend on the choice of $\fR$, which implies the statement (iv).
    For a subset $\Sigma \subset \bZ^2$,
    the natural equivalence $\bT_{\fR}^{\Sigma} [1] \simeq \bL_{\fR}^{\Sigma} $ with the isomorphisms \eqref{eq:det_shift} and \eqref{eq:det_dual} implies an isomorphism
    \begin{equation}\label{eq:det_sigma_minus}
    \rdet_{-\Sigma}(\bL_{\fR}) \simeq  \rdet_{\Sigma}(\bL_{\fR}).
    \end{equation}
    Note that we have natural equivalences
    \begin{align*}
        \cL_{\mu_1, s}^{\Ind} |_{R^{\hat{\mu}, \red}} 
        \cong \rdet_{\bZ_{>0} \times \bZ}(\bL_{\fR} )&,
        \quad
        \cL_{\mu_2, s}^{\Ind} |_{R^{\hat{\mu}, \red}} \cong \rdet_{\bZ\times \bZ_{ > 0}}(\bL_{\fR} ), \\
        (\cL^{\Ind}_{\mu_1|_{X^{\mu_2}}, s^{\mu_2}}) |_{R^{\hat{\mu}, \red}} \cong \rdet_{\bZ_{>0} \times \{ 0 \}}(\bL_{\fR} ) &, \quad
        (\cL^{\Ind}_{\mu_2|_{X^{\mu_1}}, s^{\mu_1}}) |_{R^{\hat{\mu}, \red}} \cong \rdet_{\{ 0 \} \times \bZ_{> 0}}(\bL_{\fR} ).
    \end{align*}
    Set $S_1 \coloneqq \{ (m, l) \mid (m, l) \neq (0, 0), m \geq 0, l \geq 0 \}$ and $S_2 \coloneqq \bZ_{> 0 } \times \bZ_{< 0}$.
    Then we have the following natural decompositions
    \begin{align*}
        (\bZ_{> 0 } \times \bZ) \coprod (\{ 0 \} \times \bZ_{ > 0}) &= S_1\coprod S_2 \\
        ( \bZ \times \bZ_{ > 0}) \coprod (  \bZ_{ > 0} \times \{ 0 \}) &= S_2 \coprod (- S_2).
    \end{align*}
    Then by using \eqref{eq:det_sigma_minus}, we obtain the following isomorphism
         \[
     \theta_{\hat{\mu}, s, \fR} \colon \cL^{\Ind}_{\mu_1, s} |_{R^{\hat{\mu}, \red}} \otimes (\cL^{\Ind}_{\mu_2|_{X^{\mu_1}}, s^{\mu_1}}) |_{R^{\hat{\mu}, \red}} 
        \cong 
        \cL^{\Ind}_{\mu_2, s} |_{R^{\hat{\mu}, \red}} \otimes (\cL^{\Ind}_{\mu_1|_{X^{\mu_2}}, s^{\mu_2}}) |_{R^{\hat{\mu}, \red}}.
     \]

     Take a point $\hat{x} \in R^{\hat{\mu}}_{}$ and set $x \coloneqq \eta(\hat{x})$.
     We claim that  $\theta_{\hat{\mu}, s, \fR}$ is compatible with $\kappa^{\Ind}_x$, namely, the following diagram commutes:
     \begin{equation}\label{eq:theta_kappa_Ind_pre}
    \begin{aligned}
        \xymatrix@C=45pt{
        {\cL^{\Ind}_{\mu_1, s}|_{\hat{x}}
        \otimes \left(\cL^{\Ind}_{\mu_2|_{X^{\mu_1}}, s^{\mu_1}}\right)|_{\hat{x}} 
        }  \ar[d]_-{\cong}^-{(\kappa_{x, \mu_1}^{\Ind}, \kappa_{x, \mu_2|_{X^{\mu_1}}}^{\Ind})} 
        \ar[rr]_-{\cong}^-{ (-1)^{\rank_{S_2} \Omega_{X, x}} \cdot \theta_{\hat{\mu}, s, \fR} |_{\hat{x}}}
        & {}
        & {\cL^{\Ind}_{\mu_2, s}|_{\hat{x}}
        \otimes \left(\cL^{\Ind}_{\mu_1|_{X^{\mu_2}}, s^{\mu_2}}\right)|_{\hat{x}}  } \ar[d]_-{\cong}^-{(\kappa_{x, \mu_2}^{\Ind}, \kappa_{x, \mu_1|_{X^{\mu_2}}}^{\Ind})}  \\
        {\det_{\mu_1\pm}(\Omega_{X, x}) \otimes \det_{\mu_2 \pm}(\Omega_{X^{\mu_1}, x})} \ar[rd]^-{\cong}
        & {}
        & {\det_{\mu_2\pm}(\Omega_{X, x}) \otimes \det_{\mu_1 \pm}(\Omega_{X^{\mu_2}, x})} \ar[ld]_-{\cong} \\
        {}
        & {\det_{\bZ^2 \setminus \{( 0, 0 )\}}(\Omega_{X, x}^{ }).}
        & {}
        }
        \end{aligned}
    \end{equation}
     Here we use the notation $\kappa_{x, \mu_1}^{\Ind}$ and $\kappa_{x, \mu_2}^{\Ind}$ to clarify which action we consider and similarly for $\det_{\mu_1 \pm}$ and $\det_{\mu_2 \pm}$.
    Firstly, note that we have the following natural equivalence of fibre sequences
    \begin{equation}
    \begin{aligned}
        \xymatrix{
         {(\Omega_{X, x}^{-S_2})^{\vee}[1]} \ar[r] \ar[d]_-{\simeq}
        & {(\bL_{\fR}  |_{\hat{x}}^{- S_2})^{\vee}[1]} \ar[r] \ar[d]_-{\simeq}
        & (T_{X, x}^{- S_2}[1])^{\vee}[1] \ar[d]_-{\simeq} \\
        {T_{X, x}^{ S_2}[1]} \ar[r]
        & {\bL_{\fR} |_{\hat{x}}^{S_2}} \ar[r]
        & {\Omega_{X, x}^{S_2}.}
        }
        \end{aligned}
    \end{equation}
    This combined with Corollary \ref{cor:det_shift_fibre}, \eqref{eq:double_dual} and \eqref{eq:dual_fibreseq} shows that the following diagram commutes up to the sign $(-1)^{\rank_{S_2} \Omega_{X, x}^{}}$:
    \[
    \xymatrix@C=100pt{
    {\det_{- S_2}(\bL_{\fR}  |_{\hat{x}})} \ar[r]_-{\cong}^-{\eqref{eq:det_sigma_minus}} 
    \ar[d]^-{\cong}
    & {\det_{S_2}(\bL_{\fR}  |_{\hat{x}})} 
    \ar[d]^-{\cong} \\
    {\det_{-S_2}(T_{X, x_2}[1]) \otimes \det_{-S_2}(\Omega_{X, x})} 
    \ar[d]^-{\cong}
    & {\det_{S_2}(T_{X, x_2}[1]) \otimes \det_{S_2}(\Omega_{X, x})} 
    \ar[d]^-{\cong} \\
    {\det_{\pm S_2}(\Omega_{X, x}^{})} \ar@{=}[r]
    & {\det_{\pm S_2}(\Omega_{X, x}).}
    }
    \]
    This implies the commutativity of the diagram \eqref{eq:theta_kappa_Ind_pre}

    We now prove (v).
    Let $\fR = (R, \eta, U, f, i)$ and $\fS = (S, \gamma, V, g, j)$ be $\bG_m$-equivariant
    \'etale d-critical charts of $X$ and $Y$, respectively.
    We construct an isomorphism
    \[
    \Psi_{\mu, \nu, \fR, \fS} \colon \cL_{\mu, s}^{\Ind} |_{R^{\bG_m, \red}} \boxtimes \cL_{\nu, t}^{\Ind} |_{S^{\bG_m, \red}} 
    \cong \cL_{\mu \times \nu, s \boxplus t}^{\Ind} |_{R^{\bG_m, \red} \times S^{\bG_m, \red}}
    \]
    by using the natural equivalence 
    $\bL_{\fR}|_{R^{\bG_m}}^+ \boxplus \bL_{\fS}|_{S^{\bG_m}}^+ \simeq \bL_{\fR \times \fS}|_{R^{\bG_m} \times S^{\bG_m}}^{+}$.
    Take a point $\hat{x} \in R$ and $\hat{y} \in S$ with $x = \eta(\hat{x})$ and $y = \gamma(\hat{y})$.
    Then it follows from the construction that the following diagram commutes:
    \begin{equation}\label{eq:Psi_kappa_Ind_pre}
        \begin{aligned}
        \xymatrix@C=100pt{
                    \cL_{\mu, s}^{\Ind} |_{\hat{x}} \otimes \cL_{\nu, t}^{\Ind} |_{\hat{y}} 
            \ar[r]_-{\cong}^-{\Psi_{\mu, \nu, \fR, \fS} |_{(\hat{x}, \hat{y})}}
            \ar[d]_-{\cong}^-{\kappa^{\Ind}_{x} \otimes \kappa_{y}^{\Ind}}
        & \cL_{\mu \times \nu, s \boxplus t}^{\Ind} |_{(\hat{x}, \hat{y})}
        \ar[d]_-{\cong}^-{\kappa_{(x, y)}^{\Ind}} \\
        {\det_{\pm}(\Omega_{X, x}) \otimes \det_{\pm}(\Omega_{Y, y})}
        \ar[r]^-{\cong}
        & \det_{\pm}(\Omega_{X \times Y, (x, y)}).
        }
        \end{aligned}
    \end{equation}
    In particular, $\Psi_{\mu, \nu, \fR, \fS}|_{(\hat{x}, \hat{y})}$ does not depend on the choice of \'etale d-critical charts hence glues to define an isomorphism 
    $\Psi_{\mu, \nu} \colon \cL_{\mu, s}^{\Ind} \boxtimes \cL_{\nu, t}^{\Ind} \cong \cL_{\mu \times \nu, s \boxplus t}^{\Ind}$.

    We now prove (vi). Take a point $x \in X^{\bG_m}$.
    By using Proposition \ref{prop:equivariant_Darboux}, we can find a tuple $\fR = (R, \eta, U, f, i)$ where $R$ is a derived scheme with a $\bG_m$-action, $\eta$ is a $\bG_m$-equivariant \'etale morphism, $U$ is a smooth scheme with a $\bG_m$-action whose image contains $x$, $f$ is a $\bG_m$-invariant function on $U$ with the property $f |_{\Crit(f)^{\red}} = 0$ and a $\bG_m$-equivariant locally closed embedding $i \colon R \hookrightarrow U$ with $i(R) = \Crit(f)$.
    Then $\fR^{\cl} = (R^{\cl}, \eta^{\cl}, U, f, i^{\cl})$ defines a $\bG_m$-equivariant \'etale d-critical chart for $(X, s)$ and we have a natural equivalence
    \[
    \bL_{R}^+ \simeq \bL_{\fR^{\cl}}^+.
    \]
   This equivalence induces an isomorphism
 \[
    \Lambda_{\mu, \fR}^{\Ind} \colon \rdet_+(\bL_{X} |_{R}) \cong \rdet_+(\bL_{R}) \cong \rdet_+(\bL_{\fR^{\cl}}) \xrightarrow[\cong]{\iota^{\Ind, -1}_{\fR^{\cl}}} \cL_{\mu^{\cl}, s}^{\Ind} |_{R^{\bG_m, \red}}.
 \]
 By construction, this isomorphism is compatible with $\kappa^{\Ind}_{x}$, i.e., the following diagram commutes for each lift $\hat{x} \in  R^{\bG_m}$ of $x$:
 \begin{equation}\label{eq:Lambda_kappa_Ind_pre}
     \begin{aligned}
         \xymatrix@C=100pt{
         {\det_+(\bL_{X, x} ) } \ar[r]_-{\cong}^-{ \Lambda_{\mu, \fR}^{\Ind}  }  \ar[d]_-{\cong}
         & {\cL_{\mu^{\cl}, s}^{\Ind} |_{{x}} } \ar[d]_-{\cong}^-{\kappa^{\Ind}_x} \\
         { \det_+(\tau^{\leq -1}(\bL_{X, x} )) \otimes \det_+(\tau^{\geq 0}(\bL_{X, x} ))} \ar[r]^-{\cong}
         & {\det_{\pm}(\Omega_{X, x})}
         }
     \end{aligned}
 \end{equation}
 where the bottom arrow is induced by the natural equivalence
 $\cdot \omega_{X} |_{x} \colon T_{X, x}^{+}[1] \simeq \tau^{\leq -1}(\bL_{X} |_{x}^+)  $, \eqref{eq:det_shift} and \eqref{eq:det_dual}.
 In particular, the morphism $\Lambda_{\mu, \fR}^{\Ind} |_{x} $ does not depend on the choice of $\fR$ and the lift $\hat{x}$ of $x$.
 Therefore, the isomorphism $\Lambda_{\mu, \fR}^{\Ind}$ glues to define an isomorphism $\Lambda_{\mu}^{\Ind}$.
    
\end{proof}

By construction, the maps $\xi_{\mu, s}$, $\Upsilon_h^{\Ind}$, $\theta_{\mu}$ and $\Lambda^{\Ind}_{\mu}$ are compatible with $\kappa_x^{\Ind}$, so that the following diagrams commute (we adopt the notation from the statement and the proof of Proposition \ref{prop:index_line_bundle}):
    \begin{equation}\label{eq:xi_kappa_Ind}
    \begin{aligned}
    \xymatrix@C=100pt{
    {K_{X^{\mu}, s^{\mu}}|_{x} \otimes  \cL_{\mu, s}^{\Ind, \otimes 2}|_{x}} \ar[r]_-{\cong}^-{\xi_{\mu, s}|_{x}} \ar[d]_-{\cong}^-{(\kappa_{x}, \kappa_{x}^{\Ind})}
    & {K_{X, s} |_{x}} \ar[d]_-{\kappa_{x}}^-{\cong} \\
    {\det(\Omega_{X^{\mu}, x})^{\otimes 2} \otimes \det_{\pm}(\Omega_{X, x})^{\otimes 2}} \ar[r]_-{\cong}
    & {\det(\Omega_{X, x})^{\otimes 2},}
    }
    \end{aligned}
    \end{equation}

        \begin{equation}\label{eq:Upsilon_kappa_Ind}
    \begin{aligned}
    \xymatrix@C=100pt{
        { \fL_{\mu_2, s_2}^{\Ind} |_{x_2} \otimes K_{X_1 / X_2}^{\pm}|_{x_1}} \ar[r]_-{\cong}^-{\Upsilon_{h}^{\Ind}|_{x_1}} \ar[d]_-{\cong}^-{\kappa_{x_2}^{\Ind} \otimes \id}
        & { \cL_{\mu_1, s_1}^{\Ind}|_{x_1} } \ar[d]_-{\cong}^-{\kappa_{x_1}^{\Ind}} \\
        {\det_{\pm}(\Omega_{X_2, x_2})  \otimes K_{X_1 / X_2}^{\pm}|_{x_1} } \ar[r]_-{\cong}
        & {\det_{\pm}(\Omega_{X_1, x_1}),}
     }
     \end{aligned}
    \end{equation}

     \begin{equation}\label{eq:theta_kappa_Ind}
    \begin{aligned}
        \xymatrix@C=50pt{
        {\cL^{\Ind}_{\mu_1, s}|_{{x}}
        \otimes \left(\cL^{\Ind}_{\mu_2|_{X^{\mu_1}}, s^{\mu_1}}\right)|_{{x}} 
        }  \ar[d]_-{\cong}^-{(\kappa_{x, \mu_1}^{\Ind}, \kappa_{x, \mu_2|_{X^{\mu_1}}}^{\Ind})} 
        \ar[rr]_-{\cong}^-{ (-1)^{\rank \Omega_{X, x}^{ S_2}} \cdot \theta_{{\mu}} |_{x}}
        & {}
        & {\cL^{\Ind}_{\mu_2, s}|_{{x}}
        \otimes \left(\cL^{\Ind}_{\mu_1|_{X^{\mu_2}}, s^{\mu_2}}\right)|_{{x}}  } \ar[d]_-{\cong}^-{(\kappa_{x, \mu_2}^{\Ind}, \kappa_{x, \mu_1|_{X^{\mu_2}}}^{\Ind})}  \\
        {\det_{\mu_1 \pm}(\Omega_{X, x}^{}) \otimes \det_{\mu_2 \pm}(\Omega_{X^{\mu_1}, x})} \ar[rd]^-{\cong}
        & {}
        & {\det_{\mu_2 \pm}(\Omega_{X, x}^{}) \otimes \det_{\mu_1\pm}(\Omega_{X^{\mu_2}, x})} \ar[ld]_-{\cong} \\
        {}
        & {\det_{\bZ^2 \setminus \{( 0, 0 )\} }(\Omega_{X, x}).}
        & {}
        }
        \end{aligned}
    \end{equation}

        \begin{equation}\label{eq:Psi_kappa_Ind}
        \begin{aligned}
        \xymatrix@C=100pt{
                    \cL_{\mu, s}^{\Ind} |_{{x}} \otimes \cL_{\nu, t}^{\Ind} |_{{y}} 
            \ar[r]_-{\cong}^-{\Psi_{\mu, \nu} |_{(x, y)}}
            \ar[d]_-{\cong}^-{\kappa^{\Ind}_{x} \otimes \kappa_{y}^{\Ind}}
        & \cL_{\mu \times \nu, s \boxplus t}^{\Ind} |_{(x, y)}
        \ar[d]_-{\cong}^-{\kappa_{(x, y)}^{\Ind}} \\
        {\det_{\pm}(\Omega_{X, x}) \otimes \det_{\pm}(\Omega_{Y, y})}
        \ar[r]^-{\cong}
        & \det_{\pm}(\Omega_{X \times Y, (x, y)}).
        }
        \end{aligned}
    \end{equation}
    
     \begin{equation}\label{eq:Lambda_kappa_Ind}
     \begin{aligned}
         \xymatrix@C=100pt{
         {\det_+(\bL_{X, x}) } \ar[r]_-{\cong}^-{ \Lambda_{\mu}^{\Ind}  }  \ar[d]_-{\cong}
         & {\cL_{\mu^{\cl}, s}^{\Ind} |_{{x}} } \ar[d]_-{\cong}^-{\kappa^{\Ind}_x} \\
         {\det_+(\tau^{\leq -1}(\bL_{X, x} )) \otimes \det_+(\tau^{\geq 0}(\bL_{X, x}))} \ar[r]^-{\cong}
         & {\det_{\pm}(\Omega_{X, x}).}
         }
     \end{aligned}
 \end{equation}
    The commutativity of the first diagram follows from the commutativity of \eqref{eq:xi_kappa_Ind_pre}, the second from \eqref{eq:Upsilon_kappa_Ind_pre}, the third from \eqref{eq:theta_kappa_Ind_pre}, the fourth from \eqref{eq:Psi_kappa_Ind} and the fifth from \eqref{eq:Lambda_kappa_Ind_pre}.
    Note that the diagram \eqref{eq:theta_kappa_Ind} commutes over the matching locus without the insertion of the sign in the upper horizontal map.
    We will prove that the isomorphisms constructed in Proposition \ref{prop:index_line_bundle} are compatible with each other (we only state and prove those compatibilities that we use later).
    Firstly, it is clear from the commutativity of the above diagrams that  
    $\xi_{\mu, s}$,  $\Upsilon_{h}^{\Ind}$, $\theta_{\hat{\mu}, s}$ and $\Lambda_{\mu}$ are closed under products under the identification $\Psi^{\Ind}_{\mu, \nu}$. 
    We now claim the compatibility of $\xi_{\mu, s}$ and $\Upsilon_{h}$.
    We adopt the notation from (iii) of Proposition \ref{prop:index_line_bundle}.
    Then the following diagram commutes:
    \begin{equation}\label{eq:xi_Upsilon_Ind}
    \begin{aligned}
        \xymatrix@C=80pt{
        {\substack{ \displaystyle{
        h^{\bG_m, \red, *} K_{X_2^{\mu_2}, s_2^{\mu_2}} 
 \otimes h^{\bG_m, \red, *} \fL_{\mu_2, s_2}^{\Ind, \otimes 2}} \ar[r]^-{\Upsilon_h^{\Ind}}_-{\cong} \ar[dd]_-{\cong}^-{\xi_{\mu_2, s_2}} \\
 \displaystyle{\otimes K_{X_1 / X_2}^{\otimes 2}|_{X_1^{\mu_1, \red}} }}}
        & {\substack{\displaystyle{h^{\bG_m, \red, *} K_{X_2^{\mu_2}, s_2^{\mu_2}} \otimes \fL_{\mu_1, s_1}^{\Ind, \otimes 2}} \ar[d]_-{\cong}^-{\Upsilon_{h^{\bG_m}}} \\
 \displaystyle{ \otimes K_{X_1^{\mu_1}/ X_{2}^{\mu_2}}^{\otimes 2}|_{X_1^{\mu_1, \red}}}}}  \\
        {}
        & {K_{X_1^{\mu_1}, s_1^{\mu_1}} \otimes  \fL_{\mu_1, s_1}^{\Ind, \otimes 2}} \ar[d]_-{\cong}^-{\xi_{\mu_1, s_1}} \\
        {{h^{\bG_m, \red, *} K_{X_2, s_2}|_{X_2^{\mu_2, \red}} } \otimes K_{X_1/ X_{2}}^{\otimes 2} |_{X_1^{\mu_1, \red}} } \ar[r]^-{\Upsilon_h}_-{\cong}
        & {K_{X_1, s_1}|_{X_1^{\mu_1, \red}}}.
        }
    \end{aligned}
\end{equation}
This is a direct consequence of the commutativity of the diagram \eqref{eq:xi_kappa_Ind} and \eqref{eq:Upsilon_kappa_Ind}.
Next, we claim the compatibility of $\xi_{\mu, s}$ and $\theta_{\hat{\mu}, s}$. We adopt the notation from Proposition \ref{prop:index_line_bundle} (iv):
\begin{equation}\label{eq:xi_theta_Ind}
    \begin{aligned}
        \xymatrix{
        {
        \substack{ \displaystyle
        {K_{X^{\hat{\mu}}, s^{\hat{\mu}}}}|_{X^{\hat{\mu}, \red}_{}}  
        \otimes \cL^{\Ind, \otimes 2}_{\mu_1, s} |_{X^{\hat{\mu}, \red}_{}}  \\
        \displaystyle \otimes \cL^{\Ind, \otimes 2}_{\mu_2|_{X^{\mu_1}}, s^{\mu_1}} }
        }
        \ar[rr]_-{\cong}^-{\id \otimes \theta_{\hat{\mu}, s}^{\otimes 2}} \ar[d]_-{\cong}^-{\xi_{\mu_2|_{X^{\mu_1}}, s^{\mu_1}}}
        &{}
        & 
        {
        \substack{ \displaystyle
        {K_{X^{\hat{\mu}}, s^{\hat{\mu}}} |_{X^{\hat{\mu}, \red}_{}}}   
        \otimes \cL^{\Ind, \otimes 2}_{\mu_2, s} |_{X^{\hat{\mu}, \red}_{}} \\
        \displaystyle \otimes \cL^{\Ind, \otimes 2}_{\mu_1|_{X^{\mu_2}}, s^{\mu_2}}
        } } \ar[d]_-{\cong}^-{\xi_{\mu_1 |_{X^{\mu_2}}, s^{\mu_2}}} \\
        {K_{X^{\mu_1}, s^{\mu_1} } |_{X^{\hat{\mu}, \red}_{}} \otimes \cL^{\Ind, \otimes 2}_{\mu_1, s} |_{X^{\hat{\mu}, \red}_{}}} \ar[rd]_-{\cong }^-{\xi_{\mu_1, s}}
        & {}
        &  {K_{X^{\mu_2}, s^{\mu_2} } |_{X^{\hat{\mu}, \red}_{}} \otimes \cL^{\Ind, \otimes 2}_{\mu_2, s} |_{X^{\hat{\mu}, \red}_{}}} \ar[ld]^-{\cong }_-{\xi_{\mu_2, s}}
        \\
        {}
        & {K_{X, s} |_{X^{\hat{\mu}, \red}}.}
        & {}
        }
    \end{aligned}
\end{equation}
This is a direct consequence of the commutativity of the diagram \eqref{eq:xi_kappa_Ind} and \eqref{eq:theta_kappa_Ind}.
We now prove the associativity of the isomorphism $\Upsilon^{\Ind}_{h}$.
Let $X_0$, $X_1$ and $X_2$ be algebraic spaces with $\bG_m$-action $\mu_0$, $\mu_1$ and $\mu_2$ respectively.
Assume that we are given $\bG_m$-equivariant smooth morphisms $g \colon X_0 \to X_1$ and $h \colon X_1 \to X_2$ and take a d-critical structure $s_2$ of $X_2$.
We set $s_1 \coloneqq h^{\star} s_2$ and $s_0 \coloneqq g^{\star} s_1$.
Then the following diagram commutes:
\begin{equation}\label{eq:Upsilon_Ind_assoc}
\begin{aligned}
\xymatrix{
{(g^{\bG_m, \red, *} h^{\bG_m, \red, *} \fL_{\mu_2, s_2}^{\Ind}) \otimes (g^{\bG_m, \red, *} K_{X_1 / X_2}^{\pm}) \otimes K_{X_0 / X_1}^{\pm}} \ar[r]^-{\Upsilon^{\Ind}_{h}}_-{\cong} \ar[d]^-{\cong}
& {g^{\bG_m, \red, *} \fL_{\mu_1, s_1}^{\Ind} \otimes  K_{X_0 / X_1}^{\pm} } \ar[d]_-{\cong}^-{\Upsilon_g^{\Ind}} \\
{(h^{\bG_m, \red} \circ g^{\bG_m, \red})^*  \ar[r]^-{\Upsilon_{h \circ g}^{\Ind}}_-{\cong} \cL_{\mu_2, s_2}^{\Ind} \otimes K^{\pm}_{X_0 / X_2} }
& {\cL_{\mu_0, s_0}^{\Ind}.}
}
\end{aligned}
\end{equation}
This is a direct consequence of the commutativity of the diagram \eqref{eq:Upsilon_kappa_Ind}.
Finally, we prove the compatibility between $\Upsilon^{\Ind}_h$ and $\theta_{\hat{\mu}, s}$.
Assume that we are given algebraic spaces $X$ and $Y$ with $\bG_m^2$-actions $\hat{\mu} = (\mu_1, \mu_2)$ and $\hat{\nu} = (\nu_1, \nu_2)$.
Assume further that we are given a $\bG_m^2$-equivariant smooth morphism $h \colon X \to Y$.
Take a d-critical structure $t$ of $Y$ and set $s \coloneqq h^{\star} t$.
Then the following diagram commutes:
\begin{equation}\label{eq:Upsilon_theta_Ind}
\begin{aligned}
\xymatrix@C=50pt{
{ \substack{\displaystyle h^{\bG_m^2, \red, *} \cL^{\Ind}_{\nu_1, t} |_{Y^{\hat{\nu}, \red}} \otimes h^{\bG_m^2, \red, *} \cL^{\Ind}_{\nu_2|_{Y^{\nu_1}}, t^{\nu_1}} \\
\displaystyle \otimes K_{X / Y}^{\mu_1 \pm} |_{X^{\hat{\mu}, \red}}  \otimes K_{X^{\mu_1} / Y^{\nu_1}}^{\mu_2 \pm} }} \ar[r]_-{\cong}^-{\theta_{\hat{\nu}, t}} \ar[d]_-{\cong}^-{(\Upsilon_{h, \mu_1}^{\Ind}, \Upsilon_{h |_{X^{\mu_1}}}^{\Ind})}
& {\substack{\displaystyle h^{\bG_m^2, \red, *} \cL^{\Ind}_{\nu_2, t} |_{Y^{\hat{\nu}, \red}} \otimes h^{\bG_m^2, \red, *} \cL^{\Ind}_{\nu_1|_{Y^{\nu_2}}, t^{\nu_1}} \\
\displaystyle \otimes K_{X / Y}^{\mu_2 \pm} |_{X^{\hat{\mu}, \red}}  \otimes K_{X^{\mu_2} / Y^{\nu_2}}^{\mu_1 \pm} }}
\ar[d]_-{\cong}^-{(\Upsilon_{h, \mu_2}^{\Ind}, \Upsilon_{h |_{X^{\mu_2}}}^{\Ind})}
\\
{\cL^{\Ind}_{\mu_1, s} |_{X^{\hat{\mu}, \red}} \otimes \cL^{\Ind}_{\mu_2|_{X^{\mu_1}}, s^{\mu_1}}} \ar[r]_-{\cong}^-{(-1)^{\rank_{S_2} (\Omega_{X / Y})} \cdot \theta_{\hat{\mu}, s}}
& {\cL^{\Ind}_{\mu_2, s} |_{X^{\hat{\mu}, \red}} \otimes \cL^{\Ind}_{\mu_1|_{X^{\mu_2}}, s^{\mu_2}}.}
}
\end{aligned}
\end{equation}
Here we use the notation $\Upsilon_{h, \mu_1}^{\Ind}$ and $\Upsilon_{h, \mu_2}^{\Ind}$ to clarify the $\bG_m$-action  we consider.
This statement is a direct consequence of the commutativity of the diagram \eqref{eq:Upsilon_kappa_Ind} and \eqref{eq:theta_kappa_Ind}.

\subsection{Index line bundles for d-critical stacks}\label{ssec:index_line_stack}

We will introduce the index line bundle for d-critical stacks by gluing index line bundles for 
$\bG_m$-equivariant d-critical algebraic spaces defined in \S \ref{ssec:index_line_space}:

\begin{prop}\label{prop:index_bundle_stack}
     Let $(\fX, s)$ be a d-critical stack where $\fX$ is quasi-separated with affine stabilizer groups.
     Then there exists a line bundle $\cL_{ s}^{\Ind}$ on $\Grad(\fX)^{\red}$ with the following properties:
     \begin{thmlist}
        \item For each point $x \in \Grad(\fX)$, there exists a natural isomorphism 
       \begin{equation}
           \kappa^{\Ind}_{x} \colon \cL_{ s}^{\Ind} |_{x} \cong \det_{\pm}( \tau^{\geq 0} (\bL_{\fX, u(x)})).
       \end{equation}

       \item There exists a natural isomorphism
       \begin{equation}
           \xi_{s} \colon K_{\Grad(\fX), u^{\star} s} \otimes \cL_{ s}^{\Ind, \otimes 2} \cong u^{\red, *}  K_{\fX, s} .
       \end{equation}

       \item For each $\bG_m$-equivariant chart $(X, \mu, \bar{q})$, there exists a natural isomorphism
        \begin{equation}
            \Upsilon_{q, \mu}^{\Ind} \colon q_{\mu}^{\red, *} \cL_{ s}^{\Ind} \otimes K_{X / \fX}^{\pm} \cong \cL_{\mu, q^{\star} s}^{\Ind}.
        \end{equation}

        \item 
        Let $u_1, u_2 \colon \Grad^2(\fX) \to \Grad(\fX)$ be the morphisms induced from the inclusions $i_1, i_2 \colon \bG_m \hookrightarrow \bG_m^2$ to the first and the second components respectively. Let $\sigma \colon \Grad^2(\fX) \cong \Grad^2(\fX)$ be the isomorphism induced from the swapping isomorphism of $\bG_m^2$.
        Then there exists a natural isomorphism
        \begin{equation}
            \theta_s  \colon u_1^{\red, *} \cL_{s}^{\Ind} \otimes \sigma^* \cL_{u^{\star} s}^{\Ind} \cong 
            u_2^{\red, *} \cL_{ s}^{\Ind} \otimes  \cL_{ u^{\star} s}^{\Ind}.
        \end{equation}
        Here we used the natural identification $\Grad^2(\fX) \cong \Grad(\Grad(\fX))$ given by $\Map(B\bG_{m, 1} \times B\bG_{m, 2}, \fX) \cong \Map(B\bG_{m, 1}, \Map(B\bG_{m, 2}, \fX))$.

     \item Let $(\fX, s)$ and $(\fY, t)$ be d-critical Artin stacks.
     Then there exists a natural isomorphism
     \[
     \Psi^{\Ind}_{\fX, \fY} \colon \cL_{s}^{\Ind} \boxtimes \cL_{t}^{\Ind} \cong \cL^{\Ind}_{s \boxplus t}.
     \]
     \end{thmlist}

\end{prop}

\begin{proof}
    The existence of the line bundle $\cL_{s}^{\Ind}$ with the property (iii) follows from Corollary \ref{cor:QCha_glue_stack} and the commutativity of the diagram \eqref{eq:Upsilon_Ind_assoc}.
    The existence of the isomorphisms $\kappa_x^{\Ind}$, $\xi_{s}$ and $\theta_{s}$ follows from Corollary \ref{cor:QCha_glue} and the commutativity of the diagrams \eqref{eq:Upsilon_kappa_Ind}, \eqref{eq:xi_Upsilon_Ind} and \eqref{eq:Upsilon_theta_Ind} respectively.
    The isomorphism $\Psi^{\Ind}_{\fX \times \fY}$ exists since the map $\Upsilon^{\Ind}_{h}$ in Proposition \ref{prop:index_line_bundle} is closed under products under the identification $\Psi_{\mu, \nu}^{\Ind}$.
    
\end{proof}

We now prove the compatibility between morphisms that we have constructed in Proposition \ref{prop:index_bundle_stack}.
Firstly, the isomorphisms $\xi_s$, $\Upsilon^{\Ind}_{q, \mu}$ are closed under products under the identification $\Psi^{\Ind}_{\fX \times \fY}$, since a similar statement is true for $\bG_m$-equivariant d-critical algebraic spaces.

We now prove that $\kappa^{\Ind}_{x}$ is compatible with $\xi_s$. We adopt the notation from Proposition \ref{prop:index_bundle_stack} (ii) and take $x \in \Grad(\fX)$. Then the following diagram commutes:

\begin{equation}\label{eq:xi_kappa_Ind_stack}
    \begin{aligned}
    \xymatrix@C=100pt{
    {K_{\Grad(\fX), u^{\star} s}|_{x} \otimes  \cL_{s}^{\Ind, \otimes 2}|_{x}} \ar[r]_-{\cong}^-{\xi_s|_{x}} \ar[d]_-{\cong}^-{(\kappa_{x}, \kappa_{x}^{\Ind})}
    & {K_{\fX, s} |_{u(x)}} \ar[d]_-{\kappa_{u(x)}} \\
    {\det( \tau^{\geq 0} (\bL_{\Grad(\fX), x}^{}))^{\otimes 2} \otimes \det_{\pm}( \tau^{\geq 0} (\bL_{\fX_{}, x}))^{\otimes 2}} \ar[r]_-{\cong}
    & {\det( \tau^{\geq 0} (\bL_{\fX, u(x)}))^{\otimes 2}.}
    }
    \end{aligned}
    \end{equation}
    This is a direct consequence of the commutativity of the diagram \eqref{eq:xi_kappa_Ind}. 
    Next, we prove that $\kappa^{\Ind}_{x}$ is compatible with $\Upsilon_{q, \mu}^{\Ind}$. We adopt the notation from Proposition \ref{prop:index_bundle_stack} (iii) and take $x \in \Grad(\fX)$ and its lift $\hat{x} \in X^{\mu}$. Then the following diagram commutes:

        \begin{equation}\label{eq:Upsilon_kappa_Ind_stack}
    \begin{aligned}
    \xymatrix@C=100pt{
        { \cL_{s}^{\Ind} |_{x} \otimes K_{X / \fX}^{\pm}|_{\hat{x}}} \ar[r]_-{\cong}^-{\Upsilon_{q, \mu}^{\Ind}|_{\hat{x}}} \ar[d]_-{\cong}^-{\kappa_{x}^{\Ind} \otimes \id}
        & { \cL_{\mu, q_{\mu}^{\star} s}^{\Ind}|_{\hat{x}} } \ar[d]_-{\cong}^-{\kappa_{\hat{x}}^{\Ind}} \\
        {\det_{\pm}(\tau^{\geq 0} (\bL_{\fX, u(x)}))  \otimes K_{X / \fX}^{\pm}|_{\hat{x}} } \ar[r]_-{\cong}
        & {\det_{\pm}(\tau^{\geq 0} (\bL_{X, \hat{x}})).}
     }
     \end{aligned}
    \end{equation}
    This is a direct consequence of the commutativity of the diagram \eqref{eq:Upsilon_kappa_Ind}.
    Next, we prove that $\kappa^{\Ind}_{x}$ is compatible with $\theta_s$. We adopt the notation from Proposition \ref{prop:index_bundle_stack} (iii) and take $x \in \Grad^2(\fX)$. Then the following diagram commutes:

         \begin{equation}\label{eq:theta_kappa_Ind_stack}
    \begin{aligned}
        \xymatrix@C=-50pt{
        \cL_{s}^{\Ind}|_{u_1(x)} \otimes \cL_{u^{\star} s}^{\Ind}|_{\sigma(x)}
        \ar[d]_-{\cong}^-{(\kappa_{u_1(x)}^{\Ind}, \kappa_{\sigma(x)}^{\Ind})} 
        \ar[rr]_-{\cong}^-{ (-1)^{\rank_{S_2} ( \tau^{\geq 0 } (\bL_{\fX, u(x)}) )}  \cdot \theta_s |_{x}}
        & {}
        & \cL_{s}^{\Ind}|_{u_2(x)} \otimes \cL_{u^{\star} s}^{\Ind}|_{x}
        \ar[d]_-{\cong}^-{(\kappa_{u_2(x)}^{\Ind}, \kappa_{x}^{\Ind})}  \\
        { \det_{\bZ_{\neq 0} \times \bZ}( \tau^{\geq 0} (\bL_{\fX, u(x)}))  \otimes 
        \det_{\{ 0 \} \times \bZ_{\neq 0}}( \tau^{\geq 0} (\bL_{\fX, u(x)}^{}))} \ar[rd]^-{\cong}
        & {}
        & {\det_{\bZ \times \bZ_{\neq 0}}( \tau^{\geq 0} (\bL_{\fX, u(x)}^{}))  \otimes \det_{\bZ_{\neq 0} \times \{ 0 \}}( \tau^{\geq 0} (\bL_{\fX, u(x)}^{ } ))} \ar[ld]_-{\cong} \\
        {}
        & {\rdet_{\bZ^2_{\neq (0, 0)} }( \tau^{\geq 0} (\bL_{\fX, u(x)}^{})).}
        & {}
        }
        \end{aligned}
    \end{equation}
    Here we set $S_2 \coloneqq \bZ_{> 0} \times \bZ_{ < 0}$.
    The commutativity of this diagram is a consequence of the commutativity of the diagram \eqref{eq:theta_kappa_Ind}.

     We now prove that the compatibility between $\xi_s$ and $\Upsilon_{q, \mu}^{\Ind}$.
    We adopt the notation from Proposition \ref{prop:index_bundle_stack}.
    Then the following diagram commutes:
        \begin{equation}\label{eq:xi_Upsilon_Ind_stack}
    \begin{aligned}
        \xymatrix@C=80pt{
        {\substack{ \displaystyle{
        q_{\mu}^{ \red, *} K_{\Grad(\fX), u^{\star}s} 
 \otimes q_{\mu}^{ \red, *} \cL_{ s}^{\Ind, \otimes 2}} \ar[r]^-{\Upsilon_{q, \mu}^{\Ind}}_-{\cong} \ar[dd]_-{\cong}^-{\xi_{s}} \\
 \displaystyle{\otimes K_{X / \fX}^{\otimes 2}|_{X^{\mu, \red} } }}}
        & {\substack{\displaystyle{q_{\mu}^{ \red, *} K_{\Grad(\fX), u^{\star} s} \otimes \cL_{{\mu}, q^{\star} s}^{\Ind, \otimes 2}} \\
 \displaystyle{ \otimes K_{X^{\mu}_{}/ \Grad(\fX)}^{\otimes 2}}}}
 \ar[d]_-{\cong}^-{\Upsilon_{q_{\mu}}} \\
        {}
        & {K_{X^{\mu}_{}, (q^{\star}s)^{{\mu}}} \otimes  \cL_{{\mu}, q^{\star}s}^{\Ind, \otimes 2}} \ar[d]_-{\cong}^-{\xi_{{\mu}, q^{\star} s}} \\
        {{q_{\mu}^{ \red, *} u^{\red, *} K_{\fX, s}} \otimes K_{X/ \fX}^{\otimes 2}|_{X^{\mu, \red}} } \ar[r]^-{\Upsilon_q}_-{\cong}
        & {K_{X, q^{\star} s}}|_{X^{\mu, \red}} .
        }
    \end{aligned}
\end{equation}
This is a direct consequence of the construction of $\xi_{s}$.
Next, we prove the compatibility between $\xi_{s}$ and $\theta_{s}$.
We adopt the notation from Proposition \ref{prop:index_bundle_stack}. The following diagram commutes:
\begin{equation}\label{eq:xi_theta_Ind_stack}
    \begin{aligned}
        \xymatrix{
        {
        \substack{ \displaystyle
        {K_{\Grad^2(\fX), u^{(2), \star} s}}    \\
        \displaystyle \otimes u_1^{\red, *} \cL^{\Ind, \otimes 2}_{ s} \otimes   \sigma^* \cL^{\Ind, \otimes 2}_{u^{\star} s}  }
        }
        \ar[rr]_-{\cong}^-{\id \otimes \theta_{s}^{\otimes 2}} \ar[d]_-{\cong}^-{\sigma^* \xi_{u^{\star} s}}
        &{}
        & 
        {
        \substack{ \displaystyle
        {K_{\Grad^2(\fX), u^{(2), \star} s}}    \\
        \displaystyle \otimes u_2^{\red, *} \cL^{\Ind, \otimes 2}_{ s} \otimes    \cL^{\Ind, \otimes 2}_{u^{\star} s}  }
        } \ar[d]_-{\cong}^-{ \xi_{u^{\star} s}} \\
        u_1^{\red, *} K_{\Grad(\fX), u^{ \star} s} \otimes u_1^{\red, *} \cL^{\Ind, \otimes 2}_{ s}
        \ar[rd]_-{\cong }^-{u_1^{\red, *} \xi_{s}}
        & {} 
        &  {u_2^{\red, *} K_{\Grad(\fX), u^{ \star} s} \otimes u_2^{\red, *} \cL^{\Ind, \otimes 2}_{ s}} \ar[ld]^-{\cong }_-{u_2^{\red, *} \xi_{s}}
        \\
        {}
        & {u^{(2), \red, *} K_{\fX, s}.}
        & {}
        }
    \end{aligned}
\end{equation}
Here we let $u^{(2)} \colon \Grad^2(\fX) \to \fX$ be the canonical evaluation map and we used the identification 
\[
\sigma^* K_{\Grad^2(\fX), u^{(2), \star} s} \cong K_{\Grad^2(\fX), u^{(2), \star} s}
\]
for the left vertical map.
This is a direct consequence of the commutativity of the diagram \eqref{eq:xi_theta_Ind}.

We now discuss the associativity of $\Upsilon_{q, \mu}^{\Ind}$.
    We adopt the notation from Proposition \ref{prop:index_bundle_stack}.
    Assume further that we are given a morphism $f \colon  (X_1, \mu_1,  \bar{q}_1) \to (X_2, \mu_2, \bar{q}_2)$ of $\bG_m$-equivariant charts.
    Then the following diagram commutes:
    \begin{equation}\label{eq:Upsilon_assoc_Ind_stack}
    \begin{aligned}
\xymatrix{
{(f^{\bG_m, \red, *}_{} q_{2, \mu_2}^{\red, *} \cL_{s}^{\Ind}) \otimes (f^{\bG_m, \red, *}_{} K_{X_{2}^{\mu_2} / \Grad(\fX)}^{\pm}) \otimes K_{X_1^{\mu_1} / X_2^{\mu_2}}^{\pm}} \ar[r]^-{\Upsilon^{\Ind}_{q_2, \mu_2}}_-{\cong} \ar[d]^-{\cong}
& {f^{\bG_m, \red, *}_{} \cL_{{\mu_2, q_2^{\star} s}}^{\Ind} \otimes  K_{X_1^{\mu_1} / X_2^{\mu_2}}^{\pm} } \ar[d]_-{\cong}^-{\Upsilon_{f}^{\Ind}} \\
{(q_{2, \mu_2}^{\red} \circ f_{}^{\bG_m, \red})^*  \ar[r]^-{\Upsilon_{q_1, \mu_1}^{\Ind}}_-{\cong} \cL_{s}^{\Ind} \otimes K^{\pm}_{X_1^{\mu_1} / \Grad(\fX)} }
& {\cL_{\mu_1, q_1^{\star} s}^{\Ind}.}
}
\end{aligned}
\end{equation}
This is a direct consequence of the construction of $\Upsilon_{q, \mu}^{\Ind}$.
We now prove the compatibility between $\Upsilon_{q, \mu}^{\Ind}$ and $\theta_{s}$.
We adopt the notation from Proposition \ref{prop:index_bundle_stack}.
We take a $\bG_m^2$-equivariant chart $(X, \hat{\mu}, \bar{q})$.
Then the following diagram commutes:
\begin{equation}\label{eq:Upsilon_theta_Ind_stack}
    \begin{aligned}
        \xymatrix@C=100pt{
        {
          \substack{
            \displaystyle{
            q_{\hat{\mu}}^{\red, *} u_1^{\red, *} \cL_{s}^{\Ind} \otimes q_{\hat{\mu}}^{\red, *} \sigma^* \cL_{u^{\star} s}^{\Ind}
            } \\
            \displaystyle{
            \otimes K_{X / \fX}^{\bZ^2_{\neq 0}} 
            }
        }
        } \ar[r]_-{\cong}^-{q_{\hat{\mu}}^{\red, *} \theta_{s} \otimes \id} 
        \ar[d]_-{\cong}^-{(\Upsilon_{q, \mu_1}^{\Ind}, \Upsilon_{q_{\mu_1}, \mu_2 |_{X^{\mu_1}} }^{\Ind})}
        & \substack{
            \displaystyle{
            q_{\hat{\mu}}^{\red, *} u_2^{\red, *} \cL_{s}^{\Ind} \otimes q_{\hat{\mu}}^{\red, *}  \cL_{u^{\star} s}^{\Ind}
            } \\
            \displaystyle{
            \otimes K_{X / \fX}^{\bZ^2_{\neq 0}} 
            }
        }
        \ar[d]_-{\cong}^-{(\Upsilon_{q, \mu_2}^{\Ind}, \Upsilon_{q_{\mu_2}, \mu_1 |_{X^{\mu_2}}}^{\Ind})}  \\
        {\cL^{\Ind}_{\mu_1, q^{\star}s } |_{X^{\hat{\mu}, \red}} \otimes \cL^{\Ind}_{\mu_2|_{X^{\mu_1}}, q_{\mu_1}^{\star} u^{\star} s}} \ar[r]_-{\cong}^-{(-1)^{\rank_{S_2} (\Omega_{X / \fX})} \cdot \theta_{\hat{\mu}, q^{\star} s}}
        & {\cL^{\Ind}_{\mu_2, q^{\star}s } |_{X^{\hat{\mu}, \red}} \otimes \cL^{\Ind}_{\mu_1|_{X^{\mu_2}}, q_{\mu_2}^{\star} u^{\star} s}.}
        }
    \end{aligned}
\end{equation}
This is a direct consequence of the construction of $\theta_{\tilde{s}}$.

We now describe the index line bundle for d-critical stacks underlying $(-1)$-shifted symplectic stacks.
Let $(\fX, \omega_{\fX})$ be a $(-1)$-shifted symplectic derived Artin stack which is quasi-separated with affine stabilizers.
Let $(\fX^{\cl}, s)$ denote the underlying d-critical stack.
For $x \in \Grad(\fX)$, we define an isomorphism
\begin{equation}\label{eq:kappa_Ind_stack_construction}
\kappa_x^{\mathrm{Ind}, \mathrm{der}} \colon \det_+(\bL_{\fX, u(x)}) \cong \det_{\pm}(\tau^{\geq 0} (\bL_{\fX, u(x)}))
\end{equation}
so that the following diagram commutes:
\begin{equation}
\begin{aligned}
    \xymatrix@C=80pt{
    {\det(\bL_{\fX, u(x)}^+) }
    \ar[r]_-{\cong}^-{\kappa_x^{\mathrm{Ind}, \mathrm{der}}}
    \ar[d]_-{\cong}
    & {\det_{\pm}(\tau^{\geq 0} (\bL_{\fX, u(x)}))} \\
    {\det_+(\tau^{\leq -1} (\bL_{\fX, u(x)})) \otimes \det_+(\tau^{\geq 0} (\bL_{\fX, u(x)}))}
    \ar[r]_-{\cong}^-{(- \cdot \omega_{\fX})^{-1}|_{x} \otimes \id}
    & {\det_+(\tau^{\leq  0} (\bT_{\fX, u(x)})[1]) \otimes \det_+(\tau^{\geq 0} (\bL_{\fX, u(x)})).}
    \ar[u]^-{\cong}_-{\eqref{eq:det_shift}, \eqref{eq:det_dual}}
    }
\end{aligned}
\end{equation}
We will construct an isomorphism of one-dimensional vector spaces
\[
 \Lambda_{\fX, x}^{\Ind, \mathrm{pre}} \colon \det_+(\bL_{\fX, u(x)}) \cong \cL_{s}^{\Ind}|_{x}
\]
by $\kappa_{x}^{\mathrm{Ind}, -1} \circ \kappa_{x}^{\mathrm{Ind}, \mathrm{der}}$.

\begin{prop}\label{prop:compare_index_line}

The map $x \mapsto \Lambda_{\fX, x}^{\Ind, \mathrm{pre}}$ is algebraic, hence it defines an isomorphism 
\[
\Lambda_{\fX}^{\mathrm{Ind}} \colon  \rdet_+(\bL_{\fX}) \cong \cL^{\Ind}_s.
\]

\end{prop}
 
\begin{proof}

Take morphisms $q \colon \widehat{V} \to \fX$, $\tau \colon \widehat{V} \to V$ as in Lemma \ref{lem:G_m_bbbbj_chart} below.
We let $q_{\mu} \colon \widehat{V}^{\mu} \to \Grad(\fX)$ and $\tau_{\mu} \colon \widehat{V}^{\mu} \to {V}^{\mu}$ the naturally induced morphisms.
It is enough to show that the map  $x \mapsto \Lambda_{\fX, x}^{\Ind, \mathrm{pre}}$ is algebraic after pulling back to $\widehat{V}$.
Consider the following morphism of fibre sequences
\begin{equation}\label{eq:bbbbj_fibreseq}
\begin{aligned}
\xymatrix{
{\bT_{\widehat{V} /\fX  } [1] } \ar[r] \ar@{-->}[d]^-{\simeq}_-{\eta}
& {\bT_{\widehat{V}}[1]} \ar[r] \ar[d]
& {q^* \bT_{\fX}[1]}  \ar[d] \\
{\bL_{\widehat{V} /V }[-1]} \ar[r]
& {\tau^* \bL_{V}} \ar[r]
& {\bL_{\widehat{V}}}
}
\end{aligned}
\end{equation}
where the middle and right vertical maps are induced from $(-1)$-shifted $2$-form $\tau^{\star} \omega_{V} \sim q^{\star} \omega_{\fX}$, the dotted arrow is induced by the homotopy commutativity of the right square, which is invertible by Remark \ref{rmk:auto_nondeg}.
The map $\eta^{\vee}[-1] \colon \bL_{\widehat{V} /\fX}  \simeq \bT_{\widehat{V} / V}[-2]$
induces a natural isomorphism
\begin{equation}\label{eq:detetadual_Ind}
\alpha_+ \colon \rdet_-(\bL_{\widehat{V} / \fX}) \cong \rdet_-(\bT_{\widehat{V} / V}[-2]) \cong \rdet_{+}(\bL_{\widehat{V} / V})^{\vee}
\end{equation}
where the latter isomorphism is induced from \eqref{eq:det_shift} and \eqref{eq:det_dual}.
We construct an isomorphism
\[
\Upsilon_{q}^{\Ind, \mathrm{der}} \colon q_{\mu}^{\red, *} \rdet_+(\bL_{\fX}) \otimes \rdet_+(\bL_{\widehat{V} / \fX}) \cong \rdet_+(\bL_{\widehat{V}}), \quad 
\Upsilon_{\tau}^{\Ind, \mathrm{der}} \colon  \rdet_+(\bL_{\widehat{V}}) \otimes \rdet_-(\bL_{\widehat{V} / \fX}) \cong \tau_{\mu}^{\red, *} \rdet_+(\bL_{V})
\]
using fibre sequences
 \[
 q^* \bL_{\fX} \to \bL_{\widehat{V}} \to \bL_{\widehat{V}/\fX}, \quad \tau^{*} \bL_{V} \to  \bL_{\widehat{V}} \to \bL_{\widehat{V}/V},
 \] 
 the isomorphism \eqref{eq:fibre_transform} and the isomorphism \eqref{eq:detetadual_Ind}.
 We construct an isomorphism
\[
\Upsilon_{(q, \tau)}^{\Ind, \mathrm{der}} \colon q_{\mu}^{\red, *} \rdet_+(\bL_{\fX}) \otimes \rdet_{\pm}(\bL_{\widehat{V} / \fX}) \cong \tau_{\mu}^{\red, *} \rdet_+(\bL_{V})
\]
by the composite of $\Upsilon_{\tau}^{ \Ind, \mathrm{der}}$ and $\Upsilon_{q}^{\Ind, \mathrm{der}} \otimes \id_{\rdet_{-}(\bL_{\widehat{V} / \fX})}$.
Finally, we construct an isomorphism 
\[
\Lambda_{\fX, (q, \tau)}^{\Ind, \mathrm{pre}} \colon q_{\mu}^{\red, *} \rdet_{+}(\bL_{\fX}) \cong q_{\mu}^{\red, *} \cL_{s}^{\Ind}
\]
so that the following diagram commutes:
\[
\xymatrix@C=80pt{
{q_{\mu}^{\red, *} \rdet_+(\bL_{\fX}) \otimes \rdet_{\pm}(\bL_{\widehat{V} / \fX})} \ar[r]_-{\cong}^-{\Lambda_{\fX, (q, \tau)}^{\Ind, \mathrm{pre}} \otimes \id} \ar[d]_-{\cong}^-{\Upsilon_{(q, \tau)}^{\Ind, \mathrm{der}}}
& {q_{\mu}^{\red, *} \cL_{s}^{\Ind} \otimes \rdet_{\pm}(\bL_{\widehat{V} / \fX})} \ar[d]_-{\cong}^-{\Upsilon_{q^{\cl}, \mu}^{\Ind}} \\
{\tau_{\mu} ^{\red, *}\rdet_+(\bL_{V})} \ar[r]_-{\cong}^-{\Lambda^{\Ind}_{V}}
& {\cL_{\mu, q^{\cl, \star} s}^{\Ind}.}
}
\]
Take a point $\hat{v} \in \widehat{V}^{\bG_m}$ and set $x = q_{\mu}(\hat{v})$ and $v = \tau_{\mu}(\hat{v})$.
We claim that  the following diagram commutes:
\begin{equation}\label{eq:Lambda_kappa_stack_pre_Ind}
\begin{aligned}
\xymatrix@C=80pt{
{\det_{\pm}(\tau^{\geq 0} (\bL_{\fX, u(x)}))} \ar@{=}[r] 
& {\det_{\pm}(\tau^{\geq 0} (\bL_{\fX, u(x)}))^{}} \\
{\det_{+}(\bL_{\fX, u(x)}) } \ar[u]_-{\cong}^-{\kappa_{x}^{\Ind, \mathrm{der}}} \ar[r]_-{\cong}^-{\Lambda_{\fX, (q, \tau)}^{\Ind, \mathrm{pre}}|_{\hat{v}}}
& {\cL^{\Ind}_s|_{x}} \ar[u]_-{\cong}^-{ \kappa_{x}^{\Ind}}
}
\end{aligned}
\end{equation}
which in particular implies that the map $x \mapsto \Lambda_{\fX, x}^{\Ind, \mathrm{pre}}$ is algebraic as desired.

To prove this, we first show that the map $\Upsilon_{(q, \tau)}^{\Ind, \mathrm{der}}$ is compatible with maps $\kappa_{x}^{\Ind, \mathrm{der}}$ and $\kappa_{v}^{\Ind, \mathrm{der}}$.
Namely, we first show that the following diagram commutes:
\begin{equation}\label{eq:kappa_Upsilon_derived_pre_Ind}
\begin{aligned}
\xymatrix@C=50pt{
{\det_{\pm}(\tau^{\geq 0} (\bL_{\fX, u(x)})) \otimes \det_{\pm}(\bL_{\widehat{V} / \fX, \hat{v}}) } \ar[d]_-{\cong} 
& {\det_+(\bL_{\fX, u(x)}) \otimes \det_{\pm}(\bL_{\widehat{V} / \fX, \hat{v}}) } \ar[d]_-{\cong}^-{\Upsilon_{(q, \tau)}^{\Ind, \mathrm{der}} |_{\hat{v}} } \ar[l]^-{\cong}_-{\kappa_{x}^{\Ind, \mathrm{der}} \otimes \id} \\
{\det_{\pm}(\tau^{\geq 0} (\bL_{V, v}))}
& {\det_+(\bL_{V, v}).} \ar[l]^-{\cong}_-{\kappa_{v}^{\Ind, \mathrm{der}}}
}
\end{aligned}
\end{equation}
Here the left vertical map is constructed using the fibre sequence $\tau^{\geq 0} (\bL_{\fX, u(x)}) \to \tau^{\geq 0} (\bL_{V, v}) \to \bL_{\widehat{V} / \fX, \hat{v}}$.
We define a map
\[
\kappa_{\hat{v}}^{\Ind, \mathrm{der}} \colon \det_+(\bL_{\widehat{V}, \hat{v}}) \cong \det_+(\tau^{\geq 0} (\bL_{V, v})) \otimes \det_-(\tau^{\geq 0} (\bL_{\fX, u(x)}))
\]
by the composition
\begin{align*}
\det_+(\bL_{\widehat{V}, \hat{v}}) \cong
\det_+(\tau^{\leq -1} (\bL_{\widehat{V}, \hat{v}})) \otimes \det_+(\tau^{\geq 0} (\bL_{\widehat{V}, \hat{v}})) &\cong
\det_+(\tau^{\leq -1} (\bL_{\fX, u(x)})) \otimes \det_+(\tau^{\geq 0} (\bL_{V, v})) \\
& \cong \det_-(\tau^{\geq 0} (\bL_{\fX, u(x)})) \otimes \det_+(\tau^{\geq 0} (\bL_{V, v}))
\end{align*}
where the last isomorphism is induced from the isomorphism $\tau^{\leq -1} (\bL_{\fX, u(x)}) \simeq \tau^{\geq 0} (\bL_{\fX, u(x)})^{\vee}[1]$ and the isomorphisms \eqref{eq:det_shift} and \eqref{eq:det_dual}.
The diagram \eqref{eq:kappa_Upsilon_derived_pre_Ind} can be decomposed into the following diagram:
\begin{equation}\label{eq:decomposed_kappa_Upsilon_der_pre_Ind}
\begin{aligned}
\xymatrix@C=50pt{
{\det_{\pm}(\tau^{\geq 0} (\bL_{\fX, u(x)})) \otimes \det_{\pm}(\bL_{\widehat{V} / \fX, \hat{v}}) } \ar[d]_-{\cong}
& {\det_{+}(\bL_{\fX, u(x)}) \otimes   \det_{\pm}(\bL_{\widehat{V} / \fX, \hat{v}}) } 
\ar[l]^-{\cong}_-{\kappa_{x}^{\Ind, \mathrm{der}} \otimes \id} 
\ar[d]_-{\cong}^-{\Upsilon_{q}^{\Ind, \mathrm{der}} \otimes \id} \\
{\det_-(\tau^{\geq 0} (\bL_{\fX, u(x)})) \otimes \det_+(\tau^{\geq 0} (\bL_{V, v}))  \otimes \det_{-}(\bL_{\widehat{V} / \fX, \hat{v}})} 
\ar[d]_-{\cong}^-{ \id \otimes \mathrm{sw}}
& {\det_+(\bL_{\widehat{V}, \hat{v}}) \otimes \det_{-}(\bL_{\widehat{V} / \fX, \hat{v}}) } 
\ar[l]^-{\cong}_-{\kappa_{\hat{v}}^{\Ind, \mathrm{der}} \otimes \id} 
\ar[dd]_-{\cong}^-{\Upsilon_{\tau}^{\Ind, \mathrm{der}}}  \\
{\det_-(\tau^{\geq 0} (\bL_{\fX, u(x)})) \otimes \det_{-}(\bL_{\widehat{V} / \fX, \hat{v}}) \otimes  \det_{+}(\tau^{\geq 0} (\bL_{V, v}))  } 
\ar[d]_-{\cong}
& {} \\
{\det_{\pm}(\tau^{\geq 0} (\bL_{V, v}))}
& {\det_+(\bL_{V, v})}  \ar[l]^-{\cong}_-{\kappa_{v}^{\Ind, \mathrm{der}} \otimes \id}.
}
\end{aligned}
\end{equation}
The commutativity of the upper square follows by Lemma \ref{lem:KM} applied to the positive part of the following fibre double sequence
\[
\xymatrix{
{\tau^{\leq -1} (\bL_{\fX, u(x)})} \ar@{=}[r]  \ar[d]
& {\tau^{\leq -1} (\bL_{\fX, u(x)})} \ar[r] \ar[d]
& {0} \ar[d] \\
{\bL_{\fX, u(x)}} \ar[d] \ar[r]
& {\bL_{\widehat{V}, \hat{v}}} \ar[r] \ar[d]
& {\bL_{\widehat{V} / \fX, \hat{v}}} \ar@{=}[d] \\
{\tau^{\geq 0} (\bL_{\fX, u(x)})} \ar[r]
& {\tau^{\geq 0} (\bL_{\widehat{V}, \hat{v}})} \ar[r]
& {\bL_{\widehat{V} / \fX, \hat{v}}.} \\
}
\]

Now we prove the commutativity of the lower square of \eqref{eq:decomposed_kappa_Upsilon_der_pre_Ind}.
Consider the following diagram:
\[
\xymatrix{
{\det_{+}(\bL_{\widehat{V}, \hat{v}}) \otimes \det_{-}(\bL_{\widehat{V} / \fX, \hat{v}})} 
\ar[d]_-{\cong}^-{ \id \otimes \alpha_{+}} 
\ar[r]^-{\cong} 
\ar@/_80pt/[ddd]^-{\Upsilon_{\tau}^{\Ind, \mathrm{der}}}
& {\det_{+}(\tau^{\leq -1}(\bL_{\widehat{V}, \hat{v}}) ) \otimes \det_{+}(\tau^{\geq 0}(\bL_{\widehat{V}, \hat{v}}) ) \otimes \det_{-}(\bL_{\widehat{V} / \fX, \hat{v}})} \ar[d]_-{\cong} \\
{\det_{+}(\bL_{\widehat{V}, \hat{v}}) \otimes \det_{+}(\bL_{\widehat{V} / V, v})^{\vee}} \ar[dd]^-{\cong}
& {\det_+(\bL_{\widehat{V} / V, \hat{v}}) \otimes \det_+(\tau^{\leq -1}(\bL_{{V}, {v}}) ) \otimes \det_+(\tau^{\geq 0}(\bL_{\widehat{V}, \hat{v}}) ) \otimes \det_{-}(\bL_{\widehat{V} / \fX, \hat{v}})} \ar[d]_-{\cong}^-{ \id \otimes \alpha_{+}} \\
{}
& {\det_+(\bL_{\widehat{V} / V, \hat{v}}) \otimes \det_+(\tau^{\leq -1}(\bL_{{V}, {v}}) ) \otimes \det_+(\tau^{\geq 0}(\bL_{\widehat{V}, \hat{v}}) ) \otimes \det_{+}(\bL_{\widehat{V} / V, \hat{v}})^{\vee}} \ar[d]_-{\cong} \\
{\det_+(\bL_{{V}, {v}})} \ar[r]^-{\cong}
& {\det_+(\tau^{\leq -1}(\bL_{{V}, {v}}) ) \otimes \det_+(\tau^{\geq 0}(\bL_{{V}, {v}}) ).} 
}
\]
Here the right top vertical arrow is induced from the fibre sequence
$\tau^{\leq -1}(\bL_{{V}, {v}}) \to \tau^{\leq -1}(\bL_{\widehat{V}, \hat{v}}) \to \bL_{\widehat{V}/V , \hat{v}}$.
The commutativity of this diagram follows from Lemma \ref{lem:KM} applied to the positive part of the following fibre double sequence
\[
\xymatrix{
{\tau^{\leq  -1}(\bL_{V, v})} \ar[r] \ar[d]
& {\tau^{\leq -1}(\bL_{\widehat{V}, \hat{v}})} \ar[r] \ar[d]
& {\bL_{\widehat{V} / V}} \ar@{=}[d] \\
{\bL_{V}} \ar[r] \ar[d]
& {\bL_{\widehat{V}}} \ar[r] \ar[d]
& {\bL_{\widehat{V}/V}} \ar[d] \\
{\tau^{\geq 0}(\bL_{V, v})} \ar[r]^-{\simeq}
& {\tau^{\geq 0}(\bL_{\widehat{V}, \hat{v}})} \ar[r]
& {0} 
}
\]
Therefore the commutativity of the lower square of the diagram \eqref{eq:decomposed_kappa_Upsilon_der_pre_Ind} is reduced to proving that the outer square of the following diagram commutes:
\begin{equation}\label{eq:kappa_Upsilon_der_pre_final_Ind}
\xymatrix{
{\substack{
    \displaystyle \det_-(\tau^{\geq 0}(\bL_{\fX, u(x)})) 
    \\ \displaystyle \otimes \det_{-}(\bL_{\widehat{V}/\fX, \hat{v}}) 
}
} \ar[ddd]_-{\cong} \ar[r]_-{\cong}
& {
\substack{
\displaystyle \det_+(\tau^{\geq 0}(\bL_{\fX, u(x)})^{\vee}[1])   \\
\displaystyle  \otimes \det_{+}(\bT_{\widehat{V}/\fX, \hat{v}} [1]) \ar[ddd]_-{\cong}
}} \ar[r]_-{\cong}
& {\det_+(\tau^{\leq -1}(\bL_{\widehat{V}, \hat{v}})) \otimes \det_{-}(\bL_{\widehat{V}/\fX, \hat{v}})} \ar[d]_-{\cong} \\
{}
& {}
& { \det_+(\bL_{\widehat{V} / V, \hat{v}}) \otimes \det_+(\tau^{\leq -1}(\bL_{{V}, {v}})) \otimes \det_{-}(\bL_{\widehat{V}/\fX, \hat{v}})} 
\ar[d]_-{\cong}^-{\id \otimes \alpha_{+}} \\
{}
& {}
& {\det_+(\bL_{\widehat{V} / V, \hat{v}}) \otimes \det_+(\tau^{\leq -1}(\bL_{{V}, {v}})) \otimes \det_{+}(\bL_{\widehat{V} / V, \hat{v}})^{\vee}} \ar[d]_-{\cong}  \\
{\det_{-}(\tau^{\geq 0}(\bL_{\widehat{V}, \hat{v}}))} \ar[r]_-{\cong}
& {\det_{+}(\tau^{\geq 0}(\bL_{\widehat{V}, \hat{v}})^{\vee}[1])} \ar[r]_-{\cong}
& {\det_{+}(\tau^{\leq -1}(\bL_{{V}, {v}})).}
}
\end{equation}
Here the top and bottom left horizontal maps are the isomorphisms \eqref{eq:det_shift} and \eqref{eq:det_dual} and the top (resp. bottom) right horizontal map is induced from the isomorphism $\tau^{\leq -1} (\bL_{\widehat{V}, \hat{v}}) \simeq \tau^{\geq 0} (\bL_{\fX, u(x)})^{\vee}[1]$ (resp. $\tau^{\leq -1} (\bL_{{V}, {v}}) \simeq \tau^{\geq 0} (\bL_{\widehat{V}, \hat{v}})^{\vee}[1]$).
It follows from Corollary \ref{cor:det_shift_fibre} and \eqref{eq:dual_fibreseq} that the left square commutes.
Now we will show that the right square commutes.
To see this, consider the following equivalence of fibre sequences obtained by truncating the map of fibre sequences \eqref{eq:bbbbj_fibreseq}:
\begin{equation}\label{eq:bbbbj_fibreseq_trunc}
\begin{aligned}
\xymatrix{
{\bT_{\widehat{V} /\fX, \hat{v}  } [1] } \ar[r] \ar[d]^-{\simeq}_-{\eta}
& {\tau^{\geq 0} (\bL_{\widehat{V}, \hat{v}})^{\vee}[1]} \ar[r] \ar[d]_-{\simeq}
& { \tau^{\geq 0}(\bL_{\fX, u(x)})^{\vee}[1]}  \ar[d]_-{\simeq} \\
{\bL_{\widehat{V} /V , \hat{v}}[-1]} \ar[r]
& { \tau^{\leq -1}(\bL_{V, v})} \ar[r]
& { \tau^{\leq -1} (\bL_{\widehat{V}, \hat{v}}).}
}
\end{aligned}
\end{equation}
This shows that the following diagram commutes:
\begin{equation}\label{eq:kappa_Upsilon_der_pre_Step1_Ind}
\begin{aligned}
\xymatrix{
{\substack{
\displaystyle \det_+(\tau^{\geq 0}(\bL_{\fX, u(x)})^{\vee}[1])   
\displaystyle  \otimes \det_+(\bT_{\widehat{V}/\fX, \hat{v}} [1]) 
}} \ar[d]_-{\cong} \ar[r]_-{\cong}
& { \det_+( \tau^{\leq -1} (\bL_{\widehat{V}, \hat{v}})) \otimes \det_+(\bL_{\widehat{V} /V , \hat{v}}[-1]) } \ar[d]_-{\cong} \\
{\det_+(\tau^{\geq 0}(\bL_{\widehat{V}, \hat{v}})^{\vee}[1])} \ar[r]_-{\cong}
& {\det_+(\tau^{\leq -1}(\bL_{{V}, {v}})).} 
}
\end{aligned}
\end{equation}
Next, consider the following diagram:
\begin{equation}\label{eq:kappa_Upsilon_der_pre_Step2_Ind}
\begin{aligned}
\xymatrix@C=60pt{
{}
& {}
& {} \\
{\det_-(\bL_{\widehat{V} / \fX, \hat{v}})} \ar@/^30pt/[rr]^-{\alpha_+} \ar[r]_-{\cong}^-{\det(\eta_+[-1]^{\vee} |_{\hat{v}})} \ar[d]_-{\cong}
& {\det_-(\bT_{\widehat{V} / V , \hat{v}}[2])} \ar[r]_-{\cong}
& {\det_+(\bL_{\widehat{V} / V , \hat{v}})^{\vee}} \ar@{=}[d] \\
{\det_+(\bT_{\widehat{V} / \fX, \hat{v}}[1])} \ar[r]_-{\cong}^-{\det(\eta_+ |_{\hat{v}})}
& {\det_+(\bL_{\widehat{V} / V, \hat{v}}[- 1])} \ar[r]_-{\cong}
& {\det_+(\bL_{\widehat{V} / V , \hat{v}})^{\vee}.}
}
\end{aligned}
\end{equation}
The commutativity of the diagram \eqref{eq:double_dual} and \eqref{eq:dual_shift_n} implies that this diagram commutes up to the sign $(-1)^{\rank \bL_{\widehat{V} / \fX, \hat{v}}^{-}}$.

By the commutativity of the diagram \eqref{eq:rotate_inverse} applied to the fibre sequence $\tau^{\leq -1}(\bL_{V, v}) \to \tau^{\leq -1}(\bL_{\widehat{V}, \hat{v}}) \to \bL_{\widehat{V}/V, \hat{v}}$, we see that the following diagram commutes up to the sign $(-1)^{\rank \bL_{\widehat{V}/V, \hat{v}}^+}$:
\begin{equation}\label{eq:kappa_Upsilon_der_pre_Step3_Ind}
\begin{aligned}
\xymatrix{
{\det_+(\bL_{\widehat{V}/V, \hat{v}}[-1]) \otimes \det_+(\tau^{\leq -1}(\bL_{\widehat{V}, \hat{v}})) \otimes \det_+(\bL_{\widehat{V}/V, \hat{v}})} \ar[r]_-{\cong} \ar[d]_-{\cong}
& { \det_+(\tau^{\leq -1}(\bL_{V, v})) \otimes \det_+(\bL_{\widehat{V}/V, \hat{v}})} \ar[d]_-{\cong} \\
{\det_+(\bL_{\widehat{V}/V, \hat{v}})^{\vee} \otimes \det_+(\tau^{\leq -1}(\bL_{\widehat{V}, \hat{v}})) \otimes \det_+(\bL_{\widehat{V}/V, \hat{v}})} \ar[r]_-{\cong}
& {\det_+(\tau^{\leq -1}(\bL_{V, v})).}
}
\end{aligned}
\end{equation}
By combining the commutativity properties of diagrams \eqref{eq:kappa_Upsilon_der_pre_Step1_Ind}, \eqref{eq:kappa_Upsilon_der_pre_Step2_Ind} and \eqref{eq:kappa_Upsilon_der_pre_Step3_Ind},
we see that the right square of the diagram \eqref{eq:kappa_Upsilon_der_pre_final_Ind} commutes.
This completes the proof of the commutativity of the diagram \eqref{eq:kappa_Upsilon_derived_pre_Ind}.

Now consider the following diagram:
\[
\xymatrix@C=40pt{
{
\substack{
    \displaystyle \det_{\pm}(\tau^{\geq 0}(\bL_{\fX, u(x)})) \\
     \displaystyle \otimes \det_{\pm}(\bL_{\widehat{V} / \fX, \hat{v}})
    }}
    \ar@{=}[rrr]
    \ar[ddd]_-{\cong}
& {}
& {}
& {
\substack{
    \displaystyle \det_{\pm}(\tau^{\geq 0}(\bL_{\fX, u(x)})) \\
     \displaystyle \otimes \det_{\pm}(\bL_{\widehat{V} / \fX, \hat{v}})
    }}
    \ar[ddd]_-{\cong} \\
{}
& {{
\substack{
    \displaystyle \det_+(\bL_{\fX, u(x)}) \\
     \displaystyle \otimes \det_{\pm}(\bL_{\widehat{V} / \fX, \hat{v}})
    }}}
    \ar[lu]^-{\cong}_-{\kappa_x^{\Ind, \mathrm{der}} \otimes \id} 
    \ar[r]_-{\cong}^-{\Lambda_{\fX, (q, \tau)}^{\Ind, \mathrm{pre}} |_{\hat{v}} \otimes \id}
    \ar[d]_-{\cong}^-{\Upsilon_{(q, \tau)}^{\Ind, \mathrm{der}}|_{\widehat{v}}}
& {\substack{
    \displaystyle \cL_{s}^{\Ind} |_{x} \\
     \displaystyle \otimes \det_{\pm}(\bL_{\widehat{V} / \fX, \hat{v}})
    }}
    \ar[ru]^-{\cong}_-{\kappa_x^{\Ind} \otimes \id } 
    \ar[d]_-{\cong}^-{\Upsilon_{q^{\cl}, \mu}^{\Ind}|_{\widehat{v}}}
& {} \\
{}
& {\det_+(\bL_{V, v})}
    \ar[ld]^-{\cong}_-{\kappa_v^{\Ind, \mathrm{der}}} 
    \ar[r]_-{\cong}^-{\Lambda_{V}^{\Ind}|_{v}}
& {\cL_{\mu, q^{\cl, \star} s}^{\Ind} |_{v}}
\ar[rd]^-{\cong}_-{\kappa_v^{\Ind} } 
& {} \\
\det_{\pm}(\tau^{\geq 0}(\bL_{V, v})) \ar@{=}[rrr]
& {}
& {}
& {\det_{\pm}(\tau^{\geq 0}(\bL_{V, v})).} 
}
\]
The middle square commutes by definition, the left square is nothing but the diagram \eqref{eq:kappa_Upsilon_derived_pre_Ind}, the bottom diagram is \eqref{eq:Lambda_kappa_Ind} and the right diagram is \eqref{eq:Upsilon_kappa_Ind_stack}.
In particular, we see that the top diagram, and hence the diagram \eqref{eq:Lambda_kappa_stack_pre_Ind}, commutes.
\end{proof}

\begin{lem}\label{lem:G_m_bbbbj_chart}
    Let $(\fX, \omega_{\fX})$ be a $(-1)$-shifted symplectic stack such that $\fX$ is quasi-separated with affine stabilizers
    and $x \in \Grad^n(\fX)$ be a point. 
    Then we can find the following data:
    \begin{itemize}
        \item A quasi-separated derived algebraic space $\widehat{V}$ with a $\bG_m^n$-action.
        \item A  smooth morphism $[\widehat{V} / \bG_m^n] \to \fX$ and a point $\hat{v} \in \widehat{V}^{\bG_m^n}$ such that the image of $\hat{x}$ under the natural map $\widehat{V}^{\bG_m^n} \to \Grad^n(\fX)$ is $x$. 
        \item A $\bG_m^n$-equivariant quasi-separated derived algebraic space $V$ with a
        morphism $i \colon \widehat{V} \to V$ which is an isomorphism on the classical truncations.
        \item A $\bG_m^n$-invariant $(-1)$-shifted symplectic structure $\omega_V$ on $V$ such that the correspondence $\fX \leftarrow \widehat{V} \rightarrow{} V$ carries a Lagrangian structure.
    \end{itemize}

\end{lem}

\begin{proof}
This is a combination of the Darboux theorem (Theorem \ref{thm:Darboux}) in \cite{bbbbj15} and the equivariant Darboux theorem (Proposition \ref{prop:equivariant_Darboux}) in \cite{Par24}.

We first claim that there exists a pointed smooth map
\[[R/\bG_m^n] \to \fX\]
for a derived affine scheme $R$ with an action of $\bG_m^n$, together with a fixed point $r\in R^{\bG_m^n}$ whose image under the natural map $R^{\bG_m^n} \to \Grad^n(\fX)$ is $x$. 
Using Proposition \ref{prop:QCha_atlas}, we can find a smooth morphism $[P / \bG_m^n]  \to \fX^{\cl}$ with the similar property.
By \cite[Theorem 1.13]{ahhlr} (and \cite[Proposition 6.1]{ahhlr}), the smooth morphism $[P/\bG_m^n] \to \fX^{\cl}$ lifts to a smooth morphism $[R/\bG_m^n] \to \fX$ for a derived affine scheme $R$ with a $\bG_m^n$-action.

We then find a $\bG_m^n$-invariant locally closed subscheme $\widehat{V} \subseteq R$ that contains $r$ such that the composition \[\widehat{V} \to R \to [R/\bG_m^n] \to \fX\]
is smooth of {\em minimal dimension} at $\hat{v}\in \widehat{V}$, that is, $\dim(H^i(\bT_{\widehat{V}, \hat{v}})) = \dim(H^i(\bT_{\fX,x})) $ for all $i \geq 0$.
This follows from (the third paragraph of) Step 1 of \cite[Lemma 4.1.8]{Par24}. 

By \cite[Lemma 4.1.8]{Par24}, there exists a $\bG_m^n$-equivariant closed embedding
\[i:\widehat{V} \hookrightarrow V,\]
into a quasi-smooth derived affine scheme $V$ with a $\bG_m^n$-action whose classical truncation $i^{\cl}: \widehat{V}^{\cl} \to V^{\cl}$ is an isomorphism.
Moreover, we may assume that $V$ is of minimal dimension in the sense that: 
$\dim(H^i(\bT_{\widehat{V}, \hat{v}})) = \dim(H^i(\bT_{V,\hat{v}})) $ for $i=0,1$.

By Proposition \ref{prop:closedexactdecomposition}, the symplectic form $\omega_{\fX}$ has a canonical exact structure.
By \cite[Lemma 4.1.9]{Par24}, the pullback of the exact symplectic form along $[\widehat{V}/\bG_m^n] \to \fX$ lifts to a $(-1)$-shifted exact $2$-form on $[V/\bG_m^n]$.
In particular, we have an (exact) isotropic correspondence
\[\xymatrix{
& [\widehat{V}/\bG_m^n]\ar[ld] \ar[rd]& \\ 
\fX && [V/\bG_m^n].
}\]
The desired $\bG_m^n$-invariant closed $2$-form $\omega_V$ on $V$ can be given by the image of the $(-1)$-shifted closed $2$-form on $[V/\bG_m^n]$ along $\cA^{2,\cl}([V/\bG_m^n],-1) \to  \cA^{2,\cl}(V,-1)$.

By \cite[Lemma 4.2.5]{Par24}, the closed $2$-form $\omega_V$ is symplectic and the isotropic correspondence
\[\xymatrix{
& \widehat{V}\ar[ld] \ar[rd]& \\ 
\fX && V.
}\]
is Lagrangian, after shrinking $\widehat{V}$.    
\end{proof}

We now construct $(-1)$-shifted symplectic analogue of isomorphisms in Proposition \ref{prop:index_bundle_stack}.
Let $(\fX, \omega_{\fX})$ be (possibly higher) $(-1)$-shifted symplectic stack.
Firstly, note that the natural equivalence
\begin{equation}\label{eq:Tminus=Lplus}
 \cdot \omega_{\fX} \colon \bT_{\fX}[1]  \simeq \bL_{\fX}
\end{equation}
combined with the natural isomorphisms \eqref{eq:det_dual} and \eqref{eq:det_shift} implies an isomorphism
\begin{equation}\label{eq:det_minus=plus}
\phi \colon \rdet_-(\bL_{\fX}) \cong \rdet_+(\bL_{\fX}).
\end{equation}
We define an isomorphism
\[
\xi^{\mathrm{der}}_{\fX} \colon \rdet(\bL_{\Grad(\fX)}) \otimes \rdet_+(\bL_{\fX})^{\otimes 2} \cong \rdet(\bL_{\fX}|_{\Grad(\fX)})
\]
by the following composition
\begin{align*}
\rdet(\bL_{\Grad(\fX)}) \otimes \rdet_+(\bL_{\fX})^{\otimes 2} 
\cong \rdet_0(\bL_{\fX}) \otimes \rdet_+(\bL_{\fX})^{\otimes 2} & \\
\xrightarrow[\cong]{\id \otimes \id \otimes \phi^{-1}} \rdet_0(\bL_{\fX}) \otimes \rdet_+(\bL_{\fX}) \otimes  \rdet_-(\bL_{\fX}) \cong \rdet(\bL_{\fX}|_{\Grad(\fX)}).& 
\end{align*}
Let $x \in \Grad(\fX)$ be a point. 
We now show that the map $\xi^{\der}_{\fX}$ is compatible with the map $\kappa^{\Ind, \der}_{x}$,
namely, the following diagram commutes:
\begin{equation}\label{eq:kappa_xi_der_stack}
\begin{aligned}
\xymatrix{
{\det(\bL_{\Grad(\fX), x}) \otimes \det_+(\bL_{\fX, u(x)})^{\otimes 2}} 
\ar[r]^-{\xi^{\der}_{\fX}}_-{\cong}
\ar[d]_-{\cong}^-{(\kappa_{x}^{\der}, \kappa_{x}^{\Ind, \der, \otimes 2})}
& {\det(\bL_{\fX, u(x)})} \ar[d]^-{\kappa_{x}^{\mathrm{der}}}_-{\cong} \\
{\det_0(\tau^{\geq 0}(\bL_{\fX, u(x)}))^{\otimes 2} \otimes  
\det_{\pm}(\tau^{\geq 0}(\bL_{\fX, u(x)}))^{\otimes 2}} \ar[r]^-{\cong}
& {\det(\tau^{\geq 0}(\bL_{\fX, u(x)}))^{\otimes 2}.}
}
\end{aligned}
\end{equation}
To see this,
we construct an isomorphism
\[
\kappa_{x}^{\Ind, \der, -} \colon \det_-(\bL_{\fX, u(x)}) \cong \det_{\pm}(\tau^{\geq 0}(\bL_{\fX, u(x)}))
\]
in the same manner as the construction of $\kappa_{x}^{\Ind, \der}$ in \eqref{eq:kappa_Ind_stack_construction}.
Then the following diagram commutes:
\begin{equation}\label{eq:kappa_plus_minus_difference}
\begin{aligned}
\xymatrix@C=100pt{
    {\det_-(\bL_{\fX, u(x)})} \ar[r]^-{(-1)^{\rank_{+} \tau^{\geq 0}(\bL_{\fX, u(x)})} \cdot \phi|_{x}}_-{\cong} 
    \ar[d]_{\cong}^-{\kappa_{x}^{\Ind, \der, -}}
    & {\det_+(\bL_{\fX, u(x)})}
    \ar[d]_-{\cong}^-{\kappa_{x}^{\Ind, \der}} \\
    {\det_{\pm}(\tau^{\geq 0}(\bL_{\fX, u(x)}))} \ar@{=}[r]
    & {\det_{\pm}(\tau^{\geq 0}(\bL_{\fX, u(x)})).}
    }
    \end{aligned}
\end{equation}
This follows from the commutativity of the diagrams \eqref{eq:double_dual} and \eqref{eq:dual_shift}.
We also see that the following diagram commutes by construction:
\begin{equation}\label{eq:kappa_xi_der_stack_before_swap}
\begin{aligned}
    \xymatrix{
{\det_0(\bL_{\fX, u(x)}) \otimes \rdet_+(\bL_{\fX, u(x)}) \otimes \rdet_-(\bL_{\fX, u(x)})} 
\ar[r]_-{\cong}
\ar[d]_-{\cong}^-{(\kappa_{x}^{\der}, \kappa_{x}^{\Ind, \der}, \kappa_{x}^{\Ind, \der, -})}
& {\det(\bL_{\fX, u(x)})} \ar[dd]^-{\kappa_{u(x)}}_-{\cong} \\
{\det_0(\tau^{\geq 0}(\bL_{\fX, u(x)}))^{\otimes 2} \otimes  
\det_{\pm}(\tau^{\geq 0}(\bL_{\fX, u(x)})) \otimes \det_{\pm}(\tau^{\geq 0}(\bL_{\fX, u(x)}))} \ar[d]^-{\cong}
& {} \\
{\det_{\leq 0}(\tau^{\geq 0}(\bL_{\fX, u(x)})) \otimes \det_{\geq 0}(\tau^{\geq 0}(\bL_{\fX, u(x)})) \otimes  \det_{\pm}(\tau^{\geq 0}(\bL_{\fX, u(x)}))} \ar[r]_-{\cong}
& {\det(\tau^{\geq 0}(\bL_{\fX, u(x)}))^{\otimes 2}.}
}
\end{aligned}
\end{equation}
By combining the commutativity of the diagram \eqref{eq:kappa_plus_minus_difference} and \eqref{eq:kappa_xi_der_stack_before_swap} and the fact that the brading isomorphism of the graded line bundle $\det_{+}(\tau^{\geq 0}(\bL_{\fX, u(x)}))^{\otimes 2}$ is given by multiplying $(-1)^{\rank_+ \tau^{\geq 0}(\bL_{\fX, u(x)})}$, we conclude that the diagram \eqref{eq:kappa_xi_der_stack} commutes.

Now we construct a $(-1)$-shifted symplectic analogue of $\theta_s$ in Proposition \ref{prop:index_bundle_stack}.
For a subset $\Sigma \subset \bZ^2$,
using the natural equivalence $\bT_{\fX}^{\Sigma}[1] \simeq \bL_{\fX}^{\Sigma}$ and \eqref{eq:det_shift} and $\eqref{eq:det_dual}$, 
we see that there exists a natural isomorphism
\[
\phi_{\Sigma} \colon \rdet_{- \Sigma}(\bL_{\fX}) \cong \rdet_{ \Sigma}(\bL_{\fX}).
\]
Further, for each point $x \in \Grad^2(\fX)$,
the commutativity of the diagrams \eqref{eq:double_dual} and \eqref{eq:dual_shift} implies that the following diagram commutes:
\begin{equation}\label{eq:-Sigma_det_stack}
\begin{aligned}
\xymatrix@C=100pt{
{\rdet_{- \Sigma}(\bL_{\fX})|_{x}} \ar[r]_-{\cong}^-{(-1)^{\rank_{\Sigma} \tau^{\geq 0}(\bL_{\fX, u(x)})} \cdot \phi_{\Sigma}|_{x}} 
\ar[d]_-{\cong}^-{\kappa^{ \Ind, \der, - \Sigma}_{x}}
& {\rdet_{\Sigma}(\bL_{\fX})|_{x}} 
\ar[d]_-{\cong}^-{\kappa^{\Ind, \der, \Sigma}_{x}}
\\
{\rdet_{\pm \Sigma}(\tau^{\geq 0}(\bL_{\fX, u(x)})} \ar@{=}[r]
& {\rdet_{\pm \Sigma}(\tau^{\geq 0}(\bL_{\fX, u(x)})}
}
\end{aligned}
\end{equation}
where the vertical maps are constructed in the same manner as $\kappa_{x}^{\Ind, \der}$.
 Set 
 \begin{align*}
 S_1 \coloneqq \{ (m, l) \mid (m, l) \neq (0, 0), m \geq 0, l \geq 0 \}, \quad
 S_2 \coloneqq \bZ_{> 0 } \times \bZ_{< 0}
 \end{align*}
 We construct an isomorphism
\[
\theta_{\fX}^{\der} \colon 
\rdet_{\bZ_{>0} \times \bZ} (\bL_{\fX}) \otimes \rdet_{\{ 0 \} \times \bZ_{> 0}} (\bL_{\fX})
\cong \rdet_{\bZ \times \bZ_{>0} } (\bL_{\fX}) \otimes \rdet_{\bZ_{>0} \times \{ 0 \}} (\bL_{\fX}).
\]
 using the decompositions
    \begin{align*}
        (\bZ_{> 0 } \times \bZ) \coprod (\{ 0 \} \times \bZ_{ > 0}) = S_1\coprod S_2, \quad
        ( \bZ \times \bZ_{ > 0}) \coprod (  \bZ_{ > 0} \times \{ 0 \}) = S_2 \coprod (- S_2)
    \end{align*}
    and the map $\phi_{S_2}^{-1}$.
       By the commutativity of the diagram \eqref{eq:-Sigma_det_stack}, we see that
       $\theta_{\fX}^{\der}$ is compatible with $\kappa^{\Ind, \der}_{x}$, namely, 
       the following diagram commutes:
    \begin{equation}\label{eq:theta_kappa_stack_der_Ind}
        \begin{aligned}
            \xymatrix@C=-30pt{
            {\rdet_{\bZ_{>0} \times \bZ} (\bL_{\fX, u(x)}) \otimes \rdet_{\{ 0 \} \times \bZ_{> 0}} (\bL_{\fX, u(x)})} 
            \ar[d]^-{\cong}_-{(\kappa^{ \Ind, \der, \bZ_{>0} \times \bZ}_{x}, \kappa^{ \Ind, \der, \{ 0 \} \times \bZ_{> 0}}_{x})}
            \ar[rr]_-{\cong}^-{(-1)^{\rank_{S_2} \tau^{\geq 0}(\bL_{\fX, u(x)})} \cdot \theta^{\der}_{\fX}}
            & {}
            & {\rdet_{\bZ \times \bZ_{>0} } (\bL_{\fX, u(x)}) \otimes \rdet_{\bZ_{>0} \times \{ 0 \}} (\bL_{\fX, u(x)})}
            \ar[d]^-{\cong}_-{(\kappa^{ \Ind, \der, \bZ_{} \times \bZ_{> 0}}_{x}, \kappa^{ \Ind, \der, \bZ_{> 0} \times  \{ 0 \} }_{x})}\\
             \rdet_{\bZ_{\neq 0} \times \bZ}(\tau^{\geq 0}(\bL_{\fX, u(x)}))
             \otimes \rdet_{\{ 0 \} \times \bZ_{\neq 0}}(\tau^{\geq 0}(\bL_{\fX, u(x)}))
            \ar[rd]_-{\cong}
            & {}
            & {\rdet_{\bZ \times \bZ_{\neq 0}}(\tau^{\geq 0}(\bL_{\fX, u(x)}))
             \otimes \rdet_{\bZ_{\neq 0} \times \{0 \}}(\tau^{\geq 0}(\bL_{\fX, u(x)}))} 
             \ar[ld]^-{\cong} \\
             {}
            & {\rdet_{\bZ^2_{\neq (0, 0)} }(\tau^{\geq 0}(\bL_{\fX, u(x)})).}
            & {} 
            }
        \end{aligned}
    \end{equation}

    We now construct a $(-1)$-shifted symplectic version of $\Psi_{\fX, \fY}$ in Proposition \ref{prop:index_bundle_stack}.
    Assume that we are given $(-1)$-shifted symplectic derived Artin stacks $(\fX, \omega_{\fX})$ and $(\fY, \omega_{\fY})$.
Then we construct isomorphisms 
\begin{align*}
\Psi^{\der}_{\fX, \fY} &\colon 
\rdet(\bL_{\fX}) \boxtimes \rdet(\bL_{\fY}  )
\cong \rdet(\bL_{\fX \times \fY}  )
\\
\Psi^{\Ind, \der}_{\fX, \fY} &\colon 
\rdet_+(\bL_{\fX}) \boxtimes \rdet_+(\bL_{\fY}  )
\cong \rdet_+(\bL_{\fX \times \fY}  )
\end{align*}
using the natural equivalences $\bL_{\fX} \boxplus \bL_{\fY} \simeq \bL_{\fX \times \fY}$ and 
$\bL_{\fX} |_{\Grad(\fX)} ^+ \boxplus \bL_{\fY} |_{\Grad(\fY)} ^+ \simeq \bL_{\fX \times \fY} |_{\Grad(\fX) \times \Grad(\fY)}^+$ respectively.

    We now prove the compatibility between the morphism $\Lambda_{\fX}^{\Ind}$ and the morphisms constructed in Proposition \ref{prop:index_bundle_stack}.
    Let $(\fX, \omega_{\fX})$ be a quasi-separated $(-1)$-shifted symplectic stack with affine stabilizers and $(\fX^{\cl}, s)$ be its underlying d-critical stack. By construction, $\Lambda_{\fX}^{\Ind}$ commutes with $\kappa_{x}^{\Ind}$: Namely, the following diagram commutes:
\begin{equation}\label{eq:Lambda_kappa_Ind_stack}
\begin{aligned}
\xymatrix@C=100pt{
{\det_+( \bL_{\fX, u(x)} }) \ar[r]_-{\cong}^-{\Lambda_{\fX}^{\Ind}|_{x}} \ar[d]_-{\cong}^-{\kappa_{x}^{\Ind, \der}}
& {\cL_{s}^{\Ind} |_{x}} \ar[d]_-{\cong}^-{\kappa_{x}^{\Ind}} \\
{\det_{\pm}(\tau^{\geq 0}(\bL_{\fX_{}, u(x)}))} \ar@{=}[r]
& {\det_{\pm}(\tau^{\geq 0}(\bL_{\fX_{}, u(x)}))}.
}
\end{aligned}
\end{equation}
    Next, we will show that $\xi_{\fX}$ is compatible with $\Lambda_{\fX}^{\Ind}$.
Namely, the following diagram commutes:
\begin{equation}\label{eq:Lambda_xi_Ind_stack}
\begin{aligned}
    \xymatrix@C=100pt{
    {\rdet(\bL_{\Grad(\fX)}) \otimes \rdet_+(\bL_{\fX})^{\otimes 2}} 
    \ar[d]^-{(\Lambda_{\Grad(\fX)}, \Lambda^{\Ind, \otimes 2}_{\fX})}_-{\cong}
    \ar[r]^-{\xi^{\der}_{\fX}}_-{\cong}
    & {\rdet(\bL_{\fX})}
    \ar[d]^-{\Lambda_{\fX}}_-{\cong} \\
    {K_{\Grad(\fX), u^{\star} s } \otimes \cL_{s}^{\Ind, \otimes 2}} 
    \ar[r]^-{\xi_s}_-{\cong}
    & {K_{\fX, s}|_{\Grad(\fX)}.}
    }
    \end{aligned}
\end{equation}
This is a direct consequence of the commutativity of diagrams \eqref{eq:Lambda_kappa_stack}, \eqref{eq:xi_kappa_Ind_stack}, \eqref{eq:kappa_xi_der_stack} and \eqref{eq:Lambda_kappa_Ind_stack}.
 We now show that $\theta_{\fX}^{\der}$ is compatible with $\Lambda^{\Ind}_{\fX}$.
     Namely,  the following diagram commutes: 
     \begin{equation}\label{eq:theta_Lambda_Ind_stack}
     \begin{aligned}
     \xymatrix@C=80pt{
\rdet_{\bZ_{>0} \times \bZ} (\bL_{\fX}) \otimes \rdet_{\{ 0 \} \times \bZ_{> 0}} (\bL_{\fX}) 
\ar[r]_-{\cong}^-{\theta_{\fX}^{\der}}
\ar[d]_-{\cong}^-{(\Lambda^{\Ind}_{\fX}, \Lambda^{\Ind}_{\Grad(\fX)})}
&  \rdet_{\bZ \times \bZ_{>0} } (\bL_{\fX}) \otimes \rdet_{\bZ_{>0} \times \{ 0 \}} (\bL_{\fX})
\ar[d]_-{\cong}^-{(\Lambda^{\Ind}_{\fX}, \Lambda^{\Ind}_{\Grad(\fX)})} \\
u_1^{\red, *} \cL_{s}^{\Ind} \otimes \sigma^* \cL_{u^{\star} s}^{\Ind} \ar[r]_-{\cong}^-{\theta_s}
&            u_2^{\red, *} \cL_{ s}^{\Ind} \otimes  \cL_{ u^{\star} s}^{\Ind}.
     }
     \end{aligned}
     \end{equation}
This is a direct consequence of the commutativity of diagrams \eqref{eq:theta_kappa_Ind_stack} and \eqref{eq:theta_kappa_stack_der_Ind}.
Finally, if we are given a quasi-separated $(-1)$-shifted symplectic $(\fY, \omega_{\fY})$ with affine stabilizers with classical d-critical stack $(\fY^{\cl}, t)$, it is clear that the following diagram commutes:
\begin{equation}\label{eq:Lambda_Psi_Ind_der_stack}
    \begin{aligned}
        \xymatrix@C=80pt{
        {\rdet_+(\bL_{\fX}) \boxtimes \rdet_+(\bL_{\fY})}
        \ar[r]_-{\cong}^-{\Psi_{\fX, \fY}^{\Ind}}
        \ar[d]^-{\cong}_-{\Lambda^{\Ind}_{\fX} \boxtimes \Lambda^{\Ind}_{\fY}}
        & {\rdet_+(\bL_{\fX \times \fY})}
        \ar[d]^-{\cong}_-{\Lambda^{\Ind}_{\fX \times \fY}} \\
        {\cL_{s}^{\Ind} \boxtimes \cL_{t}^{\Ind, \der}}
        \ar[r]_-{\cong}^-{\Psi_{\fX^{\cl}, \fY^{\cl}}^{\Ind}}
        & {\cL_{s \boxplus t}^{\Ind}.}
        }
    \end{aligned}
\end{equation}

    We now show that $\theta_{\fX}^{\der}$ is compatible with $\xi_{\fX}^{\der}$.
    Namely, the following diagram commutes:
    \begin{equation}\label{eq:xi_theta_der_stack}
    \begin{aligned}
    \xymatrix@C=-5pt{
    \substack{
\displaystyle    \rdet {(\bL_{\Grad^2(\fX)})}  \\
\displaystyle  \otimes \rdet_{\bZ_{>0} \times \bZ} (\bL_{\fX})^{\otimes 2}   \otimes \rdet_{\{ 0 \} \times \bZ_{> 0}} (\bL_{\fX})^{\otimes 2} 
    } 
    \ar[d]_-{\cong}^-{\xi_{\Grad(\fX)}^{\der}}
    \ar[rr]_-{\cong}^-{(\theta_{\fX}^{\der})^{\otimes 2}}
    & {}
    &{}
    \substack{
 \displaystyle   \rdet {(\bL_{\Grad^2(\fX)})}
 \ar[d]_-{\cong}^-{\xi_{\Grad(\fX)}^{\der}}\\
\displaystyle    \otimes \rdet_{\bZ \times \bZ_{>0} } (\bL_{\fX})^{\otimes 2} \otimes \rdet_{\bZ_{>0} \times \{ 0 \}} (\bL_{\fX})^{\otimes 2}
    }\\
    u_1^* \rdet{(\bL_{\Grad(\fX)})} \otimes \rdet_{\bZ_{>0} \times \bZ} (\bL_{\fX})^{\otimes 2} \ar[rd]_-{\cong}^-{u_1^*\xi_{\fX}^{\der}}
    & {}
    & {u_2^* \rdet{(\bL_{\Grad(\fX)})} \otimes \rdet_{\bZ \times \bZ_{> 0}} (\bL_{\fX})^{\otimes 2}} \ar[ld]^-{\cong}_-{u_2^*\xi_{\fX}^{\der}} \\
     {}
     & {\rdet(\bL_{\fX}).}
     & {}
    }
    \end{aligned}
    \end{equation}
    This is an immediate consequence of the commutativity of diagrams \eqref{eq:theta_kappa_stack_der_Ind} and \eqref{eq:kappa_xi_der_stack}.
Next, we note that the map $\xi_{\fX}^{\der}$ is compatible with $\Psi_{\fX, \fY}^{\Ind, \der}$:
\begin{equation}\label{eq:xi_Psi_der_commutes}
    \begin{aligned}
        \xymatrix@C=100pt{
        {(\rdet_0 (\bL_{\fX}) \otimes \rdet_+(\bL_{\fX})^{\otimes 2}) \boxtimes
        (\rdet_0 (\bL_{\fY}) \otimes \rdet_+(\bL_{\fY})^{\otimes 2})}
        \ar[r]_-{\cong}^-{\Psi^{\der}_{\Grad(\fX), \Grad(\fY)}, \Psi^{\Ind, \der}_{\fX, \fY}}
        \ar[d]^-{\cong}_-{\xi_{\fX}^{\der} \boxtimes \xi_{\fY}^{\der}}
        & {\rdet_0(\bL_{\fX \times \fY}) \otimes \rdet_+(\bL_{\fX \times \fY})^{\otimes 2} }
        \ar[d]^-{\cong}_{\xi_{\fX \times \fY}^{\der} }\\
        {\rdet(\bL_{\fX} |_{\Grad(\fX)}) \boxtimes \rdet(\bL_{\fY} |_{\Grad(\fY)})}
        \ar[r]_-{\cong}^-{\Psi_{\fX, \fY}^{\der}}
        & {\rdet(\bL_{\fX \times \fY} |_{\Grad(\fX \times \fY)}).}
        }
    \end{aligned}
\end{equation}
This is obvious from the construction.

Next, we will compare $\xi_{\fX}^{\der}$ with the map $\xi$ constructed in \eqref{eq:compositeLagrangiancorrespondencedeterminant}.
Let $(\fX, \omega_{\fX})$ be a higher $(-1)$-shifted symplectic stack.
Recall first that by Proposition \ref{prop:GradFiltcotangent} we have a natural equivalence
\[
\bT_{\Grad(\fX)} \simeq \bT_{\fX} ^{0},
\quad
\bT_{\Filt(\fX) / \Grad(\fX)}|_{\Grad(\fX)} \simeq \bT_{\fX}^{+}.
\]
The fibre sequence $\bL_{\Filt(\fX) / \fX}[-1] \to \bL_{\fX} \to \bL_{\Filt(\fX)}$ implies a natural equivalence
\[
\bL_{\Filt(\fX) / \fX}[-1] |_{\Grad(\fX)} \simeq \bL_{\fX} |_{\Grad(\fX)}^{+}.
\]
In particular, we have an isomorphism
\begin{equation}\label{eq:Filt_dual_plus}
\rdet(\bL_{\Filt(\fX) / \fX} |_{\Grad(\fX)})^{\vee} \cong \rdet_+(\bL_{\fX}).
\end{equation}

Consider the following map of fibre sequences
\[
\xymatrix{
{\bT_{\Filt(\fX) / \Grad(\fX)}}[1]|_{\Grad(\fX)} \ar[r] \ar@{-->}[d]^-{\simeq}
& {\bT_{\Filt(\fX)}[1]|_{\Grad(\fX)}} \ar[r] \ar[d]
& {\bT_{\Grad(\fX)}[1]|_{\Grad(\fX)}}  \ar[d] \\
{\bL_{\Filt(\fX) / \fX}[-1]} \ar[r]
& {\bL_{\fX}|_{\Grad(\fX)}} \ar[r]
& {\bL_{\Filt(\fX)}|_{\Grad(\fX)}} 
}
\]
where the homotopy commutativity of the right diagram is induced from the Lagrangian structure of the attractor correspondence.
Then the following diagram commutes by construction:
\begin{equation}\label{eq:Lag_corresp_vs_plus_minus}
\begin{aligned}
\xymatrix@C=60pt{
{\bT_{\Filt(\fX) / \Grad(\fX)}}[1]|_{\Grad(\fX)} \ar[r]^-{\simeq} 
\ar[d]^-{\simeq}_-{\eqref{eq:Lag_corresp_TL}}
& {\bT_{\fX}^{+}[1]} 
\ar[d]^-{\simeq}_-{\eqref{eq:Tminus=Lplus}} \\
{\bL_{\Filt(\fX) / \fX}[-1]|_{\Grad(\fX)}} \ar[r]^-{\simeq}
& {\bL_{\fX}^+.}
}
\end{aligned}
\end{equation}

We prove the following statement:
\begin{lem}\label{lem:two_localization_der_same}
    The following diagram commutes:
        \[
    \xymatrix@C=100pt{
    {\rdet(\bL_{\Grad(\fX)}) \otimes (\rdet(\bL_{\Filt(\fX) / \fX} |_{\Grad(\fX)})^{\vee})^{\otimes 2} } \ar[r]_-{\cong}^-{\eqref{eq:compositeLagrangiancorrespondencedeterminant}} \ar[d]^-{\cong}_-{\eqref{eq:Filt_dual_plus}}
    & {\rdet(\bL_{\fX})} \\
    {\rdet(\bL_{\Grad(\fX)}) \otimes \rdet_+(\bL_{\fX})^{\otimes 2}} \ar[r]_-{\cong}^-{\xi_{\fX}^{\der}}
    & {\rdet(\bL_{\fX}).} \ar@{=}[u]
    }
    \]
\end{lem}

\begin{proof}
    By the construction of $\xi_{\fX}^{\der}$, 
    the statement is equivalent to the commutativity of the following diagram:
        \[
    \xymatrix@C=80pt{
    {\rdet(\bL_{\Grad(\fX)}) \otimes (\rdet(\bL_{\Filt(\fX) / \fX} |_{\Grad(\fX)})^{\vee})^{\otimes 2} } \ar[r]_-{\cong}^-{\eqref{eq:compositeLagrangiancorrespondencedeterminant}} \ar[d]^-{\cong}_-{\eqref{eq:Filt_dual_plus}}
    & {\rdet(\bL_{\fX})} \\
    {\rdet(\bL_{\Grad(\fX)}) \otimes \rdet_+(\bL_{\fX})^{\otimes 2}} \ar[r]_-{\cong}^-{\id \otimes \id \otimes \phi^{-1}}
    & {\rdet(\bL_{\Grad(\fX)}) \otimes \rdet_+(\bL_{\fX}) \otimes \rdet_-(\bL_{\fX}).} \ar[u]^-{\cong}
    }
    \]
    Firstly, consider the following diagram:
    \begin{equation}\label{eq:xi_compare_plusminus_Step1}
    \begin{aligned}
    \xymatrix@C=50pt{
    {\rdet(\bL_{\Grad(\fX)} ) \otimes \rdet(\bL_{\Filt(\fX) / \fX}|_{\Grad(\fX)})^{\vee}}
    \ar[d]^-{\id \otimes \eqref{eq:Lag_corresp_TL_det}}_-{\cong}
    \ar[r]^-{\eqref{eq:Filt_dual_plus}}_-{\cong}
    & {\rdet_0(\bL_{\fX}) \otimes \rdet_+(\bL_{\fX}})
    \ar[d]_-{\cong}^-{\id \otimes \phi^{-1}}  \\
    {\rdet(\bL_{\Grad(\fX)} ) \otimes \rdet(\bL_{\Filt(\fX) / \Grad(\fX)} |_{\Grad(\fX)})} 
    \ar[d]_-{\cong} 
    & {\rdet_0(\bL_{\fX}) \otimes \rdet_-(\bL_{\fX}}) 
    \ar[d]^-{\cong} \\
    {\rdet(\bL_{\Filt(\fX)} |_{\Grad(\fX)})} 
    \ar[r]^-{\cong}
    & {\rdet_{\leq 0}(\bL_{\fX}). }
    }
    \end{aligned}
    \end{equation}
    This immediately follows from the commutativity of the diagram \eqref{eq:Lag_corresp_vs_plus_minus}.
    Next, we will show that the following diagram commutes up to sign $(-1)^{\rank \bL_{\Filt(\fX) / \fX}|_{\Grad(\fX)}}$:
    \begin{equation}\label{eq:xi_compare_plusminus_Step2}
    \begin{aligned}
    \xymatrix{
    {\rdet(\bL_{\Filt(\fX)} |_{\Grad(\fX)}) \otimes \rdet(\bL_{\Filt(\fX) / \fX} |_{\Grad(\fX)})^{\vee}}
    \ar[r]^-{\eqref{eq:Filt_dual_plus}}_-{\cong} \ar[d]^-{\cong}
    & {\rdet_{\leq 0}(\bL_{\fX}) \otimes \rdet_{+}(\bL_{\fX})} \ar[d]_-{\cong} \\
    {\rdet(\bL_{\fX} |_{\Grad(\fX)})} \ar@{=}[r]
    & {\rdet(\bL_{\fX} |_{\Grad(\fX)}).}
    }
    \end{aligned}
     \end{equation}
    This follows from the commutativity of \eqref{eq:rotate_det2} and the following equivalence of fibre sequences:
    \begin{equation*}
    \xymatrix{
    & {\bL_{\fX} |_{\Grad(\fX)}} \ar[r] \ar[d]^-{\simeq}
    & {\bL_{\Filt(\fX)} |_{\Grad(\fX)}} \ar[d]^-{\simeq} \ar[r]
    & {\bL_{\Filt(\fX) / \fX}|_{\Grad(\fX)}} \ar[d]^-{\simeq}\\
    & {\bL_{\fX} |_{\Grad(\fX)}} \ar[r]
    & {\bL_{\fX} |_{\Grad(\fX)}^{\leq 0}} \ar[r]
    & {\bL_{\fX}[1] |_{\Grad(\fX)}^+.}
    }
    \end{equation*}
    The statement of the lemma immediately follows from the commutativity of diagrams \eqref{eq:xi_compare_plusminus_Step1} and \eqref{eq:xi_compare_plusminus_Step2}
     and the fact that the brading isomorphism of $\rdet_+(\bL_{\fX})^{\otimes 2}$ is the multiplication by $(-1)^{\rank \bL_{\Filt(\fX) / \fX}|_{\Grad(\fX)}}$.
    
\end{proof}

\subsection{Localizing orientations}\label{ssec:Localizing_ori}

Here we discuss the localization of an orientation for $\bG_m$-equivariant algebraic spaces and d-critical stacks.

Assume that we are given a quasi-separated $\bG_m$-equivariant d-critical algebraic space $(X, \mu, s)$.
We define an isomorphism
\[
\bar{\xi}_{\mu, s}^{} \colon K_{X, s} |_{X^{\mu, \red}} \otimes (\cL_{\mu, s}^{\Ind, \vee})^{\otimes 2} \cong K_{X^{\mu}, s^{\mu}}
\]
by the following composition:
\begin{align*}
K_{X, s} |_{X^{\mu, \red}} \otimes (\cL_{\mu, s}^{\Ind, \vee})^{\otimes 2}
& \xrightarrow[\cong]{\xi_{\mu, s} \otimes \id} K_{X^{\mu}, s^{\mu}}  \otimes (\cL_{\mu, s}^{\Ind})^{\otimes 2} \otimes (\cL_{\mu, s}^{\Ind, \vee})^{\otimes 2} \\
&\cong K_{X^{\mu}, s^{\mu}}  \otimes  (\cL_{\mu, s}^{\Ind})^{\otimes 2} \otimes (\cL_{\mu, s}^{\Ind, \otimes 2})^{\vee} \cong K_{X^{\mu}, s^{\mu}}  .
\end{align*}
Note that the isomorphism $(\cL_{\mu, s}^{\Ind, \vee})^{\otimes 2} \cong (\cL_{\mu, s}^{\Ind, \otimes 2})^{\vee}$ differs from the natural isomorphism on underlying ungraded line bundles by the sign $(-1)^{I^{\mu}_{X}}$, which is the brading isomorphism of the graded line bundle $\cL_{\mu, s}^{\Ind, \otimes 2}$ of our symmetric monoidal structure on graded line bundles.
Assume now that we are given an orientation $(\cL, o)$ for $(X, s)$.
We define an orientation $o^{\mu}$ for $(X^{\mu}, s^{\mu})$ by
\begin{equation}\label{eq:localizedorientation}
o^{\mu} \colon (\cL |_{X^{\mu, \red}} \otimes \cL_{\mu, s}^{\Ind, \vee})^{\otimes 2} \xrightarrow[\cong]{o \otimes \id} K_{X, s} |_{X^{\mu, \red}} \otimes (\cL_{\mu, s}^{\Ind, \vee})^{\otimes 2} \xrightarrow[\cong]{\bar{\xi}_{\mu, s}} K_{X^{\mu}, s^{\mu}}.
\end{equation}

Let $(X_1, \mu_1, s_1)$ and $(X_2, \mu_2, s_2)$ be quasi-separated $\bG_m$-equivariant d-critical algebraic spaces.
Let $h \colon X_1 \to X_2$ be a smooth morphism such that $h^{\star} s_2 = s_1$ holds.
Assume that we are given an orientation $o_2$ for $(X_2, s_2)$. 
Then the commutativity of the diagram \eqref{eq:xi_Upsilon_Ind} implies that there exists a canonical isomorphism of orientations
\begin{equation}\label{eq:ori_loc_sm_compatible}
(h^{\star} o_2)^{\mu_1}  \cong h^{\bG_m, *} o_2^{\mu_2}.
\end{equation}

Assume that we are given a quasi-separated $\bG_m^2$-equivariant d-critical algebraic space $(X, (\mu_1, \mu_2), s)$ and an orientation $o$.
Then the commutativity of the diagram \eqref{eq:xi_theta_Ind} implies that there exists a natural isomorphism of orientations
\begin{equation}\label{eq:ori_localize_assoc_action}
    (o^{\mu_1})^{\mu_2}  \cong (o^{\mu_2})^{\mu_1} .
\end{equation}

Let $(X, \mu, s)$ and $(Y, \nu, t)$ be quasi-separated $\bG_m$-equivariant d-critical algebraic spaces with orientations $o_X$ and $o_Y$.
Then using that $\xi_{\mu, s} \boxtimes \xi_{\nu, t}$ and $\xi_{\mu \times \nu, s \boxplus t}$ are equivalent after the identification $\Psi_{X, Y}$ and $\Psi_{\mu, \nu}^{\Ind}$,
we obtain a natural isomorphism
\begin{equation}\label{eq:ori_localized_product_action}
    o_X^{\mu} \boxtimes o_Y^{\nu} \cong o_{X \times Y}^{\mu \times \nu}.
\end{equation}

Now we discuss localization of canonical orientation \eqref{eq:can_ori}.
Let $U$ be a smooth quasi-separated algebraic space with a $\bG_m$-action $\mu_U$ and $f$ be a $\bG_m$-invariant regular function on $U$ with $f|_{U^{\red}} = 0$.
Set $X = \Crit(f)^{\cl}$ and equip it with the $\bG_m$-action $\mu$ induced by $\mu_U$ and the natural d-critical structure $s$.
Then $(X, \mu, s)$ is a $\bG_m$-equivariant d-critical algebraic space.
Recall that we have defined the canonical orientation $o_{X}^{\can}$ in \eqref{eq:can_ori}.
Since we have
\[
X^{\mu} \cong \Crit(f |_{U^{\mu}})^{\cl},
\]
 d-critical structure $s^{\mu}$ coincides with the one induced by the critical locus description, 
hence one can define the canonical orientation $o_{X^{\mu}}^{\can}$ for $(X^{\mu}, s^{\mu})$.
Then the natural isomorphism $K_{U}  \otimes K_{U}^{\pm, \vee} \cong K_{U^{\mu}}$ induces an isomorphism of orientations
\begin{equation}\label{eq:can_ori_loc_action}
    (o_{X}^{\can})^{\mu} \cong o_{X^{\mu}}^{\can}.
\end{equation}
This can be proved in the same manner as the commutativity of the diagram \eqref{eq:xi_kappa_Ind_pre}: see also Remark \ref{rmk:sign_canori}.

Assume that we are given $\bG_m$-equivariant smooth morphism of smooth quasi-separated $\bG_m$-equivariant algebraic spaces $H \colon U_1 \to U_2$ and a $\bG_m$-invariant regular function $f_2 \colon U_2 \to \bA^1$ with $f_2|_{U_2^{\red}} =0$. Set $f_1 \coloneqq f_2 \circ H$ and $h \colon X_1 \to X_2$ be the natural morphism induced on the critical locus.
Let $\mu_1$ and $\mu_2$ be the $\bG_m$-actions on $X_1$ and $X_2$.
Then the following diagram commutes by the construction:
\begin{equation}\label{eq:can_ori_loc_pull_commutes}
    \begin{aligned}
        \xymatrix@C=100pt{
        {(h^{\star} o_{X_2}^{\can})^{\mu_1}}
        \ar[r]_-{\cong}^-{\eqref{eq:can_ori_pull}}
        \ar[d]_-{\cong}^-{\eqref{eq:ori_loc_sm_compatible}}
        & {(o_{X_1}^{\can})^{\mu_1}}
        \ar[dd]_-{\cong}^-{\eqref{eq:can_ori_loc_action}}
        \\
        {h^{\bG_m, \star}  (o_{X_2}^{\can})^{\mu_2}}
        \ar[d]_-{\cong}^-{\eqref{eq:can_ori_loc_action}}
        & {} \\
        {h^{\bG_m, \star} o_{X_2^{\mu_2}}^{\can}}
        \ar[r]_-{\cong}^-{\eqref{eq:can_ori_pull}}
        & {o_{X_1^{\mu_1}}^{\can}.}
        }
    \end{aligned}
\end{equation}

Now assume that $U$ is a smooth quasi-separated algebraic space with a $\bG_m^2$-action $(\mu_{U, 1}, \mu_{U, 2})$ and $f$ be a $\bG_m^2$-invariant regular function on $U$ with $f|_{U^{\red}} = 0$.
Let $(X, (\mu_1, \mu_2), s)$ be the associated $\bG_m^2$-equivariant d-critical algebraic space.
Then by repeating the proof of the commutativity of  \eqref{eq:theta_kappa_Ind_pre}, we see that the following diagram commutes:
    \begin{equation}\label{eq:can_ori_loc_assoc}
    \begin{aligned}
    \xymatrix{
    {((o^{\can}_{X})^{\mu_1})^{\mu_2 }} |_{X^{\hat{\mu}}_{\mat}} 
    \ar[r]_-{\cong}^-{\eqref{eq:can_ori_loc_action}}
    \ar[d]^-{\cong}_-{\eqref{eq:ori_localize_assoc_action}}
    & {(o^{\can}_{X^{\mu_1}})^{\mu_2} |_{X^{\hat{\mu}}_{\mat}}  } 
    \ar[r]_-{\cong}^-{\eqref{eq:can_ori_loc_action}}
    & {o_{X^{\hat{\mu}}}^{\can} |_{X^{\hat{\mu}}_{\mat}}  } \ar@{=}[d] \\
    {((o^{\can}_{X})^{\mu_2})^{\mu_1 }} |_{X^{\hat{\mu}}_{\mat}} 
    \ar[r]_-{\cong}^-{\eqref{eq:can_ori_loc_action}}
    & {(o^{\can}_{X^{\mu_2}})^{\mu_1}|_{X^{\hat{\mu}}_{\mat}}  }
    \ar[r]_-{\cong}^-{\eqref{eq:can_ori_loc_action}}
    & {o_{X^{\hat{\mu}}}^{\can} |_{X^{\hat{\mu}}_{\mat}} . }
    }
    \end{aligned}
\end{equation}

Now assume that we are given $\bG_m$-equivariant smooth quasi-separated algebraic spaces
$(U, \mu_U)$ and $(V, \nu_V)$ with $\bG_m$-invariant regular functions $f$ and $g$ such that $f|_{U^{\red}} = 0$ and $g|_{V^{\red}} = 0$.
Let $(X, \mu, s)$ and $(Y, \nu, t)$ be the associated $\bG_m$-equivariant d-critical algebraic spaces.
Then by repeating the proof of the commutativity of \eqref{eq:Psi_kappa_Ind_pre}, we see that the following diagram commutes:
\begin{equation}\label{eq:can_ori_loc_product}
    \begin{aligned}
        \xymatrix@C=20pt{
        (o^{\can}_X \boxtimes o^{\can}_Y)^{\mu \times \nu}
        \ar[d]_-{\cong}^-{\eqref{eq:can_ori_prod}}
        \ar[r]_-{\cong}^-{\eqref{eq:ori_localized_product_action}}
        & {(o^{\can}_X)^{\mu} \boxtimes (o^{\can}_Y)^{\nu}}
        \ar[r]_-{\cong}^-{\eqref{eq:can_ori_loc_action}}
        & {o^{\can}_{X^{\mu}} \boxtimes o^{\can}_{Y^{\nu}}}
        \ar[d]_-{\cong}^-{\eqref{eq:can_ori_prod}} \\
        {(o^{\can}_{X \times Y})^{\mu \times \nu}}
         \ar[rr]_-{\cong}^-{\eqref{eq:can_ori_loc_action}}
        &
        & {o^{\can}_{X^{\mu} \times Y^{\nu}}.}
        }
    \end{aligned}
\end{equation}

We now explain the localization of orientations for d-critical stacks.
Let $(\fX, s)$ be a d-critical stack and $o \colon \cL^{\otimes 2} \cong K_{\fX, s}$ be an orientation. 
We define the localized orientation $u^{\star} o$ on $(\Grad(\fX), u^{\star} s)$ by the following composition
\[
u^{\star} o \colon (u^{\red, *} \cL  \otimes \cL_{s}^{\Ind, \vee})^{\otimes  2} \xrightarrow[\cong]{o} K_{\fX, s} \otimes (\cL_{s}^{\Ind,  \vee})^{\otimes 2} \xrightarrow[\cong]{\xi_{s}} K_{\Grad(\fX), u^{\star} s}.
\]
Assume that we are given a $\bG_m$-equivariant chart $(X, \mu, \bar{q})$.
Then the commutativity of the diagram \eqref{eq:xi_Upsilon_Ind_stack} implies a canonical isomorphism of orientations:
\begin{equation}\label{eq:ori_sm_localize_stack}
(q^{\star} o)^{{\mu}} \cong q_{\mu}^{\star} u^{\star} o.
\end{equation}

The commutativity of the diagram \eqref{eq:xi_theta_Ind_stack} implies a canonical isomorphism of orientations:
\begin{equation}\label{eq:ori_localize_assoc}
    u_1^{\star} u^{\star} o \cong u_2^{\star} u^{\star} o.
\end{equation}
Here, $u_2^{\star} o$ is defined using the natural identification $\Grad^2(\fX) \cong \Grad(\Grad(\fX))$ and we define $u_{1}^{\star} \coloneqq \sigma^{\star} u_2^{\star}$.
Assume further that we are given a $\bG_m^2$-equivariant chart $(X, \hat{\mu} = (\mu_1, \mu_2), \bar{q})$. Set $\hat{\mu}' \coloneqq (\mu_2, \mu_1)$.
Then the commutativity of the diagram \eqref{eq:Upsilon_theta_Ind_stack} implies that the following diagram commutes:
\begin{equation}\label{eq:assoc_sm_ori_stack}
\begin{aligned}
\xymatrix@C=80pt{
    {((q^{\star}o)^{\mu_1})^{\mu_2}}
    \ar[r]_-{\cong}^-{\eqref{eq:ori_localize_assoc_action}}
    \ar[d]_-{\cong}^-{\eqref{eq:ori_sm_localize_stack}}
    & {((q^{\star}o)^{\mu_2})^{\mu_1}}
    \ar[d]_-{\cong}^-{\eqref{eq:ori_sm_localize_stack}}\\
    {q_{\hat{\mu}}^{\star}(u_1^{\star} u^{\star} o)}
    \ar[r]_-{\cong}^-{\eqref{eq:ori_localize_assoc}}
    &  {q_{\hat{\mu}'}^{\star}(u_2^{\star} u^{\star} o)}
    }
    \end{aligned}
\end{equation}

Assume now that we are given d-critical Artin stacks $(\fX, s, o_{\fX})$ and $(\fY, t, o_{\fY})$ such that $\fX$ and $\fY$ are quasi-separated with affine stabilizers. 
Then the isomorphism $\Psi_{\fX, \fY}$ and its compatibility between $\xi_{s}$, $\xi_{t}$ and $\xi_{s \boxplus t}$ implies an isomorphism of orientations
\begin{equation}\label{eq:ori_localize_product}
u_{\fX}^{\star} o_{\fX} \boxtimes u_{\fY}^{\star} o_{\fY} \cong (u_{\fX} \times u_{\fY})^{\star} o_{\fX} \boxtimes o_{\fY}.
\end{equation}
Assume further that  that we are given $\bG_m$-equivariant d-critical charts $(X, \mu, \bar{q})$ of $\fX$ and $(Y, \nu, \bar{r})$ of $\fY$.
Then the compatibility between $\Upsilon^{\Ind}_{q, \mu}$ and the product implies the commutativity of the following diagram:
\begin{equation}\label{eq:ori_pull_product_compat}
    \begin{aligned}
    \xymatrix{
    {(q^{\star} o_{\fX})^{\mu} \boxtimes (r^{\star} o_{\fY})^{\nu}}
    \ar[d]_-{\cong}^-{\eqref{eq:ori_sm_localize_stack}}
    \ar[r]_-{\cong}^-{\eqref{eq:ori_localized_product_action}}
    & {(q^{\star} o_{\fX} \boxtimes r^{\star} o_{\fY})^{\mu \times \nu}}
   \ar[d]_-{\cong}^-{\eqref{eq:ori_sm_localize_stack}} \\
        (q_{\mu} \times r_{\nu})^{\star} u_{\fX}^{\star} o_{\fX} \boxtimes u_{\fY}^{\star} o_{\fY} 
        \ar[r]_-{\cong}^-{\eqref{eq:ori_localize_product}}
        & (q_{\mu} \times r_{\nu})^{\star} (u_{\fX} \times u_{\fY})^{\star} o_{\fX} \boxtimes o_{\fY}.
        }
    \end{aligned}
\end{equation}

Assume that we are given an orientation $o \colon \cL^{\otimes 2} \cong \rdet(\bL_{\fX})$ for a $(-1)$-shifted symplectic derived Artin stack $\fX$ whose classical truncation is quasi-separated with affine stabilizers.
We define the localized orientation $u^{\star} o$ on $(\Grad(\fX), u^{\star} \omega_{\fX})$ by the following composition
\[
u^{\star} o \colon (u^{\red, *} \cL  \otimes \rdet_+(\bL_{\fX})^{\vee})^{\otimes  2} \xrightarrow[\cong]{o} u^* \rdet(\bL_{\fX}) \otimes (\det_+(\bL_{\fX})^{\vee})^{\otimes 2} \xrightarrow[\cong]{\xi_{\fX}^{\der}} \rdet(\bL_{\Grad(\fX)}).
\]
Using \eqref{eq:xi_theta_der_stack}, we see that there exists a natural isomorphism of orientations on $\Grad^2(\fX)$
\begin{equation}\label{eq:localize_ori_assoc_derived}
u_1^{\star} u^{\star} o  \cong u_2^{\star} u^{\star} o.
\end{equation}

\begin{rmk}\label{rmk:ori_compati_with_general}
   Using Lemma \ref{lem:two_localization_der_same}, we see that the orientation $u^{\star} o$ is equivalent to the orientation constructed using Lemma \ref{lem:Lagrangiancorrespondenceretract}.
   Also, it follows that the natural isomorphism \eqref{eq:localize_ori_assoc_derived} restricted to the matching locus is equivalent to the one defined using \eqref{eq:ori_assoc_general} twice.
\end{rmk}

Using Proposition \ref{prop:compare_index_line} and the commutativity of the diagram \eqref{eq:Lambda_xi_Ind_stack},
we see that there exists a natural isomorphism of orientations
\begin{equation}\label{eq:compare_localize_ori}
    (u^{\star} o)^{\cl} \cong u^{\star} o^{\cl}.
\end{equation}
By the commutativity of the diagram \eqref{eq:theta_Lambda_Ind_stack}, we see that the following diagram commutes:
\begin{equation}\label{eq:compare_assoc_locallize}
\begin{aligned}
\xymatrix{
{(u_1^{\star} u^{\star} o)^{\cl}} \ar[r]^-{\cong} \ar[d]^-{\cong}
& {(u_2^{\star} u^{\star} o)^{\cl}} \ar[d]^-{\cong} \\
{u_1^{\star} u^{\star} o^{\cl}} \ar[r]^-{\cong}
& {u_2^{\star} u^{\star} o^{\cl}.}
}
\end{aligned}
\end{equation}

Assume that we are given  oriented $(-1)$-shifted symplectic stacks
$(\fX, \omega_{\fX}, o_{\fX})$ and $(\fY, \omega_{\fY}, o_{\fY})$.
Then the commutativity of the diagram \eqref{eq:xi_Psi_der_commutes} implies a natural isomorphism of orientations
\begin{equation}\label{eq:ori_product_derived}
    u_{\fX}^{\star} o_{\fX} \boxtimes u_{\fY}^{\star} o_{\fY} 
    \cong u_{\fX \times \fY}^{\star} (o_{\fX} \boxtimes o_{\fY}).
\end{equation}
The commutativity of the diagram \eqref{eq:Lambda_Psi_Ind_der_stack} implies that the following diagram of orientations commute:
\begin{equation}\label{eq:ori_product_der_class_compati}
    \begin{aligned}
        \xymatrix{
        {(u_{\fX}^{\star} o_{\fX} \boxtimes u_{\fY}^{\star} o_{\fY})^{\cl}}
        \ar[r]^-{\cong}
        \ar[d]^-{\cong}
        & (u_{\fX \times \fY}^{\star} (o_{\fX} \boxtimes o_{\fY}))^{\cl}
        \ar[d]^-{\cong} \\
        {u_{\fX}^{\star} o_{\fX}^{\cl} \boxtimes u_{\fY}^{\star} o_{\fY}^{\cl}}
        \ar[r]^-{\cong}
        & {u_{\fX \times \fY}^{\star} (o^{\cl}_{\fX} \boxtimes o^{\cl}_{\fY}).}
        }
    \end{aligned}
\end{equation}

\section{Proof of the main theorem}

The aim of this section is to prove the integral isomorphism (= Theorem \ref{mainthm:hyperboliclocalization}).

In \S \ref{ssec:sheaf_theory}, we will recall some results in sheaf theory that will be used for the proof of Theorem \ref{mainthm:hyperboliclocalization}.
In \S \ref{ssec:hyp_loc}, we will recall the definition of basic properties of the hyperbolic localization functor.
In \S \ref{ssec:hyp_DT}, we will reprove Descombes's theorem \cite[Theorem 4.2]{des22} stating that hyperbolic localization of the DT perverse sheaf is again a DT perverse sheaf, with a careful choice of orientations.
In \S \ref{ssec:hyp_stack}, we will introduce the hyperbolic localization functor for stacks using the attractor correspondence.
In \S \ref{ssec:main}, we will prove Theorem \ref{mainthm:hyperboliclocalization}.
In \S \ref{ssec-nontriv-action}, we will extend = Theorem \ref{mainthm:hyperboliclocalization} to the stacks with non-trivial $\mathbb{G}_{\mathrm{m}}$-actions.
In \S \ref{ssec:int_crit}, we will describe the integral isomorphism for the critical locus inside a smooth stack.

\subsection{Some results from sheaf theory}\label{ssec:sheaf_theory}

We will discuss some results from sheaf theory that will be used in the proof of Theorem \ref{maintheorem}.

\subsubsection{Contraction lemma for monodromic complexes}

We first recall the definition of monodromic complexes:

\begin{defin}\label{defin:mon}
Let $\fY$ be an Artin stack with a $\bG_m$-action.
We say that a complex $\cF \in D(\fY)$ is \defterm{monodromic} if the restriction of $\cF$ to any $\bG_m$-orbit is locally constant.
We let $D_{\mathrm{mon}}(\fX) \subset D(\fX)$ denote the full subcategory consisting of monodromic complexes.
\end{defin}

Assume that we are given an Artin stack $\fX$ with an action of multiplicative monoid $\bA^1$.
We let $\fX^{0}$ be the fixed point substack of the submonoid $\{ 0, 1\}$ (see \cite[\S C.5]{dgcomp} for detail).
We have a contracting morphism $\pi \colon \fX \to \fX_{0}$ and its section $s \colon \fX_0 \to \fX$.
The following statement is called the contraction lemma:

\begin{prop}\label{prop:contraction}
    The following maps are invertible on monodromic objects:
    \begin{align*}
    \pi_* \xrightarrow{\mathrm{unit}} \pi_* s_* s^* \cong s^*, \\
    s^! \cong \pi_! s_! s^! \xrightarrow{\mathrm{counit}} \pi_!.
    \end{align*}
\end{prop}
The statement is proved in the context of D-module in \cite[Theorem C.5.3]{dgcomp}.
We can prove the statement by the same argument using the fact that the monodromic complex is nothing but $\tilde{\bG}_m$-equivariant complex where $\tilde{\bG}_m$ is the universal cover of $\bG_m$; see \cite[Proposition 3.7.2]{ksshv} for a related statement.
A more down-to-earth proof will be written in \cite{kk24}.

\begin{cor}\label{cor:gr_contr_lemma}
    Let $\fX$ be a quasi-separated stack with affine stabilizers.
    Let $\gr \colon \Filt(\fX) \to \Grad(\fX)$ and $\sigma \colon \Grad(\fX) \to \Filt(\fX)$ be the natural morphisms defined in \S \ref{ssec:grad_filt}.
    Then we have the following natural equivalence of functors
    \[
    \gr_* \cong \sigma^*, \quad \gr_! \cong \sigma^!.
    \]
\end{cor}

\begin{proof}
    As shown in \cite[Lemma 1.3.8]{hlp14}, the maps $\gr$ and $\sigma$ define an $\bA^1$-deformation retract from $\Filt(\fX)$ to $\Grad(\fX)$. Therefore we can apply Proposition \ref{prop:contraction}.
\end{proof}

\subsubsection{Virtual fundamental classes}

We briefly recall the theory of virtual fundamental classes following \cite{Kha19}.
It will be used only for the associativity of the integral isomorphism in Theorem \ref{maintheorem}, and in fact, we can avoid the use of the virtual fundamental classes for the proof. However, we will give a proof based on it since it gives a significant simplification. 

Let $f \colon \fX \to \fY$ be a quasi-smooth morphism of derived Artin stacks.
Then Khan \cite[Remark 3.8]{Kha19} constructed a natural transform of functors between constructible derived categories
\[
\mathrm{pur}_f \colon f^*[2 \vdim f]  \to f^! \
\]
called the \defterm{purity transform}, which is a relative version of the virtual fundamental class.
When $f$ is smooth, this isomorphism is identified with the natural isomorphism 
$f^*[2 \vdim f] \cong f^!$ induced by the trivialization $f^! \bQ_{\fY} \cong \bQ_{\fX}[2 \vdim f]$.
Assume that we are further given a quasi-smooth morphism $g \colon \fY \to \fZ$.
It follows from \cite[Proposition 2.5.4]{DFK21} that the following diagram commutes:
\begin{equation}\label{eq:purity_assoc}
\begin{aligned}
\xymatrix@C=70pt{
(g \circ f)^* [2 \vdim (g \circ f)] \ar[r]^-{\pur_{g \circ f}} \ar[d]^-{\pur_f}
& (g \circ f)^! \ar[d]^-{\cong} \\
f^! g^* [2 \vdim g]
\ar[r]^-{\pur_g}
& f^! g^!.
}
\end{aligned}
\end{equation}

\subsubsection{Generalized base change transforms}\label{sssec:gen_bc}

We will introduce a generalized version of the base change transform following \cite[\S 3]{fyz23}.
Consider the following diagram of derived Artin stacks:
\begin{equation}\label{eq:stack_square}
\begin{aligned}
\xymatrix{
{\fX_{11}}
\ar[r]^-{g'}
\ar[d]^-{f'}
& {\fX_{12}}
\ar[d]^-{f}\\
{\fX_{21}}
\ar[r]^-{g}
& {\fX_{22}.}
}
\end{aligned}
\end{equation}
We say that this diagram is \defterm{pullable} if the natural map 
$a \colon \fX_{11} \to \fX_{12} \times_{\fX_{22}} \fX_{21}$ is quasi-smooth.
We let $\tilde{f} \colon \fX_{12} \times_{\fX_{22}} \fX_{21} \to \fX_{21}$
 and $\tilde{g} \colon \fX_{12} \times_{\fX_{22}} \fX_{21} \to \fX_{12}$ the base change of $f$ and $g$ respectively.
We set $d \coloneqq \vdim a$ and define the generalized base change transform
\begin{equation}\label{eq:gen_bc}
\Exi \colon f^! g_* \to g'_* f'^! [- 2 d]
\end{equation}
by the following composition
\[
f^! g_* \cong \tilde{g}_* \tilde{f}^!  
\xrightarrow{\mathrm{unit}} \tilde{g}_* a_* a^* \tilde{f}^!
\xrightarrow{\mathrm{pur}_a} \tilde{g}_* a_* a^! \tilde{f}^! [- 2d]
\cong g'_* f'^! [- 2d].
\]
Similarly, we define the map
\begin{equation}\label{eq:gen_bc2}
\mathrm{Ex}_{{}^* {}^!} \colon g'^* f^!  \to f'^!g^* [- 2 d]
\end{equation}
by the composition
\[
g'^* f^! \cong a^* \tilde{g}^* f^! 
\to a^* \tilde{f}^! g^*
\xrightarrow{\mathrm{pur}_a} a^! \tilde{f}^! g^* [-2d]
\cong f'^!g^* [- 2 d].
\]

The following statement is an immediate consequence of \eqref{eq:purity_assoc}:

\begin{lem}\label{lem:Beck_Chevalley_bc}
    Consider the diagram \eqref{eq:stack_square} where $f$ and $f'$ are smooth.
    Set $d \coloneqq \dim f$ and $d' \coloneqq \dim f'$.
    Then it is a pullable square and the following diagram commutes:
    \[
    \xymatrix@C=70pt{
    {f^* g_*} 
    \ar[r]
    \ar[d]^-{\cong}
    & {g'_* f'^*}
    \ar[d]^-{\cong}\\
    {f^! g_* [ - 2 d]}
    \ar[r]^-{\Exi}
    & {g'_* f'^! [- 2 d']}
    }
    \]
    where the upper horizontal map is the Beck--Chevalley transform.
\end{lem}

The following statement is a direct consequence of \cite[Proposition 3.5.5]{fyz23}:

\begin{lem}\label{lem:Exi_associative}
    Assume that we are given the following composition of pullable squares:
    \[
    \xymatrix{
    {\fX_{11}}
    \ar[d]^-{f''}
    \ar[r]^-{g'}
    & {\fX_{12}}
    \ar[d]^-{f'}
    \ar[r]^-{h'}
    & {\fX_{13}}
    \ar[d]^-{f}\\
        {\fX_{21}}
        \ar[r]^-{g}
    & {\fX_{22}}
    \ar[r]^-{h}
    & {\fX_{23}.}
    }
    \]
    Let $a \colon \fX_{12} \to \fX_{13} \times_{\fX_{23}} \fX_{22}$ and $a' \colon \fX_{11} \to \fX_{12} \times_{\fX_{22}} \fX_{21}$ be the natural morphisms
    and set $d \coloneqq \vdim a$ and $d' \coloneqq \vdim a'$.
    Then the following diagram commutes:
    \[
    \xymatrix{
    {f^! (h \circ g)_*} 
    \ar[rr]^-{\Exi}
    \ar[d]^-{\cong}
    & {}
    & {(h' \circ g')_* f''^![- 2d - 2d']}
    \ar[d]^-{\cong}\\
        {f^! h_* g_*}
        \ar[r]^-{\Exi}
    & h'_* f'^! g_* [ - 2d]
    \ar[r]^-{\Exi}
    & {h'_* g'_* f''^! [- 2 d - 2 d']. } 
    }
    \]
\end{lem}

\subsection{Hyperbolic localization}\label{ssec:hyp_loc}

We will recall the definition of the hyperbolic localization functor of Braden \cite{bra03} and 
some of its properties, such as the commutativity between the vanishing cycle functor.

Let $X$ be a quasi-separated algebraic space with a $\bG_m^n$-action $\mu$.
Consider the following attractor correspondence
\[
\xymatrix{
{}
& {X^{\mu, (+, +, \ldots, +)}} \ar[rd]^-{\ev_{\mu}} \ar[ld]_-{\gr_{\mu}}
& {} \\
{X^{\mu}}
& {}
& {X.}
}
\]
We define a functor $\Loc_{\mu}^n \colon D^b_c(X) \to D^b_c(X^{\mu})$ by the following composition:
\[
\Loc_{\mu}^n = \gr_{\mu, *} \ev_{\mu}^!.
\]
When $n = 1$, we write $\Loc_{\mu} =  \Loc_{\mu}^1$ and call it the \defterm{hyperbolic localization functor}.

\begin{ex}\label{ex:hyp_const}
    Assume that $U$ is a smooth quasi-separated algebraic space with a $\bG_m$-action $\mu$.
    We let $\sigma_{\mu} \colon U^{\mu} \hookrightarrow U^{\mu, +}$ denote the natural section of $\gr_{\mu}$.
    Then Proposition \ref{prop:contraction} implies an isomorphism
    \begin{equation}\label{eq:hyp_const}
    \Loc_{\mu}(\bQ_{U}) \cong \sigma_{\mu}^* \bQ_{U^{\mu, +}}[-2d^-] \cong \bQ_{U^{\mu}}[-2d^-]
    \end{equation}
    where we set $d^- \coloneqq \rank T_{U} |_{U^{\mu}}^-$.
\end{ex}

Consider also the following repelling correspondence:
\[
\xymatrix{
{}
& {X^{\mu, -}} \ar[rd]^-{\ev_{\mu}^{-}} \ar[ld]_-{\gr_{\mu}^{-}}
& {} \\
{X^{\mu}}
& {}
& {X}
}
\]
The following result originally proved by Braden plays an important role in our paper:

\begin{thm}[{\cite{bra03}, \cite{dg14}}]\label{thm:Braden}

There exists an equivalence of functors
\[
\Loc_{\mu} = \gr_{\mu, *}\ev^! \cong \gr_{\mu, !}^- \ev_{\mu}^{-, *}
\]
on the subcategory of $D^b_c(X)$ consisting of $\bG_m$-monodromic objects.
\end{thm}

We now discuss the compatibility between the hyperbolic localization with the smooth pullback.
Assume that we are given smooth algebraic spaces $X$ and $Y$ with $\bG_m$-action $\mu$ and $\nu$ and a $\bG_m$-equivariant smooth morphism $h \colon X \to Y$.
Consider the following commutative diagram:
\begin{equation}\label{eq:attractor_corresp_morph}
\begin{aligned}
\xymatrix{
X^{\mu} \ar[dd]^-{h^{\bG_m}}
& {}
& X^{\mu, +} \ar[ll]_-{\gr_{\mu}} \ar[r]^-{\ev_{\mu}} \ar[dd]^-{h^{+}} \ar[ld]_-{\eta}
& X \ar[dd]^-{h} \\
{}
& {X^{\mu} \times_{Y^{\nu}} X^{\mu, +}} \ar[rd]^-{(h^{\bG_m})'} \ar[lu]_-{{\gr_{\nu}}'}
& {}
& {} \\
Y^{\nu} 
& {}
& Y^{\nu, +} \ar[ll]_-{\gr_{\nu}} \ar[r]^-{\ev_{\nu}}
& Y.
}    
\end{aligned}
\end{equation}

It follows from \cite[Proposition 1.3.1 (4)]{hlp14} that $h^{\bG_m}$ and $h^+$ are smooth.
We claim the natural map $h^{\bG_m, *} \gr_{\nu, *} \to  \gr_{\mu, *} h^{+, *}$ is invertible on monodromic objects.
To prove this, by the smooth base change theorem, it is enough to show that the natural map
\[
 \gr_{\nu, *}' \to  \gr_{\nu, *}' \eta_* \eta^* 
\]
is invertible. This is a direct consequence of Proposition \ref{prop:contraction},
as $\gr_{\nu}'$ and $\gr_{\mu}$ determines an $\bA^1$-deformation retract.
For a monodromic object $\cF \in D^b_c(Y)$,
we define an isomorphism
\begin{equation}\label{eq:hyp_sm_pullback}
\rho_{h} \colon h^{\bG_m, *} \Loc_{\nu}(\cF) \cong  \Loc_{\mu} h^* \cF    [ 2d^- ]
\end{equation}
by the following composition
\begin{align*}
    \rho_{h} \colon h^{\bG_m, *} \Loc_{\nu}(\cF) 
    &= h^{\bG_m, *} \gr_{\nu, *} \ev_{\nu}^! \cF \\
    &\cong  \gr_{\mu, *} h^{+, *} \ev_{\nu}^! \cF \\
    &\cong \gr_{\mu, *} h^{+, !} \ev_{\nu}^! \cF [- 2 d_0 - 2d^+] \\
    &\cong \gr_{\mu, *} \ev_{\mu}^! h^! \cF [- 2 d_0 - 2d^+] \\
    &\cong \gr_{\mu, *} \ev_{\mu}^! h^* \cF [2d^-] 
    = \Loc_{\mu} h^* \cF [2d^{-}].
\end{align*}
Here we set $d^+ \coloneqq \rank_+ T_{X / Y}$, $d^- \coloneqq \rank_- T_{X / Y} $ and $d^0 \coloneqq \rank_0 T_{X / Y}$.

We now discuss the associativity of the hyperbolic localization functor.
Assume that we are given an algebraic space $X$ with an action $\hat{\mu} = (\mu_1, \mu_2)$ of $\bG_m^2$ and a complex $\cF \in D^b_c(X)$.
Consider the following diagram:
\[
\xymatrix{
{}
& {}
& {X_{\mat}^{\hat{\mu}, +, +}} 
\ar[ld]
\ar[rd]
& {}
& {} \\
{}
& {(X^{\mu_2})^{\mu_1, +}_{\mat}}
\ar[ld]_-{\gr_{\mu_1 |_{X^{\mu_2}}}   }
\ar[rd]^-{\ev_{\mu_1 |_{X^{\mu_2}}}  }
& {}
& {X^{\mu_2, +}} 
\ar[ld]_-{\gr_{\mu_2}}
\ar[rd]^-{\ev_{\mu_2}}
& {} \\
{X^{\hat{\mu}}_{\mat}}
& {}
& {X^{\mu_2}}
& {}
& {X} 
}
\]
It follows from Proposition \eqref{prop:matching_invertible} that the middle square is Cartesian.
By using the proper base change theorem, 
we obtain an isomorphism
\[
\Loc^2_{\hat{\mu}} (\cF)|_{X^{\hat{\mu}}_{\mat}} \cong (\Loc_{\mu_1 |_{X^{\mu_2}}} \circ \Loc_{\mu_2}(\cF)) |_{X^{\hat{\mu}}_{\mat}}.
\]
By using this and a similar isomorphism for the action $(\mu_2, \mu_1)$, we obtain an isomorphism
\begin{equation}\label{eq:chi_space}
\chi_{\hat{\mu}} \colon (\Loc_{\mu_1 |_{X^{\mu_2}}} \circ \Loc_{\mu_2}(\cF)) |_{X^{\hat{\mu}}_{\mat}} \cong (\Loc_{\mu_2 |_{X^{\mu_1}}} \circ \Loc_{\mu_1}(\cF)) |_{X^{\hat{\mu}}_{\mat}}.
\end{equation}

We now discuss the compatibility between the hyperbolic localization functor and the vanishing cycle functor.
Assume now that we are given a smooth quasi-separated algebraic space $U$ with a $\bG_m$-action $\mu$ and a $\bG_m$-invariant regular function $f$ on $U$.
We let $\mu_0$ denote the induced $\bG_m$-action on the zero locus $U_0$.
The following statement is an easy consequence of Theorem \ref{thm:Braden} and the definition of the vanishing cycle functor (see e.g. \cite[Proposition 5.4.1]{nak16} or \cite[Proposition 3.3]{des22}):

\begin{prop}\label{prop:van_commute_hyp}
    The following natural morphism constructed using \eqref{eq:van_natural} is invertible on monodromic objects:
    \[
    \zeta_{\mu, f} \colon  \varphi_{f |_{U^{\mu}}} \circ \Loc_{\mu_0} \to \Loc_{\mu_0} \circ \varphi_{f}.
    \]
\end{prop}

We now prove some properties of the map $\zeta_{\mu, f}$ that we will use later.
Firstly, we now discuss the map $\zeta_{\mu, f}$ for the trivial action.
Assume that we are given a trivial $\bG_m$-action $\mu_{\mathrm{triv}}$ on an algebraic space $U$ and a regular function $f$ on $U$.
Note that we have an identification of the functor $\id \cong \Loc_{\mu_{\mathrm{triv}}}$.
Then it is clear that the following diagram commutes:
\begin{equation}\label{eq:zeta_unital}
\begin{aligned}
    \xymatrix@C=50pt{
    {\varphi_f \circ \Loc_{\mu_{\mathrm{triv}}}} \ar[r]_-{\cong}^-{\zeta_{\mu_{\mathrm{triv}, f}}} \ar[d]^-{\cong}
    & {\Loc_{\mu_{\mathrm{triv}}} \circ \varphi_f}
    \ar[d]^-{\cong} \\
    {\varphi_f } \ar@{=}[r]
    & {\varphi_f.}
    }
    \end{aligned}
\end{equation}

Assume now that we are given smooth algebraic spaces $U$ and $V$ with $\bG_m$-action $\mu$ and $\nu$ and a $\bG_m$-equivariant smooth morphism $h \colon U \to V$.
Assume further that we are given a function $g \colon V \to \bA^1$ and set $f = g \circ h$.
Take a monodromic object $\cF \in D^b_c(V)$.
Since $\zeta_{\mu, f}$ is constructed out of the standard exchange morphisms between six-operations, one sees that the following diagram commutes:
\begin{equation}\label{eq:sm_localization_van}
\begin{aligned}
\xymatrix@C=70pt{
{h_0^{\bG_m, *} \varphi_{g |_{Y^{\nu}}} \Loc_{\nu}\cF  } \ar[d]_-{\cong }^-{\rho_{h}^{}}
\ar[r]^-{h^{\bG_m, *} \zeta_{\nu, g}}_-{\cong} 
& {h_0^{\bG_m, *} \Loc_{\nu_0} \varphi_{g} \cF  } 
\ar[d]^-{\cong}_-{\rho_{h_0}^{}}
 \\
 {\varphi_{f |_{X^{\mu}}} \Loc_{\mu} h^* \cF}  \ar[r]_-{\cong}^-{\zeta_{\mu, f}} [2d^-]
 &
{\Loc_{\mu_0} \varphi_{f} h^* \cF [2d^- ].}  
}
\end{aligned}
\end{equation}

We now discuss the associativity of the isomorphism $\zeta_{\mu, f}$, which is related to the associativity of the cohomological Hall algebras.
Let $U$ be a quasi-separated algebraic space with a $\bG_m^2$-action $\hat{\mu} = (\mu_1, \mu_2)$ and 
$f \colon U \to \bA^1$ be a $\bG_m^2$-invariant function.
Then by the construction of the map $\zeta_{\mu, f}$ and $\chi_{\hat{\mu}}$,
it is obvious that the following diagram commutes:

\begin{equation}\label{eq:zeta_assoc_functor}
    \begin{aligned}
        \xymatrix@C=80pt{
        \varphi_{f |_{U^{\hat{\mu}}} } \circ \Loc_{\mu_1 |_{U^{\mu_2}}} \circ \Loc_{\mu_2}( - ) |_{U^{\hat{\mu}}_{0, \mat}} 
        \ar[d]^-{\chi_{\hat{\mu}}}_-{\cong}
        \ar[r]_-{\cong}^-{\zeta_{\mu_2, f} \circ \zeta_{\mu_1|_{U^{\mu_2}}, f|_{U^{\mu_2}} }}
        & {\Loc_{\mu_{1, 0} |_{U_0^{\mu_{2, 0}}}} \circ \Loc_{\mu_{2, 0}} \circ \varphi_f( - ) |_{U^{\hat{\mu}}_{0, \mat}}}
        \ar[d]^-{\chi_{\hat{\mu}_0}}
        \\
        \varphi_{f|_{U^{\hat{\mu}}}} \circ \Loc_{\mu_2 |_{U^{\mu_1}}} \circ \Loc_{\mu_1}( -)  |_{U^{\hat{\mu}}_{0, \mat}} \ar[r]_-{\cong}^-{\zeta_{\mu_1, f} \circ \zeta_{\mu_2|_{U^{\mu_1}}, f|_{U^{\mu_1}} }}
        & {\Loc_{\mu_{2, 0} |_{U_0^{\mu_{1, 0}}}} \circ \Loc_{\mu_{1, 0}} \circ \varphi_f( - ) |_{U^{\hat{\mu}}_{0, \mat}} .}
        }
    \end{aligned}
\end{equation}

We now prove that $\zeta_{\mu, f}$ is compatible with the Thom--Sebastiani isomorphism.
Let $U$ and $V$ be smooth quasi-separated algebraic spaces with $\bG_m$-action $\mu$ and $\nu$ and  $f$ and $g$ be $\bG_m$-invariant functions on $U$ and $V$. Take monodromic objects $\cF \in D^b_c(U)$ and $\cG \in D^b_c(V)$.
Since $\zeta_{\mu, f}$ is constructed out of the standard exchange morphisms between six-operations,
the commutativity of the following diagram is obvious:
\begin{equation}\label{eq:zeta_TS_functor}
    \begin{aligned}
        \xymatrix@C=70pt{
        { \varphi_{f |_{U^{\mu}}} \left( \Loc_{\mu}(\cF) \right) \boxtimes \varphi_{g |_{V^{\nu}}} \left( \Loc_{\nu}(\cG) \right)  } \ar[r]_-{\cong}^-{\zeta_{\mu, f} \boxtimes \zeta_{\nu, g}}  \ar[d]_-{\cong}^-{\TS}
        & {\Loc_{\mu_0}\varphi_{f}(\cF) \boxtimes \Loc_{\nu_0}\varphi_{g}(\cG)} \ar[d]_-{\cong}^-{\TS} 
          \\
          {\varphi_{f|_{U^{\mu}} \boxplus g _{V^{\nu}}} \left( \Loc_{\mu \times \nu}(\cF \boxtimes \cG) \right) |_{U_0^{\mu_0} \times V_0^{\nu_0}}} \ar[r]_-{\cong}^-{\zeta_{\mu \times \nu, f \boxplus g}}
          &
        {\Loc_{\mu_0 \times \nu_0}(\varphi_{f \boxplus g}(\cF \boxtimes \cG)) |_{U_0^{\mu_0} \times V_0^{\nu_0}} .} 
        }
    \end{aligned}
\end{equation}

\subsection{Hyperbolic localization of DT perverse sheaves}\label{ssec:hyp_DT}

We now discuss Descombes's Theorem \cite[Theorem 4.2]{des22} stating that the hyperbolic localization of the DT perverse sheaf is again a DT perverse sheaf.
Since we will use the construction of the isomorphism later and the paper \cite{des22} seems to ignore possible automorphisms of orientations \footnote{After our paper appeared on arXiv, the author posted an updated version \cite{descombes2025hyperbolic} of \cite{des22} including more careful discussions on orientations.},
we will give a detailed proof here.

\begin{thm}[{\cite[Theorem 4.2]{des22}}]\label{thm:hyp_DT}

    Let $(X, \mu, s)$ be a $\bG_m$-equivariant d-critical algebraic space with an orientation $o$.
    Assume that $X$ is quasi-separated.
    Then there exists a natural isomorphism
    \begin{equation}\label{eq:hyp_DT}
        \zeta_{\mu, s, o} \colon   \varphi_{X^{\mu}, s^{\mu}, o^{\mu}}[ I^{\mu}_{X}] \cong \Loc_{\mu}(\varphi_{X, s, o}).
    \end{equation}
    
\end{thm}

\begin{proof}

We first construct $\zeta_{\mu, s, o}$ over each \'etale d-critical chart.
Take a $\bG_m$-equivariant \'etale d-critical chart  $\fR = (R, \eta, U, f, i)$.
By abuse of notation, we write the $\bG_m$-action on $R$ and $U$ by $\mu$.
We first construct an isomorphism
\[
\beta_{ \fR} \colon Q_{\fR}^{o} |_{R^{\mu} } \cong Q_{\fR^{\mu}}^{o^{\mu}}.
\]
Recall from \S \ref{ssec:DT_perv} that $Q_{\fR}^{o}$ parametrizes local isomorphisms between $o$ and the canonical orientation.
Therefore, it is enough to construct an isomorphism
\begin{equation}\label{eq:ori_beta}
(o_{\fR}^{\can})^{\mu_{}} \cong o_{\fR^{\mu}}^{\can}.
\end{equation}
which we have already seen in \eqref{eq:can_ori_loc_action}.
 
We now construct 
\[
\zeta_{\mu, s, o, \fR} \colon   \varphi_{X^{\mu}, s^{\mu}, o^{\mu}}[ I^{\mu}_X] |_{R^{\mu}} \cong \Loc_{\mu}(\varphi_{X, s, o})|_{R^{\mu}}
\]
by the following composition
\begin{align*}
    \varphi_{X^{\mu}, s^{\mu}, o^{\mu}}[ I^{\mu}_X] 
    &\xrightarrow[\cong]{\omega_{\fR^{\mu}}} 
    \varphi_{f |_{U^{\mu}}} \otimes Q^{o^{\mu}}_{\fR^{\mu}}  [d^+ - d^{-}] |_{R^{\mu}} \\
    &\xrightarrow[\cong]{\zeta_{\mu, f} \otimes \beta_{\fR}^{-1}} \Loc_{\mu}(\varphi_{f} )  \otimes Q^{o}_{\fR} |_{R^{\mu}} 
    \xrightarrow[\cong]{\omega_{\fR}^{-1}, \eqref{eq:hyp_sm_pullback}}  \Loc_{\mu}(\varphi_{X, s, o})|_{R^{\mu}}.
\end{align*}
Here, we set $d^{+} \coloneqq \rank_+ T_{U}$ and $d^{-} \coloneqq \rank_- T_{U}$.
The shift in the third complex disappears by \eqref{eq:we_are_shifting_vanishing} and \eqref{eq:hyp_const}.
We now prove that $\zeta_{\mu, s, o, \fR}$ does not depend on the choice of $\fR$.
Since it is clear that $\zeta_{\mu, s, o, \fR}$ is compatible with $\bG_m$-equivariant \'etale pullback of \'etale d-critical charts,
by using Theorem \ref{thm:T_compare_d-critical}, it is enough to prove the identity
\begin{equation}
\zeta_{\mu, s, o, \fR} = \zeta_{\mu', s, o, \fR'} \label{eq-zeta-indep}
\end{equation}
for $\fR'$ the \'etale d-critical chart defined by
\[
\fR' = (R, \eta,  U \times \bA^{2N}, f \boxplus q, (i, 0))
\]
for some $\bG_m$-equivariant non-degenerate quadratic function $q$ on $\bA^{2N}$.
We write
\[
q = (x_1 y_1 + \cdots + x_n y_n) + (z_1 w_1 + \cdots + z_m w_m)
\]
where $x_i$ (resp. $y_i$) has positive (resp. negative) $\bG_m$-weights and $z_j$ and $w_j$ are invariant under the $\bG_m$-action.
We write $q_{x, y} = x_1 y_1 + \cdots + x_n y_n$ and $q_{z, w} = z_1 w_1 + \cdots + z_m w_m$.
We let $\mu' = (\mu, \nu)$ denote the $\bG_m$-action on $U \times \bA^{2N}$.
Consider the following diagram:
\[
\xymatrix@C=50pt{
{Q^{o}_{\fR'}|_{R^{\mu}}} \ar[r]_-{\cong}^-{\beta_{\fR'}}  \ar[d]_-{\cong}^-{b_{y, w}}
& {Q^{o^{\mu}}_{\fR'^{\mu}}} \ar[d]_-{\cong}^-{b_{w}} \\ 
{Q^{o}_{\fR} |_{R^{\mu}}} \ar[r]_-{\cong}^-{\beta_{\fR}} 
& {Q^{o^{\mu}}_{\fR^{\mu}}} .
}
\]
We claim that it commutes. To see this,  consider the \'etale d-critical chart
\[
\fR'_{{x, y}} = (R, \eta,  U \times \bA^{2n}, f \boxplus q_{x, y}, (i, 0)).
\]
By using the identity \eqref{eq:b_assoc}, it is enough to prove that the following diagram commutes:
\[
\xymatrix@C=50pt{
{Q^{o}_{\fR'_{{x, y}}}|_{R^{\mu}}} \ar[r]_-{\cong}^-{\beta_{\fR'_{{x, y}}}}  \ar[d]_-{\cong}^-{b_{y}}
& {Q^{o^{\mu}}_{\fR'^{\mu}_{{x, y}}}} \ar@{=}[d]^-{} \\ 
{Q^{o}_{\fR} |_{R^{\mu}}} \ar[r]_-{\cong}^-{\beta_{\fR}} 
& {Q^{o^{\mu}}_{\fR^{\mu}}} .
}
\]
It is equivalent to the commutativity of the following diagram:
\begin{equation}\label{eq:comm_b_y_and_qis}
\begin{aligned}
\xymatrix{
{K_{U \times \bA^{2n}} |_{R^{\mu, \red}} \otimes K_{U \times \bA^{2n}} |_{R^{\mu, \red}}^{\pm, \vee} } \ar[r]^-{\cong} \ar[d]^-{\cong}
& {K_{U^{\mu} |_{R^{\mu, \red}}}}
\ar@{=}[d] \\
{K_{U } |_{R^{\mu, \red}} \otimes K_{U} |_{R^{\mu, \red}}^{\pm, \vee}}
\ar[r]^-{\cong}
& {K_{U^{\mu} |_{R^{\mu, \red}}}}
}
\end{aligned}
\end{equation}
where the left vertical map is constructed using \eqref{eq:isom_canori_quadratic} and \eqref{eq:embed_L_Ind_isom}.
Consider the following acyclic complex
\[
\bL_{\fR'_{x, y} /\fR}^+ \coloneqq [T_{\bA^{2n}}^+ \xrightarrow{\Hess_q^+} \Omega_{\bA^{2n}}^+], \quad 
\partial x_i \mapsto d y_i.
\]
The fibre sequence
\[
\Omega_{\bA^{2n}}^+ \to \bL_{\fR' /\fR}^+ \to T_{\bA^{2n}}^+[1]
\]
together with \eqref{eq:det_shift} and \eqref{eq:det_dual} 
induces a trivialization
\begin{equation}\label{eq:trivialize_K_A2n}
K_{\bA^{2n}} \cong K_{\bA^{2n}}^{\pm} \cong \cO_{\bA^{2n}}.
\end{equation}
By the construction of \eqref{eq:embed_L_Ind_isom}, it is enough to compare this isomorphism with 
\eqref{eq:isom_canori_quadratic}.
By \cite[p.8]{Con}, we have an equivalence of fibre sequences
\[
\xymatrix{
{\Omega_{\bA^{2n}}^+} \ar@{=}[d] \ar[r]
& {\bL_{\fR' /\fR}^+} \ar[d]_-{\simeq} \ar[r]
& {T_{\bA^{2n}}^+[1]} \ar[d]^-{\Hess_{q_{x, y}}}_-{\simeq} \\
{\Omega_{\bA^{2n}}^+} \ar[r]
& 0 \ar[r]
& {\Omega_{\bA^{2n}}^+[1],}
}
\]
where the bottom sequence is $\Delta_{\Omega_{\mathbb{A}^{2n}}^+}$  defined just before \eqref{eq:det_shift}. 
Therefore by using the construction of \eqref{eq:det_shift}, \eqref{eq:det_dual}, the commutativity of \eqref{eq:rotate_det1} for $\Delta_{\Omega_{\mathbb{A}^{2n}}^+}$  and that the double dual of a graded line bundle is trivialized with the insertion of a sign of the grading,
we see that the trivialization \eqref{eq:trivialize_K_A2n} is given by
\[
d x_1 \wedge \cdots \wedge d x_n \wedge d y_n \wedge \cdots \wedge d y_1 \mapsto 1,
\]
hence we conclude the commutativity of the diagram \eqref{eq:comm_b_y_and_qis}.

On the other hand, consider the following diagram
\[
\xymatrix@C=70pt{
{
{\varphi_{f |_{U^{\mu}} \boxplus q_{z, w} }[I^{\mu}_X]} \ar[d]_-{\cong}^-{\TS^{-1}}
\ar[r]_-{\cong}^-{\zeta_{\mu', f \boxplus q }}}
& \Loc_{\mu_{}} (\varphi_{f \boxplus q}) 
\ar[d]_-{\cong}^-{\TS^{-1}}    \\
{\varphi_{f |_{U^{\mu}}} \boxtimes \varphi_{q_{z, w}} [I^{\mu}_X] } 
\ar[d]_-{\cong}^-{c_w}
\ar[r]_-{\cong}^-{\zeta_{\mu, f} \boxtimes \zeta_{\nu , q} }
&
{\Loc_{\mu_{}} (\varphi_{f} \boxtimes \varphi_{q}) } 
\ar[d]_-{\cong}^-{c_{y, w}} 
\\
{\varphi_{f |_{U^{\mu}}}[I^{\mu}_X] }
\ar[r]_-{\cong}^-{\zeta_{\mu, f}}
&{\Loc_{\mu} (\varphi_{f})} .
}
\]
The commutativity of the upper square follows from the commutativity of the diagram \eqref{eq:zeta_TS_functor}.
To prove the commutativity of the lower diagram, using the commutativity of the diagram \eqref{eq:c_associative}, it is enough to prove the commutativity of the following diagram
\[
\xymatrix@C=60pt{
{\Loc_{\nu_{x, y}}(\varphi_{q_{x, y}}) |_{0}} \ar[r]^-{\zeta_{\nu_{x, y}, q_{x, y}}}_-{\cong} \ar[d]_-{\cong}^-{c_{y}}
& {\varphi_{0}} \ar[d]_-{\cong} \\
{\bQ_{0}} \ar@{=}[r]
& {\bQ_{0}}
}
\]
where $\nu_{x, y}$ denotes the $\bG_m$-action on $\bA^{2n}$.
Note that we have an identification $(\bA^{2n})^{\nu_{x, y}, -} = V_{y}$ where $V_{y}$ denote the subspace of $\bA^{2n}$ spanned by $y$-coordinates.
Then the claim follows from the construction of $c_{y}$ and $\zeta_{\nu_{x, y}, q_{x, y}}$ since both of them are constructed using the fundamental class of $V_{y}$.

\end{proof}

We now prove some properties of $\zeta_{\mu, s, o}$ that we will use later. 
Firstly, let $\mu^{\mathrm{triv}}$ be the trivial $\bG_m$-action on an oriented d-critical algebraic space $(X, s, o)$.
Then by the commutativity of \eqref{eq:zeta_unital},
we see that the following diagram commutes:

\begin{equation}\label{eq:zeta_unital_algsp}
    \begin{aligned}
    \xymatrix@C=50pt{
    {\varphi_{X, s, o}} \ar[r]_-{\cong}^-{\zeta_{\mu_{\mathrm{triv}, s, o}}} 
    \ar@/_30pt/[rr]_-{\id}
    & {\Loc_{\mu_{\mathrm{triv}}} (\varphi_{X, s, o}) }
    \ar[r]_-{\cong}
    &  \varphi_{X, s, o} \\
    {}
    &
    &
    }
    \end{aligned}
\end{equation}

We now prove that the map $\zeta_{\mu, s, o}$ is compatible with smooth pullback.
Though this is already proved in \cite[Proposition 4.5]{des22}, 
we will include the proof for reader's convenience:

\begin{prop}[{\cite[Proposition 4.5]{des22}}]\label{prop:zeta_sm_space}
    Let $(X, \mu, s)$ and $(Y, \nu, t)$ be quasi-separated $\bG_m$-equivariant algebraic spaces and $h \colon X \to Y$ be a $\bG_m$-equivariant smooth morphism with $s = h^{\star} t$.
    Let $o$ be an orientation for $(Y, \nu, t)$.
    Then the following diagram commutes:
    \[
    \xymatrix@C=30pt{
    {h^{\bG_m, *} \varphi_{Y^{\nu}, t^{\nu}, o^{\nu}}[I^{\nu}_Y]} \ar[rr]_-{\cong}^-{\Theta_{h^{\bG_m}, t^{\nu}, o^{\nu} }}
    \ar[d]_-{\cong}^-{\zeta_{\nu, t, o}}
    & {}
    & {\varphi_{X^{\mu}, s^{\mu}, (h^{\star }o)^{\mu} }}[I^{\nu}_Y - d_0]
    \ar[d]_-{\cong}^-{\zeta_{\mu,  s, h^{\star} o}}
    \\
        {h^{\bG_m, *} {\Loc_{\nu}}(\varphi_{Y, t, o})} 
         \ar[r]_-{\cong}^-{\rho_{h}}
    &{\Loc_{\mu}(h^* \varphi_{Y, t, o})[2 d^- ] } 
    \ar[r]_-{\cong}^-{\Theta_{h, t, o}} 
    & {\Loc_{\mu} (\varphi_{X, s, h^{\star}o})[2 d^- - d] } 
    }
    \]
    where we set $d \coloneqq \dim h$, $d_0 \coloneqq \dim h^{\bG_m}$, $d^+ \coloneqq \rank_+ T_{X/Y} $ and $d^- \coloneqq \rank_- T_{X/Y} $
   and we used the identity  $I^{\mu}_X - I^{\nu}_Y = d^+ - d^{-}$ from \eqref{eq:smooth_index}.
    
\end{prop}

\begin{proof}
    It is enough to prove the commutativity for each $\bG_m$-equivariant \'etale d-critical chart around each point of $X^{\mu}$.
    Using Proposition \ref{prop:T_sm_chart}, we may take $\bG_m$-equivariant \'etale d-critical chart $\fR = (R, \eta, U, f, i)$ of $X$ and $\fS = (S, \gamma, V, g, i)$ and a $\bG_m$-equivariant smooth morphism $(H_0, H) \colon \fR \to \fS$.
    By the commutativity of the diagram \eqref{eq:sm_localization_van}, it is enough to prove the commutativity of the following diagram:
    \[
    \xymatrix@C=50pt{
    {H_0^* Q^{o}_{\fS} |_{R^{\mu}}} \ar[r]_-{\cong}^-{b_H} \ar[d]_-{\cong}^-{\beta_{\fS}} \ar[r]
    & {Q^{h^{\star} o}_{\fR}|_{R^{\mu}}} \ar[d]_-{\cong}^-{\beta_{\fR}} \\
    {H_0^{\bG_m, *} Q^{o^{\nu}}_{\fS^{\nu}}} \ar[r]_-{\cong}^-{b_{H^{\bG_m}}}
    & {Q^{(h^{\star} o)^{\mu}}_{\fR^{\mu}}}
    }
    \]
    where the $\bG_m$-action on $R$ and $S$ are denoted by $\mu$ and $\nu$ by abuse of notation, and
    we use the identification \eqref{eq:ori_loc_sm_compatible}.
    This follows from the commutativity of the diagram \eqref{eq:can_ori_loc_pull_commutes}.
\end{proof}

We now prove the associativity of the morphism $\zeta_{\mu, s, o}$:

\begin{prop}\label{prop:zeta_assoc_space}
    Let $(X, (\mu_1, \mu_2), s)$ be a quasi-separated $\bG_m^2$-equivariant d-critical algebraic space.
    Then the following diagram commutes:
    \begin{equation}\label{eq:zeta_assoc_space}
    \begin{aligned}
        \xymatrix@C=100pt{
        \varphi_{X^{\hat{\mu}}, s^{\hat{\mu}}, (o^{\mu_2})^{\mu_1 }}[ I^{\mu_2}_{X} + I^{\mu_1}_{X^{\mu_2}}] |_{X^{\hat{\mu}}_{\mat}} 
        \ar[d]^-{}_-{\cong}
        \ar[r]_-{\cong}^-{\zeta_{\mu_2, s, o} \circ \zeta_{\mu_1|_{X^{\mu_2}}, s^{\mu_2}, o^{\mu_2} }}
        & \Loc_{\mu_1 } (\Loc_{\mu_2} ( \varphi_{X, s, o} )) |_{X^{\hat{\mu}}_{\mat}} 
        \ar[d]^-{\chi_{\hat{\mu}}}_-{\cong}
        \\
        \varphi_{X^{\hat{\mu}}, s^{\hat{\mu}}, (o^{\mu_1})^{\mu_2 }} [ I^{\mu_1}_{X} + I^{\mu_2}_{X^{\mu_1}}] |_{X^{\hat{\mu}}_{\mat}}
        \ar[r]_-{\cong}^-{\zeta_{\mu_1, s, o} \circ \zeta_{\mu_2|_{X^{\mu_1}}, s^{\mu_1}, o^{\mu_1} }}
        & \Loc_{\mu_2} (\Loc_{\mu_1} ( \varphi_{X, s, o} )) |_{X^{\hat{\mu}}_{\mat}}.
        }
    \end{aligned}
    \end{equation}
    Here, the left vertical map is constructed using the isomorphism \eqref{eq:ori_localize_assoc_action}.

\end{prop}

\begin{proof}
    Take a point $x \in X^{\hat{\mu}}_{\mathrm{mat}}$.
    It is enough to prove the commutativity for a $\bG_m^2$-equivariant \'etale d-critical charts around $x$.
    Take a $\bG_m^2$-equivariant \'etale d-critical chart $\fR = (R, \eta, U, f, i)$ of $x$.
    By abuse of notion, we will denote by $\hat{\mu} = (\mu_1, \mu_2)$ the $\bG_m^2$-actions on $R$ and $U$.
    Using Proposition \ref{prop:minimal_chart_space}, we may assume that there exists a point $\hat{x} \in R^{\hat{\mu}}$ lifting $x$ which is contained in $U^{\hat{\mu}}_{\mat}$.
    By the commutativity of the diagram \eqref{eq:zeta_assoc_functor}, it is enough to prove the commutativity of the diagram
    \[
    \xymatrix@C=50pt{
    {Q_{\fR^{\hat{\mu}}} ^{(o^{\mu_2})^{\mu_1  }} } |_{R^{\hat{\mu}}_{\mat}} 
    \ar[d]^-{\cong}
    \ar[r]_-{\cong}^-{\beta_{\fR^{\mu_2}}}
    & {Q_{\fR^{\mu_2}} ^{o^{\mu_2} } } |_{R^{\hat{\mu}}_{\mat}}
    \ar[r]_-{\cong}^-{\beta_{\fR, \mu_2}}
    & {Q_{\fR}^{o} |_{R^{\hat{\mu}}_{\mat}}} \ar@{=}[d]  \\
    {{Q_{\fR^{\hat{\mu}}} ^{(o^{\mu_1})^{\mu_2  }} } |_{R^{\hat{\mu}}_{\mat}}}
    \ar[r]_-{\cong}^-{\beta_{\fR^{\mu_1}}}
    & {Q_{\fR^{\mu_1}} ^{o^{\mu_1} } } |_{R^{\hat{\mu}}_{\mat}}
    \ar[r]_-{\cong}^-{\beta_{\fR, \mu_1}}
    & {Q_{\fR}^{o} |_{R^{\hat{\mu}}_{\mat}}.}
    }
    \]
    Here the subscript of $\beta_{\fR, \mu_1}$ and $\beta_{\fR, \mu_2}$ denotes $\bG_m$-action that we consider. 
    Recall that $Q_{\fR}^{o}$ parametrizes local isomorphisms between $o_{\fR}^{\can}$ and $o$.
   The commutativity of the diagram follows from \eqref{eq:can_ori_loc_assoc}.
\end{proof}

We now show that the map $\zeta_{\mu, s, o}$ is compatible with the Thom--Sebastiani isomorphism.

\begin{prop}\label{prop:zeta_TS_space}
    Let $(X, \mu, s)$ and $(Y, \nu, t)$ be quasi-separated $\bG_m$-equivariant d-critical algebraic spaces with orientations $o_X$ and $o_Y$.
    Then the following diagram commutes:
    \[
    \xymatrix@C=70pt{
    {\varphi_{X^{\mu}, s^{\mu}, o_X^{\mu}}} \boxtimes  {\varphi_{Y^{\nu}, t^{\nu}, o_Y^{\nu}}} [I^{\mu}_X + I^{\nu}_Y]
    \ar[r]^-{\zeta_{\mu, s, o_X} \boxtimes \zeta_{\nu, t, o_Y}}_-{\cong}
    \ar[d]^-{\TS}_-{\cong}
    &  \Loc_{\mu}(\varphi_{X, s, o_X}) \boxtimes \Loc_{\nu}(\varphi_{Y, t, o_Y})
    \ar[d]^-{\TS}_-{\cong}
    \\
    {\varphi_{X^{\mu} \times Y^{\nu}, s^{\mu} \boxplus t^{\nu}, o_X^{\mu} \boxtimes o_Y^{\nu}}} [I^{\mu \times \nu}_{X \times Y}] \ar[r]_-{\cong}^-{\zeta_{\mu \times \nu, s \boxplus t, o_X \boxtimes o_Y}}
    & \Loc_{\mu \times \nu}(\varphi_{X \times Y, s \boxplus t, o_X \boxtimes o_Y}).
    }
    \]
\end{prop}

\begin{proof}
    It is enough to prove the commutativity for each $\bG_m$-equivariant \'etale d-critical charts around each point of $(x, y) \in X^{\mu} \times Y^{\nu}$.
    Take an \'etale d-critical chart $\fR = (R, \eta, U, f, i)$ of $x$ and $\fS = (S, \gamma, V, g, j)$ of $y$.
    By the commutativity of the diagram \eqref{eq:zeta_TS_functor} and the construction of Thom--Sebastiani isomorphism of d-critical algebraic spaces \eqref{eq:TS_chart},
    it is enough to prove the commutativity of the following diagram:
    \[
    \xymatrix@C=70pt{
    {Q_{\fR}^{o_{X}} \boxtimes Q_{\fS}^{o_{Y}} |_{X^{\mu} \times Y^{\nu}}}
    \ar[r]_-{\cong}^-{\beta_{\fR} \boxtimes \beta_{\fS}}
    \ar[d]_-{\cong}^-{\alpha_{\fR, \fS}}
    & {Q_{\fR^{\mu}}^{o_{X}^{\mu}} \boxtimes Q_{\fS^{\mu}}^{o_{Y}^{\nu}}}
    \ar[d]_-{\cong}^-{\alpha_{\fR^{\mu}, \fS^{\nu}}} \\
    {Q_{\fR \times \fS}^{o_{X} \boxtimes o_{Y}}  |_{X^{\mu} \times Y^{\nu}}.}
    \ar[r]_-{\cong}^-{\beta_{\fR \times \fS}}
    & {Q_{\fR^{\mu} \times \fS^{\nu}}^{o_X^{\mu} \boxtimes o_{Y}^{\nu}}.}
    }
    \]
    This follows from \eqref{eq:can_ori_loc_product}.
\end{proof}

\subsection{Hyperbolic localization for stacks}\label{ssec:hyp_stack}

Here we will introduce the hyperbolic localization functor for Artin stacks.
Let $\fX$ be an Artin stack with quasi-separated diagonal and affine stabilizer groups.
Consider the following correspondence:
\begin{equation}\label{eq:gr_filt_second}
\begin{aligned}
\xymatrix{
{}
& {\Filt^n(\fX)} \ar[ld]_-{\gr} \ar[rd]^-{\ev}
& {} \\
{\Grad^n(\fX)}
& {}
& {\fX}
}
\end{aligned}
\end{equation}
We define a functor $\Loc^n_{\fX} \coloneqq D^b_c(\fX) \to D(\Grad^n(\fX))$ by
\[
\Loc_{\fX}^n \coloneqq \gr_{ *} \ev^!.
\]
When $n = 1$, we write $\Loc_{\fX} = \Loc^1_{\fX}$ and we will call it the \defterm{hyperbolic localization functor}.

\begin{rmk}
    More generally, for a quasi-separated Artin stack $\fX$ with affine stabilizers with a $\bG_m^n$-action $\mu$, once can define the functor
    $\Loc_{\mu}^n \colon D(\fX) \to D(\fX^{\mu})$ using the correspondence
    \begin{equation}\label{eq:attractor_general}
    \fX^{\mu} \leftarrow \fX^{\mu, +} \to \fX.
    \end{equation}
    The correspondence \eqref{eq:gr_filt_second} corresponds to the case of the trivial action.
    Since the correspondence \eqref{eq:attractor_general} can be embedded to the correspondence \eqref{eq:gr_filt_second} for $[\fX / \bG_m^n]$, we do not lose generality by restricting to the case of the trivial action.
\end{rmk}

Let $\fX$ be a quasi-separated Artin stack with affine stabilizers. By using Proposition \ref{prop:QCha_atlas}, take a collection of $\bG_m$-equivariant charts $(X, \mu, \bar{q})$ such that
$q_\mu \colon X^{\mu} \to \Grad(\fX)$ is jointly surjective.
Consider the following diagram:
\[
\xymatrix{
{X^{\mu}}
\ar[dd]^-{q_{\mu}}
& {}
& {X^{\mu, +}} 
\ar[r]^-{\ev_{\mu}}
\ar[ll]_-{\gr_{\mu}}
\ar[dd]^-{q_{\mu}^{+}}
\ar[ld]_-{\eta}
& {X}
\ar[dd]^-{q} \\
{}
& {X^{\mu} \times_{\Grad(\fX)} \Filt(\fX)}
\ar[rd]^-{q_{\mu}'}
\ar[lu]_-{\gr'}
& {} 
& {} \\
{\Grad(\fX)}
& {}
& {\Filt(\fX)}
\ar[r]^-{\ev}
\ar[ll]_-{\gr}
& {\fX.}
}
\]
It follows from \cite[Proposition 1.3.1 (4)]{hlp14} that $q_{\mu}$ and $q_{\mu}^+$ are smooth.
Further, by \cite[Lemma 1.3.8]{hlp14}, the map $\gr$ is $\bA^1$-deformation retract.
Therefore, by repeating the construction of the isomorphism \eqref{eq:hyp_sm_pullback} using Corollary \ref{cor:gr_contr_lemma},
we obtain an isomorphism of functors
\begin{equation}\label{eq:rho_q}
\rho_{q} \colon q_{\mu}^*  \circ \Loc_{\fX} \cong \Loc_{\mu} \circ q^*[2d^-]
\end{equation}
where we set $d^{-} \coloneqq \rank T_{X/ \fX}^{-}$.
In particular, the functor $\Loc_{\fX}$ preserves constructible objects.

\begin{lem}\label{lem:rho_q_assoc}
    Let $\fX$ be a quasi-separated Artin stack with affine stabilizers and
    $h \colon (X_1, \mu_1, \bar{q}_1) \to (X_2, \mu_2, \bar{q}_2)$ be a smooth morphism of $\bG_m$-equivariant charts.
    Then the following diagram commutes:
\[
\xymatrix@C=80pt{
{q_{1, \mu_1}^* \Loc_{\fX}  }
\ar[d]^-{\cong}
\ar[r]_-{\cong}^-{\rho_{q_1}}
& { \Loc_{\mu_1} q_1^* [2d_1^{-}] }\\
 h^{\bG_m, *} q_{2, \mu_2}^* \Loc_{\fX}
 \ar[r]_-{\cong}^-{\rho_{q_2}}
& { h^{\bG_m, *} \Loc_{\mu_2} q_2^* [2d_2^{-}] }
\ar[u]_-{\cong}^-{\rho_h}
}
\]
where we set $d_1^{-} \coloneqq \rank_- T_{X_1 / \fX}$
and $d_2^{-} \coloneqq \rank_- T_{X_2 / \fX}$.
    
\end{lem}

\begin{proof}
    The statement can be decomposed into the commutativity of the following diagrams:
\[
\xymatrix{
q_{1, \mu_1}^{+, *} \ev^!  
\ar[r]^-{\cong}
\ar[d]^-{\cong}
& {\ev_{\mu_1}^! q_{1}^{ *}[2 d_{1}^-]}
\ar[d]^-{\cong }\\
{h^{+, *} q_{2, \mu_2}^{+, *} \ev^!}
\ar[r]^-{\cong}
& {\ev_{\mu_1}^! h^{*} q_{2}^{ *}[2 d_{1}^-], }
} \quad
\xymatrix{
{q_{1, \mu_1}^* \gr_* }
\ar[r]^-{\cong}
\ar[d]^-{\cong}
& {\gr_{\mu_1, *} q_{1, \mu_1}^{+, *} }
\ar[d]^-{\cong}\\
{h^{\bG_m, *} q_{2, \mu_2}^*  \gr_*}
\ar[r]^-{\cong}
& {\gr_{\mu_1, *} {h^{+, *} q_{2, \mu_2}^{+, *} }.}
}
\]

The commutativity of the left diagram follows from the associativity of the purity transform \eqref{eq:purity_assoc}.
The commutativity of the right diagram follows from the associativity of the Beck--Chevalley transform.
\end{proof}

We now discuss the associativity of the hyperbolic localization functor.
By repeating the construction of $\chi_{\hat{\mu}}$ in \eqref{eq:chi_space} using Proposition \ref{prop:matching_invertible_stack}, we obtain the following natural isomorphism
\[
\chi_{\fX} \colon \Loc_{\Grad(\fX)} \circ \Loc_{\fX} (-) |_{\Grad^2(\fX)_{\mat}}
\cong \sigma^* \Loc_{\Grad(\fX)} \circ \Loc_{\fX} (-) |_{\Grad^2(\fX)_{\mat}}
\]
where $\sigma \colon \Grad^2(\fX) \cong \Grad^2(\fX)$ denotes the swapping isomorphism.

\begin{lem}\label{lem:chi_assoc}
    Let $\fX$ be a quasi-separated Artin stack with affine stabilizers and
    $(X, (\mu_1, \mu_2), \bar{q})$ be a  $\bG_m^2$-equivariant chart
    such that $q_{\hat{\mu}} \colon X^{\hat{\mu}} \to \Grad^2(\fX)$ restricts to 
    $ X^{\hat{\mu}}_{\mat} \to \Grad^2(\fX)_{\mat}$.
    Then the following diagram commutes:
    \[
    \xymatrix@C=50pt{
     q_{\hat{\mu}}^* \circ  \Loc_{\Grad(\fX)} \circ \Loc_{\fX} (-) |_{X^{\hat{\mu}}_{\mat}}
    \ar[r]_-{\cong}^-{\chi_{\fX}}
    \ar[d]_-{\cong}^-{\rho_{q} \circ \rho_{q_{\mu_2}}} 
    & q_{\hat{\mu}}^* \circ \sigma^* \Loc_{\Grad(\fX)} \circ \Loc_{\fX} (-) |_{X^{\hat{\mu}}_{\mat}} 
    \ar[d]_-{\cong}^-{\rho_{q} \circ \rho_{q_{\mu_1}}} 
    \\
    {\Loc_{\mu_1 |_{X^{\mu_2}}} \circ \Loc_{\mu_2} \circ q^* ( - )[2e] |_{X^{\hat{\mu}}_{\mat}}  }
    \ar[r]^-{\chi_{\hat{\mu}}}_-{\cong}
    & {\Loc_{\mu_2 |_{X^{\mu_1}}} \circ \Loc_{\mu_1} \circ q^* ( - )[2e] |_{X^{\hat{\mu}}_{\mat}}. }
    }
    \]
    where we set $e \coloneqq \rank_{\mu_1-}(T_{X / \fX}) + \rank_{\mu_2-}(T_{X^{\mu_1} / \Grad(\fX)})$.
\end{lem}

\begin{proof}
    Note that the map $\chi_{\fX}$ can be written as the composition
    \[
    \Loc_{\Grad(\fX)} \circ \Loc_{\fX} (-) |_{\Grad^2(\fX)_{\mat}} 
    \xrightarrow[\cong]{\chi_{\fX}'} \Loc_{\fX}^2 (-) |_{\Grad^2(\fX)_{\mat}} 
    \xrightarrow[\cong]{\chi_{\fX}''} \sigma^* \Loc_{\Grad(\fX)} \circ \Loc_{\fX} (-) |_{\Grad^2(\fX)_{\mat}} 
    \]
    where $\chi_{\fX}'$ and $\chi_{\fX}''$ are constructed using the base change transform 
    \eqref{eq:gen_bc} for the Cartesian diagram 
    \[
    \xymatrix@C=80pt{
    {\Filt^2(\fX)_{\mat}}
    \ar[r]^-{\Filt(\gr_{\fX})}
    \ar[d]^-{\ev_{\Filt(\fX)}}
    & {\Filt(\Grad(\fX))_{\mat}}
    \ar[d]^-{\ev_{\Grad(\fX)}} \\
    {\Filt(\fX)}
    \ar[r]^-{\gr_{\fX}}
    & {\Grad(\fX)}.
    }
    \]
    Therefore it is enough to prove the commutativity of the following diagram:
    \[
    \xymatrix@C=50pt{
     q_{\hat{\mu}}^* \circ  \Loc_{\Grad(\fX)} \circ \Loc_{\fX} (-) |_{X^{\hat{\mu}}_{\mat}}
    \ar[r]_-{\cong}^-{\chi_{\fX}'}
    \ar[d]_-{\cong}^-{\rho_{q} \circ \rho_{q_{\mu_2}}} 
    & q_{\hat{\mu}}^* \circ \Loc^2_{\fX} (-) |_{X^{\hat{\mu}}_{\mat}}
    \ar[d]_-{\cong}^-{} 
    \\
    {\Loc_{\mu_1 |_{X^{\mu_2}}} \circ \Loc_{\mu_2} \circ q^* ( - ) |_{X^{\hat{\mu}}_{\mat}} [2e]  }
    \ar[r]^-{\chi_{\hat{\mu}}'}_-{\cong}
    & {\Loc^2_{\hat{\mu}} \circ q^* ( - ) |_{X^{\hat{\mu}}_{\mat}}[2e], }
    }
    \]
    where the right vertical map is defined in the same manner as $\rho_q$ in $\eqref{eq:rho_q}$.
    By the construction of $\rho_q$ and $\chi_{\fX}'$, this can be further reduced to the commutativity of the following diagram:
    \[
    \xymatrix@C=50pt{
    q_{\hat{\mu}}^{(+, 0), *}  \ev_{\Grad(\fX)}^! \gr_{\fX, *}(-) |_{(X^{\mu_2})^{\mu_1, +}_{\mat}}
    \ar[r]^-{\cong}
    \ar[d]^-{\cong}
    & q_{\hat{\mu}}^{(+, 0), *}\Filt(\gr_{\fX})_*  \ev_{\Filt(\fX)}^! (-) |_{(X^{\mu_2})^{\mu_1, +}_{\mat}}
    \ar[d]^-{\cong} \\
    {\ev_{\mu_1|_{X^{\mu_2}}}^! \gr_{\mu_2, *} q_{\mu_2}^{+, *} ( - ) [2e']|_{(X^{\mu_2})^{\mu_1, +}_{\mat}}} 
    \ar[r]^-{\cong}
    & {   \gr_{\mu_2|_{X^{\mu_1, +}}, *}  \ev_{\mu_1|_{X^{\mu_2, +}}}^!  q_{\mu_2}^{+, *} ( - )[2e']|_{(X^{\mu_2})^{\mu_1, +}_{\mat}}  } 
    }
    \]
    where $q_{\hat{\mu}}^{(+, 0)} \colon (X^{\mu_2})^{\mu_1, +} \to \Filt(\Grad(\fX))$ is the natural morphism and we set $e' \coloneqq \rank_{\mu_1-} T_{X^{\mu_2} / \Grad(\fX)}$.
    The horizontal map is the base change map \eqref{eq:gen_bc} and the vertical map
    is also the base change map \eqref{eq:gen_bc} 
    after the identification $q_{\hat{\mu}}^{(+, 0), *} \cong q_{\hat{\mu}}^{(+, 0), !}[- 2 \dim q_{\hat{\mu}}^{(+, 0)}]$ and $q_{\mu_2}^{+, *} \cong q_{\mu_2}^{+, !}[- 2 \dim q_{\mu_2}^+]$ by Lemma \ref{lem:Beck_Chevalley_bc}.
    Then the commutativity of the above diagram follows by Lemma \ref{lem:Exi_associative} applied to the following compositions of pullable squares:
    \[
    \xymatrix{
    {X^{\hat{\mu}, +, +}_{\mat}}
    \ar[r]
    \ar[d]
    & {(X^{\mu_2})^{\mu_1, +}_{\mat}}
    \ar[d] \\
    {\Filt^2(\fX)_{\mat}}
    \ar[r] 
    \ar[d]
    & {\Filt(\Grad(\fX))_{\mat}}
    \ar[d] \\
    {\Filt(\fX)}
    \ar[r]
    & {\Grad(\fX)},
    }
    \quad
        \xymatrix{
    {X^{\hat{\mu}, +, +}_{\mat}}
    \ar[r]
    \ar[d]
    & {(X^{\mu_2})^{\mu_1, +}_{\mat}}
    \ar[d] \\
    {X^{\mu_2, +}}
    \ar[r] 
    \ar[d]
    & {X^{\mu_2}}
    \ar[d] \\
    {\Filt(\fX)}
    \ar[r]
    & {\Grad(\fX)}.
    }
    \]
\end{proof}

\subsection{Proof of the main theorem}\label{ssec:main}

In this section, we will prove Theorem \ref{maintheorem}.

\begin{thm}\label{thm:main_thm}
    Let $(\fX, s, o)$ be an oriented d-critical stack where $\fX$ is quasi-separated and has affine stabilizers.
    Then there exists a natural isomorphism
    \[
    \zeta_{\fX, s, o} \colon  \varphi_{\Grad(\fX), u^{\star} s, u^{\star} o}[I_{\fX}] \cong 
    \Loc_{\fX} (\varphi_{\fX, s, o}).
    \]
\end{thm}

\begin{proof}

Take a $\bG_m$-equivariant chart $\fQ = (X, \mu, q)$.
Set $d \coloneqq \rank T_{X / \fX}$, 
$d^+ \coloneqq \rank_+ T_{X / \fX} $,
$d^- \coloneqq \rank_- T_{X / \fX} $
and $d^0 \coloneqq \rank_0 T_{X / \fX}$.
We construct an isomorphism 
\[
\zeta_{\fX, s, o, \fQ} \colon  q_{\mu}^* \varphi_{\Grad(\fX), u^{\star} s, u^{\star} o}[I_{\fX}] \cong 
    q_{\mu}^*  \Loc_{\fX} (\varphi_{\fX, s, o})
\]
so that the following diagram commutes:
\[
\xymatrix@C=50pt{
{q_{\mu}^* \varphi_{\Grad(\fX), u^{\star} s, u^{\star} o}[I_{\fX}]}
\ar[d]^-{\Theta_{q_{\mu}, u^{\star} s, u^{\star} o}}_-{\cong}
\ar[rr]_-{\cong}^-{\zeta_{\fX, s, o, \fQ}}
& {}
& {q_{\mu}^* \Loc_{\fX}(\varphi_{\fX, s, o})}
\ar[d]_-{\cong}^-{\rho_q}\\
{\varphi_{X^{\mu}, (q^{\star} s)^{\mu}, (q^{\star} o)^{\mu}}[I_{\fX} - d_0]}
\ar[r]_-{\cong}^-{\zeta_{\mu, q^{\star} s, q^{\star} o}}
& {\Loc_{\mu}( \varphi_{X, q^{\star} s, q^{\star}o}) [d^- - d^+ - d_0 ] }
\ar[r]_-{\cong}^-{\Theta_{q, s, o}^{-1}}
& {\Loc_{\mu}(q^* \varphi_{\fX, s, o}) [2d^-].} 
}
\]
As shown in Proposition \ref{prop:QCha_atlas}, 
the family $\{ q_{\mu} \colon X^{\mu} \to \Grad(\fX) \}_{\fQ \in \mathrm{Cha}_{\fX}^{\bG_m}}$ defines a covering of $\Grad(\fX)$.
In particular, the complex $\Loc_{\fX} (\varphi_{\fX, s, o}) [- I_{\fX}] $ is perverse.
Therefore the assignment
\[
\mathrm{Sch}_{/ \Grad(\fX)}^{\mathrm{sm}} \ni (t \colon T \to \Grad(\fX)) \mapsto 
\Hom(t^* \varphi_{\Grad(\fX), u^{\star} s, u^{\star} o}[I_{\fX}],
t^* \Loc_{\fX}( \varphi_{\fX, s, o}))
\]
defines a lisse-\'etale sheaf.
Therefore by using Corollary \ref{cor:QCha_glue}, we are reduced to prove the following statement:
For a morphism $h \colon \fQ_1 = (X_1, \mu_1, \bar{q}_1) \to \fQ_2 = (X_2, \mu_2, \bar{q}_2)$ of $\bG_m$-equivariant charts, the following diagram commutes:
\[
\xymatrix@C=70pt{
{q_{\mu_1}^* \varphi_{\Grad(\fX), u^{\star} s, u^{\star} o}[I_{\fX}]}
\ar[r]_-{\cong}^-{\zeta_{\fX, s, o, \fQ_1}}
\ar[d]^-{\cong}
& {q_{\mu_1}^* \Loc_{\fX}(\varphi_{\fX, s, o})}
\ar[d]^-{\cong}\\
{h^{\bG_m, *} q_{\mu_2}^* \varphi_{\Grad(\fX), u^{\star} s, u^{\star} o}[I_{\fX}]}
\ar[r]_-{\cong}^-{\zeta_{\fX, s, o, \fQ_2}}
& {h^{\bG_m, *} q_{\mu_2}^* \Loc_{\fX}(\varphi_{\fX, s, o}).}
}
\]
This follows by the construction of $\zeta_{\fX, s, o, \fQ}$, Proposition \ref{prop:zeta_sm_space}, the commutativity of the diagram \eqref{eq:assoc_pull_van_stack} and Lemma \ref{lem:rho_q_assoc}.
\end{proof}

We now state the unitality property of $\zeta_{\fX, s, o}$. Consider a morphism $\iota\colon \fX\to\Grad(\fX)$, which is a section of $u\colon \Grad(\fX)\rightarrow \fX$.
Then by the commutativity of the diagram \eqref{eq:zeta_unital_algsp} as well as Lemma \ref{lem:Gradprojectionisomorphism} we see that the following diagram commutes:
\begin{equation}\label{eq:zeta_unital_stack}
\begin{aligned}
    \xymatrix{
    {\varphi_{\fX, s, o}} 
    \ar[r]_-{\cong}^-{\zeta_{\fX, s, o}}
    \ar@/_30pt/[rr]^-{\id}
    &  \iota^* \varphi_{\Grad(\fX), u^{\star} s, u^{\star} o}
    \ar[r]^-{\cong}
    &  \varphi_{\fX, s, o} \\
    {}
    & {}
    & {}
    }
    \end{aligned}
\end{equation}
where the latter isomorphism is constructed using $\iota^{\star} u^{\star} s = s$
and the natural isomorphism $\iota^{\star} u^{\star} o \cong o$.

We now prove the  associativity of the morphism $\zeta_{\fX, s, o}$.

\begin{prop}\label{prop:zeta_assoc_stack}
    The following diagram commutes:
    \[
    \xymatrix@C=100pt{
    {\varphi_{\Grad^2(\fX)_{\mat}, u^{(2), \star} s, u_2^{\star} u^{\star} o} [I_{\fX}'] }
    \ar[d]^-{\cong}
    \ar[r]_-{\cong}^-{\zeta_{\fX, s, o} \circ \zeta_{\Grad(\fX), u^{\star} s, u^{\star} o}}
    & {\Loc_{\Grad(\fX)} \circ \Loc_{\fX} (\varphi_{\fX}) |_{\Grad^2(\fX)_{\mat}}}
    \ar[d]_-{\cong}^-{\chi_{\fX}} \\
    {\varphi_{\Grad^2(\fX)_{\mat}, u^{(2), \star} s, u_1^{\star} u^{\star} o} [I_{\fX}']}
    \ar[r]_-{\cong}^-{\sigma^{*} (\zeta_{\fX, s, o} \circ \zeta_{\Grad(\fX), u^{\star} s, u^{\star} o})}
    & {\sigma^* \Loc_{\Grad(\fX)} \circ \Loc_{\fX} (\varphi_{\fX}) |_{\Grad^2(\fX)_{\mat}}}
    }
    \]
    where we set $I_{\fX}' \coloneqq u_1^* I_{\fX} + I_{\Grad(\fX)}$ and
    the left vertical isomorphism is induced by the isomorphism of orientations 
    \eqref{eq:ori_localize_assoc}.
\end{prop}

\begin{proof}
    Take a $\bG_m^2$-equivariant chart $(X, \hat{\mu}, \bar{q})$ of $\fX$.
    Using Lemma \ref{lem:matching_chart}, 
    it is enough to prove the statement after pulling back to $X_{\mat}^{\hat{\mu}}$.
    Therefore the claim follows from the commutativity of the diagram \eqref{eq:assoc_sm_ori_stack}, Proposition \ref{prop:zeta_assoc_space} and Lemma \ref{lem:chi_assoc}.
\end{proof}

We now show that the isomorphism $\zeta_{\fX, s, o}$ is compatible with the Thom--Sebastiani isomorphism:

\begin{prop}\label{prop:zeta_TS_stack}
    Let $(\fX, s, o_X)$ and $(\fY, t, o_Y)$ be oriented d-critical stacks.
    Assume that $\fX$ and $\fY$ are quasi-separated with affine stabilizers.
    Then the following diagram commutes:
     \[
    \xymatrix@C=70pt{
    {\varphi_{\Grad(\fX), u_{\fX}^{\star} s, u_{\fX}^{\star} o_{\fX}} \boxtimes \varphi_{\Grad(\fY), u_{\fY}^{\star} t, u_{\fY}^{\star} o_{\fY}} }  [I_{\fX} + I_{\fY}]
    \ar[r]^-{\zeta_{\fX, s, o_{\fX}} \boxtimes \zeta_{\fY, t, o_{\fY}}}_-{\cong}
    \ar[d]^-{\TS}_-{\cong}
    &  \Loc_{\fX}(\varphi_{\fX, s, o_{\fX}}) \boxtimes \Loc_{\fY}(\varphi_{\fY, t, o_{\fY}})
    \ar[d]^-{\TS}_-{\cong}
    \\
    {\varphi_{\Grad(\fX \times \fY), u_{\fX \times \fY}^{\star} (s \boxplus t), u_{\fX \times \fY}^{\star}  (o_{\fX} \boxtimes o_{\fY}})} [I_{\fX \times \fY}] \ar[r]_-{\cong}^-{\zeta_{\fX \times \fY, s \boxplus t, o_{\fX} \boxtimes o_{\fY}}}
    & \Loc_{\fX \times \fY}(\varphi_{\fX \times \fY, s \boxplus t, o_{\fX} \boxtimes o_{\fY}}).
    }
    \]
\end{prop}

\begin{proof}
    Take $\bG_m$-equivariant charts $(X, \mu, \bar{q})$ of $\fX$ and $(Y, \nu, \bar{r})$ of $\fY$.
    By the commutativity of the diagram \eqref{eq:ori_pull_product_compat}, Lemma \ref{lem:TS_DT_smooth_compatible} (which holds more generally for stacks by construction)  and Proposition \ref{prop:zeta_TS_space}, it is enough to prove the commutativity of the following diagram of functors:
    \[
    \xymatrix{
    {q_{\mu}^* \Loc_{\fX} ( - ) \boxtimes r_{\nu}^* \Loc_{\fY} ( - )} 
    \ar[r]^-{\rho_q \boxtimes \rho_r}
    \ar[d]^-{\cong}
    & {\Loc_{\mu} q^*( - )[2 d^{-}] \boxtimes \Loc_{\nu} r^*( - )[2 e^{-}] }
    \ar[d]^-{\cong} \\
    {(q_{\mu} \times  r_{\nu})^* \Loc_{\fX \times \fY} ( - \boxtimes - )}
    \ar[r]_-{\cong}^-{\rho_{q \times r}}
    & {\Loc_{\mu \times \nu} ( q \times r)^* ( - \boxtimes -) [2d^{-} + 2 e^{-}]}
    }
    \]
    where we set $d^{-} \coloneqq \rank_- T_{X / \fX}$ and $e^{-} = \rank_- T_{Y / \fY}$.
    This is obvious from the construction of the map $\rho_q$.
\end{proof}

Now let $(\fX, \omega_{\fX}, o)$ be an oriented $(-1)$-shifted symplectic stack and assume that $\fX^{\cl}$ is quasi-separated with affine stabilizers.
Then by Proposition \ref{prop:comparison_grad_dcrit}, an isomorphism of orientations \eqref{eq:compare_localize_ori} and equality \eqref{eq:Ind=vdim}, 
we obtain the following corollary, which summarizes Theorem \ref{thm:main_thm}, Proposition \ref{prop:zeta_assoc_stack} and Proposition \ref{prop:zeta_TS_stack} for the underlying d-critical stacks for $(-1)$-shifted symplectic stacks:

\begin{cor}[Integral isomorphism]\label{cor:Joyce_conj_attractor}
    Let $(\fX, \omega_{\fX}, o)$ be a $(-1)$-shifted symplectic stack such that $\fX^{\cl}$ is quasi-separated with affine stabilizers.
    We equip $\Grad(\fX)$ with the $(-1)$-shifted symplectic structure $u^{\star} \omega_{\fX}$ and the orientation $u^{\star} o$.
    Set $d \coloneqq \vdim \Filt(\fX)$.
    Then there exists a natural isomorphism
    \[
    \zeta_{\fX} = \zeta_{\fX, \omega_{\fX}, o} \colon 
    \varphi_{\Grad(\fX)} [d] \cong \gr_* \ev^! \varphi_{\fX}
    \]
    which we call the \defterm{integral isomorphism} with the following properties:
    \begin{itemize}[align=left]
    
        \item \textup{(Unitality) :}    
        Let $\iota \colon \fX \hookrightarrow \Grad(\fX)$ be a section of $u\colon \Grad(\fX)\rightarrow \fX$. Then 
        \begin{equation*}
        \iota^* \zeta_{\fX} = \id_{\varphi_{\fX}}.
        \end{equation*}
        
        \item \textup{(Associativity) : } 
        Set $d' \coloneqq \vdim \Filt(\Grad(\fX))$.
        The following diagram commutes:
        \begin{equation*}
        \begin{aligned}
        \xymatrix@C=40pt{
        {\varphi_{\Grad^2(\fX)_{\mat}}[d + d']}
        \ar[rr]_-{\cong}^-{\zeta_{\Grad(\fX)}}
        \ar[d]^-{\cong}
        & {}
        & {\gr_* \ev^! \varphi_{\Grad(\fX)} |_{\Grad^2(\fX)_{\mat}} [d] }
        \ar[d]_-{\cong}^-{\zeta_{\fX}} \\
        {\sigma^* \varphi_{\Grad^2(\fX)_{\mat}}[d + d'] }
        \ar[r]_-{\cong}^-{\sigma^* \zeta_{\Grad(\fX)}}
        & {\gr_* \ev^! \varphi_{\Grad(\fX)} |_{\Grad^2(\fX)_{\mat}} [d] }
        \ar[r]_-{\cong}^-{\zeta_{\fX}}
        & {\gr^{(2)}_* \ev^{(2), !} \varphi_{\fX} |_{\Grad^2(\fX)_{\mat}} }
        }
        \end{aligned}
        \end{equation*}
        where the left vertical map is constructed using the isomorphism of orientations 
        \eqref{eq:localize_ori_assoc_derived}.

        \item \textup{(Multiplicativity) : }
        Let $(\fY, \omega_{\fY}, o_{\fY})$ be oriented $(-1)$-shifted symplectic stack such that $\fY^{\cl}$ is quasi-separated with affine stabilizers.
        Set $e \coloneqq \vdim \Filt(\fY)$. 
        Then the following diagram commutes:
        \[
        \xymatrix@C=80pt{
        {\varphi_{\Grad(\fX)} \boxtimes \varphi_{\Grad(\fY)}[d + e]}
        \ar[r]_-{\cong}^-{\zeta_{\fX} \boxtimes \zeta_{\fY}}
        \ar[d]_-{\cong}^-{\TS}
        & {\gr_* \ev^! \varphi_{\fX} \boxtimes \gr_* \ev^! \varphi_{\fY}}
        \ar[d]_-{\cong}^-{\TS} \\
        {\varphi_{\Grad(\fX \times \fY)}[d + e]}
        \ar[r]_-{\cong}^-{\zeta_{\fX \times \fY}}
        & {\gr_* \ev^! (\varphi_{\fX \times \fY}).}
        }
        \]
    \end{itemize}
\end{cor}

\begin{proof}
The existence of the map $\zeta_{\fX}$ follows from Theorem \ref{thm:main_thm} combined with the isomorphism of orientations \eqref{eq:compare_localize_ori}.
The unitality follows from the commutativity of the diagram \eqref{eq:zeta_unital_stack}.
The associativity follows from Proposition \ref{prop:zeta_assoc_stack} together with the commutativity of the diagram \eqref{eq:compare_assoc_locallize}.
The multiplicativity follows from Proposition \ref{prop:zeta_TS_stack} together with the commutativity of the diagram \eqref{eq:ori_product_der_class_compati}.
\end{proof}

\begin{rmk}[Other sheaf-theoretic contexts]
    Though we have proved the statement only for perverse sheaves on the analytic topology,
    the same argument works for $\ell$-adic perverse sheaves over  $(-1)$-shifted symplectic stacks defined over algebraically closed fields of characteristic zero.
    Also, using the  six-functor  formalism for mixed Hodge modules recently constructed by     Tubach \cite{Tub2},
    one can upgrade Corollary \ref{cor:Joyce_conj_attractor} to monodromic mixed Hodge modules with the same proof by replacing shifts with half Lefschetz twists.
    
\end{rmk}

\subsection{Nontrivial actions}\label{ssec-nontriv-action}

In this section we explain how to extend Theorem \ref{thm:main_thm} to the case of a nontrivial $\bG_m$-action on $\fX$.

\begin{defin}
Let $\fX$ be an Artin stack with a $\bG_m$-action $\mu$. The \defterm{index function} is the function $I_\fC\colon \fX^{\mu}\rightarrow \bZ$ given by
\[
I^\mu_{\fX}(x) \coloneqq (\rank_+ H^0(\bT_{\fX, x}) - \rank_- H^0(\bT_{\fX, x})) - (\rank_+ H^{-1}(\bT_{\fX, x}) - \rank_- H^{-1}(\bT_{\fX, x}))
\]
for a point $x\in \fX^{\mu}$.
\end{defin}

If the $\bG_m$-action on $\fX$ is trivial, the previous definition reduces to Definition \ref{def:indexstack} and if $\fX$ is an algebraic space, it reduces to Definition \ref{def:indexalgspace}.

\begin{lem}\label{lem:indexproperties}
The index function satisfies the following properties:
\begin{enumerate}
    \item The restriction of $I_{[\fX/\bG_m]}\colon \Grad([\fX/\bG_m])\rightarrow \bZ$ along the morphism $\fX^\mu\rightarrow \Grad([\fX/\bG_m])$ from Remark \ref{rmk:nontrivialactionGrad} coincides with $I^\mu_\fX$.
    \item If $\fX$ admits a $\bG_m$-invariant d-critical structure, $I^\mu_\fX\colon \fX^\mu\rightarrow \bZ$ is locally constant.
    \item If $\fX$ is an underlying stack of a $(-1)$-shifted symplectic stack $\widehat{\fX}$ equipped with a $\bG_m$-action preserving the $(-1)$-shifted symplectic structure, then $I^\mu_\fX = \vdim \widehat{\fX}^{\mu, +}|_{\fX^{\mu}}$.
\end{enumerate}
\end{lem}
\begin{proof}$ $
\begin{enumerate}
    \item For $x\in\fX^{\mu}$ we have a long exact sequence
    \[0\longrightarrow H^{-1}(\bT_{\fX, x})\longrightarrow H^{-1}(\bT_{[\fX/\bG_m], x})\longrightarrow \C\longrightarrow H^0(\bT_{\fX, x})\longrightarrow H^0(\bT_{[\fX/\bG_m], x})\longrightarrow 0,\]
    where $\C$ is concentrated in weight $0$. Comparing the formulas for the indices $I^\mu_\fX(x)$ and $I_{[\fX/\bG_m]}(x)$ we get the result.
    \item If $\fX$ admits a $\bG_m$-invariant d-critical structure, $[\fX/\bG_m]$ admits a d-critical structure, so
    \[I_{[\fX/\bG_m]}\colon \Grad([\fX/\bG_m])\rightarrow \bZ\]
    is locally constant. Using part (1) this implies that $I^\mu_\fX\colon \fX^\mu\rightarrow \bZ$ is locally constant.
    \item By Proposition \ref{prop:closedexactdecomposition} the $(-1)$-shifted symplectic structure on $\widehat{\fX}$ is canonically exact. Therefore, by \cite[Section 3.3]{Par24} the $\bG_m$-action on $\widehat{\fX}$ admits a canonical moment map $M\colon \fX\rightarrow \bA^1[-1]$, so that $\widehat{\fX}/\!/\bG_m=[M^{-1}(0)/\bG_m]$ admits a $(-1)$-shifted symplectic structure whose underlying classical truncation gives the d-critical structure on $[\fX/\bG_m]$. Using \eqref{eq:Ind=vdim} we get
    \[I_{[\fX/\bG_m]} = \vdim \Filt(\widehat{\fX}/\!/\bG_m)|_{\Grad([\fX/\bG_m])}.\]
    Using Remark \ref{rmk:nontrivialactionGrad} and part (1) we get
    \[I^\mu_\fX = \vdim ([M^{-1}(0)^{\mu, +}/\bG_m])|_{\fX^\mu}.\]
    But since $\vdim([M^{-1}(0)^{\mu, +}/\bG_m])=\vdim(\widehat{\fX}^{\mu, +})$, we get the result.
\end{enumerate}
\end{proof}

Suppose $\fX$ is a quasi-separated Artin stack with affine stabilizers equipped with a $\bG_m$-invariant d-critical structure $s$ and orientation $o$. Consider the inclusion $u\colon \fX^\mu\rightarrow \fX$ of fixed points. We can localize the d-critical structure and orientation as follows:
\begin{itemize}
\item We claim that $u^\ast s$ is a d-critical structure on $\fX^\mu$. Indeed, consider the commutative diagram
\begin{equation}\label{eq:localizationcriticalorientation}
\xymatrix{
\fX^\mu \ar^{u}[rr] \ar[d] && \fX \ar[d] \\
\fX^\mu\times B\bG_m \ar[r] & \Grad([\fX/\bG_m]) \ar[r] & [\fX/\bG_m]
}
\end{equation}
By Proposition \ref{prop:grad_dcrit} the pullback of the d-critical structure on $[\fX/\bG_m]$ to $\Grad([\fX/\bG_m])$ is still a d-critical structure. The other maps in the diagram are smooth, which implies that $u^\ast s$ is also a d-critical structure.
\item It is explained in Section \ref{ssec:Localizing_ori} that an orientation of $[\fX/\bG_m]$ restricts to an orientation of $\Grad([\fX/\bG_m])$. Pulling it back along the other smooth maps in \eqref{eq:localizationcriticalorientation} we obtain an orientation $u^\ast o$ of $\fX^\mu$.
\end{itemize}

Recall the localization functor
\[\Loc_\mu\colon D^b_c(\fX)\longrightarrow D^b_c(\fX^\mu)\]
given by
\[\Loc_\mu = \gr_*\ev^!.\]
We are ready to state a version of Theorem \ref{thm:main_thm} for nontrivial $\bG_m$-actions.

\begin{thm}\label{thm:main_thm_nontrivial}
Let $(\fX, s, o)$ be an oriented d-critical stack equipped with a $\bG_m$-action preserving $s$ and $o$, where $\fX$ is quasi-separated and has affine stabilizers. Then there exists a natural isomorphism
\[\zeta_{\mu, s, o}\colon \varphi_{\fX^\mu, u^\ast s, u^\ast o}[I^\mu_\fX]\cong \Loc_\mu(\varphi_{\fX, s, o}).\]
\end{thm}
\begin{proof}
Consider morphisms of correspondences
\[
\xymatrix{
& \fX^{\mu, +} \ar_{\gr}[dl] \ar^{\ev}[dr] \ar^{f_2}[d] & \\
\fX^{\mu} \ar^{f_1}[d] & [\fX^{\mu, +}/\bG_m] \ar_{\overline{\gr}}[dl] \ar^{\overline{\ev}}[dr] \ar^{g_2}[d] & \fX \ar^{f_3}[d] \\
\fX^{\mu}\times B\bG_m \ar^{g_1}[d] & \Filt([\fX/\bG_m]) \ar_{\widetilde{\gr}}[dl] \ar^{\widetilde{\ev}}[dr] & [\fX/\bG_m] \ar@{=}[d] \\
\Grad([\fX/\bG_m]) && [\fX/\bG_m]
}
\]

Let $u\colon \fX^{\mu}\rightarrow \fX$ and $\widetilde{u}\colon \Grad([\fX/\bG_m])\rightarrow [\fX/\bG_m]$. We denote the d-critical structure on $[\fX/\bG_m]$ by $\widetilde{s}$ and the orientation by $\widetilde{o}$, so that $s = f_3^\ast \widetilde{s}$ and $o = f_3^\ast \widetilde{o}$. By construction, we have
\[u^\ast s = f_1^\ast g_1^\ast\widetilde{u}^\ast\widetilde{s},\qquad u^\ast o = f_1^\ast g_1^\ast \widetilde{u}^\ast \widetilde{o}.\]

By Theorem \ref{thm:main_thm} we obtain an isomorphism
\[\varphi_{{\Grad([\fX/\bG_m])}, \widetilde{u}^\ast \widetilde{s}, \widetilde{u}^\ast \widetilde{o}}[I_{[\fX/\bG_m]}]\cong \widetilde{\gr}_*\widetilde{\ev}^! \varphi_{[\fX/\bG_m], \widetilde{s}, \widetilde{o}}.\]

Applying $g_1^!$, using the base change isomorphism for the bottom-left square and using Lemma \ref{lem:indexproperties}(1), we obtain an isomorphism
\[\varphi_{\fX^\mu\times B\bG_m, g_1^\ast\widetilde{u}^\ast \widetilde{s}, g_1^\ast\widetilde{u}^\ast \widetilde{o}}[I^\mu_\fX]\cong \overline{\gr}_*\overline{\ev}^! \varphi_{[\fX/\bG_m], \widetilde{s}, \widetilde{o}}.\]

Applying $f_1^!$, using the base change isomorphism for the top-left square and using $f_1^! \varphi_{\fX^\mu\times B\bG_m}\cong \varphi_{\fX^\mu}[1]$ and $f_3^! \varphi_{[\fX/\bG_m]}\cong \varphi_{\fX}[1]$ we obtain the required isomorphism
\[\varphi_{\fX^\mu, u^\ast s, u^\ast o}[I^\mu_\fX]\cong \overline{\gr}_*\overline{\ev}^! \varphi_{\fX, s, o}.\]
\end{proof}

\subsection{Integral isomorphism for critical loci}\label{ssec:int_crit}

We will describe the isomorphism $\zeta_{\fX}$ explicitly in the case of global critical locus.
Let $\fU$ be a quasi-separated smooth Artin stack with affine stabilizers.
Let $f \colon \fU \to \bA^1$ be a regular function and define $\fX \coloneqq \Crit(f)$ and equip it with the $(-1)$-shifted symplectic structure defined in Example \ref{ex:critical} and equip it with the canonical orientation $o_{\fX}^{\can}$ in
Example \ref{ex:can_ori}.

\begin{lem}\label{lem:can_ori_vanishing_stack}
    There exists a natural isomorphism
    \begin{equation}\label{eq:DT=van_stack}
    \varphi_{\fX, \omega_{\fX}, o_{\fX}^{\can}} \cong \varphi_{f}(\bQ_{\fU}[n]).
    \end{equation}
    where we set $n \coloneqq \dim \fU$.
\end{lem}

\begin{proof}
    We let $(\fX^{\cl}, s)$ denote the underlying d-critical structure of $(\fX, \omega_{\fX})$ and  $o_{\fX^{\cl}}^{\can}$ denotes the underlying orientation.
    Take a smooth cover $\tilde{q} \colon U \to \fU$ and let $\tilde{f} \colon U \to \bA^1$ be the composition $U \to \fU \to \bA^1$ and set
    $X \coloneqq \Crit(\tilde{f})$.
    We let $q \colon X^{\cl} \to \fX^{\cl}$ be the canonical morphism.
    Then, arguing as the proof of \eqref{eq:can_ori_pull} and using \eqref{eq:can_ori_cl},
    we see that the natural isomorphism $\tilde{q}^* K_{\fU} \otimes K_{U / \fU} \cong K_{U}$ defines an isomorphism of orientations 
    \begin{equation}\label{eq:can_ori_smpullback_stack}
    o_{X^{\cl}}^{\can} \cong q^{\star} o_{\fX^{\cl}}^{\can}.
    \end{equation}
    We let $\fR = (\Crit(f)^{\cl}, \id, U, f, i)$ be the tautological \'etale d-critical chart for $(X^{\cl}, q^{\star} s)$. Then the $\bZZ$-local system $Q^{o}_{\fR}$ is canonically trivial.
    In particular, we have a natural isomorphism
    \[
    \varphi_{X^{\cl}, q^{\star}s, o_{X^{\cl}}^{\can}} \cong \varphi_{\tilde{f}}.
    \]
    Then we obtain a natural isomorphism
    \[
    q^{*} \varphi_{\fX^{\cl}, s, o_{\fX^{\cl}}^{\can}} \xrightarrow[\cong]{\Theta_{q, s, o_{\fX^{\cl}}^{\can }}}  \varphi_{X^{\cl}, q^{\star}s, o_{X^{\cl}}^{\can}}[-d] \cong \varphi_{\tilde{f}}[-d] \cong q^{*}\varphi_{f}(\bQ_{\fU}[n])
    \]
    where we set $d \coloneqq \dim q$ and the final isomorphism is constructed using \eqref{eq:van_natural}.
    It is clear from the commutativity of the diagram \eqref{eq:assoc_pull_van_stack}
    that it descends to the desired isomorphism.
\end{proof}

We keep the notation. 
By Lemma \ref{lem:crit_localize}, we have a natural equivalence of $(-1)$-shifted symplectic stacks
\[
\Grad(\fX) \cong \Crit (f |_{\Grad(\fU)}).
\]
By arguing as the construction of the map \eqref{eq:can_ori_loc_action}, we can construct a natural isomorphism
\begin{equation}\label{eq:can_ori_localize_stack}
    o_{\Grad(\fX)}^{\can} \cong u^{\star} o_{\fX}^{\can}.
\end{equation}
By arguing as Example \ref{ex:hyp_const}, we obtain a natural isomorphism
\begin{equation}\label{eq:const_loc_stack}
\Loc_{\fU}(\bQ_{\fU}) \cong \bQ_{\Grad(\fU)}[- 2n^-]
\end{equation}
where we set $n^- \coloneqq \rank T_{\fU}|_{\Grad(\fU)}^- $.
With these preparations, we can explicitly describe $\zeta_{\fX, \omega_{\fX}, o_{\fX}^{\can}}$:

\begin{prop}\label{prop:compati_CoHA_pre}
    Set $f_0 \coloneqq f |_{\Grad(\fU)}$, $n \coloneqq \dim \fU$ and $n_0 \coloneqq \dim \Grad(\fU)$.
    Then the following diagram commutes:
    \begin{equation}\label{eq:compati_CoHA_pre}
    \begin{aligned}
    \xymatrix@C=50pt{
    {\varphi_{\Grad(\fX), u^{\star} \omega_{\fX}, u^{\star} o^{\can}_{\fX}}[\vdim \Filt(\fX)]} 
    \ar[d]_-{\cong}^-{\eqref{eq:DT=van_stack}}
    \ar[rr]_-{\cong}^-{\zeta_{\fX, \omega_{\fX}, o_{\fX}^{\can}}}
    & {}
    & {\Loc_{\fX}(\varphi_{\fX, \omega_{\fX}, o_{\fX}^{\can}}) }
    \ar[d]_-{\cong}^-{\eqref{eq:DT=van_stack}} \\
    {\varphi_{f_0}(\bQ_{\Grad(\fU)})[\vdim \Filt(\fX) + n_0]}
    \ar[r]_-{\cong}^-{\eqref{eq:const_loc_stack}}
    & {\varphi_{f_0} (\Loc_{\fU}(\bQ_{\fU}))[n]}
    \ar[r]_-{\cong}^-{\eqref{eq:van_natural}}
    & {\Loc_{\fX}(\varphi_{f} (\bQ_{\fU}) ) [n] }
    }
    \end{aligned}
    \end{equation}
    where the left vertical map is constructed using the identification \eqref{eq:can_ori_localize_stack}.
    
\end{prop}

\begin{proof}
    Take a $\bG_m$-equivariant chart $(U, \mu, \bar{\tilde{q}})$ of $\fU$.
    It is enough to prove the claim after pulling back the diagram to $X^{\mu}$.
    Set $\tilde{f} \coloneqq f \circ q$, $\tilde{f}_0 \coloneqq f|_{U^{\mu}}$ and $X \coloneqq \Crit(f)$.
    Let $q \colon X^{\cl} \to \fX^{\cl}$ be the natural morphism.
    Arguing as the proof of the commutativity of the diagram \eqref{eq:can_ori_loc_pull_commutes},
    we see that the following diagram commutes:
    \[
    \xymatrix@C=50pt{
    {(q^{\star} o^{\can}_{\fX^{\cl}})^{\mu}}
    \ar[r]_-{\cong}^-{\eqref{eq:can_ori_smpullback_stack}}
    \ar[d]_-{\cong}^-{\eqref{eq:ori_sm_localize_stack}}
    & {(o^{\can}_{X^{\cl}})^{\mu}}
    \ar[r]_-{\cong}^-{\eqref{eq:ori_beta}}
    & o_{X^{\mu, \cl}}^{\can}
    \ar@{=}[d] \\
    {q_{\mu}^{\star} u^{\star} o_{\fX^{\cl}}^{\can}}
    \ar[r]_-{\cong}^-{\eqref{eq:can_ori_localize_stack}}
    & q_{\mu}^{\star}  o_{\Grad(\fX)^{\cl}}^{\can}
    \ar[r]_-{\cong}^-{\eqref{eq:can_ori_smpullback_stack}}
    & {o^{\can}_{X^{\mu, \cl}}}.
    }
    \]
    Using the commutativity of this diagram and the fact that
     each map appearing in the diagram \eqref{eq:compati_CoHA_pre} commutes with smooth pullback (which follows immediately from the construction),
     it is enough to prove the commutativity of the following diagram:
     \[
     \xymatrix@C=50pt{
     {\varphi_{X^{\mu, \cl}, \tilde{s}^{\mu}, (o_{X^{\cl}}^{\can})^{\mu}}[I^{\mu}_{X^{\cl}}]}
     \ar[rr]_-{\cong}^-{\zeta_{\mu, \tilde{s}, o_{X^{\cl}}^{\can}}}
     \ar[d]_-{\cong}
     & {}
     & {\Loc_{\mu} (\varphi_{X^{\cl},  \tilde{s}, o_{X^{\cl}}^{\can}}) }
     \ar[d]_-{\cong} 
     \\
     {\varphi_{\tilde{f}_0} (\bQ_{U^{\mu}})[I^{\mu}_{X^{\cl}} + m_0] }
     \ar[r]_-{\cong}^-{\eqref{eq:hyp_const}}
     & {\varphi_{\tilde{f}_0} ( \Loc_{\mu}(\bQ_{U}))[m]}
     \ar[r]_-{\cong}^-{\eqref{eq:van_natural}}
     & {\Loc_{\mu}(\varphi_{\tilde{f}} (\bQ_{U}) )[m]}
     }
     \]
     where we set $m \coloneqq \dim U$, $m_0 \coloneqq \dim U^{\mu}$, 
     $\tilde{s}$ is the canonical d-critical structure on $X^{\cl}$ defined using the critical locus description and
     the left vertical map is constructed using the natural isomorphism 
     $(o_{X^{\cl}}^{\can})^{\mu} \cong o_{X^{\mu, \cl}}^{\can}$.
     This is immediate from the construction of $\zeta_{\mu, \tilde{s}, o_{X^{\cl}}^{\can}}$ in Theorem \ref{thm:hyp_DT}.
     
\end{proof}

\subsection{Ungraded cases}\label{ssec-sign-dep}

Throughout the paper, we have been working with orientations with gradings, even though the DT perverse sheaves themselves depend only on the underlying ungraded orientations.
In this subsection, we discuss how to obtain results for d-critical stacks with ungraded orientations from the graded cases.

Let $\mathrm{pt}$ be a point with the trivial d-critical structure $s^{\mathrm{triv}}$.
We let $o_{\mathrm{pt}}^{\mathrm{odd}}$ and $o_{\mathrm{pt}}^{\mathrm{even}}$ be orientations on $\mathrm{pt}$ with the odd and even gradings respectively.
Consider an isomorphism of orientations \footnote{One may assume that the role of $\sqrt{-1}$ and $- \sqrt{-1}$ are interchangeable here. However, we have already implicitly chosen $\sqrt{-1}$ when constructing the isomorphism $c_y$ in \eqref{eq:quad_trivialize} to orient the complex plane.}
\[
o_{\mathrm{pt}}^{\mathrm{odd}} \boxtimes o_{\mathrm{pt}}^{\mathrm{odd}} \cong o_{\mathrm{pt}}^{\mathrm{even}}, \quad 1 \boxtimes 1 \mapsto - \sqrt{-1}.
\]
which, together with the Thom--Sebastiani isomorphism, induces an isomorphism
\begin{equation}
\varphi_{\mathrm{pt}, s^{\mathrm{triv}}, o_{\mathrm{pt}}^{\mathrm{odd}}} \boxtimes \varphi_{\mathrm{pt}, s^{\mathrm{triv}}, o_{\mathrm{pt}}^{\mathrm{odd}}} \cong \varphi_{\mathrm{pt}, s^{\mathrm{triv}}, o_{\mathrm{pt}}^{\mathrm{even}}} \cong \mathbb{Q}_{\mathrm{pt}}. \label{eq-square-triv}
\end{equation}
In particular, there is an isomorphism
\[
 a \colon \varphi_{\mathrm{pt}, s^{\mathrm{triv}}, o_{\mathrm{pt}}^{\mathrm{odd}}} \cong \mathbb{Q}_{\mathrm{pt}}.
\]
It is shown in \cite[Lemma 3.4]{kin21} one can choose $a$ so that composite 
\begin{equation}
\mathbb{Q}_{\mathrm{pt}} \cong 
\mathbb{Q}_{\mathrm{pt}} \boxtimes \mathbb{Q}_{\mathrm{pt}} \xrightarrow[\cong]{a^{-1} \boxtimes a^{-1}}\varphi_{\mathrm{pt}, s^{\mathrm{triv}}, o_{\mathrm{pt}}^{\mathrm{odd}}} \boxtimes \varphi_{\mathrm{pt}, s^{\mathrm{triv}}, o_{\mathrm{pt}}^{\mathrm{odd}}} \xrightarrow[\cong]{\eqref{eq-square-triv}}  \mathbb{Q}_{\mathrm{pt}}
\end{equation}
is the identity map, and we will fix such an $a$ throughout the subsection.

Let $(\fX, s)$ be a d-critical stack with $\fX$ quasi-separated and having affine stabilizers.
We let $o^{\mathrm{ungr}}$ be an ungraded orientation, i.e., a choice of an ungraded line bundle $\cL$ on $\fX^{\mathrm{red}}$ together with an isomorphism of ungraded line bundles $o^{\mathrm{ungr}} \colon \mathcal{L}^{\otimes 2} \cong K_{\fX, s}$.
We let $o^{\mathrm{even}}$ and $o^{\mathrm{odd}}$ be the orientations on $(\fX, s)$ with odd and even gradings having $o^{\mathrm{ungr}}$ as the underlying ungraded orientations.
Then we have a natural isomorphism
\begin{equation}\label{eq-even-odd-exchange}
\varphi_{\fX, s, o^{\mathrm{odd}}} \xrightarrow[\cong]{\mathrm{TS}^{-1}} \varphi_{\fX, s, o^{\mathrm{even}}} \boxtimes \varphi_{\mathrm{pt}, s^{\mathrm{triv}}, o_{\mathrm{pt}}^{\mathrm{odd}}} \xrightarrow[\cong ]{\id \boxtimes a^{-1}}
\varphi_{\fX, s, o^{\mathrm{even}}}. 
\end{equation}
The associativity of the Thom--Sebastiani isomorphism \eqref{eq:TS_assoc_stack} implies that this is inverse to the following composite:
\begin{equation}\label{eq-even-odd-inverse}
\varphi_{\fX, s, o^{\mathrm{even}}} \xrightarrow[\cong]{\mathrm{TS}^{-1}} \varphi_{\fX, s, o^{\mathrm{odd}}} \boxtimes \varphi_{\mathrm{pt}, s^{\mathrm{triv}}, o_{\mathrm{pt}}^{\mathrm{odd}}} \xrightarrow[\cong ]{\id \boxtimes a^{-1}}
\varphi_{\fX, s, o^{\mathrm{odd}}}. 
\end{equation}
We will set
\[
\varphi_{\fX, s, o^{\mathrm{ungr}}} \coloneqq \varphi_{\fX, s, o^{\mathrm{even}}}.
\]
By Theorem \ref{thm:main_thm} and \eqref{eq-even-odd-exchange}, we have a natural isomorphism
\begin{equation}
    \zeta_{\fX, s, o^{\mathrm{ungr}}} \colon  \varphi_{\Grad(\fX), u^{\star} s, u^{\star} o^{\mathrm{ungr}}}[I_{\fX}] \cong 
    \Loc_{\fX} (\varphi_{\fX, s, o^{\mathrm{ungr}}}), \label{eq-ungr-zeta}
\end{equation}
where $u^{\star} o^{\mathrm{ungr}}$ is defined as the underlying ungraded line bundle of $u^{\star} o^{\mathrm{even}}$.
For an orientation $o$ on $\fX$ with $o^{\mathrm{ungr}}$ the underlying ungraded orientation, we construct an isomorphism
\begin{equation}
 (u^{\star } o )^{\mathrm{ungr}} \cong u^{\star} o^{\mathrm{ungr}}\label{eq-ungr-local-compat}
\end{equation}
with the underlying map of line bundles given by $\sqrt{-1}^{I_{\fX} \cdot |o|}$.
Using \eqref{eq-even-odd-exchange}, we define an isomorphism 
\begin{equation}
\varphi_{\fX, s, o} \cong \varphi_{\fX, s, o^{\mathrm{ungr}}}. \label{eq-phi-ungr}
\end{equation}
Then one easily sees that the following diagram commutes:
\begin{equation}\label{eq-zeta-ungr-compati}
\begin{tikzcd}
	{\varphi_{\Grad(\fX), u^\star s, u^{\star} o}[I_{\fX}]} &[20pt] &[20pt] {\mathrm{Loc}_{\fX}(\varphi_{\fX, s, o})} \\
	{\varphi_{\Grad(\fX), u^\star s, (u^{\star} o)^{\mathrm{ungr}}}[I_{\fX}]} & {\varphi_{\Grad(\fX), u^\star s, u^{\star} o^{\mathrm{ungr}}}} & {\mathrm{Loc}_{\fX}( \varphi_{\fX, s, o^{\mathrm{ungr}}})}
	\arrow["{{\zeta_{\fX, s, o}}}", from=1-1, to=1-3]
	\arrow["{\eqref{eq-phi-ungr}}"', from=1-1, to=2-1]
	\arrow["{\eqref{eq-phi-ungr}}", from=1-3, to=2-3]
	\arrow["{{\eqref{eq-ungr-local-compat}}}"', from=2-1, to=2-2]
	\arrow["\zeta_{\fX, s,o^{\mathrm{ungr}}}"', from=2-2, to=2-3]
\end{tikzcd}
\end{equation}
We construct an isomorphism
\begin{equation}
u_1^{\star} u^{\star} o_{\mathrm{ungr}} \cong u_2^{\star} u^{\star} o_{\mathrm{ungr}}.
\label{eq:ori_localize_assoc_ungr}
\end{equation}
so that the following diagram commutes:
\[\begin{tikzcd}
	{u_1^{\star} u^{\star}o^{\mathrm{ungr}}} & {u_1^{\star} (u^{\star}o^{\mathrm{even}})^{\mathrm{ungr}}} & {(u_1^{\star} u^{\star}o^{\mathrm{even}})^{\mathrm{ungr}}} \\
	{u_2^{\star} u^{\star}o^{\mathrm{ungr}}} & {u_2^{\star} (u^{\star}o^{\mathrm{even}})^{\mathrm{ungr}}} & {(u_2^{\star} u^{\star}o^{\mathrm{even}})^{\mathrm{ungr}}.}
	\arrow["\eqref{eq-ungr-local-compat}", from=1-1, to=1-2]
	\arrow["\eqref{eq:ori_localize_assoc_ungr}"',from=1-1, to=2-1]
	\arrow["\eqref{eq-ungr-local-compat}",from=1-2, to=1-3]
	\arrow["\eqref{eq:ori_localize_assoc}", from=1-3, to=2-3]
	\arrow["\eqref{eq-ungr-local-compat}",from=2-1, to=2-2]
	\arrow["\eqref{eq-ungr-local-compat}",from=2-2, to=2-3]
\end{tikzcd}\]
Then it follows from the commutativity of \eqref{eq-zeta-ungr-compati} that the associativity statement analogous to Proposition \ref{prop:zeta_assoc_stack} holds for the isomorphisms \eqref{eq-ungr-zeta}.

\section{Applications to the moduli of objects in 3CY categories}

In the previous sections, we have established a hyperbolic localization statement for d-critical stacks. In this section, we explain how to use this to define cohomological Hall algebras for $3$Calabi--Yau categories.

In \S \ref{ssec:CoHA}, we will construct the cohomological Hall algebras for a moduli-like stack introduced in \S \ref{ssec:Moduli-like}.
In \S \ref{ssec:moduli_dg} --- \S \ref{ssec:moduli_attractor}, we will show that the setup in \S \ref{ssec:CoHA} can be used to construct the cohomological Hall algebras for $3$-Calabi--Yau categories with  strong orientation data.
In \S \ref{ssec:CYcompori}, we will give examples of $3$-Calabi--Yau categories with strong orientation data using deformed Calabi--Yau completion.
In \S \ref{ssec:comparisonKS}, we will show that our construction recovers Kontsevich--Soibelman's cohomological Hall algebras for quivers with potentials.

\subsection{Moduli-like stacks}\label{ssec:Moduli-like}

In this section we introduce the geometric setting necessary to define the cohomological Hall algebra. It will be given by a $(-1)$-shifted symplectic stack with a $B\bG_m$-action and an addition map satisfying certain conditions. For a $B\bG_m$-action on a stack $\fX$, let $\Psi\colon B\bG_m\times \fX\rightarrow \fX$ be the corresponding action map. For $n\in\bZ$ we obtain maps
\[\iota_n\colon \fX\longrightarrow \Grad(\fX)\]
adjoint to
\[B\bG_m\times \fX\xrightarrow{(-)^n\times \id} B\bG_m\times \fX\xrightarrow{\Psi}\fX,\]
where the first map is induced by the $n$-th power map on $\bG_m$. For instance, $\iota_0=\iota\colon \fX\rightarrow \Grad(\fX)$ is induced by the projection $B\bG_m\rightarrow \pt$. Similarly, for $(n,m)\in\bZ$ we obtain maps
\[\iota_{n, m}\colon \fX\rightarrow \Grad^2(\fX)\]
adjoint to
\[B\bG_m\times B\bG_m\times \fX\xrightarrow{(-)^n\times (-)^m\times \id}B\bG_m\times B\bG_m\times \fX\xrightarrow{\id\times \Psi}B\bG_m\times \fX\xrightarrow{\Psi} \fX.\]

\begin{defin}\label{def:modulilike}
A \defterm{moduli-like stack} is a derived stack $\fM$ which is locally geometric and locally of finite presentation, together with the following data:
\begin{enumerate}
    \item A $B\bG_m$-action on $\fM$.
    \item An $\bE_\infty$-structure on $\fM$ commuting with the $B\bG_m$-action. We denote by $\Phi_n\colon \fM^n\rightarrow \fM$ the corresponding addition maps and by $\eta\colon \pt\rightarrow \fM$ the unit map.
    \item A $(-1)$-shifted symplectic structure $\omega$ on $\fM$. We require it to be additive, in the sense that there is a homotopy $\Phi_2^\ast\omega\sim \omega\boxplus\omega$ of $(-1)$-shifted closed two-forms on $\fM\times \fM$ and a homotopy of homotopies of $(-1)$-shifted closed two-forms witnessing commutativity of the pullbacks along
    \[
    \xymatrix@C=1cm{
    \fM\times\fM\times \fM \ar^-{\id\times \Phi_2}[r] \ar^{\Phi_2\times \id}[d] & \fM\times \fM \ar^{\Phi_2}[d] \\
    \fM\times \fM \ar^{\Phi_2}[r] & \fM
    }
    \]
\end{enumerate}
These data are required to satisfy the following conditions:
\begin{enumerate}
    \item The composites
    \[\tilde{\Phi}_2\colon \fM^2\xrightarrow{\iota_0\times \iota_1}\Grad(\fM^2)\xrightarrow{\Grad(\Phi_2)}\Grad(\fM)\]
    and
    \[\tilde{\Phi}_3\colon \fM^3\xrightarrow{\iota_{0,0}\times \iota_{1,0}\times \iota_{1,1}} \Grad^2(\fM^3)\xrightarrow{\Grad^2(\Phi_3)}\Grad^2(\fM)\]
    are open immersions.
    \item The map $\eta\colon \pt\rightarrow \fM$ is an open immersion.
    \item The image of $\tilde{\Phi}_3$ lands in the matching locus of $\Grad^2(\fM)$.
    \item The base change of $\Filt(\fM^{\cl})\rightarrow \Grad(\fM^{\cl})$ along $\iota_n\colon \fM^{\cl}\rightarrow \Grad(\fM^{\cl})$ is an isomorphism.
\end{enumerate}
\end{defin}

\begin{rmk}
As we will explain in Proposition \ref{prop:modulilikemoduli}, a natural example of a moduli-like stack is the derived moduli stack of objects in a $3$-Calabi--Yau category.
\end{rmk}

\begin{rmk}
A structure similar to that of a moduli-like stack was introduced in \cite[Assumption 3.1]{JoyceVertex}.
\end{rmk}

Now let $\fM$ be a moduli-like stack. By Corollary \ref{cor:lagattractorcorrespondence} we have a $(-1)$-shifted Lagrangian structure on the attractor correspondence $\Grad(\fM)\xleftarrow{\gr}\Filt(\fM)\xrightarrow{\ev}\fM$. Its pullback along $\tilde{\Phi}_2\colon \fM^2\rightarrow \Grad(\fM)$ gives a $(-1)$-shifted Lagrangian correspondence
\begin{equation}\label{eq:M2filt}
\begin{aligned}
\xymatrix{
& \fM^{2-\filt} \ar_{\gr}[dl] \ar^{\ev}[dr] & \\
\fM^2 && \fM
}
\end{aligned}
\end{equation}
together with a splitting $\fM^2\rightarrow \fM^{2-\filt}$ of the map $\gr\colon \fM^{2-\filt}\rightarrow \fM^2$. Let
\[\chi\colon \fM\times \fM\rightarrow \bZ\]
be the locally-constant function obtained by restricting the index function $I_{\fM}\colon \Grad(\fM)\rightarrow \bZ$ along $\tilde{\Phi}_2$.

Using Lemma \ref{lem:Lagrangiancorrespondenceretract} the above Lagrangian correspondence gives an isomorphism
\begin{equation}\label{eq:orientationmultiplicative}
\Phi_2^*\ori_{\fM}\cong \ori_{\fM}\boxtimes \ori_{\fM}
\end{equation}
of graded $\mu_2$-gerbes on $\fM^2$. In particular, if $o$ is an orientation of $\fM$, there is a canonical orientation $\Phi_2^\star o$ of $\fM^2$.

Moreover, by Corollary \ref{cor:matching_composite} we obtain that there is an equivalence between the composite Lagrangian correspondences
\[
\xymatrix{
& \fM\times \fM^{2-\filt} \ar_{\id\times\gr}[dl] \ar^{\id\times\ev}[dr] && \fM^{2-\filt} \ar_{\gr}[dl] \ar^{\ev}[dr] & \\
\fM^3 && \fM^2 && \fM
}
\]
and
\[
\xymatrix{
& \fM^{2-\filt}\times \fM \ar_{\gr\times \id}[dl] \ar^{\ev\times \id}[dr] && \fM^{2-\filt} \ar_{\gr}[dl] \ar^{\ev}[dr] & \\
\fM^3 && \fM^2 && \fM.
}
\]
In particular, we get that $\chi$ satisfies a cocycle relation:
\begin{equation}\label{eq:chicocycle}
(\Phi_2^\star\boxtimes \id) \chi + \pr^\star_{12}\chi = (\id\boxtimes \Phi_2^\star)\chi + \pr^\star_{23}\chi.
\end{equation}

We denote the composite Lagrangian correspondence by
\begin{equation}\label{eq:M3filt}
\begin{aligned}
\xymatrix{
& \fM^{3-\filt} \ar_{\gr^{(2)}}[dl] \ar^{\ev^{(2)}}[dr] & \\
\fM^3 && \fM.
}
\end{aligned}
\end{equation}

Using \eqref{eq:ori_assoc_general} we obtain 2-isomorphisms witnessing commutativity of the diagram
\begin{equation}\label{eq:orientationassociative}
\begin{aligned}
\xymatrix{
(\Phi_2^*\boxtimes \id)\Phi_2^*\ori_{\fM} \ar^{\eqref{eq:orientationmultiplicative}}[d] \ar^{\sim}[rr] && (\id\boxtimes \Phi_2^*)\Phi_2^*\ori_{\fM} \ar^{\eqref{eq:orientationmultiplicative}}[d] \\
(\Phi_2^*\ori_{\fM})\boxtimes \ori_{\fM} \ar_{\eqref{eq:orientationmultiplicative}}[dr] && \ori_{\fM}\boxtimes (\Phi_2^*\ori_{\fM}) \ar^{\eqref{eq:orientationmultiplicative}}[dl] \\
& \ori_{\fM}\boxtimes \ori_{\fM} \boxtimes \ori_{\fM}.
}
\end{aligned}
\end{equation}

In particular, if $o$ is an orientation of $\fM$, there is an isomorphism of orientations
\begin{equation}
\label{eq:moduli_ori_assoc}
(\Phi_2^\star\boxtimes \id)\Phi_2^\star o\cong (\id\boxtimes \Phi_2^\star)\Phi_2^\star o
\end{equation}
of $\fM^3$.

\begin{defin}\label{defin:CYorientationdata}
Let $\fM$ be a moduli-like stack. An \defterm{orientation data on $\fM$} is given by the following data\footnote{Our definition of the orientation data slightly differs from the one in \cite[Definition 4.2]{ju21}, where they defined the orientation data to be an isomorphism class of $o$ such that the isomorphisms $\lambda$ and $\tau$ exist.}:
\begin{enumerate}
    \item An orientation $o$ of $\fM$.
    \item An isomorphism of orientations $\lambda\colon \Phi_2^\star o\cong o\boxtimes o$.
    \item An isomorphism of orientations $\tau\colon \eta^\star o\cong o_{\triv}$, where $o_{\triv}$ is the trivial orientation of $\pt$.
\end{enumerate}
An orientation data is called a \defterm{strong orientation data} if the diagrams
\[
\xymatrix{
(\Phi_2^\star\boxtimes \id)\Phi_2^\star o \ar_{\sim}^{\eqref{eq:moduli_ori_assoc}}[rr] \ar^{(\Phi_2^\star\boxtimes \id)\lambda}[d] && (\id\boxtimes \Phi_2^\star) \Phi_2^\star o \ar^{(\id\boxtimes \Phi_2^\star)\lambda}[d] \\
(\Phi_2^\star o)\boxtimes o \ar_{\lambda\boxtimes \id}[dr] && o\boxtimes (\Phi_2^\star o) \ar^{\id\boxtimes \lambda}[dl] \\
& o\boxtimes o\boxtimes o, &
}
\]
and
\[
\xymatrix@C=1cm{
(\eta^\star\boxtimes \id)\Phi_2^\star o\ar^-{(\eta^\star\boxtimes \id)\lambda}[r] \ar_{\sim}[d] & \eta^\star o\boxtimes o \ar^{\tau\boxtimes \id}[d] \\
o \ar^{\sim}[r] & o_{\triv}\boxtimes o
}\qquad \xymatrix@C=1cm{
(\id\boxtimes \eta^\star)\Phi_2^\star o\ar^-{(\id\boxtimes \eta^\star)\lambda}[r] \ar_{\sim}[d] & o\boxtimes \eta^\star o \ar^{\id\boxtimes \tau}[d] \\
o \ar^{\sim}[r] & o\boxtimes o_{\triv}
}
\]
commute.
\end{defin}

\begin{rmk}
    Let $X$ be a smooth projective Calabi--Yau threefold and $\fM_{\Perf(X)}$ the moduli stack of perfect complexes on $X$.
    In \cite{ju21}, Joyce and Upmeier constructed an orientation data for $\fM_{\Perf(X)}$.
    However, it is not known whether their orientation data upgrades to a strong orientation data.
\end{rmk}

\subsection{Cohomological Hall algebra}\label{ssec:CoHA}

As an application of Corollary \ref{cor:Joyce_conj_attractor},
we will construct the cohomological Hall algebra for moduli-like stacks $\fM$ with a strong orientation data.

Set $\Gamma \coloneqq \pi_0(\fM_{})$ and write the connected component decomposition by
\[
\fM = \coprod_{\alpha \in \Gamma} \fM_{\alpha}.
\]

Since $\gr\colon \Filt(\fM)\rightarrow \Grad(\fM)$ is an $\bA^1$-deformation retract by \cite[Lemma 1.3.8]{hlp14}, so is $\fM^{n-\filt}\rightarrow \fM^n$ for $n=2, 3$. Therefore, $\pi_0(\fM^{n-\filt})\cong \pi_0(\fM^n)\cong \pi_0(\fM)^n$, where the first isomorphism follows from \cite[Lemma 1.3.7]{hlp14} and the second isomorphism follows from the assumption that $\fM$ is locally of finite presentation. Let
\[\fM^{n-\filt} = \coprod_{\alpha_1, \cdots, \alpha_n \in \Gamma} \fM^{n-\filt}_{\alpha_1, \dots, \alpha_n}\]
be the connected component decomposition.

The $\bE_\infty$-structure on $\fM$ induces an abelian monoid structure on $\Gamma$.
The Lagrangian correspondences \eqref{eq:M2filt} and \eqref{eq:M3filt} restrict to the following correspondences:
\[
\xymatrix{
{}
& {\fM_{\alpha_1, \alpha_2}^{2-\filt}}
\ar[ld]_-{\gr}
\ar[rd]^-{\ev}
& {} \\
{\fM_{\alpha_1} \times \fM_{\alpha_2}}
& {}
& {\fM_{\alpha_1 + \alpha_2}}
}
\qquad
\xymatrix{
{}
& {\fM_{\alpha_1, \alpha_2, \alpha_3}^{3-\filt}}
\ar[ld]_-{\gr^{(2)}}
\ar[rd]^-{\ev^{(2)}}
& {} \\
{\fM_{\alpha_1} \times \fM_{\alpha_2} \times \fM_{\alpha_3}}
& {}
& {\fM_{\alpha_1 + \alpha_2 + \alpha_3}.}
}
\]

By \eqref{eq:chicocycle} the index function $I_\fM\colon \Grad(\fM)\rightarrow \bZ$ restricts under the inclusion $\tilde{\Phi}_2\colon \fM^2\rightarrow \Grad(\fM)$ to a 2-cocycle $\chi\colon \Gamma\times \Gamma\rightarrow \bZ$.

\begin{thm}\label{thm:integral_isomorphism_moduli}
Let $\fM$ be a moduli-like stack with an orientation data such that $\fM^{\cl}$ is quasi-separated and has affine stabilizers. Then there exists a natural isomorphism
\[
\zeta_{\alpha_1, \alpha_2} \colon \varphi_{\fM_{\alpha_1}}  \boxtimes \varphi_{\fM_{\alpha_2}} [\chi(\alpha_1, \alpha_2)] \cong \gr_{ *} \ev^! \varphi_{\fM_{\alpha_1 + \alpha_2}} 
\] with the following properties:
\begin{itemize}
    \item \textup{(Unitality)} Let $0 \in \Gamma$ be the unit element. 
    Then the maps
    \begin{align*}
        \varphi_{\fM_{\alpha}} &\cong \varphi_{\fM_{\alpha}} \boxtimes \bQ \cong  \varphi_{\fM_{\alpha}} \boxtimes \eta^*\varphi_{\fM_0} 
        \xrightarrow[\cong]{\zeta_{\alpha, 0}} \varphi_{\fM_{\alpha}}, \\
        \varphi_{\fM_{\alpha}} &\cong \bQ \boxtimes \varphi_{\fM_{\alpha}} \cong \eta^*\varphi_{\fM_{0}} \boxtimes \varphi_{\fM_{\alpha}} 
        \xrightarrow[\cong]{\zeta_{0, \alpha}} \varphi_{\fM_{\alpha}}
    \end{align*}
    are the identities.
    
    \item \textup{(Associativity)} Assume further that our orientation data is a strong orientation data. Then the following diagram commutes:
    \[
    \xymatrix@C=50pt{
    {\varphi_{\fM_{\alpha_1}}} \boxtimes {\varphi_{\fM_{\alpha_2}}} \boxtimes {\varphi_{\fM_{\alpha_3}}} [\chi(\alpha_1, \alpha_2, \alpha_3)]
    \ar[r]_-{\cong}^-{\zeta_{\alpha_1, \alpha_2} \boxtimes \id}
    \ar[d]_-{\cong}^-{\id  \boxtimes \zeta_{\alpha_2, \alpha_3}}
    & { (\gr_{ *} \ev^! \varphi_{\fM_{\alpha_1 + \alpha_2}}) \boxtimes \varphi_{\fM_{\alpha_3}}}[{\chi(\alpha_1 + \alpha_2, \alpha_3)}]
    \ar[d]_-{\cong}^-{\zeta_{\alpha_{1} + \alpha_2, \alpha_3}} \\
    {\varphi_{\fM_{\alpha_1}} \boxtimes 
    (\gr_{*} \ev_{}^! \varphi_{\fM_{\alpha_2 + \alpha_3}})}[{\chi(\alpha_1, \alpha_2 +  \alpha_3)}]
    \ar[r]_-{\cong}^-{\zeta_{\alpha_1, \alpha_2 + \alpha_3}}
    & {\gr^{(2)}_{*} \ev^{(2), !}_{} \varphi_{\fM_{\alpha_1 + \alpha_2 + \alpha_3}}},
    }
    \]
    where we set $\chi(\alpha_1, \alpha_2, \alpha_3) \coloneqq \chi(\alpha_1, \alpha_2) + \chi(\alpha_1 + \alpha_2, \alpha_3) = \chi(\alpha_2, \alpha_3) + \chi(\alpha_1, \alpha_2 + \alpha_3)$.
\end{itemize}
\end{thm}

\begin{proof}
    Consider the inclusion
    $\tilde{\Phi}_2 \colon \fM_{\alpha_1} \times \fM_{\alpha_2} \to \Grad(\fM)$.
    By assumption, there exists a natural equivalence of symplectic structures
    $\omega_{\fM_{\alpha_1}} \boxplus \omega_{\fM_{\alpha_2}} \sim \tilde{\Phi}_2^{\star} u^{\star} \omega_{\fM}$.
    Also, by the definition of the orientation data and Remark \ref{rmk:ori_compati_with_general},
    we have a natural isomorphism of orientations 
    $o_{\alpha_1 } \boxtimes o_{\alpha_2} \cong \tilde{\Phi}_2^{\star} u^{\star} o$,
    where $o$ denotes the orientation for $\fM$ and $o_{\alpha}$ is its restriction to $\fM_{\alpha}$.
    Therefore, by using Corollary \ref{cor:TS_derived} and Corollary \ref{cor:Joyce_conj_attractor}, we obtain an isomorphism $\zeta_{\alpha_1, \alpha_2}$ by the composition
    \[
    \varphi_{\fM_{\alpha_1}, o_{\alpha_1}}  \boxtimes \varphi_{\fM_{\alpha_2}, o_{\alpha_2}} [\chi(\alpha_1, \alpha_2)]
    \xrightarrow[\cong]{\TS} 
    \varphi_{\fM_{\alpha_1} \times \fM_{\alpha_2} \tilde{\Phi}_2^{\star} u^{\star} o}  [\chi(\alpha_1, \alpha_2)]
    \xrightarrow[\cong]{\zeta_{\fM_{\alpha_1  + \alpha_2}}} 
    \gr_! \ev^* \varphi_{\fM_{\alpha_1+\alpha_2}, o}.
    \]

    The unitality of $\zeta_{\alpha_1, \alpha_2}$ is an immediate consequence of the unitality property of $\zeta_{\fX, \omega_{\fX}, o}$ proved in Corollary \ref{cor:Joyce_conj_attractor}.

    For the associativity of $\zeta_{\alpha_1, \alpha_2}$, 
    consider the inclusion $\tilde{\Phi}_3 \colon \fM_{\alpha_1}\times \fM_{\alpha_2}\times \fM_{\alpha_3} \to \Grad^2(\fM)$ and let $\tilde{\Phi}_3'\colon \fM_{\alpha_1}\times \fM_{\alpha_2}\times \fM_{\alpha_3} \to \Grad^2(\fM)$ its post-composition with the swapping isomorphism $\sigma\colon \Grad^2(\fM)\rightarrow \Grad^2(\fM)$. By the associativity property of $\zeta_{\fX, \omega_{\fX}, o}$ proved in Corollary \ref{cor:Joyce_conj_attractor}, we see that the following diagram commutes:
    \begin{equation}\label{eq:CoHa_assoc_without_orientationdata}
    \begin{aligned}
    \xymatrix{
    {\varphi_{\fM_{\alpha_1} \times \fM_{\alpha_2} \times \fM_{\alpha_3}, \tilde{\Phi}_3^{\star} u_1^{\star} u^{\star} o  }[n_{123}]}
    \ar[d]_-{\cong}
    \ar[r]_-{\cong}^-{\zeta_{\fM_{\alpha_1 + \alpha_2} \times \fM_{\alpha_3}}}
    & {(\gr \times \id)_* (\ev \times  \id)^! \varphi_{\fM_{\alpha_1 + \alpha_2} \times \fM_{\alpha_3}, \tilde{\Phi}_2^\star u^\star o \boxtimes o_{\alpha_3}}  
    [n_{12,3}] }
    \ar[dd]_-{\cong}^-{\zeta_{\fM_{\alpha_1 + \alpha_2 + \alpha_3}}} \\
    {\varphi_{\fM_{\alpha_1} \times \fM_{\alpha_2} \times \fM_{\alpha_3}, \tilde{\Phi}_3^{',\star} u_1^{\star} u^{\star} o  }[n_{123}]}
    \ar[d]_-{\cong}^-{\zeta_{\fM_{\alpha_1} \times \fM_{\alpha_2 + \alpha_3}}}
    & {} \\
    {( \id \times \gr)_* ( \id \times \ev)^! \varphi_{\fM_1 \times \fM_{\alpha_2 + \alpha_3}, o_{\alpha_1} \boxtimes \tilde{\Phi}_2^{\star} u^{\star} o }   [n_{1,23}] }
    \ar[r]_-{\cong}^-{\zeta_{\fM_{\alpha_1 + \alpha_2 + \alpha_3}}}
    & {\gr^{(2)}}_* \ev^{(2), !} \varphi_{\fM_{\alpha_1 + \alpha_2 + \alpha_3}, o}.
    }
    \end{aligned}
    \end{equation}
    Here we set $n_{123} \coloneqq \chi (\alpha_1, \alpha_2, \alpha_3)$,
    $n_{12, 3} \coloneqq \chi(\alpha_1 + \alpha_2, \alpha_3)$ and
    $n_{1, 23} \coloneqq \chi(\alpha_1, \alpha_2 + \alpha_3)$
    and the upper left vertical map is induced from the isomorphism of orientations \eqref{eq:localize_ori_assoc_derived}.
    By the definition of the strong orientation data and the associativity statement in Remark \ref{rmk:ori_compati_with_general},
    the following diagram of orientations commutes:
    \[
    \xymatrix{
    {(o_{\alpha_1} \boxtimes o_{\alpha_2}) \boxtimes o_{\alpha_3}}
    \ar[r]^-{\cong}
    \ar[d]^-{\cong}
    & {(\tilde{\Phi}_2 \times \id)^{\star} (u^{\star} \boxtimes \id^{\star}) (o_{\alpha_1 + \alpha_2} \boxtimes o_{\alpha_3})}
    \ar[r]^-{\cong}
    & {\tilde{\Phi}_3^{\star} u_1^{\star} u^{\star} o}
    \ar[d]^-{\cong} \\
    {o_{\alpha_1} \boxtimes (o_{\alpha_2} \boxtimes o_{\alpha_3})}
    \ar[r]^-{\cong}
    & {(\id \times \tilde{\Phi}_2)^{\star} (\id^{\star} \boxtimes u^{\star}) (o_{\alpha_1} \boxtimes o_{\alpha_2 + \alpha_3})}
    \ar[r]^-{\cong}
    & {\tilde{\Phi}_3^{', \star} u_1^{\star} u^{\star} o}.
    }
    \]
    Therefore the statement holds by the commutativity of the diagram \eqref{eq:CoHa_assoc_without_orientationdata}, the associativity of the Thom--Sebastiani isomorphism in Corollary \ref{cor:TS_derived} together with the unitality and the multiplicativity of the integral isomorphism in Corollary \ref{cor:Joyce_conj_attractor}.
    
\end{proof}

\begin{rmk}\label{rmk:sign_assoc}
If the orientation data is not strong, it follows from the proof of Theorem \ref{thm:integral_isomorphism_moduli} that the associativity property of the map $\zeta_{\alpha_1, \alpha_2}$ holds true only up to some choice of sign.
\end{rmk}

For $\Gamma$-graded dg vector spaces $V$ and $W$, 
we define the twisted tensor product by
\[
V \otimes^{\mathrm{tw}} W \coloneqq \bigoplus_{\alpha, \beta \in \Gamma} V_{\alpha} \otimes W_{\beta} [- \chi(\alpha, \beta)].
\]
It defines a monoidal structure on the $\infty$-category of $\Gamma$-graded complexes.
The following corollary is an immediate consequence of Theorem \ref{thm:integral_isomorphism_moduli}:

\begin{cor}\label{cor:CoHA}
Let $\fM$ be a moduli-like stack with a strong orientation data satisfying the following conditions:
\begin{itemize}
    \item $\fM^{\cl}$ is a quasi-separated Artin stack with affine stabilizers.
    \item $\fM^{\cl}$ is $\Theta$-reductive, i.e. the map $\ev\colon \Filt(\fM^{\cl})\rightarrow \fM^{\cl}$ is proper over each connected component of the source.
    \item $\eta\colon \pt\rightarrow \fM^{\cl}$ is a closed immersion.
\end{itemize}
For each $\alpha \in \Gamma$, set $\cH_{\alpha} \coloneqq  H^{\bullet}(\fM_{\alpha}, \varphi_{\fM_{\alpha}})$ and $\cH \coloneqq \bigoplus_{\alpha \in \Gamma} \cH_{\alpha}$.
Then there are natural morphisms
\[
*^{\Hall} \colon  \cH \otimes^{\mathrm{tw}} \cH \to \cH,\qquad \bQ \rightarrow \cH_0
\]
which define an associative algebra, the \defterm{cohomological Hall algebra} for $\fM$.
\end{cor}

\begin{ex}
    Let $X$ be a smooth Calabi--Yau threefold and set $\eC = \Perf(X)$.
    In this case, points in $\fM_{\eC}$ correspond to compactly supported perfect complexes on $X$.
    Let $\fM \subset \fM_{\eC}$ be the locus consisting of compactly supported coherent sheaves (i.e., complexes concentrated in degree zero).
    As we will see in Proposition \ref{prop:modulilikemoduli}, $\fM$ admits a moduli-like structure in a natural way.
    Assume that $\fM$ admits an orientation data: this is satisfied for projective $X$ by \cite{ju21} and for local curves and surfaces by Corollary \ref{cor:strong_orientation_completion}.
    For each $\gamma \in H_*(X)$, define
    \[
    \fM_{\gamma} \coloneqq \{ [E] \in \fM \mid \ch(E) = \gamma \}, \quad
    \cH_{\gamma} \coloneqq H^{\bullet}(\fM_{\gamma}, \varphi_{\fM_{\gamma}}).
    \]
    Set $\cH_X \coloneqq \bigoplus_{\gamma} \cH_{\gamma}$.
    Then Corollary \ref{cor:CoHA} implies that there exists a $H_*(X)$-graded multiplication
    \[
   *^{\Hall} \colon \cH_X \otimes^{\mathrm{tw}} \cH_{X} \to \cH_{X}
    \]
    which is associative up to some choice of sign by Remark \ref{rmk:sign_assoc}.
    If the orientation data is strong (which is satisfied for local curves and local surfaces as we will see in Corollary \ref{cor:strong_orientation_completion}),
    the multiplication $*^{\Hall}$ defines an associative algebra structure on $\cH_X$.
\end{ex}

\subsection{Moduli of objects in dg categories}\label{ssec:moduli_dg}

Let $\dgcat$ be the $\infty$-category of $\C$-linear idempotent complete small stable $\infty$-categories and exact functors; equivalently it is the $\infty$-category of $\C$-linear dg categories localized along Morita equivalences \cite[Corollary 5.7]{coh13}. We will simply refer to objects of $\dgcat$ as dg categories. For $\eC\in\dgcat$ we denote by $\Fun^{\exact}(\eC, -)$ the stable $\infty$-category of $\C$-linear exact functors. There is a natural symmetric monoidal structure on $\dgcat$ with $\Perf_\C\subset \Mod_\C$, the $\infty$-category of perfect complexes of $\C$-vector spaces, the unit. Binary coproducts and binary products in $\dgcat$ coincide, and we denote them by $\eC\oplus \eD$.

\begin{lem}\label{lem:posetenvelope}
Let $P$ be a finite poset. Then its stable envelope is $\Perf(P^{\op})=\Fun(P^{\op}, \Perf_\C)$. In other words, restriction along the Yoneda embedding $P\rightarrow \Perf(P^{\op})$ given by $p\mapsto \C[\Hom_P(-, p)]$ induces an equivalence
\[\Fun^{\exact}(\Perf(P^{\op}), \eC)\longrightarrow \Fun(P, \eC)\]
for any $\eC\in\dgcat$.
\end{lem}
\begin{proof}
Let $\eC\in\dgcat$. Its ind-completion $\Ind(\eC)\in\Mod_{\Mod_\C}(\mathrm{Pr}^L)$ is a stable presentable $\C$-linear $\infty$-category. Thus, by \cite[Theorem 5.1.5.6]{htt} restriction along the Yoneda embedding induces an equivalence
\[\Fun^L(\Fun(P^{\op}, \eS), \Ind(\eC))\longrightarrow \Fun(P, \Ind(\eC)),\]
where on the left we consider the $\infty$-category of functors preserving small colimits. By \cite[Corollary 2.2]{ao23} the natural functor $\Fun(P^{\op}, \eS)\otimes \Mod_\C\rightarrow \Fun(P^{\op}, \Mod_\C)$ is an equivalence, and so we get that the restriction functor
\[\Fun^{L, \C}(\Fun(P^{\op}, \eS)\otimes \Mod_\C, \Ind(\eC))\longrightarrow \Fun^L(\Fun(P^{\op}, \eS), \Ind(\eC)),\]
where on the left we consider colimit-preserving $\C$-linear functors, is an equivalence. Since $\Fun(P^{\op}, \Mod_\C)$ is compactly generated,
using \cite[Proposition 5.3.5.10]{htt} we obtain that the restriction
\[\Fun^{\exact}(\Fun(P^{\op}, \Mod_\C)^\omega, \Ind(\eC))\longrightarrow \Fun(P, \Ind(\eC))\]
is an equivalence. Since $P$ is a finite poset, by \cite[Proposition 2.8]{ao23} we conclude that
\[\Fun^{\exact}(\Perf(P^{\op}), \eC)\longrightarrow \Fun(P, \eC)\]
is an equivalence.
\end{proof}

For $\eC\in\dgcat$ we denote its $\infty$-category of (right) modules by
\[\Mod_\eC=\Fun^{\exact}(\eC^{\op}, \Mod_\C).\]
We refer to objects of $\Mod_{\eC^{\op}\otimes\eD}$ as $(\eC, \eD)$-bimodules. There are relative tensor product functors and internal Hom functors; we refer to \cite[Section 2]{bdi} for more details on the Morita theory of dg categories.

For $\eC\in\dgcat$ we denote $\eC^e=\eC^{\op}\otimes\eC$. We also denote by $\eC\in\Mod_{\eC^e}$ the diagonal bimodule $x, y\mapsto \Hom_\eC(x, y)$. The \defterm{Hochschild complex} is $\HH(\eC)=\eC\otimes_{\eC^e} \eC$. It carries a natural $S^1$-action and the \defterm{negative cyclic complex} is $\HC^{-}(\eC)=\HH(\eC)^{S^1}$.

\begin{defin}$ $
\begin{itemize}
    \item A dg category $\eC$ is of \defterm{finite type} if $\eC\in\dgcat$ is compact.
    \item A dg category $\eC$ is \defterm{smooth} if the diagonal bimodule $\eC\in\Mod_{\eC^e}$ is compact. In this case, we have the dual bimodule $\eC^!=\Hom_{\eC^e}(\eC, \eC^e)\in\Mod_{\eC^e}$.
\end{itemize}
\end{defin}

We have the following basic facts:
\begin{itemize}
    \item A finite type dg category is smooth \cite[Proposition 2.14]{tv07}.
    \item If $\eC$ is smooth, there is an isomorphism $\HH(\eC)\cong \Hom_{\eC^e}(\eC^!, \eC)$.
\end{itemize}

Next, let us recall the construction of the moduli stack of objects. Let $\eC$ be a dg category.
\begin{defin}
The \defterm{moduli stack $\fM_{\eC}$ of objects in $\eC$} is the functor
\[
    \cdga^{\leq 0} \ni R \mapsto \Hom_{\dgcat}(\eC, \Perf_R) \in \eS.
\]
We refer to functors $\eC\rightarrow \Perf_R$ as \defterm{pseudo-perfect $\eC$-modules}.
\end{defin}

We also have a similar description of mapping stacks.
\begin{lem}\label{lem:mapsintomoduli}
Let $X$ be a derived prestack and $\eC$ a dg category. Then there is a natural equivalence
\[\Map(X, \fM_{\eC})(R)\cong \Hom_{\dgcat}(\eC, \Perf(X\times \Spec R))\]
for $R\in\cdga^{\leq 0}$.
\end{lem}
\begin{proof}
It is enough to prove that $\Hom(X, \fM_{\eC})\cong \Hom_{\dgcat}(\eC, \Perf(X))$ for any derived prestack $X$. If $X=\Spec A$ is affine, this holds by the definition of the moduli stack of objects. As any derived prestack is obtained as a colimit of derived affine schemes and both $\Hom(X, \fM_{\eC})$ and $\Hom_{\dgcat}(\eC, \Perf(X))$ send colimits in $X$ to limits, the claim follows.
\end{proof}
If $\eC$ is of finite type, it is shown in \cite[Theorem 3.6]{tv07} that $\fM_{\eC}$ is a locally geometric derived stack.

\begin{defin}
Let $\eC$ be a smooth dg category. A \defterm{$d$-dimensional Calabi-Yau structure on $\eC$} is a negative cyclic class $\tilde{c} \in |\HC^{-}(\eC)[-d]|$ whose underlying Hochschild class $c \in | \HH(\eC)[-d] | \cong |\Hom_{\eC^e}(\eC^!, \eC)[-d]|$ defines an equivalence $\eC^!\rightarrow \eC[-d]$ of $\eC^e$-modules.
\end{defin}

\begin{rmk}
In this paper we only consider left Calabi--Yau structures in the terminology of \cite{bdi}, so the adjective ``left'' will often be dropped.
\end{rmk}

\begin{ex}
The dg category $\Perf_\C^{\oplus n}$ has $\HC^-(\Perf_\C^{\oplus n})\cong \C[\![u]\!]^{\oplus n}$. The class $\tilde{c}_n=(1, \dots, 1)$ defines a $0$-Calabi--Yau structure on $\Perf_\C^{\oplus n}$.
\end{ex}

Given an exact functor $F\colon \eC\rightarrow \Perf_R$ we get a map
\[\HC^-(\eC)\xrightarrow{\HC^-(F)} \HC^-(R)\xrightarrow{\tilde{\kappa}_2} \Omega^{\geq 2}(R)[2],\]
where the map $\tilde{\kappa}_2$ is defined using the HKR isomorphism as in \cite[Proposition 5.2]{bdii}. By naturality, we get an induced map
\begin{equation}\label{eq:HC_to_forms}
\tilde{\kappa}_2\colon \HC^-(\eC)\longrightarrow \Omega^{\geq 2}(\fM_{\eC})[2].
\end{equation}
If $\eC$ is of finite type, it is shown in \cite[Theorem 5.5]{bdii} that it sends $d$-Calabi--Yau structures on $\eC$ to $(2-d)$-shifted symplectic structures on $\fM_{\eC}$.

    Let $P_n=\{0<1<\dots < n-1\}$ be the poset of integers between $0$ and $n-1$.
    
    \begin{defin}
    The \defterm{moduli stack of $n$-filtered objects in $\eC$} is
    \[\fM^{n-\filt}_{\eC} = \fM_{\eC\otimes \Perf(P_n)}.\]
    \end{defin}
    By Lemma \ref{lem:posetenvelope} the $R$-points of $\fM^{n-\filt}_{\eC}$ are given by
    \[\fM^{n-\filt}_{\eC}(R) = \Map(P_n^{\op}, \Fun^{\exact}(\eC, \Perf_R)),\]
    so $\fM^{n-\filt}_{\eC}$ parametrizes sequences $M_{n-1}\rightarrow \dots\rightarrow M_0$ of pseudo-perfect modules. We will be interested in the following functors:
    \begin{itemize}
        \item $\iota_1\colon \Perf_\C\rightarrow \Perf(P_n)$ given by the inclusion of the constant sheaf $\cO_{P_n}\in\Perf(P_n)$ sending $i\in P_n\mapsto \C$ and any $i\rightarrow j$ to the identity. By restriction, it induces a map $\ev\colon \fM^{n-\filt}_{\eC}\rightarrow \fM_{\eC}$ sending a filtered pseudo-perfect module $M_{n-1}\rightarrow \dots\rightarrow M_0$ to $M_0$.
        \item $P^\delta_n\rightarrow \Perf(P_n)$, where $P^\delta_n$ is the same underlying set as $P_n$ considered as a discrete category, given by sending $i\in P^\delta_n$ to the skyscraper sheaf $i$. By the universal property of $\Perf(P^\delta_n)$ it induces an exact functor $\iota_\gr\colon \Perf(P^\delta_n)\rightarrow \Perf(P_n)$. Since $\eC\otimes\Perf(P^\delta_n)\cong \eC^{\oplus n}$, it induces a map $\gr\colon \fM^{n-\filt}_{\eC}\rightarrow \fM_{\eC}^n$ sending a filtered pseudo-perfect module to its associated graded.
        \item The functor $\pi\colon \Perf(P_n)\rightarrow \Perf(P^\delta_n)$ left inverse to $\iota_\gr$ given by restriction along $P^\delta_n\rightarrow P_n$. By restriction, it induces a map $\sigma\colon \fM_{\eC}^n\rightarrow \fM^{n-\filt}_{\eC}$ which is a section of $\gr$. It sends $M_0, \dots, M_{n-1}$ to $M_{n-1}\rightarrow M_{n-1}\oplus M_{n-2}\rightarrow \dots$
        \item The diagonal functor $\Perf_\C\rightarrow \Perf(P^\delta_n)=\Perf_\C^{\oplus n}$ given by the inclusion of the constant sheaf $\cO_{P^\delta_n}\in\Perf(P^\delta_n)$. By restriction, it induces a map $\Phi_n\colon \fM^n_{\eC}\rightarrow \fM_{\eC}$ given by sending $M_0, \dots, M_{n-1}$ to $M_0\oplus \dots \oplus M_{n-1}$. We have $\Phi_n=\ev\circ \sigma$.
    \end{itemize}

    We obtain a correspondence
    \begin{equation} \label{eq:exactsequencecorrespondence}
    \begin{aligned}
    \xymatrix{
    & \fM^{n-\filt}_{\eC} \ar[ld]_-{\gr} \ar[rd]^-{\ev} & \\
    \fM_{\eC}^n \ar@/_10pt/[ru]_-{\sigma} && \fM_{\eC},
    }
    \end{aligned}
    \end{equation}

    Let us now endow this correspondence with a Lagrangian structure when $\eC$ carries a $d$-Calabi--Yau structure. The following is essentially shown in \cite[Theorem 5.14]{bdi}.
    
    \begin{prop}\label{prop:filteredcospan}
    The cospan of dg categories $\Perf(P^\delta_n)\rightarrow \Perf(P_n)\leftarrow \Perf_\C$ carries a unique $0$-Calabi--Yau structure which restricts to the $0$-Calabi--Yau structure $\tilde{c}_n$ on $\Perf(P^\delta_n)\cong \Perf_\C^{\oplus n}$ and the $0$-Calabi--Yau structure $\tilde{c}_1$ on $\Perf_\C$.
    \end{prop}
    \begin{proof}
    Since the functor $\HC^-$ is additive, the functors
    \[\pi\colon \Perf(P_n)\longrightarrow \Perf(P^\delta_n),\qquad \iota_\gr\colon \Perf(P^\delta_n)\longrightarrow \Perf(P_n)\]
    induce inverse maps on $\HC^-$ \cite[Proposition 5.15]{bdi}. Therefore, a $0$-Calabi--Yau structure on the cospan is given by a homotopy commutative diagram
    \[
    \xymatrix{
    \C \ar^{c_1}[r] \ar^{c_n}[d] & \C[\![u]\!] \ar^{\Delta}[d] \\
    \C[\![u]\!]^{\oplus n} \ar^{\id}[r] & \C[\![u]\!]^{\oplus n}
    }
    \]
    where the vertical map on the right is the diagonal. Since all complexes are concentrated in non-negative cohomological degrees, there is a unique homotopy which makes this diagram commutative.
    \end{proof}

    \begin{cor}[{\cite[Corollary 6.5]{bdii}}]\label{cor:BDcorrespondence}
    Let $\eC$ be a finite type dg category equipped with a $d$-Calabi--Yau structure. Then there is a canonical $(2-d)$-shifted Lagrangian structure on the correspondence
    \begin{equation}\label{eq:moduli_attractor_corresp}
    \begin{aligned}
    \xymatrix{
    & \fM^{n-\filt}_{\eC} \ar[ld]_-{\gr} \ar[rd]^-{\ev} & \\
    \fM_{\eC}^n && \fM_{\eC},
    }
    \end{aligned}
    \end{equation}
    In particular, restricting this Lagrangian structure along the section $\sigma\colon \fM_{\eC}^n\rightarrow \fM^{n-\filt}_{\eC}$ we get that $\Phi_n\colon \fM_{\eC}^n\rightarrow \fM_{\eC}$ is compatible with $(2-d)$-shifted symplectic structures.
    \end{cor}

    Let us now describe how to compose these Lagrangian correspondences.

    \begin{prop}\label{prop:filtrationcomposition}
    Let $\eC$ be a finite type dg category. The composition of the correspondences
    \[
    \xymatrix{
    & \fM^{n-\filt}_{\eC}\times \fM^{m-1}_{\eC} \ar_{\gr\times \id}[dl] \ar^{\ev\times \id}[dr] && \fM^{m-\filt}_{\eC} \ar_{\gr}[dl] \ar^{\ev}[dr] & \\
    \fM^{n+m-1}_{\eC} && \fM^m_{\eC} && \fM_{\eC}
    }
    \]
    and
    \[
    \xymatrix{
    & \fM^{n-1}_{\eC} \times \fM^{m-\filt}_{\eC}\ar_{\id\times\gr}[dl] \ar^{\id\times \ev}[dr] && \fM^{n-\filt}_{\eC} \ar_{\gr}[dl] \ar^{\ev}[dr] & \\
    \fM^{n+m-1}_{\eC} && \fM^n_{\eC} && \fM_{\eC}
    }
    \]
    are both equivalent to the same correspondence
    \[
    \xymatrix{
    & \fM^{(n+m-1)-\filt}_{\eC} \ar_{\gr}[dl] \ar^{\ev}[dr] & \\
    \fM^{n+m-1}_{\eC} && \fM_{\eC}.
    }
    \]
    Moreover, if $\eC$ is equipped with a $d$-Calabi--Yau structure, this equivalence is compatible with Lagrangian structures.
    \end{prop}
    \begin{proof}
    It is shown in \cite[Section 6.1.2]{bcs20} that the functor of moduli of objects sends compositions of Calabi--Yau cospans to compositions of Lagrangian correspondences. Thus, it is enough to prove the claim for Calabi--Yau cospans. The two assertions are proven analogously, so we will only deal with the first claim. Consider a diagram
    \[
    \xymatrix{
    && \Perf(P_{n+m-1}) && \\
    & \Perf_\C^{\oplus(n-1)} \oplus \Perf(P_m) \ar[ur] && \Perf(P_n) \ar[ul] & \\
    \Perf_\C^{\oplus(n+m-1)} \ar[ur] && \Perf_\C^{\oplus n} \ar[ul] \ar[ur] && \Perf_\C \ar[ul]
    }
    \]
    of dg categories. To show that the top square is coCartesian in $\dgcat$, using Lemma \ref{lem:posetenvelope} we have that for any dg category $\eD$ we have $\Fun^{\exact}(\Perf(P_n), \eD)\cong \Fun((\Delta^{n-1})^{\op}, \eD)$, so we have to show that
    \[
    \xymatrix{
    \Delta^0 \ar[r] \ar[d] & \Delta^{m-1} \ar[d] \\
    \Delta^{n-1} \ar[r] & \Delta^{n+m-2}
    }
    \]
    is a homotopy coCartesian square of simplicial sets with respect to the Joyal model structure. By induction, it is enough to prove the claim for $n=m=1$ in which case we have that $\Lambda^2_1=\Delta^1\coprod_{\Delta^0} \Delta^1\rightarrow \Delta^2$ is a categorical equivalence since it is inner anodyne \cite[Lemma 2.2.5.2]{htt}. It remains to show that the above diagram of dg categories represents a composition of Calabi--Yau cospans; but it is obvious, as there is a unique structure of a $0$-Calabi--Yau cospan on $\Perf_\C^{\oplus(n+m-1)}\rightarrow \Perf(P_{n+m-1})\leftarrow \Perf_\C$ by Proposition \ref{prop:filteredcospan}.
    \end{proof}

\subsection{From moduli of objects to mapping stacks}

In this section we explain how to relate moduli of objects to mapping stacks.

\begin{defin}
Let $X$ be a derived stack and $\eD$ a dg category. The $\infty$-category of \defterm{perfect complexes on $X$ with coefficients in $\eD$} is
\[\Perf(X; \eD) = \lim_{\Spec A\rightarrow X} (\Perf_A\otimes \eD).\]
\end{defin}

\begin{ex}\label{ex:PerfPerfcoefficients}
Let $R$ be a connective commutative dg algebra. Then
\[\Perf(X; \Perf_R)\cong \Perf(X\times\Spec R).\]
\end{ex}

\begin{ex}\label{ex:Perfsmoothproper}
If $\eD$ is smooth and proper, it is dualizable in $\dgcat$ and hence $(-)\otimes\eD$ preserves limits. Therefore, the natural functor $\eD\otimes\Perf(X)\rightarrow \Perf(X; \eD)$ is an equivalence.
\end{ex}

An object $E\in\Perf(X; \eD)$ induces a functor
\[T_{E, R}\colon \Fun^{\exact}(\eD, \Perf_R)\longrightarrow \Perf(X; \Perf_R)\cong \Perf(X\times \Spec R)\]
which is natural in $R$.

\begin{prop}\label{prop:modulitomap}
Let $X$ be a derived stack and $\eC,\eD$ a pair of dg categories. An object $E\in\Perf(X; \eD)$ gives rise to a morphism
\[F_E\colon \fM_{\eC\otimes\eD}\longrightarrow \Map(X, \fM_{\eC})\]
of derived stacks.
\end{prop}
\begin{proof}
Using the description of the mapping stacks to the moduli of objects given in Lemma \ref{lem:mapsintomoduli} we have to construct, for every connective cdga $R$, a morphism
\[\Hom_{\dgcat}(\eC\otimes\eD, \Perf_R)\longrightarrow \Hom_{\dgcat}(\eC, \Perf(X\times\Spec R))\]
natural in $R$. By adjunction the left-hand side is isomorphic to
\[\Hom_{\dgcat}(\eC\otimes\eD, \Perf_R)\cong \Hom_{\dgcat}(\eC, \Fun^{\exact}(\eD, \Perf_R))\]
and the map $F_E$ is simply given by post-composition with $T_{E, R}$.
\end{proof}

If $X$ is a derived stack and $\eD$ a dg category, consider the morphisms
\begin{align*}
\HC^-(\Perf(X; \eD))&\longrightarrow \lim_{\Spec A\rightarrow X} \HC^-(\Perf_A\otimes \eD)\\
&\cong \lim_{\Spec A\rightarrow X} (\HH(A)\otimes \HH(\eD))^{S^1} \\
&\longrightarrow \lim_{\Spec A\rightarrow X} (A\otimes \HC^-(\eD)) \\
&\longleftarrow \bR\Gamma(X, \cO)\otimes \HC^-(\eD)
\end{align*}
where we have used the natural projection $\HH(A)\rightarrow A$ which is $S^1$-equivariant. If $X$ is $\cO$-compact, the last backward map is an isomorphism, so in total we obtain a morphism
\[\kappa_{\eD, X}\colon \HC^-(\Perf(X; \eD))\longrightarrow \bR\Gamma(X, \cO)\otimes \HC^-(\eD).\]

\begin{thm}\label{thm:modulitomapsymplectic}
Consider the following data:
\begin{enumerate}
    \item $X$ is an $\cO$-compact derived stack equipped with a $d_1$-preorientation $[X]\colon \bR\Gamma(X, \cO)\rightarrow \C[-d_1]$.
    \item $\eC$ is a dg category equipped with a class $[\eC]\in|\HC^-(\eC)[-d_2]|$.
    \item $E\in\Perf(X; \eD)$.
    \item $\eD$ is a dg category.
\end{enumerate}
Consider the class $[\eD]\in|\HC^-(\eD)[-d_1]|$ given by the composite
\begin{align*}
\C&\xrightarrow{[E]} \HC^-(\Perf(X; \eD))\\
&\xrightarrow{\kappa_{\eD, X}} \bR\Gamma(X, \cO)\otimes \HC^-(\eD) \\
&\xrightarrow{[X]\otimes \id} \HC^-(\eD)[-d_1]
\end{align*}
Then the morphism
\[F_E\colon \fM_{\eC\otimes\eD}\longrightarrow \Map(X, \fM_{\eC})\]
is compatible with $(2-d_1-d_2)$-shifted presymplectic structures, where the $(2-d_1-d_2)$-shifted presymplectic structure on the left is defined by the image of $[\eC]\otimes [\eD]$ under $\HC^-(\eC)\otimes \HC^-(\eD)\rightarrow \HC^-(\eC\otimes \eD)$ and on the right induced by the AKSZ procedure from the $(2-d_2)$-shifted presymplectic structure on $\fM_{\eC}$ defined by $[\eC]$.
\end{thm}
\begin{proof}
Consider an $R$-point of $\fM_{\eC\otimes \eD}$ specified by an exact functor $F\colon \eC\otimes\eD\rightarrow \Perf_R$ and the corresponding $R$-point of $\Map(X, \fM_{\eC})$ specified by an exact functor $F'\colon \eC\rightarrow \Perf(X\times \Spec R)$. The pullback of the $(2-d_1-d_2)$-shifted presymplectic structure on $\fM_{\eC\otimes\eD}$ to $\Spec R$ is given by the composite
\begin{align*}
\C&\xrightarrow{[\eC]\otimes [\eD]} \HC^-(\eC)\otimes \HC^-(\eD)[-d_1-d_2] \\
&\rightarrow \HC^-(\eC\otimes\eD)[-d_1-d_2] \\
&\xrightarrow{\HC^-(F)} \HC^-(R)[-d_1-d_2] \\
&\xrightarrow{\tilde{\kappa}_2} \Omega^{\geq 2}(R)[2-d_1-d_2].
\end{align*}

Similarly, the pullback of the $(2-d_1-d_2)$-shifted presymplectic structure on $\Map(X, \fM_{\eC})$ to $\Spec R$ is given by the composite
\begin{align*}
\C&\xrightarrow{[\eC]} \HC^-(\eC)[-d_2] \\
&\xrightarrow{\HC^-(F')} \HC^-(\Perf(X\times \Spec R))[-d_2] \\
&\xrightarrow{\kappa_{\Spec R, X}} \bR\Gamma(X, \cO)\otimes \HC^-(R)[-d_2] \\
&\xrightarrow{[X]} \HC^-(R)[-d_1-d_2] \\
&\xrightarrow{\tilde{\kappa}_2} \Omega^{\geq 2}(R)[2-d_1-d_2]
\end{align*}

The claim then follows from the fact that $F'$ is given by the composite
\begin{align*}
\eC&\xrightarrow{\id\otimes E} \eC\otimes\Perf(X; \eD) \\
&\rightarrow \Perf(X; \eC\otimes \eD) \\
&\xrightarrow{F} \Perf(X; \Perf_R) \\
&\cong \Perf(X\times \Spec R).
\end{align*}
\end{proof}

The constructions described in Proposition \ref{prop:modulitomap} and Theorem \ref{thm:modulitomapsymplectic} are functorial in the following way. Suppose $I$ is a small $\infty$-category, $X_\bullet$ is a diagram $I\rightarrow \dSt$ of derived stacks and $\eD_\bullet$ is a diagram $I\rightarrow \dgcat$ of dg categories. Then:
\begin{itemize}
    \item Using the description of natural transformations as an end \cite[Proposition 5.1]{ghn17}, by Proposition \ref{prop:modulitomap} an object $E\in\int_{i\in I} \Perf(X_i;\eD_i)$ produces a natural transformation
    \[F_E\colon \fM_{\eC\otimes\eD_\bullet}\longrightarrow \Map(X_\bullet, \fM_{\eC})\]
    of functors $I^{\op}\rightarrow \dSt$.
    \item If, in addition, we have a class $[\eC]\in|\HC^-(\eC)[-d_2]|$, for every $i\in I$ the derived stack $X_i$ is $\cO$-compact and we are given a $d_1$-preorientation
    \[[X_\bullet]\colon \colim_{i\in I} \bR\Gamma(X_i, \cO)\rightarrow \C[-d_1],\]
    then by Theorem \ref{thm:modulitomapsymplectic} the natural transformation $F_E$ is compatible with $(2-d_1-d_2)$-shifted presymplectic structures.
\end{itemize}

\subsection{Attractor correspondence for the moduli of objects}\label{ssec:moduli_attractor}

Let us now work out a description of the attractor correspondence for the moduli stack of objects. Let $\eC$ be a dg category of finite type. For a finite connected subset $P\subset \Z_{\geq}$ (which is isomorphic to $P^{\op}_n$ for some $n$) consider the correspondence
\[
\xymatrix{
& \fM_{\eC\otimes\Perf(P^{\op})} \ar_{\gr}[dl] \ar^{\ev}[dr] & \\
\fM_{\eC\otimes \Perf(P^\delta)} && \fM_{\eC}
}
\]
An inclusion $P_1\subset P_2$ produces restriction functors $\Perf(P_2^{\op})\rightarrow \Perf(P_1^{\op})$ and hence a morphism of correspondences
\[
\left(
\begin{gathered}
\xymatrix{
& \fM_{\eC\otimes\Perf(P_1^{\op})} \ar_{\gr}[dl] \ar^{\ev}[dr] & \\
\fM_{\eC\otimes \Perf(P_1^\delta)} && \fM_{\eC}
}
\end{gathered}
\right)\longrightarrow\left(
\begin{gathered}
\xymatrix{
& \fM_{\eC\otimes\Perf(P_2^{\op})} \ar_{\gr}[dl] \ar^{\ev}[dr] & \\
\fM_{\eC\otimes \Perf(P_2^\delta)} && \fM_{\eC}
}    
\end{gathered}\right)
\]

\begin{prop}\label{prop:gradedmoduli}
Let $\eC$ be a dg category of finite type. Then
\[
\begin{gathered}
\xymatrix{
& \Filt(\fM_{\eC}) \ar_{\gr}[dl] \ar^{\ev}[dr] & \\
\Grad(\fM_{\eC}) && \fM_{\eC}
}    
\end{gathered}
\cong\colim_{P\subset \bZ_{\geq}\text{, finite, connected}}\left(\begin{gathered}
\xymatrix{
& \fM_{\eC\otimes\Perf(P^{\op})} \ar_{\gr}[dl] \ar^{\ev}[dr] & \\
\fM_{\eC\otimes \Perf(P^\delta)} && \fM_{\eC}
}
\end{gathered}\right)
\]
Moreover, the transition maps $\fM_{\eC\otimes \Perf(P_1^{\op})}\rightarrow \fM_{\eC\otimes \Perf(P_2^{\op})}$ and $\fM_{\eC\otimes \Perf(P_1^{\delta})}\rightarrow \fM_{\eC\otimes \Perf(P_2^{\delta})}$ are open immersions.
\end{prop}
\begin{proof}
Let us prove that
\[\Filt(\fM_{\eC})\cong \colim_{P\subset \bZ_{\geq}\text{, finite, connected}} \fM_{\eC\otimes\Perf(P^{\op})}.\]
By Lemma \ref{lem:mapsintomoduli} we have
\[\Filt(\fM_{\eC})(R)\simeq \Hom_{\dgcat}(\eC, \Perf([\bA^1/\bG_m]\times \Spec R))\]
for any $R\in\cdga^{\leq 0}$. By Proposition \ref{prop:filteredperfect} we further identify
\[\Hom_{\dgcat}(\eC, \Perf([\bA^1/\bG_m]\times \Spec R))\simeq \Hom_{\dgcat}(\eC, \colim_{P\subset \bZ_{\geq}\text{, finite, connected}}\Fun(P, \Perf_R)).\]
Since $\eC$ is of finite type, the natural morphism
\[\colim_{P\subset \bZ_{\geq}\text{, finite, connected}}\Hom_{\dgcat}(\eC, \Fun(P, \Perf_R))\longrightarrow \Hom_{\dgcat}(\eC, \colim_{P\subset \bZ_{\geq}\text{, finite, connected}}\Fun(P, \Perf_R))\]
is an equivalence. We have $\Fun(P, \Perf_R)\cong \Fun^{\exact}(\Perf(P^{\op}), \Perf_R)$. Therefore,
\[\Hom_{\dgcat}(\eC, \Fun(P, \Perf_R))\simeq \Hom_{\dgcat}(\eC\otimes \Perf(P^{\op}), \Perf_R).\]
Finally, the right-hand side is, by definition, $\fM_{\eC\otimes \Perf(P^{\op})}(R)=\Fun(P, \Fun^{\exact}(\eC, \Perf_R))$. Thus,
\[\Filt(\fM_{\eC})(R)\cong \colim_{P\subset \bZ_{\geq}\text{, finite, connected}} \fM_{\eC\otimes\Perf(P^{\op})}(R)\]
which shows that
\[\Filt(\fM_{\eC})\cong \colim_{P\subset \bZ_{\geq}\text{, finite, connected}} \fM_{\eC\otimes\Perf(P^{\op})},\]
where we consider the colimit in the $\infty$-category of derived prestacks. As $\Filt(\fM_{\eC})$ is a stack, this shows that the right-hand side coincides with the colimit computed in the $\infty$-category of derived stacks.

For a subset $P_1\subset P_2$, we will now show that $\fM_{\eC\otimes\Perf(P_1^{\op})}\rightarrow \fM_{\eC\otimes \Perf(P_2^{\op})}$ is an open immersion. On $R$-points this map is
\[\Map(P_1, \Fun^{\exact}(\eC, \Perf_R))\longrightarrow \Map(P_2, \Fun^{\exact}(\eC, \Perf_R))\]
given by the left Kan extension. The right-hand side is given by the space of $P_2$-labelled diagrams $\{V_m\}_{m\in P_2}$ of exact functors $\eC\rightarrow \Perf_R$ and the left-hand side is given by the union of connected components of those $P_2$-labelled diagrams such that $V_m=0$ if $m>i$ for any $i\in P_1$ and $V_{m+1}\rightarrow V_m$ is an isomorphism if $m<i$ for any $i\in P_1$.

As $\eC$ is of finite type, it admits a generator $E\in\eC$, so that the functor $\Fun^{\exact}(\eC, \Perf_R)\rightarrow \Perf_R$ given by evaluation at $E$ is conservative. Therefore, the claim is reduced to $\eC=\Perf_\C$. Therefore, we have to show the following two claims:
\begin{enumerate}
    \item For any perfect complex $V$ of $R$-modules the locus where it is zero is open in $\Spec R$.
    \item For a morphism of perfect complexes $V_1\rightarrow V_2$ of $R$-modules, the locus where it is an isomorphism is open in $\Spec R$.
\end{enumerate}
By taking the cone, the second claim is reduced to the first claim. Moreover, by induction on the Tor-amplitude, we are reduced to showing the first claim for $V$ a vector bundle on $\Spec R$. Since open immersions are local on the target and a vector bundle is locally trivial, we are reduced to showing the claim for $V$ the trivial vector bundle, where the claim is obvious.

The claims for $\Grad(\fM_{\eC})$ as well as for the morphisms $\Filt(\fM_{\eC})\rightarrow \Grad(\fM_{\eC})$ and $\Filt(\fM_{\eC})\rightarrow \fM_{\eC}$ are proven similarly using the functoriality described in Proposition \ref{prop:filteredperfect}.
\end{proof}

\begin{ex}\label{ex:GradVect}
Setting $\eC = \Perf_\C$ in Proposition \ref{prop:gradedmoduli} we get a description of the moduli stack of graded perfect complexes
\[\Grad(\Perf)\cong \colim_{P\subset \bZ\text{, finite, connected}} \Perf^P.\]
If we set $\Vect\subset \Perf$ the full substack of vector bundles, then we similarly get a description of the moduli stack of graded vector spaces
\[\Grad(\Vect)\cong \colim_{P\subset \bZ\text{, finite, connected}} \Vect^P.\]
\end{ex}

Let us now show that the isomorphism exhibited in Proposition \ref{prop:gradedmoduli} is compatible with Lagrangian structures. Let $\eC$ be a finite type dg category equipped with a $d$-Calabi--Yau structure. For a finite connected subset $P\subset \Z_{\geq}$ Corollary \ref{cor:BDcorrespondence} constructs a $(2-d)$-shifted symplectic structure on $\fM_{\eC}$ as well as a Lagrangian structure on the correspondence
\[
\xymatrix{
& \fM_{\eC\otimes\Perf(P^{\op})} \ar_{\gr}[dl] \ar^{\ev}[dr] & \\
\fM_{\eC\otimes \Perf(P^\delta)} && \fM_{\eC}
}
\]

Similarly, Corollary \ref{cor:lagattractorcorrespondence} constructs a Lagrangian structure on the correspondence
\[
\xymatrix{
& \Filt(\fM_{\eC}) \ar_{\gr}[dl] \ar^{\ev}[dr] & \\
\Grad(\fM_{\eC}) && \fM_{\eC}
} 
\]

\begin{prop}\label{prop:moduliattractorlagrangian}
Let $\eC$ be a finite type dg category equipped with a $d$-Calabi--Yau structure. Let $P\subset \Z_{\geq}$ be a finite connected subset. Then the morphism
\[
\left(\begin{gathered}
\xymatrix{
& \fM_{\eC\otimes\Perf(P^{\op})} \ar_{\gr}[dl] \ar^{\ev}[dr] & \\
\fM_{\eC\otimes \Perf(P^\delta)} && \fM_{\eC}
}
\end{gathered}\right)
\longrightarrow
\begin{gathered}
\xymatrix{
& \Filt(\fM_{\eC}) \ar_{\gr}[dl] \ar^{\ev}[dr] & \\
\Grad(\fM_{\eC}) && \fM_{\eC}
}    
\end{gathered}
\]
is compatible with Lagrangian structures.
\end{prop}
\begin{proof}
Let $\spncat$ be the category $0\rightarrow 1\leftarrow 2$ with three objects. Let $X_\bullet\colon \spncat\rightarrow \dSt$ be the diagram of $\cO$-compact derived stacks $B\bG_m\xrightarrow{i_0}[\bA^1/\bG_m]\xleftarrow{i_1}\pt$ equipped with a $0$-orientation defined in Proposition \ref{prop:orientedcospan}. Let $\eD_\bullet\colon \spncat\rightarrow \dgcat$ be the diagram of dg categories $\Perf(P^\delta)\rightarrow \Perf(P^{\op})\leftarrow \Perf_\C$. 

Let us first prove that for every $i\in\spncat$ the natural functor
\[\Fun^{\exact}(\eD_i, \Perf_\C)\otimes \Perf_R\longrightarrow \Fun^{\exact}(\eD_i, \Perf_R)\]
is an equivalence. This is obvious for $i=0, 2$, where $\eD_i$ is a finite direct sum of $\Perf_\C$, so let us just show the claim for $i=1$. Using Lemma \ref{lem:posetenvelope} we are reduced to showing that
\[\Fun(P, \Perf_\C)\otimes \Perf_R\longrightarrow \Fun(P, \Perf_R)\]
is an equivalence. By \cite[Proposition 2.8]{ao23} we have $\Fun(P, \Mod_R)^\omega = \Fun(P, \Perf_R)$ (as well as the same claim with $R$ replaced by $\C$). If we denote by $\widehat{\otimes}$ the symmetric monoidal structure in $\Pr^L$, we are reduced to showing that
\[\Fun(P, \eS)\widehat{\otimes} \Mod_R\longrightarrow \Fun(P, \Mod_R)\]
is an equivalence for any commutative dg algebra $R$, which is \cite[Corollary 2.2]{ao23}.

The natural functor
\[\Perf(X_\bullet)\otimes \Perf(R)\longrightarrow \Perf(X_\bullet\times\Spec R)\]
is also an equivalence, since $X_\bullet$ is perfect \cite[Proposition 4.6]{bzfn}.

Therefore, the equivalence established in Proposition \ref{prop:filteredperfect}, natural in $R$, is uniquely determined by its value on $R=\C$, so that the functor $\Fun^{\exact}(\eD_\bullet, \Perf_R)\rightarrow \Perf(X_\bullet\times\Spec R)$ defined in Proposition \ref{prop:filteredperfect} is given by
\[\Fun^{\exact}(\eD_\bullet, \Perf_R)\xrightarrow{T_{E, R}} \Perf(X_{\bullet} \times \Spec R)\]
for a universal object
\[E\in\int_{i\in\spncat}\Perf(X_i; \eD_i),\]
where we note that by Example \ref{ex:Perfsmoothproper} we have $\Perf(X_i; \eD_j)\cong \Perf(X_i)\otimes \eD_j$ since $\eD_j$ is smooth and proper for every $j\in\spncat$.

Using Theorem \ref{thm:modulitomapsymplectic} we deduce that the morphism of correspondences $F_E$ is compatible with Lagrangian structures, where $\eD_\bullet$ is equipped with a certain structure of a $0$-Calabi--Yau cospan $[\eD_\bullet]$ described in the statement. It remains to show that this $0$-Calabi--Yau structure on $\eD_\bullet$ is equivalent to the one defined in Proposition \ref{prop:filteredcospan}. By uniqueness, it is enough to identify the $0$-Calabi--Yau structures on $\Perf(P^\delta)$ and $\Perf_\C$. The latter one is given by the composite of the identity maps, so it coincides with $\tilde{c}_1$. Let us now compute $[\Perf(P^\delta)]$. By definition, it is given by the composite
\begin{align*}
\C&\xrightarrow{[E_0]} \HC^-(\Perf(B\bG_m; \Perf(P^\delta))) \\
&\xrightarrow{\kappa_{\Perf(P^\delta), B\bG_m}} \bR\Gamma(B\bG_m, \cO)\otimes \HC^-(\Perf(P^\delta)) \\
&\xrightarrow{[B\bG_m]\otimes \id} \HC^-(\Perf(P^\delta)).
\end{align*}
Recall that $[B\bG_m]$ is defined by the pullback
\[\bR\Gamma(B\bG_m, \cO)\xrightarrow{i^*} \bR\Gamma(\pt, \cO)\cong \C\]
along the inclusion of the basepoint $i\colon \pt\rightarrow B\bG_m$. Using naturality of $\kappa_{\Perf(P^\delta), -}$ for the morphism $i\colon \pt\rightarrow B\bG_m$ we get that the class $[\Perf(P^\delta)]$ is given by the composite
\[
\C\xrightarrow{[E_0]} \HC^-(\Perf(B\bG_m; \Perf(P^\delta)))\xrightarrow{i^*} \HC^-(\Perf(P^\delta)).
\]
The object $E_0\in\Perf(B\bG_m; \Perf(P^\delta))\cong \Fun(P^\delta, \Perf(B\bG_m))$ is given by sending $p\in P^\delta$ to the 1-dimensional representation of $\bG_m$ with weight $p$. Thus, $i^* E_0=\cO_{P^\delta}$ is the constant sheaf $p\mapsto \C$ whose class $[i^* E_0]\in \HC^-(\Perf(P^\delta))$ coincides with $\tilde{c}_n$.
\end{proof}

\begin{prop}\label{prop:modulilikemoduli}
Let $\eC$ be a finite type dg category equipped with a $3$-Calabi--Yau structure and $\fM\subset \fM_{\eC}$ an open substack satisfying the following conditions:
\begin{enumerate}
    \item $\fM^{\cl}$ is a quasi-separated Artin stack with affine stabilizers.
    \item $\fM$ is closed under direct sums and it contains the zero module.
\end{enumerate}
Then $\fM$ is a moduli-like stack in the sense of Definition \ref{def:modulilike}.
\end{prop}
\begin{proof}
By \cite[Theorem 3.6]{tv07} the stack $\fM_{\eC}$ is locally geometric and locally of finite presentation, so $\fM$ is locally geometric and locally of finite presentation. The stack $\fM_{\eC}$ carries a natural $B\bG_m$-action, which is given on $R$-points by tensoring a pseudo-perfect module $\eC\rightarrow \Perf_R$ with a line bundle on $\Spec R$ considered as a constant pseudo-perfect module. Since $\fM$ is open, this action restricts to $\fM$. Since $\fM$ is closed under direct sums, the $\bE_\infty$-structure on $\fM_{\eC}$ given by direct sum restricts to $\fM$. Finally, by Corollary \ref{cor:BDcorrespondence} we obtain a $(-1)$-shifted symplectic structure on $\fM$ which is additive.

Let us now check the conditions:
\begin{enumerate}
    \item Consider the map $\fM_{\eC}^2\rightarrow \Grad(\fM_{\eC})$ given by the inclusion of graded objects in weights $0$ and $1$. By Proposition \ref{prop:gradedmoduli} it is an open immersion. Since $\fM_{\eC}$ is locally geometric and locally of finite presentation, by Proposition \ref{prop:gradedmoduli} $\Grad(\fM_{\eC})$ is locally geometric and locally of finite presentation. Since $\fM\rightarrow \fM_{\eC}$ is a monomorphism, so is $\Grad(\fM)\rightarrow \Grad(\fM_{\eC})$ (see \cite[Proposition 1.3.1]{hlp14}). Since $\fM\rightarrow \fM_{\eC}$ is formally \'etale, $\Grad(\fM)\rightarrow \Grad(\fM_{\eC})$ is formally \'etale by Proposition \ref{prop:GradFiltcotangent}. Therefore, by \cite[Corollary 2.2.5.6]{HAGII} the morphism $\Grad(\fM)\rightarrow \Grad(\fM_{\eC})$ is \'etale and hence flat. Therefore, $\Grad(\fM)\rightarrow \Grad(\fM_{\eC})$ is an open immersion. In the diagram
    \[
    \xymatrix{
    \fM^2 \ar^{\tilde{\Phi}_2}[r] \ar[d] & \Grad(\fM) \ar[d] \\
    \fM_{\eC}^2 \ar[r] & \Grad(\fM_{\eC})
    }
    \]
    we have proven that all morphisms except for $\tilde{\Phi}_2$ are open immersions. Therefore, $\tilde{\Phi}_2$ is an open immersion. The statement about $\tilde{\Phi}_3$ is proven analogously.

    \item This claim follows from Proposition \ref{prop:gradedmoduli}.
    \item As $\fM\rightarrow \fM_{\eC}$ is formally \'etale, it is enough to show that $\fM_{\eC}^3\rightarrow \Grad^2(\fM_{\eC})$ given by the inclusion of graded objects in weights $(0, 0)$, $(1, 0)$ and $(1, 1)$ lands in the matching locus. Using the description of the tangent complex of $\fM_{\eC}$ from \cite[Corollary 3.17]{tv07}, we see that $\bT_{\fM_{\eC}}|_{\fM_{\eC}^3}$ is concentrated in weights $(0, 0)$, $(\pm 1, 0)$, $(\pm 1, \pm 1)$ and $(0, \pm 1)$. In particular, the morphism $\fM_{\eC}^3\rightarrow \Grad^2(\fM_{\eC})$ lands in the matching locus.
    \item This claim follows from Lemma \ref{lem:Gradprojectionisomorphism}.
\end{enumerate}
\end{proof}

\begin{rmk}
Proposition \ref{prop:moduliattractorlagrangian} shows that for an open substack $\fM\subset \fM_{\eC}$ as in Proposition \ref{prop:modulilikemoduli}, the $(-1)$-shifted Lagrangian correspondence $\fM^2\leftarrow \fM^{2-\filt}\rightarrow \fM$ is a restriction of the $(-1)$-shifted Lagrangian correspondence $\fM_{\eC}^2\leftarrow \fM_{\eC}^{2-\filt}\rightarrow \fM_{\eC}$ from \cite[Corollary 6.5]{bdii}.
\end{rmk}

\subsection{Orientation data on deformed $3$-Calabi--Yau completions}\label{ssec:CYcompori}

We will now prove the existence of a canonical orientation data for deformed $3$-Calabi--Yau completions of a dg category \cite{kel11}.
Let $\eC$ be a finite type dg category and $\tilde{c} \in \HC^{-}_1(\eC)$ be a negative cyclic class whose underlying Hochschild class is denoted by $c$.
Following \cite[\S 5.1]{kel11}, we defined the deformed Calabi--Yau completion $\Pi_3(\eC, \tilde{c})$ of $\eC$ associated with $\tilde{c}$ by the following pushout diagram:
\[
\xymatrix{
{T_{\eC}(\eC^![2])} \ar[r]^-{[\id, 0]} \ar[d]^-{[\id, c]}
& {\eC} \ar[d] \\
{\eC} \ar[r]
& {\Pi_3(\eC, \tilde{c})} \pocorner
}
\]
where $\eC^!$ denotes the inverse dualizing bimodule and $T_{\eC}(\eC^![2])$ denotes the tensor algebra dg category.
It is shown in \cite[Theorem 5.9]{bcs20} that $\Pi_3(\eC, \tilde{c})$ is finite type and \cite{kel18} that the class $\tilde{c}$ induces a $3$-Calabi--Yau structure on $\Pi_3(\eC, \tilde{c})$.

\begin{ex}\label{ex:Ginzburg}
    Let $Q$ be a (finite) quiver.
    We let $\Rep_{Q}^{\Perf}$ denote the dg category of perfect $Q$-representations.
    Take a potential $W \in \HH_0(\Rep_{Q}^{\Perf})$.
    Then Ginzburg \cite[\S 5]{gin06} constructed a non-positively graded dg algebra $\Gamma_3(Q, W)$ whose zeroth cohomology is the Jacobi algebra $\mathrm{Jac}(Q, W)$.
    Then it is shown in \cite[Theorem 6.2]{kel11} that there exists a Morita equivalence
    \[
    \Pi_3(\Rep_{Q}^{\Perf}, \delta W) \simeq \Mod_{\Gamma_3(Q, W)}^{\Perf}.
    \]
    where $\delta$ denotes the mixed differential.
\end{ex}

\begin{ex}\label{ex:local_surface}
    Let $S$ be a smooth algebraic surface and set $X \coloneqq \Tot_S(\omega_S)$.
    A Calabi--Yau threefold of this form is called a \defterm{local surface}.
    Then it is shown in \cite[Corollary 4.2]{km21} that there exists a natural Morita equivalence of $3$-Calabi--Yau categories
    \[
    \Pi_3(\Perf(S), 0) \simeq \Perf(X).
    \]
\end{ex}

\begin{ex}\label{ex:local_curve}
    Let $C$ be a smooth projective curve and $N$ be a rank two vector bundle with $\det(N) \simeq \omega_C$ and set $X \coloneqq \Tot_C(N)$.
    A Calabi--Yau threefold of this form is called a \defterm{local curve}.
    Then by \cite[Lemma 3.3]{kk21} we may find a line bundle $L$ with a surjective morphism $N \to L$.
    Set $Y \coloneqq \Tot_C(L)$. Then it is shown in \cite[Theorem 4.3]{km21} there exists a Hochschild homology class $c \in \HH_0(\Perf(Y))$ with a Morita equivalence of $3$-Calabi--Yau categories
    \[
    \Pi_3(\Perf(Y), \delta c) \simeq \Perf(X).
    \]
\end{ex}

Let $\eC$ be a finite type dg category and $c \in \HH_0(\eC)$ be a Hochschild homology class.
We have a natural map
    \begin{equation}\label{eq:HH_to_function}
    |\HH(\eC)| \to \Gamma(\fM_{\eC}, \cO_{\fM_{\eC}})
    \end{equation}
as constructed in \cite[\S 5.2]{bdii}
and let $f_c \colon \fM_{\eC} \to \bA^1$ be the corresponding function. 
Then it is shown in \cite[Corollary 6.19]{bcs20} that there exists a natural equivalence of $(-1)$-shifted symplectic stacks
\begin{equation}\label{eq:BCS}
\fM_{\Pi_3(\eC, \delta c)} \simeq \Crit(f_c)
\end{equation}
where the $(-1)$-shifted symplectic structure for the left-hand side is constructed using  \eqref{eq:HC_to_forms} and for the right-hand side is constructed using Example \ref{ex:critical}.
Using this equivalence, we will prove the following result:

\begin{prop}\label{prop:strong_orientation_completion}
    The deformed $3$-Calabi--Yau completion $\Pi_3(\eC, \delta c)$ admits a strong orientation data.
\end{prop}

\begin{proof}
        We first discuss the generality of the canonical orientation for the derived critical locus defined in Example \ref{ex:can_ori}.
        If we are given regular functions $f_1$ and $f_2$ on derived Artin stacks $\fY_1$ and $\fY_2$ with $\Crit(f_i) = \fX_i$, there exists a natural isomorphism
        \begin{equation}\label{eq:canonical_orientation_product_general}
        o_{\fX_1}^{\can} \boxtimes o_{\fX_2}^{\can} \simeq o_{\fX_1 \times \fX_2}^{\can}
        \end{equation}
        where the canonical orientation $o_{\fX_1 \times \fX_2}^{\can}$ is constructed using the critical locus description $\fX_1 \times \fX_2 \simeq \Crit(f_1 \boxplus f_2)$.
        Further, if we are given yet another derived Artin stack $\fY_3$ with a regular function $f_3$ on it, if we write $\fX_3 = \Crit(f_3)$, the following diagram commutes:
        \begin{equation}\label{eq:product_can_associative}
        \begin{aligned}
        \xymatrix{
        {o_{\fX_1}^{\can} \boxtimes o_{\fX_2}^{\can}  \boxtimes o_{\fX_3}^{\can}} \ar[r]^-{\simeq} \ar[d]^-{\simeq}
        & {o_{\fX_1 \times \fX_2} ^{\can} \boxtimes o_{\fX_3}^{\can}} \ar[d]^-{\simeq} \\
        {o_{\fX_1}^{\can} \boxtimes o_{\fX_2 \times \fX_3} ^{\can}} \ar[r]^-{\simeq}
        & {o_{\fX_1 \times \fX_2 \times \fX_3}^{\can}.}
        }
        \end{aligned}
        \end{equation}
        
        We now claim that the localization of the canonical orientation for $\fX = \Crit(f)$ to $\Grad(\fX)$ is again a canonical orientation.
        By  Lemma \ref{lem:crit_localize}, there exists an equivalence of $(-1)$-shifted symplectic stacks
        $\Grad(\fX) \simeq \Crit(f |_{\Grad(\fX)})$.
        The direct sum decomposition $\bL_{\fY} |_{\Grad(\fY)} = \bL_{\Grad(\fY)} \oplus \bL_{\fY} |_{\Grad(\fY)}^{\pm}$ 
        induces a natural isomorphism
        \[
        \rdet(\bL_{\fY}) \otimes \rdet(\bL_{\fY} |_{\Grad(\fY)}^{\pm})^{\vee } \simeq \rdet(\bL_{\Grad(\fY)})
        \]
        which induces an isomorphism of orientations 
        \begin{equation}\label{eq:ori_localize_canonical}
        u^{\star} o_{\fX}^{\can} \simeq o_{\Grad(\fX)}^{\can}.
        \end{equation}
        See \eqref{eq:can_ori_localize_stack}
        for a related discussion.         
        It is clear that the following diagram commutes:
        \begin{equation}\label{eq:can_ori_localize_assoc}
            \begin{aligned}
             \xymatrix{
             {u_1^{\star} u^{\star} o_{\fX}^{\mathrm{can}} |_{\Grad^2(\fX)_{\mat}} }
             \ar[d]^-{\cong}
             \ar[r]^-{\cong}
             & {u_1^{\star} o_{\Grad(\fX)}^{\can} |_{\Grad^2(\fX)_{\mat}}}
             \ar[r]^-{\cong}
             & {o_{\Grad^2(\fX)_{\mat}}^{\can}}
             \ar@{=}[d]
             \\
             u_2^{\star} u^{\star} o_{\fX}^{\mathrm{can} } |_{\Grad^2(\fX)_{\mat}}
             \ar[r]^-{\cong}
             & {u_2^{\star} o_{\Grad(\fX)}^{\can}} |_{\Grad^2(\fX)_{\mat}}
             \ar[r]^-{\cong}
             & {o_{\Grad^2(\fX)_{\mat}}^{\can}.}
             }
            \end{aligned}
        \end{equation}
        If we are given derived Artin stacks $\fX = \Crit(f \colon \fY \to \bA^1)$, $\fX' = \Crit(f' \colon \fY' \to \bA^1)$, 
         the following diagram commutes by construction:
        \begin{equation}\label{eq:ori_product_localize_compati}
        \begin{aligned}
        \xymatrix{
        {(u \times u')^{\star}(o_{\fX}^{\can} \boxtimes o_{\fX'}^{\can})} \ar[r]^-{\simeq} \ar[d]^-{\simeq}
        & {(u \times u')^{\star}(o_{\fX \times \fX'}^{\can})} \ar[d]^{\simeq} \\
        {u^{\star} o_{\fX}^{\can} \boxtimes u'^{\star} o_{\fX'}^{\can}  } \ar[r]^-{\simeq}
        & {o_{\Grad(\fX) \times 
        \Grad(\fX')}^{\can}.}
        }
        \end{aligned}
        \end{equation}

        We now return to the statement of the proposition.
        We equip $\fM_{\Pi_3(\eC, \delta_c )} \simeq \Crit(f_c)$ with the canonical orientation described above.
        Consider the diagonal map $\Delta \colon \eC \to \eC \times \eC$ .
        Under the identification of $\HH(\eC \times \eC) \simeq \HH(\eC) \oplus \HH(\eC)$,
        we can identify $\Delta_* c$ and $(c, c)$.
        In particular, the composition
        \[
        \fM_{\eC} \times \fM_{\eC} \xrightarrow{{\Phi_2}}  \fM_{\eC} \xrightarrow{f_c} \bA^1
        \]
        is given by $f_{c} \boxplus f_{c}$.
        Then by using \eqref{eq:canonical_orientation_product_general} and \eqref{eq:ori_localize_canonical},
        we obtain a natural isomorphism 
        \[
        \Phi_2^{\star} (o_{\fM_{\Pi_3(\eC, \delta_c )}}^{\can}) \simeq o_{\fM_{\Pi_3(\eC, \delta_c )}}^{\can, \boxtimes 2}.
        \]
        Then the existence of the strong orientation data follows from the commutativity of the diagrams \eqref{eq:product_can_associative}, \eqref{eq:can_ori_localize_assoc} and \eqref{eq:ori_product_localize_compati}.

\end{proof}

\begin{rmk}
    By the same proof, we can show that the deformed $3$-Calabi--Yau completion $\Pi_3(\eC, \tilde{c})$ admits a strong orientation data for arbitrary negative cyclic class $\tilde{c} \colon k \to |\HC^{-}(\eC)|$.
\end{rmk}

The following statement is a consequence of Proposition \ref{prop:strong_orientation_completion} and Example \ref{ex:Ginzburg}, \ref{ex:local_surface} and \ref{ex:local_curve}:

\begin{cor}\label{cor:strong_orientation_completion}
    The following $3$-Calabi--Yau categories admit a strong orientation data:
    \begin{itemize}
        \item[(1)] The dg category of perfect modules $\Mod_{\Gamma_3(Q, W)}^{\Perf}$ over the Ginzburg dg algebra $\Gamma_3(Q, W)$ associated with a quiver with potential $(Q, W)$.

        \item[(2)] The dg category of perfect complexes $\Perf(X)$ for a local surface $X = \Tot_C(\omega_C)$.

          \item[(3)] The dg category of perfect complexes $\Perf(X)$ for a local curve $X = \Tot_C(N)$.
    \end{itemize}
\end{cor}

\subsection{Comparison with the Kontsevich--Soibelman's CoHA for quivers with potentials}\label{ssec:comparisonKS}

In this section, we will prove that the cohomological Hall algebra structure for the cohomological 
Donaldson--Thomas invariants for quivers with potentials constructed in Corollary \ref{cor:CoHA} coincide with the one constructed in \cite[\S 7.6]{ks10}.

Let $Q$ be a finite quiver
with $Q_0$ the set of vertices and $Q_1$ the set of edges.
For a dimension vector $ \mathbf{d} \in \bZ_{\geq 0}^{Q_0}$, we let $\fM_{\mathbf{d}}$ denote the moduli stack of representations of $Q$ with dimension vector $\mathbf{d}$.
We define a pairing $(-, -)_Q$ of dimension vectors by
\[
(\mathbf{d}, \mathbf{d}')_Q \coloneqq \sum_{(e \colon i \to j) \in Q_1} \mathbf{d}_i \mathbf{d}_j' - \mathbf{d}_j \mathbf{d}_i'.
\]
Consider the following correspondence:
\[
\xymatrix{
{}
& \fM_{\mathbf{d}, \mathbf{d}'}^{2-\filt}
\ar[ld]_-{\gr}
\ar[rd]^-{\ev}
&  \\
{\fM_{\mathbf{d} } \times \fM_{\mathbf{d}' } }
& 
& {\fM_{\mathbf{d} + \mathbf{d}'}.} 
}
\]
Since all spaces appearing in this diagram are smooth,
we obtain a natural isomorphism
\begin{equation}\label{eq:quiver_without_potential_attractor}
\bQ_{\fM_{\mathbf{d}}}^{\mathrm{Perv}} \boxtimes \bQ_{\fM_{\mathbf{d}'}}^{\mathrm{Perv}} [(\mathbf{d}, \mathbf{d}')_{Q}] \cong \gr_* \ev^! \bQ_{\fM_{\mathbf{d} + \mathbf{d}'}}^{\mathrm{Perv}}
\end{equation}
where  $\bQ^{\mathrm{Perv}}$ denotes the constant sheaf shifted to be perverse.
Assume now that we are given a potential $W \colon \bC \to |\HH(\Rep_Q^{\Perf})|$
and $f_{W, \mathbf{d}} \colon \fM_{Q, \mathbf{d}} \to \bA^1$ be the associated function.
Then it is clear that we have an identity $f_{W, \mathbf{d} + \mathbf{d}'} \circ \ev = (f_{W, \mathbf{d}} \boxplus f_{W, \mathbf{d}'}) \circ \gr $.
Therefore, by applying the vanishing cycle functor $\varphi_{f_{W, \mathbf{d}} \boxplus f_{W, \mathbf{d}'}}$
to the isomorphism \eqref{eq:quiver_without_potential_attractor} and
using the Thom--Sebastiani isomorphism \eqref{eq:Thom_Sebastiani_functor} and \eqref{eq:van_natural}, we obtain the following map\footnote{
One can show that this map is an isomorphism by repeating the proof of Theorem \ref{thm:main_thm}
or by Proposition \ref{prop:Braden_appendix}.
} of shifted perverse sheaves
\begin{equation}\label{eq:quiver_integrality}
\varphi_{f_{W}, \mathbf{d}} \boxtimes \varphi_{f_{W}, \mathbf{d}'} [(\mathbf{d}, \mathbf{d}')_Q]
\to \gr_* \ev^! \varphi_{f_{W}, \mathbf{d} + \mathbf{d}'}.
\end{equation}
Write $\cH_{Q, W, \mathbf{d}} \coloneqq H^*(\fM_{\mathbf{d}}, \varphi_{f_{W}, \mathbf{d}})$.
By passing the map \eqref{eq:quiver_integrality} to the left adjoint and using the properness of $\ev$, we obtain a morphism
\[
*^{\Hall}_{\mathrm{KS}} \colon \cH_{Q, W, \mathbf{d}} \otimes^{\mathrm{tw}} \cH_{Q, W, \mathbf{d}'} \to \cH_{Q, W, \mathbf{d} + \mathbf{d}'},
\]
which defines an algebra structure on $\cH_{Q, W} \coloneqq \bigoplus_{\mathbf{d}} \cH_{Q, W, \mathbf{d}}$ and is called Kontsevich--Soibelman's cohomological Hall algebras for $(Q, W)$.

Now take $\eC = \Mod^{\Perf}_{\Gamma_3(Q, W)}$ where $\Gamma_3(Q, W)$ denotes the Gizburg dg algebra.
Let $\fM \subset \fM_{\eC}$ denote the locus consisting of modules concentrated in degree zero.
Then there exists a natural identification $\pi_0(\fM) \cong \mathbb{Z}_{\geq 0}^{Q_0}$
and we have natural isomorphism of $(-1)$-shifted symplectic stacks
\[
\fM \cong \coprod_{\mathbf{d} \in \bZ_{\geq 0}^{Q_0}} \Crit(f_{W, \mathbf{d}})
\]
as we have seen in \eqref{eq:BCS}. We let $\fM_{\mathbf{d}} \subset \fM$ denotes the open and closed substack corresponding to $\Crit(f_{W, \mathbf{d}})$.
We equip $\fM$ with the strong orientation data constructed in Proposition \ref{prop:strong_orientation_completion} and write
\[
\cH_{Q, W}^{3\text{-}\mathrm{CY}} \coloneqq \bigoplus_{\mathbf{d} \in \bZ^{Q_0}} H^{\bullet}(\fM_{\mathbf{d}}, \varphi_{\fM_{\mathbf{d}}}).
\]
We denote the multiplication constructed in Corollary $\ref{cor:CoHA}$ by $*^{\mathrm{Hall}}_{3\text{-}\mathrm{CY}}$.
The following statement is an immediate consequence of Lemma \ref{lem:can_ori_vanishing_stack}
 and Proposition \ref{prop:compati_CoHA_pre}:

 \begin{prop}\label{prop:CoHAcomparison}
     There exists a natural isomorphism of associative algebras
     \[
     (\cH_{Q, W}, *^{\mathrm{Hall}}_{\mathrm{KS}}) \cong
     (\cH_{Q, W}^{3\text{-}\mathrm{CY}}, *^{\mathrm{Hall}}_{3\text{-}\mathrm{CY}}).
     \]
 \end{prop}

\section{Applications to cohomological DT invariants of $3$-manifolds}

In this section we explain how to construct a parabolic induction map for cohomological DT invariants of $3$-manifolds as in \cite[Theorem 3.47]{ns23} as well as define a cohomological Hall algebra structure.

\subsection{Local systems}

Let $M\in\eS$ be a space and $\LocSys(M)=\Fun(M, \Mod_\C)$ the presentable stable $\infty$-category of chain complexes of local systems of $\C$-vector spaces on $M$. Let $\eC=\LocSys(M)^\omega\subset \LocSys(M)$ be the subcategory of compact objects. For a morphism $f\colon M\rightarrow N$ of spaces we have the restriction functor $f^*\colon \LocSys(N)^\omega\rightarrow \LocSys(M)^\omega$ and its left adjoint $f_\sharp\colon \LocSys(M)^\omega\rightarrow \LocSys(N)^\omega$. Recall that $M$ is \defterm{finitely dominated} if it is a compact object in $\eS$. In this case, we have the following objects:
\begin{itemize}
    \item The constant local system $\C_M\in\LocSys(M)$ with fibre $\C$ is compact, i.e. it defines an object $\C_M\in\LocSys(M)^\omega$.
    \item There is an orientation local system $\zeta_M\in\LocSys(M)^\omega$ which represents the functor of homology $\rC_\bullet(M; -)\colon \LocSys(M)^\omega\rightarrow \Perf_\C$.
\end{itemize}

\begin{lem}\label{lem:connectedcoproduct}
Let $M$ be a connected space. Then the functor $\Map(M_\B, -)\colon \dSt\rightarrow \dSt$ commutes with coproducts.
\end{lem}
\begin{proof}
Let $I$ be a set. The colimit functor $\Fun(I, \dSt)\rightarrow \dSt$ lifts to an equivalence $\Fun(I, \dSt)\cong \dSt_{/I_\B}$. As $\Map(M_\B, -)$ is an $M$-indexed limit, we have to show that the forgetful functor $U\colon \dSt_{/I_\B}\rightarrow \dSt$ preserves $M$-indexed limits. For a diagram $F\colon M\rightarrow \dSt_{/I_\B}$ we have
\[U\left(\lim_{x\in M} F(x)\right)\cong \left(\lim_{x\in M} UF(x)\right)\times_{\lim_{x\in M} I_\B} I_\B.\]
The limit $\lim_{x\in M} I_\B$ is equivalent to $\Hom_{\eS}(M, I)_B$ and since $I$ is discrete, the constant map $I\rightarrow \Hom_{\eS}(M, I)$ is an equivalence.
\end{proof}

\begin{defin}
An \defterm{oriented $d$-dimensional Poincar\'e space} is a finitely dominated space $M\in\eS$ equipped with a fundamental class $[M]\in\rC_d(M; \C)$ such that the induced morphism $\zeta_M\rightarrow \C_M[-d]$ is an isomorphism.
\end{defin}

It is shown in \cite[Proposition 5.1]{bdi} that if $M$ is finitely dominated, $\LocSys(M)^\omega$ is a finite type dg category. Moreover, if $M$ is an oriented $d$-dimensional Poincar\'e space, the image of the fundamental class under the inclusion of constant loops
\[i\colon \rC_\bullet(M; \C)\longrightarrow (\rC_\bullet(LM; \C))^{S^1}\simeq \HC^{-}(\LocSys(M)^\omega)\]
defines a $d$-Calabi--Yau structure on $\LocSys(M)^\omega$ by \cite[Theorem 5.4]{bdi}.

\begin{thm}\label{thm:Mlocsys}
Let $M\in\eS$ be a finitely dominated space. Then there is an equivalence of derived stacks
\[\Map(M_\B, \Perf)\simeq \fM_{\LocSys(M)^\omega}.\]
Moreover, if $M$ is an oriented $d$-dimensional Poincar\'e space, this is an equivalence of $(2-d)$-shifted symplectic stacks, where on the left the $(2-d)$-shifted symplectic structure is induced by the AKSZ procedure from the $2$-shifted symplectic structure on $\Perf$ and the $d$-orientation of $M_\B$ while on the right the $(2-d)$-shifted symplectic structure is defined by the $d$-Calabi--Yau structure on $\LocSys(M)^\omega$.
\end{thm}
\begin{proof}
Consider the functor, natural in $M\in\eS$,
\[M\longrightarrow \LocSys(M)^\omega\]
given by $\pt\mapsto \C\in\Perf_\C$ for $M=\pt$. Such a functor is unique as $M=\colim_{x\in M} \pt$. This functor defines a universal object
\[E_M\in\Perf(M_\B; \LocSys(M)^\omega)\cong\lim_{x\in M} (\LocSys(M)^\omega).\]
Applying Proposition \ref{prop:modulitomap} we obtain a morphism
\[F_{E_M}\colon \fM_{\LocSys(M)^\omega}\longrightarrow \Map(M_\B, \Perf).\]
The object $E_M$ is functorial in $M$ in the sense that it is the value on $M$ of a universal object
\[E\in\int_{M\in\eS}\Perf(M_\B; \LocSys(M)^\omega).\]
The construction $M\mapsto \LocSys(M)^\omega$ defines a functor $\eS\rightarrow \dgcat$ which preserves colimits. Therefore, both $\Map(M_\B, \Perf)$ and $\fM_{\LocSys(M)^\omega}$ viewed as functors $\eS^{\op}\rightarrow \dSt$ send colimits of spaces to limits of derived stacks. Moreover, $F_{E_M}$ is obviously an isomorphism for $M=\pt$ and hence it is an isomorphism for any $M$. This proves the first statement.

Next, suppose $M$ is an oriented $d$-dimensional Poincar\'e space with fundamental class $[M]\in\rC_d(M; \C)$. Using Theorem \ref{thm:modulitomapsymplectic} we get that the morphism
\[\fM_{\LocSys(M)^\omega}\longrightarrow \Map(M_\B, \Perf)\]
is compatible with $(2-d)$-shifted symplectic structures, with a certain $d$-Calabi--Yau structure $[\LocSys(M)^\omega]$ on $\LocSys(M)^\omega$ described in the statement. It remains to show that
\[[\LocSys(M)^\omega]\sim i([M])\in|\HC^-(\LocSys(M)^\omega)[-d]|.\]
By definition the class $[\LocSys(M)^\omega]$ is given by the composite
\begin{align*}
\C&\xrightarrow{[E_M]} \HC^-(\Perf(M_\B; \LocSys(M)^\omega)) \\
&\xrightarrow{\kappa_{\LocSys(M)^\omega, M_\B}} \rC^\bullet(M)\otimes \HC^-(\LocSys(M)^\omega) \\
&\xrightarrow{[M]\otimes \id} \HC^-(\LocSys(M)^\omega)[-d]
\end{align*}
By definition of $E_M$ this composite is equivalent to the image of $[M]$ under a map
\[\rC_\bullet(M;\C[\![u]\!])\longrightarrow \rC_\bullet(LM;\C)^{S^1}\cong \HC^-(\LocSys(M)^\omega)\]
which is natural in $M$ and which is the identity for $M=\pt$. The inclusion of constant loops
\[i\colon \rC_\bullet(M;\C[\![u]\!])\longrightarrow \rC_\bullet(LM;\C)^{S^1}\]
is also natural in $M$ and is the identity for $M=\pt$. But since $\rC_\bullet(M;\C[\![u]\!])$ preserves colimits in $M$, such a natural transformation is unique, which implies the claim.
\end{proof}

\begin{ex}
Let $M=S^1$ be the circle which has a fundamental class $[S^1]\in\rC_1(S^1; \C)$. Then we can identify $\LocSys(S^1)^\omega\cong \Perf_{\C[x^{\pm 1}]}$ and Theorem \ref{thm:Mlocsys} in this case reduces to a $1$-shifted symplectic equivalence
\[\Map(S^1_\B, \Perf)\cong \fM_{\Perf_{\C[x^{\pm 1}]}}\]
from \cite[Proposition 5.1]{bcs23}.
\end{ex}

\subsection{Parabolic restriction for character stacks and its orientation}

Let us now consider a more general setting. For an algebraic group $G$ we denote by
\[\Loc_G(M) = \Map(M_\B, BG)\]
the \defterm{character stack}. Fixing a point $x\in M$, we have the \defterm{framed character stack}
\[\Loc^{\fr}_G(M) = \Loc_G(M) \times_{B G} \pt,\]
where the morphism $\Loc_G(M)\rightarrow B G$ is given by evaluation at $x$. If $M$ is connected, the framed character stack is represented by an affine derived scheme and
\[\Loc_G(M)\cong [\Loc^{\fr}_G(M) / G].\]

Recall the following facts about shifted symplectic structures on classifying stacks:
\begin{itemize}
    \item If $G$ is equipped with a non-degenerate $G$-invariant symmetric bilinear pairing $c\in\Sym^2(\fg^*)^G$ on its Lie algebra $\fg$, there is a canonical $2$-shifted symplectic structure on $BG$.
    \item If $G$ is connected reductive with $P\subset G$ a parabolic subgroup with Levi factor $L$, then there is a unique $2$-shifted Lagrangian structure on the correspondence
    \begin{equation}\label{eq:parabolicLagrangian}
    \xymatrix{
    & BP \ar[dl] \ar[dr] & \\
    BL && BG
    }
    \end{equation}
    which restricts to the chosen $2$-shifted symplectic structure on $BG$, see \cite[Lemma 3.4]{saf17}.
\end{itemize}

\begin{prop}\label{prop:parabolicintersectionLagrangian}
Let $G$ be a connected reductive group equipped with a non-degenerate pairing $c$ and $P_1, P_2\subset G$ parabolic subgroups with the same Levi factor $L\subset G$. Then the diagram
\begin{equation}\label{eq:parabolicintersectiondiagram}
\xymatrix{
& BP_1 \ar[dl] \ar[dr] & \\
BL & B(P_1\cap P_2) \ar[u] \ar[d] & BG \\
& BP_2 \ar[ul] \ar[ur]
}
\end{equation}
has a unique structure of a 2-fold $2$-shifted Lagrangian correspondence, which restricts to the chosen $2$-shifted symplectic structure on $BG$.
\end{prop}
\begin{proof}
If $H$ is an affine algebraic group, the space $\cA^{2, \cl}(BH, 2)$ of closed two-forms on $BH$ of degree $2$ is equivalent to the set $\Sym^2(\mathfrak{h}^*)^H$ of $H$-invariant symmetric bilinear pairings on the Lie algebra of $H$.

We first prove that there is a unique structure of a 2-fold $2$-shifted isotropic correspondence which restricts to the chosen $2$-shifted presymplectic structure on $BG$. Let $\fp_1$, $\fp_2$ and $\fp_{12}$ be the Lie algebras of $P_1$, $P_2$ and $P_1\cap P_2$. Consider the correspondence
\[
\xymatrix{
& \Sym^2(\fp_1^*)^{P_1} \ar[d] & \\
\Sym^2(\fl^*)^L \ar[ur] \ar[dr] & \Sym^2(\fp_{12}^*)^{P_1\cap P_2} & \Sym^2(\fg^*)^G \ar[ul] \ar[dl] \\
& \Sym^2(\fp_2^*)^{P_2} \ar[u] &
}
\]
given by the pullback of closed two-forms along the diagram 
\eqref{eq:parabolicintersectiondiagram}.

Using \cite[Theorem 6.2.7]{springer} we may choose Borel subgroups $B_i\subset G$ contained in $P_i$. By \cite[Corollary 8.3.10]{springer} the intersection $B_1\cap B_2$ contains a maximal torus $T\subset G$; let $\ft$ be its Lie algebra. Let
\[\fp_1 = \fl\oplus \fu_1,\qquad \fp_2 = \fl\oplus \fu_2\]
be $L$-invariant splittings.

Let $(-, -)$ be an $L$-invariant pairing on $\fp_i$. For any $x,y\in\fp_i$ of $T$-weights $\alpha,\beta$ we have $(\alpha+\beta)(x, y) = 0$. In particular, this implies that $(\fu_i, \fp_i) = 0$. Thus, the map $\Sym^2(\fl^*)^L\rightarrow \Sym^2(\fp_i^*)^L$ is an isomorphism. We have $[\fu_i,\fp_i]\subset \fu_i$, so the map $\Sym^2(\fp_i^*)^{P_i}\rightarrow \Sym^2(\fp_i^*)^L$ is an isomorphism. An analogous argument shows that $\Sym^2(\fl^*)^L\rightarrow \Sym^2(\fp_{12}^*)^{P_1\cap P_2}$ is an isomorphism. The two maps $\Sym^2(\fg^*)^G\rightarrow \Sym^2(\fp_i^*)^{P_i}\xleftarrow{\sim} \Sym^2(\fl^*)^L$ are equal to the restriction along $L\subset G$. Since the restriction $\Sym^2(\fg^*)^G\rightarrow \Sym^2(\ft^*)$ is injective, so is the restriction $\Sym^2(\fg^*)^G\rightarrow \Sym^2(\fp_i^*)^{P_i}$. This finishes the construction of the unique structure of a 2-fold $2$-shifted isotropic correspondence.

Let us now show that if the chosen pairing on $\fg$ is non-degenerate, then this 2-fold $2$-shifted isotropic correspondence is Lagrangian. The fact that $BL\leftarrow BP_i\rightarrow BG$ is a $2$-shifted Lagrangian correspondence is shown in \cite[Lemma 3.4]{saf17}. So, we have to check that $B(P_1\cap P_2)\rightarrow BP_1\times_{BG\times BL} BP_2$ is a $1$-shifted Lagrangian morphism. This boils down to showing that the complex
\[\fp_{12}\rightarrow \fp_1\oplus \fp_2\rightarrow \fg\oplus \fl\rightarrow \fp_{12}^*\]
concentrated in degrees $[0, 3]$ is acyclic. Since $\fp_{12}\rightarrow \fp_1\oplus \fp_2$ is injective, $H^0$ vanishes. Its $H^1$ vanishes since $(\fp_1\oplus \fp_2) / \fp_{12}\rightarrow \fg$ is injective. Its $H^3$ vanishes since the composite $\fg\cong \fg^*\rightarrow \fp_{12}^*$ is surjective. The Euler characteristic of this complex is $\dim\fg + \dim\fl - \dim \fp_1 - \dim\fp_2$. Since $\dim \fp_i = (\dim \fg + \dim \fl)/2$, the Euler characteristic of this complex vanishes and hence $H^2=0$, i.e. this complex is acyclic.
\end{proof}

In particular, if $M$ is a closed oriented $3$-manifold and $G$ is equipped with a non-degenerate pairing on its Lie algebra, there is a $(-1)$-shifted symplectic structure on $\Loc_G(M)$ defined by the AKSZ procedure. Let us recall the orientation on character stacks following \cite{ns23}.

\begin{prop}\label{prop:characterstackorientation}
Let $M$ be a closed spin $3$-manifold.
\begin{enumerate}
    \item For a connected reductive group $G$ equipped with a non-degenerate pairing $c\in\Sym^2(\fg^*)^G$ there is an orientation $o_G$ on $\Loc_G(M)$.
    \item For a pair of connected reductive groups $G_1, G_2$ equipped with non-degenerate pairings $c_i\in\Sym^2(\fg_i^*)^{G_i}$ there is an isomorphism $o_{G_1\times G_2}\simeq o_{G_1}\boxtimes o_{G_2}$ of orientation on $\Loc_{G_1\times G_2}(M)\simeq \Loc_{G_1}(M)\times \Loc_{G_2}(M)$ which satisfies an associativity relation.
    \item For a parabolic subgroup $P\subset G$ whose Levi factor is $L$ there is an orientation $\rho_P$ of the Lagrangian correspondence
    \[
    \xymatrix{
    & \Loc_P(M) \ar[dl] \ar[dr] & \\
    \Loc_L(M) && \Loc_G(M)
    }
    \]
    which restricts to $o_L$ and $o_G$ on $\Loc_L(M)$ and $\Loc_G(M)$. For $P=G$ the orientation $\rho_P$ is the identity orientation of the diagonal Lagrangian correspondence. The orientation $\rho_P$ for conjugate parabolic subgroups are equal.
    \item Given a pair of parabolic subgroups $P_1\subset G_1$ and $P_2\subset G_2$ with Levi factors $L_1$ and $L_2$, under the isomorphisms $\lambda_{L_1, L_2}$ and $\lambda_{G_1, G_2}$ we have an equality $\rho_{P_1}\boxtimes \rho_{P_2} = \rho_{P_1\times P_2}$ of orientations of the product correspondence
    \[
    \xymatrix{
    & \Loc_{P_1}(M)\times \Loc_{P_2}(M) \ar[dl] \ar[dr] & \\
    \Loc_{L_1}(M)\times \Loc_{L_2}(M) && \Loc_{G_1}(M)\times \Loc_{G_2}(M)
    }
    \]
\end{enumerate}
\end{prop}
\begin{proof}
The claim is essentially shown in \cite[Theorem 3.45]{ns23} to which we refer for details. Let us sketch how to construct the relevant compatibility isomorphisms. First, as explained in \cite[Remark 3.15]{ns23}, there is a canonical homology orientation of $M$, i.e. a trivialization of the (graded) line $\det H_\bullet(M; \C)$ (note that it lies in degree zero as $\chi(M)=0$). The spin structure on $M$ allows us to choose a canonical Euler structure in the sense of \cite[Section 3.2]{ft99}, i.e. a trivialization of the homological Euler class $\xi\colon e(M)\sim 0$ whose characteristic class is $c(\xi)=0\in H_1(M; \Z)$. Equivalently, it is the choice of a $\mathrm{Spin}^c$ structure with the trivial $c_1$. Let $G$ be an algebraic group and consider the derived stack $\fX=\Loc_G(M)$. By \cite[Proposition 3.23]{ns23} we get that $\vdim(\Loc_G(M)) = 0$ and there is a trivialization of the determinant of the cotangent complex of $\Loc_G(M)$ given by the torsion volume form $\vol_G$. Let $\overline{\vol}_G$ be the corresponding section of $\rdet(\bL_{\fX})$. For a pair of algebraic groups $G_1, G_2$ we have an equality $\overline{\vol}_{G_1\times G_2} = \overline{\vol}_{G_1}\wedge \overline{\vol}_{G_2}$ as sections of
\[\rdet(\bL_{\Loc_{G_1\times G_2}(M)})\simeq \rdet(\bL_{\Loc_{G_1}(M)}) \boxtimes \rdet(\bL_{\Loc_{G_2}(M)}).\]
Note that for sections of line bundles on classical stacks there is no room for nontrivial homotopies).

\begin{enumerate}
    \item We define the orientation $o_G$ to be given by the structure sheaf $\cO_{\Loc_G(M)^{\red}}$ equipped with an isomorphism $\cO_{\Loc_G(M)^{\red}}^{\otimes 2}\simeq \rdet(\bL_{\Loc_G(M)})$ given by $1\mapsto \overline{\vol}_G$.
    \item We have an isomorphism
    \[\cO_{\Loc_{G_1\times G_2}(M)^{\red}}\simeq \cO_{\Loc_{G_1}(M)^{\red}}\boxtimes \cO_{\Loc_{G_2}(M)^{\red}}\]
    given by $1\mapsto 1\boxtimes 1$. The equality $\overline{\vol}_{G_1\times G_2} = \overline{\vol}_{G_1}\wedge \overline{\vol}_{G_2}$ then implies that it defines an isomorphism of orientations $o_{G_1\times G_2}\simeq o_{G_1}\boxtimes o_{G_2}$. The associativity of this isomorphism follows from the associativity of the above isomorphism of structure sheaves.
    \item We define the orientation $\rho_P$ using the isomorphism $\cO_{\Loc_P(M)^{\red}}\simeq \rdet(\bL_{\Loc_G(M)})$ given by $1\mapsto \overline{\vol}_P$. It is shown in \cite[Theorem 3.45]{ns23} that it is compatible with the orientation $o_G$ and $o_L$ of $\Loc_G(M)$ and $\Loc_L(M)$. Since it depends on the stack $BP$ and not on the subgroup $P$, the orientation $\rho_P$ is invariant under conjugation.
    \item The multiplicativity of $\rho_P$ follows from the relation $\overline{\vol}_{P_1\times P_2}=\overline{\vol}_{P_1}\wedge \overline{\vol}_{P_2}$.
\end{enumerate}
\end{proof}

So, if $M$ is a closed spin $3$-manifold and $G$ a connected reductive group equipped with a non-degenerate pairing $c\in\Sym^2(\fg^*)^G$, we obtain a perverse sheaf 
\[\varphi_{\Loc_G(M)}\in\Perv(\Loc_G(M))\]
given by the DT perverse sheaf associated to the $(-1)$-shifted symplectic derived Artin stack $\Loc_G(M)$ equipped with the orientation $o_G$ from Proposition \ref{prop:characterstackorientation}. For a pair $G_1,G_2$ of such groups, the Thom--Sebastiani isomorphism from Proposition \ref{prop:TS_DT_stack} as well as the isomorphism of orientations $o_{G_1\times G_2}\simeq o_{G_1}\boxtimes o_{G_2}$ from Proposition \ref{prop:characterstackorientation} gives an isomorphism
\[
\varphi_{\Loc_{G_1\times G_2}(M)}\cong \varphi_{\Loc_{G_1}(M)}\boxtimes \varphi_{\Loc_{G_2}(M)}
\]
of perverse sheaves on $\Loc_{G_1\times G_2}(M)\simeq \Loc_{G_1}(M)\times \Loc_{G_2}(M)$ which satisfies an associativity relation.

One can compose parabolic subgroups in the following way.

\begin{lem}\label{lem:paraboliccomposition}
Let $G_3$ be a connected reductive group with a parabolic subgroup $P_{23}\subset G_3$ with Levi factor $G_2$ and $P_{12}\subset G_2$ is a parabolic subgroup with Levi factor $G_1$. Consider the fibre product $P_{13} = P_{12}\times_{G_2} P_{23}$. Then $P_{13}\subset G_3$ is a parabolic subgroup with Levi factor $G_1$.
\end{lem}
\begin{proof}
Consider the diagram
\[
\xymatrix{
&& BP_{13} \ar[dl] \ar[dr] && \\
& BP_{12} \ar[dl] \ar[dr] && BP_{23} \ar[dl] \ar[dr] & \\
BG_1 && BG_2 && BG_3
}
\]
where the top square is Cartesian. Since $BP_{12}\rightarrow BG_2$ is proper, the morphism $BP_{13}\rightarrow BP_{23}$ is proper. Since $BP_{23}\rightarrow BG_3$ is also proper, this implies that $BP_{13}\rightarrow BG_3$ is proper, i.e. $P_{13}\subset G_3$ is parabolic. Now, consider a Levi decomposition $P_{23} = G_2\ltimes U_{23}$, where $U_{23}$ is the unipotent radical of $G_2$. Therefore, $P_{13} = P_{12}\ltimes U_{23}$. Its unipotent radical is $U_{12}\ltimes U_{23}$, where $U_{12}\subset P_{12}$ is the unipotent radical of $P_{12}$. Therefore, the Levi factor of $P_{13}$ is $P_{13} / (U_{12}\ltimes U_{23})\simeq G_1$.
\end{proof}

\begin{prop}\label{prop:paraboliccompositionorientation}
Consider the setting of Lemma \ref{lem:paraboliccomposition} and let $M$ be a closed spin $3$-manifold. Consider the diagram
\[
\xymatrix{
&& \Loc_{P_{13}}(M) \ar[dl] \ar[dr] && \\
& \Loc_{P_{12}}(M) \ar[dl] \ar[dr] && \Loc_{P_{23}}(M) \ar[dl] \ar[dr] & \\
\Loc_{G_1}(M) && \Loc_{G_2}(M) && \Loc_{G_3}(M)
}
\]
where the top square is Cartesian. Then the composite of the orientation $\rho_{P_{23}}$ and $\rho_{P_{12}}$ is equal to $\rho_{P_{13}}$.
\end{prop}
\begin{proof}
The orientation are defined using the torsion volume forms $\overline{\vol}_{P_{ij}}$ described in Proposition \ref{prop:characterstackorientation}. The composite of the orientation $\rho_{P_{23}}$ and $\rho_{P_{12}}$ is defined using the volume form $\overline{\vol}_{P_{12}}(M)\wedge \overline{\vol}_{P_{23}}(M)\wedge \overline{\vol}_{G_2}(M)^{-1}$ on $\Loc_{P_{12}}(M)\times_{\Loc_{G_2}(M)} \Loc_{P_{23}}(M)$. This is precisely the glued volume form described in \cite[Section 2.2]{ns23}. But by construction the torsion volume forms are compatible with gluing, i.e. under the isomorphism
\[\Loc_{P_{13}}(M)\simeq \Loc_{P_{12}}(M)\times_{\Loc_{G_2}(M)} \Loc_{P_{23}}(M)\]
we have $\overline{\vol}_{P_{13}}\mapsto \overline{\vol}_{P_{12}}(M)\wedge \overline{\vol}_{P_{23}}(M)\wedge \overline{\vol}_{G_2}(M)^{-1}$.
\end{proof}

\subsection{Cohomological Hall algebra for $3$-manifolds}

In this section $M$ denotes a connected closed spin $3$-manifold and $G$ a connected reductive group. Recall that for a cocharacter $\lambda\colon \bG_m\rightarrow G$ we have defined the stabilizer and attractor subgroups $G_\lambda,G^+_\lambda\subset G$ and for a pair of commuting cocharacters $\hat{\lambda}=(\lambda_1,\lambda_2)\colon \bG_m^2\rightarrow G$ the subgroups $G_{\hat{\lambda}}, G^{+, +}_{\hat{\lambda}}\subset G$.

\begin{prop}\label{prop:LocGattractor}
Let $T\subset G$ be a maximal torus and $W$ the Weyl group. Choose a non-degenerate $G$-invariant symmetric bilinear pairing $c\in\Sym^2(\fg^*)^G$ on the Lie algebra of $G$. Let $M$ be a closed oriented $3$-manifold.

\begin{enumerate}
    \item For any cocharacter $\lambda\colon \bG_m\rightarrow G$ the subgroup $G^+_\lambda\subset G$ is parabolic and $G_\lambda\subset G^+_\lambda$ is the Levi factor.
    \item Any parabolic subgroup $P\subset G$ arises as $G^+_\lambda\subset G$ for some cocharacter $\lambda\colon \bG_m\rightarrow G$.
    \item The natural morphisms $X_\bullet(T)/W\rightarrow X_\bullet(G)/G$ and $X^2_\bullet(T)/W\rightarrow X^2_\bullet(G)/G$ are isomorphisms.
    \item For a pair of commuting cocharacters $\hat{\lambda}=(\lambda_1,\lambda_2)\colon \bG_m^2\rightarrow G$ the subgroup $(G^+_{\lambda_1})_{\lambda_2}\times_{G_{\lambda_2}} G^+_{\lambda_2}\subset G$ is parabolic with Levi factor $G_{\hat{\lambda}}$.
    \item We have
    \[\left((G^+_{\lambda_1})_{\lambda_2}\times_{G_{\lambda_2}} G^+_{\lambda_2}\right)\cap \left((G^+_{\lambda_2})_{\lambda_1}\times_{G_{\lambda_1}} G^+_{\lambda_1}\right) = G^{+,+}_{\hat{\lambda}}.\]
    \item The attractor correspondence \eqref{eq:grad_filt_diagram} for $\Loc_G(M)$ is equivalent to
    \[
    \xymatrix{
    & \coprod_{\lambda \in X_\bullet(T) / W} \Loc_{G_\lambda^+}(M) \ar[dl] \ar[dr] & \\
    \coprod_{\lambda \in X_\bullet(T) / W} \Loc_{G_\lambda}(M) && \Loc_G(M)
    }
    \]
    as a $(-1)$-shifted Lagrangian correspondence, where the latter structure is obtained by the AKSZ procedure from \eqref{eq:parabolicLagrangian} with $P=G^+_\lambda$ and $L=G_\lambda$.
    \item The diagram \eqref{eq:GradFiltLagrangiancorrespondence} for $\Loc_G(M)$ is equivalent to
    \[
    \xymatrix{
    & \coprod_{\hat{\lambda} \in X^2_\bullet(T) / W} \Loc_{\left((G^+_{\lambda_1})_{\lambda_2}\times_{G_{\lambda_2}} G^+_{\lambda_2}\right)}(M) \ar[dl] \ar[dr] & \\
    \coprod_{\hat{\lambda} \in X^2_\bullet(T) / W} \Loc_{G_{\hat{\lambda}}}(M) & \coprod_{\hat{\lambda} \in X^2_\bullet(T) / W} \Loc_{G^{+,+}_{\hat{\lambda}}}(M) \ar[u] \ar[d] & \Loc_G(M) \\
    & \coprod_{\hat{\lambda} \in X^2_\bullet(T) / W} \Loc_{\left((G^+_{\lambda_2})_{\lambda_1}\times_{G_{\lambda_1}} G^+_{\lambda_1}\right)}(M) \ar[ul] \ar[ur] &
    }
    \]
    as a 2-fold $(-1)$-shifted Lagrangian correspondence, where the latter structure is obtained by the AKSZ procedure from Proposition \ref{prop:parabolicintersectionLagrangian} with $L=G_{\hat{\lambda}}$, $P_1 = (G^+_{\lambda_1})_{\lambda_2}\times_{G_{\lambda_2}} G^+_{\lambda_2}$, $P_2= (G^+_{\lambda_2})_{\lambda_1}\times_{G_{\lambda_1}} G^+_{\lambda_1}$ and $P_1\cap P_2 = G^{+, +}_{\hat{\lambda}}$.
\end{enumerate}
\end{prop}
\begin{proof}
The first two statements are shown in \cite[Proposition 8.4.5]{springer}. The third statement is shown in \cite[Lemma C.3.5]{cgp}.

The fourth statement follows from Lemma \ref{lem:paraboliccomposition}: $(G^+_{\lambda_1})_{\lambda_2}\subset G_{\lambda_2}$ is parabolic with Levi factor $G_{\hat{\lambda}}$ and $G^+_{\lambda_2}\subset G$ is parabolic with Levi factor $G_{\lambda_2}$.

Let us now prove the fifth statement. Since $G$ is affine, it boils down to proving that the intersection of the two subspaces $\C[t_1, t_2^{\pm 1}]\times_{\C[t_2^{\pm 1}]} \C[t_2]$ and $\C[t_1^{\pm 1},t_2]\times_{\C[t_1^{\pm 1}]} \C[t_1]$ of $\C[t_1^{\pm 1}, t_2^{\pm 1}]$ is equal to $\C[t_1, t_2]\subset \C[t_1^{\pm 1}, t_2^{\pm 1}]$ which is obvious.

Now we prove the sixth statement. Specializing Theorem \ref{thm:Grad_quotient} to $X=\pt$ we get that the attractor correspondence for $BG$ is
\[
\xymatrix{
& \coprod_{\lambda \in X_\bullet(T) / W} BG^+_\lambda \ar[dl] \ar[dr] & \\
\coprod_{\lambda \in X_\bullet(T) / W} BG_\lambda && BG
}
\]

Using the $2$-shifted symplectic structure on $BG$ determined by the non-degenerate pairing $c$, there is an induced structure of a $2$-shifted Lagrangian correspondence using Corollary \ref{cor:lagattractorcorrespondence}. By uniqueness of a Lagrangian structure on $BG_\lambda\leftarrow BG^+_\lambda\rightarrow BG$, it coincides with the one defined in \cite{saf17}.

Since $\Loc_G(-)$, $\Filt(-)$ and $\Grad(-)$ are all mapping stacks, we get that the attractor correspondence for $\Loc_G(M)$ is obtained by applying $\Map(M_\B, -)$ to this $2$-shifted Lagrangian correspondence. We can then commute the coproduct over $X_\bullet(T)/W$ with $\Map(M_B, -)$ since $M$ is connected and its homotopy type is finitely dominated.

The proof of the seventh statement is identical by using uniqueness of the $2$-fold $2$-shifted Lagrangian correspondence shown in Proposition \ref{prop:parabolicintersectionLagrangian}.
\end{proof}

Let $B\subset G$ be a Borel subgroup, $T\subset B$ a maximal torus, $\Phi^+$ the corresponding set of positive roots of $G$, $\Delta\subset \Phi^+$ the subset of simple roots. Recall that a \defterm{standard parabolic subgroup} $P\subset G$ is a parabolic subgroup containing $B$. By \cite[Theorem 8.4.3]{springer} we have a bijection of sets
\begin{equation}\label{eq:standardparabolicbijection}
\{\text{standard parabolic subgroups $P\subset G$}\}\leftrightarrow \{\text{subsets $I\subset \Delta$}\}
\end{equation}
given by assigning to a standard parabolic subgroup $P$ with Levi factor $L$ the subset of simple roots appearing in the root decomposition of the Lie algebra of $L$.

\begin{lem}\label{lem:standardparabolic}$ $
\begin{enumerate}
    \item Consider a cocharacter $\lambda\colon \bG_m\rightarrow T$. Then the parabolic subgroup $G^+_\lambda\subset G$ is standard if, and only if, $\lambda(\alpha)\geq 0$ for every $\alpha\in\Delta$. Under the bijection \eqref{eq:standardparabolicbijection} the corresponding subset of simple roots is
    \[I = \{\alpha\in\Delta\mid \lambda(\alpha) = 0\}.\]
    \item Consider a pair of cocharacters $\hat{\lambda}=(\lambda_1,\lambda_2)\colon \bG_m^2\rightarrow T$ such that $\lambda_1(\alpha)\lambda_2(\alpha)\geq 0$ for every $\alpha\in\Phi^+$. Then the image of the inclusion $\Loc_{G_{\hat{\lambda}}}(M)\subset \Grad^2(\Loc_G(M))$ described in Proposition \ref{prop:LocGattractor}(7) lies in the matching locus.
\end{enumerate}
\end{lem}
\begin{proof}
Let $\fg$ be the Lie algebra of $G$ and let
\[\fg = \ft\oplus \bigoplus_{\alpha\in\Phi} \fg_\alpha\]
be the root decomposition.
\begin{enumerate}
    \item The Lie algebra of $G^+_\lambda$ has the root decomposition
    \[(\fg)^+_\lambda = \ft\oplus \bigoplus_{\alpha\in\Phi,\ \lambda(\alpha)\geq 0} \fg_\alpha.\]
    So, $G^+_\lambda$ is standard if, and only if, $\lambda(\alpha)\geq 0$ for every $\alpha\in\Phi^+$. Since every positive root is a sum of simple roots, this condition is equivalent to $\lambda(\alpha)\geq 0$ for every $\alpha\in\Delta$. The Lie algebra of $G_\lambda$ has the root decomposition
    \[(\fg)_\lambda = \ft\oplus \bigoplus_{\alpha\in\Phi,\ \lambda(\alpha) = 0} \fg_\alpha.\]
    The subset $I\subset \Delta$ of simple roots corresponding to a standard parabolic subgroup $P\subset G$ is precisely the set of simple roots appearing in the root decomposition of the Lie algebra of the Levi subgroup.
    \item Consider the bigrading of $\fg$ given by $\hat{\lambda}$. By the condition in the statement, we have that the weight $(m, l)$-part $\fg_{(m, l)}$ of $\fg$ is nonzero only if $ml\geq 0$. Therefore, using the description of the cotangent complex of the mapping stack (see e.g. \cite[Proposition 1.5]{ns23}) we obtain that the image of $\Loc_{G_{\hat{\lambda}}}(M)\rightarrow \Grad^2(\Loc_G(M))$ lands in the matching locus.
\end{enumerate}
    
\end{proof}

We will now show that the hyperbolic localization for the perverse sheaf $\varphi_{\Loc_G(M)}$ gives rise to a parabolic restriction isomorphism.

\begin{prop}\label{prop:parabolicrestriction}
Let $P\subset G$ be a parabolic subgroup with Levi factor $L$ and consider the correspondence
\[
\xymatrix{
& \Loc_P(M) \ar_{\pi_L}[dl] \ar^{\pi_G}[dr] & \\
\Loc_L(M) && \Loc_G(M)
}
\]
There is an isomorphism
\[\zeta_P\colon \varphi_{\Loc_L(M)}\cong (\pi_L)_*\pi_G^! \varphi_{\Loc_G(M)}\]
invariant under conjugation of $P$ with the following properties:
\begin{enumerate}
    \item For the parabolic subgroup $G\subset G$ we have $\zeta_G=\id$.
    \item For a pair of parabolic subgroups $P_1\subset G_1$ and $P_2\subset G_2$ with Levi factors $L_1$ and $L_2$ the diagram
    \[
    \xymatrix@C=1.5cm{
    \varphi_{\Loc_{L_1}(M)}\boxtimes \varphi_{\Loc_{L_2}(M)} \ar^-{\zeta_{P_1}\boxtimes \zeta_{P_2}}_-{\cong}[r] \ar^{\TS}_{\cong}[d] & (\pi_{L_1})_*\pi_{G_1}^!(\varphi_{\Loc_{G_1}(M)})\boxtimes (\pi_{L_2})_*\pi_{G_2}^!(\varphi_{\Loc_{G_2}(M)}) \ar^{\TS}_{\cong}[d] \\
    \varphi_{\Loc_{L_1\times L_2}(M)} \ar^-{\zeta_{P_1\times P_2}}_-{\cong}[r] & (\pi_{L_1\times L_2})_*\pi_{G_1\times G_2}^!(\varphi_{\Loc_{G_1\times G_2}(M)})
    }
    \]
    \item For parabolic subgroups $P_{12}, P_{23}, P_{13}$ as in Lemma \ref{lem:paraboliccomposition} consider the diagram
    \[
    \xymatrix{
    && \Loc_{P_{13}}(M) \ar_{\pi_{P_{12}}}[dl] \ar^{\pi_{P_{23}}}[dr] && \\
    & \Loc_{P_{12}}(M) \ar_{\pi_{G_1}}[dl] \ar^{\pi'_{G_2}}[dr] && \Loc_{P_{23}}(M) \ar_{\pi''_{G_2}}[dl] \ar^{\pi_{G_3}}[dr] & \\
    \Loc_{G_1}(M) && \Loc_{G_2}(M) && \Loc_{G_3}(M).
    }
    \]
    Then the diagram
    \[
    \xymatrix{
    \varphi_{\Loc_{G_1}(M)} \ar^-{\zeta_{P_{13}}}_-{\cong}[r] \ar^{\zeta_{P_{12}}}_{\cong}[d] & (\pi_{G_1})_*(\pi_{P_{12}})_*\pi_{P_{23}}^!\pi_{G_3}^! \varphi_{\Loc_{G_3}(M)} \ar_{\cong}[d] \\
    (\pi_{G_1})_*(\pi'_{G_2})^! \varphi_{\Loc_{G_2}(M)} \ar^-{\zeta_{P_{23}}}_-{\cong}[r] & (\pi_{G_1})_*(\pi'_{G_2})^!(\pi''_{G_2})_*\pi_{G_3}^! \varphi_{\Loc_{G_3}(M)}
    }
    \]
    commutes.
\end{enumerate}
\end{prop}
\begin{proof}
By Proposition \ref{prop:LocGattractor}(2) we may find a cocharacter $\lambda\colon \bG_m\rightarrow G$ such that $P = G^+_\lambda$ and $L=G_\lambda$. By \cite[Proposition 3.23]{ns23} we have $\vdim(\Loc_P(M)) = 0$. In particular, we have $I_{\Loc_{G_\lambda}(M)} = 0$ by \eqref{eq:Ind=vdim}.

Let $u\colon \Loc_L(M)\rightarrow \Loc_G(M)$ be the map induced by the inclusion $G_\lambda\subset G$. Let $\omega_{\Loc_G(M)}$ be the natural $(-1)$-shifted symplectic structure on $\Loc_G(M)$. The natural structure of a Lagrangian correspondence on $\Loc_L(M)\leftarrow \Loc_P(M)\rightarrow \Loc_G(M)$ induces an equivalence between the $(-1)$-shifted symplectic structure $u^\ast\omega_{\Loc_G(M)}$ and $\omega_{\Loc_L(M)}$. Using the orientation $\rho_{G^+_\lambda}$ of the Lagrangian correspondence $\Loc_{G_\lambda}(M)\leftarrow \Loc_{G^+_\lambda}(M)\rightarrow \Loc_G(M)$ from Proposition \ref{prop:characterstackorientation} we obtain an isomorphism between the localized orientation $u^\star o_G$ and $o_L$. Therefore, using the isomorphism $\zeta_{\Loc_G(M)}$ from Corollary \ref{cor:Joyce_conj_attractor} we obtain an isomorphism
\[\zeta_{G^+_\lambda}\colon \varphi_{\Loc_{G_\lambda}(M)}\cong (\pi_{G_\lambda})_*\pi_G^!\varphi_{\Loc_G(M)}.\]

Let us now show that it satisfies the relevant properties for \emph{some} realizations of the parabolic subgroups as $P=G^+_\lambda$:
\begin{enumerate}
    \item As $\Loc_G(M_1\amalg M_2)=\Loc_G(M_1)\times \Loc_G(M_2)$, it is enough to prove the claim for $M$ connected with a chosen basepoint. In this case, we have a  chart $q\colon X=\Loc^{\fr}_G(M)\rightarrow \Loc_G(M)$ given by the framed character stack. Suppose the cocharacter $\lambda\colon \bG_m\rightarrow G$ is such that $G_\lambda = G$. Then $\lambda$ factors through the centre of $G$. The pair $(X, \lambda)$ defines a $\bG_m$-equivariant chart of $\Loc_G(M)$ and, since $\eta$ is surjective, it is enough to prove the statement after the pullback by $q$. The isomorphism $q^*(\zeta)$ constructed in Theorem \ref{thm:main_thm} coincides with the hyperbolic localization isomorphism from Theorem \ref{thm:hyp_DT} for the attractor correspondence $X^{\lambda}\leftarrow X^{\lambda,+}\rightarrow X$. But since the $\bG_m$-action given by $\lambda$ is trivial, this isomorphism is the identity by the commutativity of the diagram \eqref{eq:zeta_unital}.
    \item The multiplicativity of $\zeta$ follows from the multiplicativity property of the isomorphism $\zeta$ from Corollary \ref{cor:Joyce_conj_attractor} as well as the equality $\rho_{P_1}\boxtimes \rho_{P_2} = \rho_{P_1\times P_2}$ of orientations of the parabolic restriction correspondence shown in Proposition \ref{prop:characterstackorientation}(4).
    \item Fix a minimal parabolic subgroup $B\subset P_{13}$ and a maximal torus $T\subset B$ with the corresponding set $\Delta$ of simple roots. Then $P_{23}, P_{13}\subset G$ are standard parabolic subgroups. Moreover, $B\cap G_2\subset G_2$ is a Borel subgroup and $P_{12}\supset B\cap G_2$. Let $I_1\subset I_2\subset \Delta$ be the sets of simple roots of $G_1\subset G_2\subset G_3$. Then $I_2\subset \Delta$ is the subset of simple roots corresponding to the parabolic $P_{23}$ under \eqref{eq:standardparabolicbijection}, $I_1\subset \Delta$ is the subset corresponding to $P_{13}$ and $I_1\subset I_2$ is the subset corresponding to $P_{12}$.

    Consider cocharacters
    \[\lambda_{23}\colon \bG_m\rightarrow T,\qquad \lambda_{12}\colon \bG_m\rightarrow T\]
    satisfying the following conditions:
    \begin{align*}
    \lambda_2(\alpha) &= 0\quad \forall\alpha\in I_2,\qquad \lambda_2(\alpha) > 0\quad \forall \alpha\in \Delta\setminus I_2 \\
    \lambda_1(\alpha) &= 0\quad \forall\alpha\in I_1,\qquad \lambda_1(\alpha) > 0\quad \forall \alpha\in \Delta\setminus I_1.
    \end{align*}
    Then using Lemma \ref{lem:standardparabolic}(1) as well as the fact that \eqref{eq:standardparabolicbijection} is a bijection, we get
    \[P_{23} = (G_3)^+_{\lambda_2},\qquad P_{12} = (G_2)^+_{\lambda_1},\qquad P_{13} = (G_3)^+_{\lambda_1}.\]
    Let $\hat{\lambda}=(\lambda_1, \lambda_2)\colon \bG_m^2\rightarrow T$. Then by Lemma \ref{lem:standardparabolic}(2) the image of $\Loc_{G_{\hat{\lambda}}}(M)\rightarrow \Grad^2(\Loc_G(M))$ lands in the matching locus. Using the associativity property of the isomorphism $\zeta$ from Corollary \ref{cor:Joyce_conj_attractor} as well as the fact that the orientation of the parabolic restriction diagram is compatible with compositions shown in Proposition \ref{prop:paraboliccompositionorientation} we get a commutative diagram
    \[
    \xymatrix{
    & \varphi_{\Loc_{G_1}(M)} \ar_{\zeta_{(G_2)^+_{\lambda_1}}}[dl] \ar^{\zeta_{(G_1)^+_{\lambda_2}}}[dr] & \\
    (\pi_{G_1})_*(\pi'_{G_2})^! \varphi_{\Loc_{G_2}(M)} \ar_{\zeta_{(G_3)^+_{\lambda_2}}}[d] && \varphi_{\Loc_{G_1}(M)} \ar^{\zeta_{(G_3)^+_{\lambda_1}}}[d] \\
    (\pi_{G_1})_*(\pi'_{G_2})^!(\pi''_{G_2})_*\pi_{G_3}^! \varphi_{\Loc_{G_3}(M)} \ar_{\cong}[rr] && (\pi_{G_1})_*(\pi_{P_{12}})_*\pi_{P_{23}}^!\pi_{G_3}^! \varphi_{\Loc_{G_3}(M)}
    }
    \]
    But by part (1) the map $\zeta_{(G_1)^+_{\lambda_2}}$ is the identity, so we get the claim.
\end{enumerate}

Let us finally show that the isomorphism $\zeta_{G^+_\lambda}$ is independent of the choice of $\lambda$ and merely depends on the underlying subgroup $P\subset G$. Namely, given two cocharacters $\lambda_1,\lambda_2\colon \bG_m\rightarrow G$ such that $P=G^+_{\lambda_1}=G^+_{\lambda_2}$, we want to show that the isomorphisms $\zeta_{G^+_{\lambda_1}}$ and $\zeta_{G^+_{\lambda_2}}$ are equal. Since $G_{\lambda_1} = G_{\lambda_2}$ and $\lambda_i(t)\in G_{\lambda_i}$, the cocharacters $\lambda_1$ and $\lambda_2$ commute, so they define a homomorphism $\hat{\lambda}=(\lambda_1, \lambda_2)\colon \bG_m^2\rightarrow G$. Consider the bigrading of $\fg$ given by $\hat{\lambda}$. Since $G^+_{\lambda_1} = G^+_{\lambda_2}$, we have that the weight $(m, l)$-part $\fg_{(m, l)}$ of $\fg$ is nonzero only if $ml\geq 0$. As in the proof of Lemma \ref{lem:standardparabolic}(2) we obtain that the image of $\Loc_{G_{\hat{\lambda}}}(M)\rightarrow \Grad^2(\Loc_G(M))$ lands in the matching locus. Now consider the 2-fold correspondence
\[
\xymatrix{
& \Loc_{(G^+_{\lambda_1})_{\lambda_2}\times_{G_{\lambda_2}} G^+_{\lambda_2}}(M) \ar[dl] \ar[dr] & \\
\Loc_{G_{\hat{\lambda}}}(M) & \Loc_{G^{+,+}_{\hat{\lambda}}}(M) \ar[u] \ar[d] & \Loc_G(M) \\
& \Loc_{(G^+_{\lambda_2})_{\lambda_1}\times_{G_{\lambda_1}} G^+_{\lambda_1}}(M), \ar[ul] \ar[ur] &
}
\]
where the maps $G_{\hat{\lambda}}\leftarrow (G^+_{\lambda_1})_{\lambda_2}\rightarrow G_{\lambda_2}$, $G_{\hat{\lambda}}\leftarrow (G^+_{\lambda_2})_{\lambda_1}\rightarrow G_{\lambda_1}$ and $\left((G^+_{\lambda_1})_{\lambda_2}\times_{G_{\lambda_2}} G^+_{\lambda_2}\right)\leftarrow G^{+, +}_{\hat{\lambda}}\rightarrow \left((G^+_{\lambda_2})_{\lambda_1}\times_{G_{\lambda_1}} G^+_{\lambda_1}\right)$ are all equalities as subgroups of $G$.

Using properties (1) and (3) we get that
\[\zeta_{(G^+_{\lambda_1})_{\lambda_2}\times_{G_{\lambda_2}} G^+_{\lambda_2}} = \zeta_{G^+_{\lambda_2}}\]
and
\[\zeta_{(G^+_{\lambda_2})_{\lambda_1}\times_{G_{\lambda_1}} G^+_{\lambda_1}} = \zeta_{G^+_{\lambda_1}}.\]
Using the associativity property of the isomorphism $\zeta$ from Corollary \ref{cor:Joyce_conj_attractor} we have
\[\zeta_{(G^+_{\lambda_1})_{\lambda_2}\times_{G_{\lambda_2}} G^+_{\lambda_2}} = \zeta_{(G^+_{\lambda_2})_{\lambda_1}\times_{G_{\lambda_1}} G^+_{\lambda_1}}.\]
\end{proof}

Consider the \defterm{cohomological DT invariants}
\[H_{G, M} = H^\bullet(\Loc_G(M), \varphi_{\Loc_G(M)}).\]
Using multiplicativity of the orientation $o_G$ we get a map
\[H_{G_1, M}\otimes H_{G_2, M}\longrightarrow H_{G_1\times G_2, M}.\]
Using the parabolic restriction isomorphism constructed in Proposition \ref{prop:parabolicrestriction} we obtain the following construction.

\begin{cor}\label{cor:parabolicinduction}
Let $P\subset G$ be a parabolic subgroup with Levi factor $L$. There is a \defterm{parabolic induction map}
\[\ind_P\colon H_{L, M}\longrightarrow H_{G, M}\]
with the following properties:
\begin{enumerate}
    \item For the parabolic subgroup $G\subset G$ the parabolic induction
    \[\ind_G\colon H_{G, M}\longrightarrow H_{G, M}\]
    is the identity.
    \item For a pair of parabolic subgroups $P_1\subset G_1$ and $P_2\subset G_2$ with Levi factors $L_1$ and $L_2$ we have
    \[\ind_{P_1}\otimes \ind_{P_2} = \ind_{P_1\times P_2}\colon H_{L_1, M}\otimes H_{L_2, M}\longrightarrow H_{G_1\times G_2, M}.\]
    \item For parabolic subgroups $P_{12}, P_{23}, P_{13}$ as in Lemma \ref{lem:paraboliccomposition} we have
    \[\ind_{P_{23}}\circ \ind_{P_{12}} = \ind_{P_{13}}\colon H_{G_1, M}\longrightarrow H_{G_3, M}.\]
\end{enumerate}
\end{cor}
\begin{proof}
Consider the correspondence
\[
\xymatrix{
& \Loc_P(M) \ar_{\pi_L}[dl] \ar^{\pi_G}[dr] & \\
\Loc_L(M) && \Loc_G(M)
}
\]
It is shown in the proof of \cite[Theorem 3.47]{ns23} that the morphism $\pi_G\colon \Loc_P(M)\rightarrow \Loc_G(M)$ is representable and proper. In particular, there is a natural transformation $(\pi_G)_*\pi_G^!\rightarrow \id$. We define the parabolic induction map $\ind_P$ as the composite
\begin{align*}
H^\bullet(\Loc_L(M), \varphi_{\Loc_L(M)})&\xrightarrow{\zeta_P} H^\bullet(\Loc_L(M), (\pi_L)_*\pi_G^!\varphi_{\Loc_G(M)}) \\
&\cong H^\bullet(\Loc_G(M), (\pi_G)_*\pi_G^!\varphi_{\Loc_G(M)}) \\
&\rightarrow H^\bullet(\Loc_G(M), \varphi_{\Loc_G(M)}).
\end{align*}
The three properties of $\ind_P$ follow from the corresponding properties of the isomorphism $\zeta_P$.
\end{proof}

\begin{rmk}
The parabolic induction map $\ind_P$ depends on the choice of a spin structure on $M$ as follows. The volume form $\overline{\vol}_G$ that goes into the definition of the orientation depends on the canonical Euler structure (and hence the spin structure) as shown in \cite[Proposition 3.23]{ns23}. In particular, the orientation $o_G$ for two different spin structures are canonically isomorphic since $G$ is unimodular; however, the orientation $\rho_P$ does depend on the spin structure. The difference of two spin structures $s_1, s_2$ defines an element of $H^1(M; \Z/2\Z)$. Denote its image under the composite
\[H^1(M; \Z/2\Z)\longrightarrow H^2(M; \Z)\cong H_1(M; \Z)\]
of the Bockstein homomorphism and Poincar\'e duality by $h$. It is a 2-torsion element of $H_1(M; \Z)$ which represents the difference of the corresponding canonical Euler structures. The modular character $P\rightarrow \bG_m$ of $P$ factors through $\Delta_P\colon L\rightarrow \bG_m$, and hence it defines a rank 1 local system on $\Loc_L(M)$. Denote by $\langle \Delta_P, h\rangle$ the monodromy of this local system along $h$ which defines a $\mu_2$-valued function on $\Loc_L(M)$. Then the two parabolic induction maps differ by a sign:
\[\ind^{s_2}_P = \ind^{s_1}_P\cdot \langle \Delta_P, h\rangle.\]
\end{rmk}

We now collect all the results we have proven to define a cohomological Hall algebra for $\GL_n$ and cohomological Hall modules for $\SO_n$ and $\Sp_{2n}$. For this, let us first describe a class of parabolic subgroups of classical groups (we denote by $\{e_i\}$ the standard basis of $\C^n$):
\begin{itemize}
    \item Let $P^{\GL}_{n,m}\subset \GL_{n+m}$ for $n,m\geq 0$ be the parabolic subgroup preserving the flag $\spn\{e_1, \dots, e_m\}\subset \spn\{e_1, \dots, e_{n+m}\}$. Its Levi factor is $\GL_n\times \GL_m$.
    \item Let $P^{\GL}_{n_1, n_2, n_3}\subset \GL_{n_1+n_2+n_3}$ for $n_1,n_2,n_3\geq 0$ be the parabolic subgroup preserving the flag
    \[\spn\{e_1, \dots, e_{n_1}\}\subset \spn\{e_1, \dots, e_{n_1+n_2}\}\subset \spn\{e_1, \dots, e_{n_1+n_2+n_3}\}.\]
    Its Levi factor is $\GL_{n_1}\times \GL_{n_2}\times \GL_{n_3}$.
    \item Let $\SO_n$ be the group of isometries of the symmetric bilinear form on $\C^n$ with $(e_i, e_j) = \delta_{i,n-j}$ of determinant $1$. Let $P^{\SO}_{n, m}\subset \SO_{n+2m}$ for $n,m\geq 0$ be the parabolic subgroup preserving the isotropic flag $\C^m=\spn\{e_1, \dots, e_m\}\subset \spn\{e_1, \dots, e_{n+2m}\}$. Its Levi factor is $\SO_n\times \GL_m$.
    \item Let $P^{\SO}_{n_1, n_2, n_3}\subset \SO_{n_1+2n_2+2n_3}$ for $n_1,n_2,n_3\geq 0$ be the parabolic subgroup preserving the isotropic flag
    \[\spn\{e_1, \dots, e_{n_1}\}\subset \spn\{e_1, \dots, e_{n_1+n_2}\}\subset \spn\{e_1, \dots, e_{n_1+2n_2+2n_3}\}.\]
    Its Levi factor is $\SO_{n_1}\times \GL_{n_2}\times \GL_{n_3}$.
    \item Let $\Sp_{2n}$ be the group of invertible transformations preserving the symplectic form
    \[(e_i, e_j) = \delta_{i,2n-j},\ i\leq n,\qquad (e_i, e_j) = -\delta_{i, 2n-j},\ i > n.\]
    Let $P^{\Sp}_{2n, m}\subset \Sp_{2n+2m}$ be the parabolic subgroup preserving the isotropic flag $\spn\{e_1, \dots, e_m\}\subset \spn\{e_1, \dots, e_{2n+2m}\}$. Its Levi factor is $\Sp_{2n}\times \GL_m$.
    \item Let $P^{\Sp}_{2n_1, n_2, n_3}\subset \Sp_{2n_1+2n_2+2n_3}$ for $n_1,n_2,n_3\geq 0$ be the parabolic subgroup preserving the isotropic flag
    \[\spn\{e_1, \dots, e_{n_1}\}\subset \spn\{e_1, \dots, e_{n_1+n_2}\}\subset \spn\{e_1, \dots, e_{2n_1+2n_2+2n_3}\}.\]
    Its Levi factor is $\Sp_{2n_1}\times \GL_{n_2}\times \GL_{n_3}$.
\end{itemize}

We have
\begin{equation}\label{eq:PGLcomposition}
P^{\GL}_{n_1, n_2} \times_{\GL_{n_1+n_2}} P^{\GL}_{n_1+n_2, n_3} = P^{\GL}_{n_1,n_2,n_3} = P^{\GL}_{n_2, n_3}\times_{\GL_{n_2+n_3}} P^{\GL}_{n_1, n_2+n_3},
\end{equation}
\begin{equation}\label{eq:PSOcomposition}
P^{\SO}_{n_1, n_2} \times_{\SO_{n_1+2n_2}} P^{\SO}_{n_1+2n_2, n_3} = P^{\SO}_{n_1,n_2,n_3} = P^{\GL}_{n_2, n_3}\times_{\GL_{n_2+n_3}} P^{\SO}_{n_1, n_2+n_3}
\end{equation}
and
\begin{equation}\label{eq:PSpcomposition}
P^{\Sp}_{2n_1, n_2} \times_{\Sp_{2n_1+2n_2}} P^{\Sp}_{2n_1+2n_2, n_3} = P^{\Sp}_{2n_1,n_2,n_3} = P^{\GL}_{n_2, n_3}\times_{\GL_{n_2+n_3}} P^{\Sp}_{2n_1, n_2+n_3}.
\end{equation}

Using the orientation $o_G$ and $\rho_P$ we obtain an orientation data for the category of local systems on $M$. Recall that by Theorem \ref{thm:Mlocsys} we have
\[\fM_{\LocSys(M)^\omega}\cong \Map(M_\B, \Perf).\]
Let
\[\fM = \Map(M_\B, \Vect)\]
be the substack of local systems of vector spaces.

\begin{prop}\label{prop:LocSysorientation}
Let $M$ be a connected closed spin $3$-manifold. Let $\eC = \LocSys(M)^\omega$ and $\fM\subset \fM_{\eC}$ the substack of local systems of vector spaces.
\begin{enumerate}
    \item There is an isomorphism of $(-1)$-shifted symplectic stacks $\fM\cong \coprod_n \Loc_{\GL_n}(M)$.
    \item $\fM$ has affine diagonal and is $\Theta$-reductive.
    \item $\fM$ is a moduli-like stack.
    \item There is a strong orientation data on $\fM$.
\end{enumerate}
\end{prop}
\begin{proof}$ $
\begin{enumerate}
    \item Write
    \[\Vect\simeq \colim_n B \GL_n.\]
    Since $M$ is connected and finitely dominated, by Lemma \ref{lem:connectedcoproduct} we have
    \[\Map(M_\B, \Vect)\cong \coprod_n \Map(M_\B, B\GL_n).\]

    \item Choosing a basepoint, we may identify
    \[\Loc_{\GL_n}(M)\cong [\Loc^{\fr}_{\GL_n}(M) / \GL_n].\]
    This is a quotient of an affine scheme by an affine algebraic group, so the diagonal is affine. Moreover, by \cite[Example 5.1.8]{hlp14} it is $\Theta$-reductive.

    \item The claim follows from Proposition \ref{prop:modulilikemoduli}.

    \item By Proposition \ref{prop:gradedmoduli} the moduli stack $\fM^{2-\filt}$ parametrizes fibre sequences $\cL_1\rightarrow \cL_2\rightarrow \cL_3$ of local systems of perfect complexes on $M$, such that $\cL_1$ and $\cL_3$ are local systems of vector spaces. Therefore, $\cL_2$ is a local system of vector spaces and this is an exact sequence. Thus, the correspondence $\fM^2\leftarrow \fM^{2-\filt}\rightarrow \fM$ is equivalent to
    \[
    \xymatrix{
    & \coprod_{a,b} \Loc_{P_{a,b}}(M) \ar[dl] \ar[dr] & \\
    \coprod_{a,b} \Loc_{\GL_a}(M)\times \Loc_{\GL_b}(M) && \coprod_n \Loc_{\GL_n}(M)
    }
    \]
    Moreover, by Proposition \ref{prop:LocGattractor}(6) this equivalence is compatible with $(-1)$-shifted Lagrangian structures. Therefore, the orientation on $\fM$ follows from the orientation $o_{\GL_n}$ of $\Loc_{\GL_n}(M)$ and orientation $\rho_{P_{a, b}}$ of $\Loc_{\GL_a}(M)\times \Loc_{\GL_b}(M)\leftarrow \Loc_{P_{a,b}}(M)\rightarrow \Loc_{\GL_{a+b}}(M)$ constructed in Proposition \ref{prop:characterstackorientation}.

    Using \eqref{eq:PGLcomposition} we have
    \[(P^{\GL}_{n_1, n_2}\times\GL_{n_3}) \times_{\GL_{n_1+n_2}\times \GL_{n_3}} P^{\GL}_{n_1+n_2, n_3} = P^{\GL}_{n_1,n_2,n_3} = (\GL_{n_1}\times P^{\GL}_{n_2, n_3})\times_{\GL_{n_1}\times \GL_{n_2+n_3}} P^{\GL}_{n_1, n_2+n_3}.\]

    Thus, the associativity of the orientation data follows from Proposition \ref{prop:paraboliccompositionorientation}. The unitality of the orientation data follows from the fact that $\rho_{P_{n,0}}$ and $\rho_{P_{0, n}}$ are the identity orientations for the correspondence $\Loc_{\GL_n}(M)\leftarrow \Loc_{\GL_n}(M)\rightarrow \Loc_{\GL_n}(M)$.
\end{enumerate}
\end{proof}

\begin{thm}\label{thm:manifoldhall}
Let $M$ be a closed spin $3$-manifold. Consider the graded vector spaces
\begin{align*}
\cH^{\GL} &= \bigoplus_{n\geq 0} H_{\GL_n, M} \\
\cH^{\SO} &= \bigoplus_{n\geq 0} H_{\SO_n, M} \\
\cH^{\Sp} &= \bigoplus_{n\geq 0} H_{\Sp_{2n}, M}
\end{align*}
Then the maps
\begin{align*}
H_{\GL_n, M}\otimes H_{\GL_m, M}&\rightarrow H_{\GL_n\times \GL_m, M}\xrightarrow{\ind_{P^{\GL}_{n,m}}} H_{\GL_{n+m}, M} \\
H_{\SO_n, M}\otimes H_{\GL_m, M}&\rightarrow H_{\SO_n\times \GL_m, M}\xrightarrow{\ind_{P^{\SO}_{n,m}}} H_{\SO_{n+2m}, M} \\
H_{\Sp_{2n}, M}\otimes H_{\GL_m, M}&\rightarrow H_{\Sp_{2n}\times \GL_m, M}\xrightarrow{\ind_{P^{\Sp}_{2n,m}}} H_{\Sp_{2n+2m}, M}
\end{align*}
as well as the obvious isomorphism $e\colon \C\cong H_{\GL_0, M}$ define a graded algebra structure on $\cH^{\GL}$ and graded $\cH^{\GL}$-module structures on $\cH^{\SO}$ and $\cH^{\Sp}$.
\end{thm}
\begin{proof}
The fact that $\cH^{\GL}$ is a unital associative algebra is shown in Corollary \ref{cor:CoHA}. The unitality and associativity of the $\cH^{\GL}$-actions on $\cH^{\SO}$ and $\cH^{\Sp}$ are checked analogously, so we will only prove the claim for $\cH^{\SO}$.

We begin by establishing unitality. The multiplication map
\[H_{\SO_n, M}\otimes H_{\GL_0, M}\longrightarrow H_{\SO_n, M}\]
consists of the following two composites:
\[H_{\SO_n, M}\xrightarrow{\id\otimes e} H_{\SO_n, M}\otimes H_{\GL_0, M}\rightarrow H_{\SO_n\times \GL_0, M}\cong H_{\SO_n, M}\]
and
\[H_{\SO_n, M}\xrightarrow{\ind_{P^{\SO}_{n, 0}}} H_{\SO_n, M}.\]
The first composite is clearly the identity. The second composite involves parabolic induction along the parabolic subgroup $P^{\SO}_{n, 0}=\SO_n\subset \SO_n$ and is hence also the identity.

Next, let us establish associativity. Namely, we have to show that the two composites
\[
H_{\SO_{n_1}, M}\otimes H_{\GL_{n_2}, M}\otimes H_{\GL_{n_3}, M}\xrightarrow{\ind_{P^{\SO}_{n_1,n_2}}\otimes \id} H_{\SO_{n_1+2n_2}, M}\otimes H_{\GL_{n_3}, M}\xrightarrow{\ind_{P^{\SO}_{n_1+2n_2,n_3}}} H_{\SO_{n_1+2n_2+2n_3}, M}
\]
and
\[
H_{\SO_{n_1}, M}\otimes H_{\GL_{n_2}, M}\otimes H_{\GL_{n_3}, M}\xrightarrow{\id\otimes \ind_{P^{\GL}_{n_2,n_3}}} H_{\SO_{n_1}, M}\otimes H_{\GL_{n_2+n_3}, M}\xrightarrow{\ind_{P^{\SO}_{n_1,n_2+n_3}}} H_{\SO_{n_1+2n_2+2n_3}, M}
\]
are equal. Using the multiplicativity and unitality properties of the parabolic induction maps (Corollary \ref{cor:parabolicinduction}) we are reduced to showing that the two composites
\begin{equation}\label{eq:Hallmanifoldassociative1}
H_{\SO_{n_1}\times \GL_{n_2}\times \GL_{n_3}}\xrightarrow{\ind_{P^{\SO}_{n_1,n_2}\times \GL_{n_3}}} H_{\SO_{n_1+2n_2}\times \GL_{n_3}}\xrightarrow{\ind_{P^{\SO}_{n_1+2n_2,n_3}}} H_{\SO_{n_1+2n_2+2n_3}, M}
\end{equation}
and
\begin{equation}\label{eq:Hallmanifoldassociative2}
H_{\SO_{n_1}\times \GL_{n_2}\times \GL_{n_3}}\xrightarrow{\ind_{\SO_{n_1}\times P^{\GL}_{n_2,n_3}}} H_{\SO_{n_1}\times \GL_{n_2+n_3}}\xrightarrow{\ind_{P^{\SO}_{n_1,n_2+n_3}}} H_{\SO_{n_1+2n_2+2n_3}, M}
\end{equation}
are equal.

Using \eqref{eq:PSOcomposition} we have
\[(P^{\SO}_{n_1, n_2}\times\GL_{n_3}) \times_{\SO_{n_1+2n_2}\times \GL_{n_3}} P^{\SO}_{n_1+2n_2, n_3} = P^{\SO}_{n_1,n_2,n_3} = (\SO_{n_1}\times P^{\GL}_{n_2, n_3})\times_{\SO_{n_1}\times \GL_{n_2+n_3}} P^{\SO}_{n_1, n_2+n_3}.\]
Therefore, by the associativity property of parabolic induction shown in Corollary \ref{cor:parabolicinduction}(3) the two composites \eqref{eq:Hallmanifoldassociative1} and \eqref{eq:Hallmanifoldassociative2} are both equal to $\ind_{P^{\SO}_{n_1,n_2,n_3}}$.
\end{proof}

\appendix

\section{Braden's theorem for stacks}

The aim of this appendix is to prove a generalization of Braden's theorem on hyperbolic localization (= Theorem \ref{thm:Braden}) to the attractor correspondences of Artin stacks.
Though we have not used this result in our paper, we record it here since this can be proved analogously to Theorem \ref{thm:main_thm} and is interesting in its own right.
A similar result in the context of D-modules can be found in \cite[Theorem 4.3.4]{dg16}.

Let $\fX$ be a quasi-separated Artin stack with affine stabilizers.
We set $\Filt^{-}(\fX) \coloneqq \Map([\bA^{1, -} / \bG_m], \mathfrak{X})$ where $\bA^{1, -}$ denotes the affine line equipped with $\bG_m$-action of weight $-1$.
We let $\gr^{-} \colon \Filt^{-}(\fX) \to \Grad(\fX)$ and $\ev^{-} \colon \Filt^{-}(\fX) \to \fX$ denote the natural maps.
The following is the main result in this appendix, which generalizes Braden's theorem (= Theorem \ref{thm:Braden}):

\begin{prop}\label{prop:Braden_appendix}
    There exists a natural isomorphism of functors between monodromic constructible derived categories:
    \[
    \gr_* \ev^! \cong \gr^{-}_! \ev^{-, *}.
    \]
\end{prop}

We need the following lemma, which is immediate from the description of the cotangent complexes on both sides\footnote{One can more generally show that the map is an open immersion.}:

\begin{lem}\label{lem:app_etale}
    The natural map $\Grad(\fX) \to \Filt(\fX) \times_{\fX} \Filt^{-}(\fX)$ is an \'etale morphism.
\end{lem}

\begin{proof}[Proof of Proposition \ref{prop:Braden_appendix}]
    Let $\sigma \colon \Grad(\fX) \to \Filt(\fX)$ and $\sigma^{-} \colon \Grad(\fX) \to \Filt^{-}(\fX)$
    be the canonical sections of $\gr$ and $\gr^{-}$, respectively.
    Then by using Proposition \ref{prop:contraction}, there exists a natural isomorphism
    \[
    \gr_* \cong \sigma^*, \quad \gr^{-}_! \cong \sigma^!.
    \]
    Therefore it is enough to construct an equivalence between
    $\sigma^* \ev^!$ and $\sigma^{-, !} \ev^{-, *}$.
    Consider the following diagram:
    \[
    \xymatrix{
    {\Grad(\fX)} 
    \ar[r]^-{\sigma} 
    \ar[d]^-{\sigma^{-}}
    & {\Filt(\fX)} 
    \ar[d]^-{\ev} \\
    {\Filt^{-}(\fX).} \ar[r]^-{\ev^{-}}
    & {\fX.}
    }
    \]
    By Lemma \ref{lem:app_etale}, this diagram is pullable (see \S \ref{sssec:gen_bc}).
    Therefore, we can define the base change transform \eqref{eq:gen_bc2}
    \[
    \eta \colon \sigma^* \ev^! \to \sigma^{-, !} \ev^{-, *}.
    \]
    We will show that this map is invertible.
    To see this, take a $\bG_m$-equivariant chart $(X, \mu, \bar{q})$ of $\fX$ such that 
    $q_{\mu} \colon X^{\mu} \to \Grad(\fX)$ is surjective using Proposition \ref{prop:QCha_atlas}.
    It is enough to show that $q_{\mu}^* \eta$ is invertible.
    Consider the following diagram:
    \[
    \xymatrix{
    {X^{\mu}}
    \ar[r]^-{q_{\mu}}
    \ar[d]^{\sigma_{\mu}^{-}}
    & {\Grad(\fX)} 
    \ar[r]^-{\sigma}
    \ar[d]^{\sigma^{-}}
    & {\Filt(\fX)}
    \ar[d]^-{\ev} \\
    {X^{\mu, -}}
    \ar[r]^-{q_{\mu}^{-}}
    & {\Filt^{-}(\fX)}
    \ar[r]^-{\ev^{-}}
    & {\fX.}
    }
    \]
    Since $q_{\mu}$ and $q_{\mu}^{-}$ are smooth, the left square is pullable and the base change map 
    $q_{\mu}^* \sigma^{-, !}[2\dim q_{\mu}] \to \sigma_{\mu}^{-, !} q_{\mu}^{-, *}[2 \dim q_{\mu}^{-}]$ is invertible.
    Using the associativity of the base change map \eqref{eq:gen_bc2} for pullable squares \cite[Proposition 3.5.4]{fyz23}, it is enough to prove the invertibility of the base change map  \eqref{eq:gen_bc2} for the outer squares.
    Now consider the following diagram:
       \[
    \xymatrix{
    {X^{\mu}}
    \ar[r]^-{\sigma_{\mu}}
    \ar[d]^{\sigma_{\mu}^{-}}
    & {X^{\mu, +}} 
    \ar[r]^-{q_{\mu}^{+}}
    \ar[d]^{\ev_{\mu}}
    & {\Filt(\fX)}
    \ar[d]^-{\ev} \\
    {X^{\mu, -}}
    \ar[r]^-{\ev_{\mu}^{-}}
    & {X}
    \ar[r]^-{q}
    & {\fX.}
    }
    \]
    These diagrams are obviously pullable. 
    The base change map \eqref{eq:gen_bc2} for the left square is invertible by \cite[Theorem 3.1.6]{dg14}
    and the base change map \eqref{eq:gen_bc2} for the right square is invertible by the smoothness of $q$ and $q_{\mu}$.
    Therefore, we conclude the invertibility of the base change map \eqref{eq:gen_bc2} for the outer square.
    
\end{proof}

\begin{rmk}
    Though we have proved it for analytic sheaves on complex Artin stacks,
    the same method works for any sheaf theory context for which Braden's theorem on hyperbolic localization for $\bG_m$-action on algebraic spaces (= Theorem \ref{thm:Braden}) holds.
    For example, one can prove the same result for $\ell$-adic sheaves on Artin stacks over any field
    and also for mixed Hodge modules using the six-functor formalism \cite{Tub2}.
\end{rmk}

\begin{rmk}
    By repeating the discussion in \S \ref{ssec-nontriv-action}, we may extend Proposition \ref{prop:Braden_appendix} to stacks with non-trivial $\mathbb{G}_{\mathrm{m}}$-actions.
\end{rmk}

\section{Sign convention in shifted symplectic geometry}

In this appendix, we will explain some sign conventions regarding shifted symplectic perfect complexes.
Throughout, we fix a derived stack $\mathfrak{X}$ and $E$ be a perfect complex over $\mathfrak{X}$.

\subsection{Commutation of the duality and the shift functors}

Here we will explain our choice of the equivalence 
\[
\beta_n \colon E[n]^{\vee} \cong E^{\vee}[-n].
\]
We warn the reader that our choice is different from the standard one used in the literature e.g. in \cite[\S 1.3]{Con}.
However, our choice of the sign has an advantage that the diagrams \eqref{eq:dual_shift} and \eqref{eq:dual_shift_n}
 commute.
Also, it seems that most of the literature discussing symmetric perfect obstruction theory implicitly uses the same sign convention as ours: see \eqref{eq-symplectic_sign_-1}.

For a perfect complex $F \in \mathrm{Perf}(\mathfrak{X})$, we let 
\[
\Delta_F \colon F \to 0 \to F[1]
\]
the fibre sequence given by the rotation of $F \xrightarrow{\id} F \to 0$.
Consider a map
\[
\alpha_1 \colon (E[1])^{\vee} \cong E^{\vee}[-1]
\]
which naturally induce the following equivalence of fibre sequences:
\begin{equation}\label{eq-tauto-fibreseq-equiv}
    \begin{tikzcd}
	{\Delta_{E}^{\vee}:E[1]^{\vee}} & 0 & {E^{\vee}} \\
	{\Delta_{E^{\vee}[-1]}:E^{\vee}[-1]} & 0 & {E^{\vee}.}
	\arrow[from=1-1, to=1-2]
	\arrow["{\alpha_1}"', shift left=5, from=1-1, to=2-1]
	\arrow[from=1-2, to=1-3]
	\arrow[Rightarrow, no head, from=1-2, to=2-2]
	\arrow[Rightarrow, no head, from=1-3, to=2-3]
	\arrow[from=2-1, to=2-2]
	\arrow[from=2-2, to=2-3]
\end{tikzcd}
\end{equation}
For an integer $n$, we define
\[
\alpha_n \colon (E[n])^{\vee} \cong E^{\vee}[-n]
\]
inductively so that the following diagram commutes for any $n$ and $m$:
\[\begin{tikzcd}
	{E[n + m]^{\vee}} & {E[n]^{\vee}[-m]} \\
	{E^{\vee}[-n -m]} & {E^{\vee}[-n -m].}
	\arrow["{\alpha_m}", from=1-1, to=1-2]
	\arrow["{\alpha_{n + m}}"', from=1-1, to=2-1]
	\arrow["{\alpha_n}", from=1-2, to=2-2]
	\arrow[Rightarrow, no head, from=2-1, to=2-2]
\end{tikzcd}\]
Finally, we set
\begin{equation}\label{eq-beta-def}
\beta_n \coloneqq (-1)^n \cdot \alpha_n .
\end{equation}

Assume now that $E$ is equipped with an $k$-shifted symplectic form $\omega \in \Gamma(\mathfrak{X}, \wedge^2 E^{\vee} [k])$. Let 
\[
\phi_{E} \colon E^{\vee} \to E[k]
\]
be the map given by the composite
\[
 E \xrightarrow{\omega} (\wedge^2 E^{\vee})[k] \otimes E \cong (\wedge^2 E^{\vee} \otimes E)[k] \hookrightarrow (E^{\vee} \otimes E^{\vee} \otimes E)[k] \xrightarrow[]{\id \otimes \mathrm{ev}} E^{\vee}[k].
\]
Then an easy computation implies that the following diagram commutes:
\[\begin{tikzcd}
	E &[2cm] {E^{\vee}[k]} \\
	{E^{\vee \vee}} \\
	{(E^{\vee}[k])^{\vee}[k]} & {E^{\vee}[k].}
	\arrow["\cong"', from=1-1, to=1-2]
	\arrow["{(-1)^{k(k-1)/2 + 1} \cdot \phi_E}", from=1-1, to=1-2]
	\arrow["\cong", from=1-1, to=2-1]
	\arrow[Rightarrow, no head, from=1-2, to=3-2]
	\arrow["{\beta_k}"', from=2-1, to=3-1]
	\arrow["\cong", from=2-1, to=3-1]
	\arrow["{\phi_{E}^{\vee}[k]}", from=3-2, to=3-1]
	\arrow["\cong"', from=3-2, to=3-1]
\end{tikzcd}\]
See \cite[Lemma 0.2]{symplecticsign} for the detail.
In particular, when $k = -1$, we have
\begin{equation}\label{eq-symplectic_sign_-1}
\phi_E = \phi_E^{\vee}[-1]
\end{equation}
under the identification $\beta_{-1} \colon E^{\vee}[-1]^{\vee} \cong E^{\vee \vee}[1] \cong E[1]$.

\printbibliography

\end{document}

%% file: macros.tex
\theoremstyle{plain}
\newtheorem{maintheorem}{Theorem}

\newtheorem{thm}{Theorem}[section]
\newtheorem{prop}[thm]{Proposition}
\newtheorem{def-prop}[thm]{Definition-Proposition}

\newtheorem{lem}[thm]{Lemma}
\newtheorem{cor}[thm]{Corollary}
\newtheorem*{thm*}{Theorem}

\theoremstyle{definition}
\newtheorem{defin}[thm]{Definition}

\newtheorem*{NaC}{Notation and Convention}
\newtheorem*{ACK}{Acknowledgement}

\theoremstyle{remark}
\newtheorem{rmk}[thm]{Remark}
\newtheorem*{rmk*}{Remark}

\newtheorem{ex}[thm]{Example}

\usepackage{enumitem}
\newlist{thmlist}{enumerate}{1}
\setlist[thmlist]{label=(\roman{thmlisti}), ref=\thethm(\roman{thmlisti})}

\DeclareMathOperator{\ch}{ch}

\DeclareMathOperator{\Spec}{Spec}
\DeclareMathOperator{\id}{id}
\newcommand{\B}{\mathrm{B}}
\newcommand{\rC}{\mathrm{C}}

\newcommand{\bA}{\mathbb{A}}

\newcommand{\bC}{\mathbb{C}}
\newcommand{\C}{\bC}
\newcommand{\bD}{\mathbb{D}}
\newcommand{\bE}{\mathbb{E}}
\newcommand{\bG}{\mathbb{G}}
\newcommand{\bL}{\mathbb{L}}
\newcommand{\bN}{\mathbb{N}}
\newcommand{\bQ}{\mathbb{Q}}
\newcommand{\bR}{\mathbb{R}}

\newcommand{\bT}{\mathbb{T}}
\newcommand{\bZ}{\mathbb{Z}}
\newcommand{\Z}{\mathbb{Z}}

\DeclareMathOperator{\Hom}{Hom}
\DeclareMathOperator{\Tot}{Tot}
\DeclareMathOperator{\Coh}{Coh}

\DeclareMathOperator{\Pic}{Pic}

\DeclareMathOperator{\Sym}{Sym}

\DeclareMathOperator{\cl}{cl}

\DeclareMathOperator{\GL}{GL}

\DeclareMathOperator{\Sp}{Sp}
\DeclareMathOperator{\SL}{SL}
\DeclareMathOperator{\SO}{SO}

\DeclareMathOperator{\Perv}{Perv}

\DeclareMathOperator{\pt}{pt}

\DeclareMathOperator{\pr}{pr}
\DeclareMathOperator{\op}{op}

\DeclareMathOperator{\HC}{HC}

\DeclareMathOperator{\HH}{HH}
\DeclareMathOperator{\Perf}{Perf}

\DeclareMathOperator{\Mod}{Mod}

\DeclareMathOperator{\Map}{Map}

\DeclareMathOperator{\rank}{rank}

\DeclareMathOperator{\vir}{vir}
\DeclareMathOperator{\Crit}{Crit}
\DeclareMathOperator{\red}{red}
\DeclareMathOperator{\vdim}{vd}
\DeclareMathOperator{\Lag}{Lag}

\DeclareMathOperator{\can}{can}

\DeclareMathOperator{\cofib}{cofib}
\DeclareMathOperator{\fib}{fib}

\newcommand\cdga{\mathop{\mathbf{cdga}}\nolimits}

\newcommand\dSt{\mathop{\mathbf{dSt}}\nolimits}
\newcommand\dgcat{\mathop{\mathbf{dgcat}}\nolimits}

\newcommand{\HBM}{H^{\mathrm{BM}}}
\newcommand{\sHom}{\mathop{\mathcal{H} \mathit{om}}\nolimits}

\renewcommand{\Re}{\operatorname{Re}}

\newcommand\ddr{d_{\mathrm{dR}}}

\newcommand{\cA}{{\mathcal A}}

\newcommand{\cE}{{\mathcal E}}
\newcommand{\cF}{{\mathcal F}}
\newcommand{\cG}{{\mathcal G}}
\newcommand{\cH}{{\mathcal H}}

\newcommand{\cL}{{\mathcal L}}
\newcommand{\cM}{{\mathcal M}}
\newcommand{\cN}{{\mathcal N}}
\newcommand{\cO}{{\mathcal O}}

\newcommand{\cS}{{\mathcal S}}

\newcommand{\fC}{{\mathfrak C}}

\newcommand{\fF}{{\mathfrak F}}

\newcommand{\fg}{\mathfrak{g}}

\newcommand{\fL}{{\mathfrak L}}
\newcommand{\fl}{{\mathfrak l}}
\newcommand{\fM}{{\mathfrak M}}

\newcommand{\fp}{{\mathfrak p}}
\newcommand{\fQ}{{\mathfrak Q}}
\newcommand{\fR}{{\mathfrak R}}
\newcommand{\fS}{{\mathfrak S}}

\newcommand{\ft}{{\mathfrak t}}
\newcommand{\fU}{{\mathfrak U}}
\newcommand{\fu}{{\mathfrak u}}

\newcommand{\fX}{{\mathfrak X}}
\newcommand{\fY}{{\mathfrak Y}}
\newcommand{\fZ}{{\mathfrak Z}}

\newcommand{\ev}{\mathrm{ev}}

\newcommand{\filt}{\mathrm{filt}}
\newcommand{\gr}{\mathrm{gr}}

\newcommand{\pbcorner}[1][ul]{\save*!/#1+1.6pc/#1:(1,-1)@^{|-}\restore}
\newcommand{\pocorner}[1][dr]{\save*!/#1+1.6pc/#1:(1,-1)@^{|-}\restore}

\newcommand{\eC}{{\EuScript C}}
\newcommand{\eD}{{\EuScript D}}
\newcommand{\eS}{{\EuScript S}}
\newcommand{\GIT}{{/\! \! /}}

\newcommand\freefootnote[1]{%
  \let\thefootnote\relax%
  \footnotetext{#1}%
  \let\thefootnote\svthefootnote%
}

\newcommand{\bDR}{\mathbf{DR}}
\newcommand{\Fun}{\mathrm{Fun}}
\newcommand{\Grad}{\mathrm{Grad}}

\newcommand{\Filt}{\mathrm{Filt}}
\newcommand{\LocSys}{\mathrm{LocSys}}
\newcommand{\Vect}{\mathrm{Vect}}
\newcommand{\QCoh}{\mathrm{QCoh}}

\newcommand{\Loc}{\mathrm{Loc}}
\newcommand{\Ind}{\mathrm{Ind}}

\newcommand{\ori}{\mathrm{or}}

\newcommand{\attr}{\mathrm{attr}}
\newcommand{\colim}{\operatornamewithlimits{colim}}
\newcommand{\exact}{\mathrm{ex}}
\newcommand{\fr}{\mathrm{fr}}
\newcommand{\ind}{\mathrm{ind}}
\newcommand{\sPic}{\mathrm{Pic}^{\mathrm{gr}}}
\newcommand{\rdet}{\widehat{\mathrm{det}}}
\newcommand{\Hess}{\mathrm{Hess}}
\newcommand{\der}{\mathrm{der}}

\newcommand{\TS}{\mathrm{TS}}
\newcommand{\bZZ}{\bZ / 2 \bZ}

\newcommand{\Rep}{\mathrm{Rep}}
\newcommand{\vol}{\mathrm{vol}}
\newcommand{\spn}{\mathrm{span}}
\newcommand{\spncat}{\mathrm{spn}}
\newcommand{\triv}{\mathrm{triv}}
\newcommand{\mat}{\mathrm{mat}}

\renewcommand{\det}{\mathrm{det}}
\newcommand{\Spin}{\mathrm{Spin}}

\newcommand{\defterm}[1]{\textbf{\emph{#1}}}

\newcommand{\pur}{\mathrm{pur}}

\newcommand{\Exi}{\mathrm{Ex}_{{}^!{}_*}}
\newcommand{\Hall}{\mathrm{Hall}}